\documentclass[12pt,twoside]{amsart}
\usepackage[margin=3cm]{geometry}
\usepackage[colorlinks=false]{hyperref}
\usepackage[english]{babel}
\usepackage{graphicx,titling}
\usepackage{float}
\usepackage{amsmath,amsfonts,amssymb,amsthm}
\usepackage{lipsum}
\usepackage[T1]{fontenc}
\usepackage{fourier}
\usepackage{color}
\usepackage[latin1]{inputenc}
\usepackage{esint}
\usepackage{caption}
\usepackage{ccicons}

\makeatletter
\def\blfootnote{\xdef\@thefnmark{}\@footnotetext}
\makeatother

\newcommand\ccnote{
    \blfootnote{\copyright\,\, Dallas Albritton, Jacob Bedrossian, and Matthew Novack}
    \blfootnote{\ccLogo\, \ccAttribution\,\, Licensed under a \href{https://creativecommons.org/licenses/by/4.0/}{Creative Commons Attribution License (CC-BY)}.}
}

\usepackage[export]{adjustbox}
\numberwithin{equation}{section}
\usepackage{setspace}

\renewcommand{\leq}{\leqslant}

\renewcommand{\geq}{\geqslant}
\renewcommand{\mathbb}{\varmathbb}
\usepackage{fancyhdr}
\newtheorem{theorem}{Theorem}[section]
\newtheorem{lemma}[theorem]{Lemma}
\newtheorem{corollary}[theorem]{Corollary}
\newtheorem{proposition}[theorem]{Proposition}
\newtheorem{remark}[theorem]{Remark}
\usepackage{comment} 
\usepackage{xfrac}

\usepackage{imakeidx}
\makeindex

\newcommand{\eps}{\epsilon}

\newcommand{\norm}[1]{\left\| #1 \right\|}
\newcommand{\normm}[1]{\left| #1 \right|}
\newcommand{\abs}[1]{\left| #1 \right|}
\newcommand{\set}[1]{\left\{ #1 \right\}}
\newcommand{\brak}[1]{\left\langle #1 \right\rangle} 

\newcommand{\R}{\mathbb{R}}

\newcommand{\N}{\mathbb{N}}

\newcommand{\C}{\mathbb{C}}

\newcommand{\cM}{\mathcal{M}}

\newcommand{\dee}{d} 

\newcommand{\red}[1]{\textcolor{red}{#1}}

\DeclareMathOperator{\Div}{\mathrm{div}}
\DeclareMathOperator{\Id}{\mathrm{Id}}

\usepackage{bbm}

\newcommand{\p}{\partial}

\renewcommand{\div}{\Div}

\newtheorem*{lemma*}{Lemma}

\theoremstyle{definition}

\newcommand{\dacomment}[1]{\marginpar{\raggedright\scriptsize{\textcolor{red}{#1}}}}

\renewcommand{\leq}{\leqslant}
\renewcommand{\geq}{\geqslant}

\renewcommand{\N}{\mathbb{N}}
\renewcommand{\R}{\mathbb{R}}

\renewcommand{\eps}{\varepsilon}

\renewcommand{\subset}{\subseteq}

\renewcommand{\subset}{\subseteq}

\newcommand{\les}{\lesssim}

\newcommand{\la}{\langle}
\newcommand{\ra}{\rangle}

\renewcommand{\sf}[1]{\mathsf{#1}}

\numberwithin{equation}{section}

\address{Dallas Albritton, Department of Mathematics, University of Wisconsin--Madison, Madison, WI 53706, USA}
\email{dalbritton@wisc.edu}
\medskip
\address{Jacob Bedrossian, Department of Mathematics, University of California, Los Angeles, Los Angeles, CA 90095, USA} 
\email{jacob@math.ucla.edu}
\medskip
\address{Matthew Novack, Purdue University, West Lafayette, IN 47907-2067, USA}
\email{mdnovack@purdue.edu}

\begin{document}

\thispagestyle{empty}

\begin{minipage}{0.28\textwidth}
\begin{figure}[H]
\includegraphics[width=2.5cm,height=2.5cm,left]{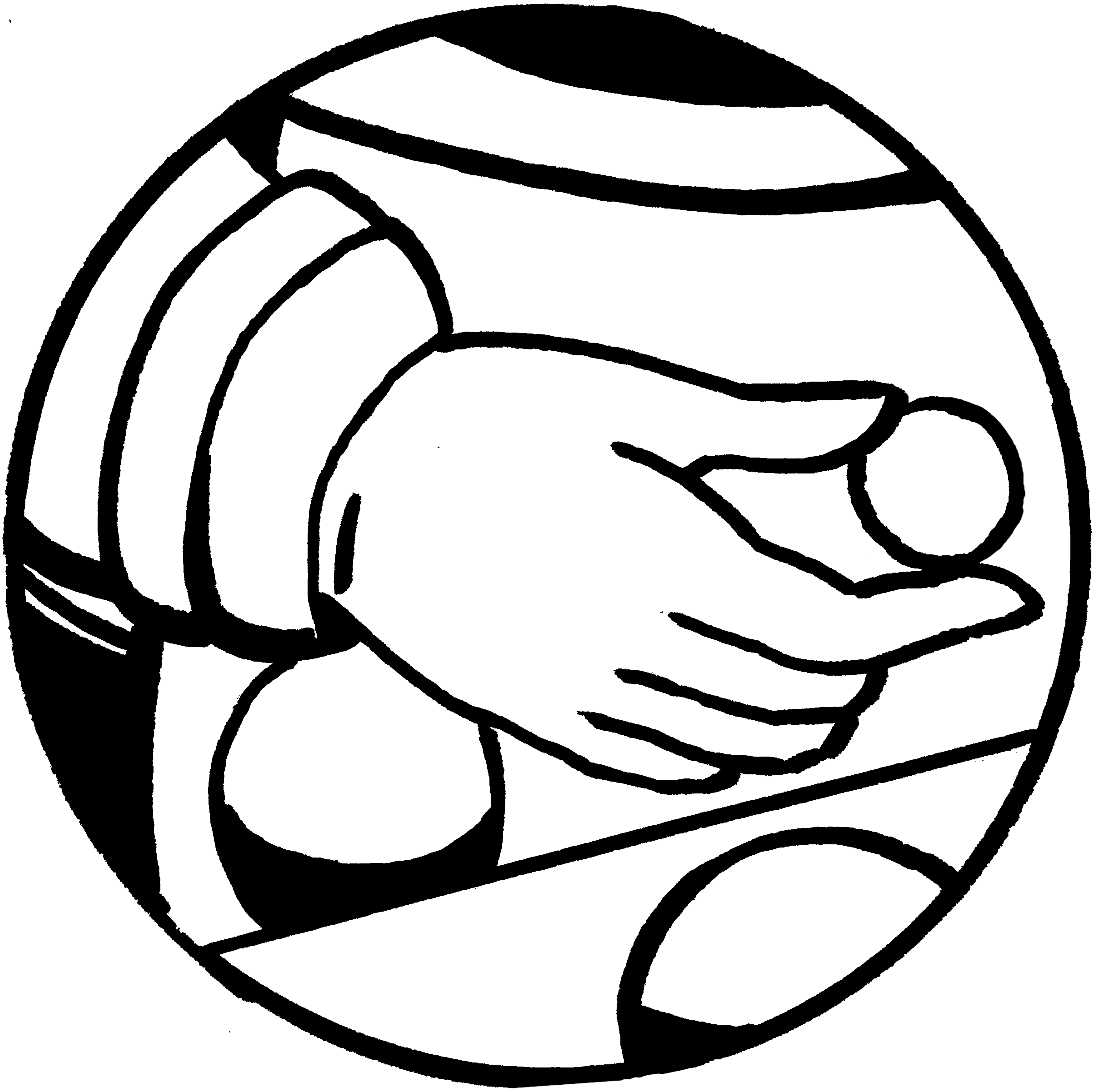}
\end{figure}
\end{minipage}
\begin{minipage}{0.7\textwidth} 
\begin{flushright}
Ars Inveniendi Analytica (2026), Paper No. 1, 87 pp.
\\
DOI 10.15781/5z5b3583
\\
ISSN: 2769-8505
\end{flushright}
\end{minipage}

\ccnote

\vspace{1cm}


\begin{center}
\begin{huge}
\textit{Kinetic shock profiles for the Landau equation}


\end{huge}
\end{center}

\vspace{1cm}


\begin{minipage}[t]{.28\textwidth}
\begin{center}
{\large{\bf{Dallas Albritton}}} \\
\vskip0.15cm
\footnotesize{University of Wisconsin--Madison}
\end{center}
\end{minipage}
\hfill
\noindent
\begin{minipage}[t]{.28\textwidth}
\begin{center}
{\large{\bf{Jacob Bedrossian}}} \\
\vskip0.15cm
\footnotesize{University of California, \\ Los Angeles}
\end{center}
\end{minipage}
\hfill
\noindent
\begin{minipage}[t]{.28\textwidth}
\begin{center}
{\large{\bf{Matthew Novack}}} \\
\vskip0.15cm
\footnotesize{Purdue University} 
\end{center}
\end{minipage}

\vspace{1cm}


\begin{center}
\noindent \em{Communicated by Laurent Desvillettes}
\end{center}
\vspace{1cm}


\noindent \textbf{Abstract.} \textit{The physical quantities in a gas should vary continuously across a shock. However, the physics inherent in the compressible Euler equations is insufficient to describe the width or structure of the shock. We demonstrate the existence of weak shock profiles to the kinetic Landau equation, that is, traveling wave solutions with Maxwellian asymptotic states whose hydrodynamic quantities satisfy the Rankine-Hugoniot conditions. These solutions serve to capture the structure of weak shocks at the kinetic level. Previous works considered only the Boltzmann equation with hard sphere and angular cut-off potentials.}
\vskip0.3cm

\setcounter{tocdepth}{2}
{\small\tableofcontents}

\vspace{0.5cm}

\section{Introduction}\label{sec:Intro}

The most classical model of gas dynamics is the compressible Euler equations
\begin{align*}
  & \partial_t \varrho + \div_x (\varrho u) = 0 \\
  & \partial_t (\varrho u) + \div_x (\varrho u \otimes u + p {\rm I}) = 0 \\
  & \partial_t \left(\varrho \left( e + \frac{\abs{u}^2}{2} \right) \right) + \div_x \left( u \varrho \left(e + \frac{1}{2}\abs{u}^2\right) + up \right) = 0 \, ,
\end{align*}
with equation of state $p = (\gamma - 1) \varrho e = \varrho \theta$ and adiabatic exponent $\gamma > 1$. As usual, $\varrho$ represents the density, $u$ the fluid velocity, $\theta$ the temperature, $e$ the internal energy density, and $p$ the pressure. 
While this system can form discontinuous shock waves in finite time, the Rankine-Hugoniot conditions correctly predict the speed of planar shocks when compared to experimental evidence. 
However, the physical quantities in a gas should vary continuously, and the physics inherent in the compressible Euler equations is insufficient to describe the width or structure of the inner `shock layer' (see early discussions in, e.g.,~\cite{stokes1848liv,Taylor1910}). 
On the other hand, the compressible Navier-Stokes equations~\eqref{eq:compressibleNSintro} admit smooth traveling wave solutions, alternatively called `viscous shocks' or `shock profiles', moving at the speed $s$ dictated by the Rankine-Hugoniot conditions applied to the asymptotic values $(\varrho_L,u_L,\theta_L)$ as $x \to -\infty$ and $(\varrho_R, u_R, \theta_R)$ as $x \to\infty$~\cite{weyl1949shock,gilbarg1951existence}.
For weak shocks, for which the jump between end states is small, the Navier-Stokes profiles are observed to be approximately consistent with experiments.  However, it has been long known that the Navier-Stokes shock profiles differ significantly from experiments in the large shock regime; see, for example, the early experiments \cite{schmidt1969,Alsmeyer1976} and the early numerical computations \cite{Liepmann1962,Bird1970} and discussions in the survey \cite{fiszdon1976plane}. Various models for larger shocks have also been studied, such as bimodal models as in \cite{MottSmith1951} or the Burnett equations \cite{Salomons1992}.  However, simulations of kinetic models such as the Boltzmann equation have yielded consistently better results \cite{Bird1970,Ohwada1993,Dodulad2013,malkov2015high,sharipov2018} for larger shocks.

Shock layer profiles at the kinetic level are traveling wave solutions which approach local Maxwellians asymptotically as $x\rightarrow \mp \infty$ with the associated hydrodynamic moments 
\begin{equation}\notag
\lim_{x \to \mp \infty} F(x,v) = \frac{ \varrho_{\sfrac{L}{R}} }{( 2\pi \theta_{\sfrac{L}{R}} )^{3/2} } e^{- \frac{ \abs{v-u_{\sfrac{L}{R}} }^2 }{ 2\theta_{\sfrac{L}{R}} } }  \, . 
\end{equation}
The work of Caflisch and Nicolaenko~\cite{CN82} provided a mathematically rigorous construction of small shock profiles for the Boltzmann equation with hard cut-off potentials (see also the earlier work~\cite{NT75}). 
Later, M{\'e}tivier and Zumbrun~\cite{MZ08} gave an alternative proof with quantitative improvements for hard spheres.

In this work, we remove the restriction to hard, cut-off potentials by constructing small shock profiles governed by the Landau collision operator.
More precisely, we construct solutions $F(t,x,v) = F(x-s_0 t,v)$ to the traveling wave Landau equation
\begin{align}
(v_1 - s_0) \partial_x F = Q(F,F)  \, , \qquad (x,v)\in \R \times \R^3 \, ,\label{def:LandauIntro}
\end{align}
where $s_0$ is the shock speed.
Note that these are solutions to the full three-dimensional equations with one-dimensional spatial dependence. 
The collision operator is defined by\index{$Q(F,G)$}
\begin{align}
Q(F_1,F_2)(v) & = \nabla_v \cdot \left( \int_{\R^3} \phi(v-u) \left( F_1(u)\nabla_v F_2(v) - F_2(v) \nabla_u F_1(u) \right) \, \dee u \right) \notag\\
\phi^{ij}(v) &= \left( \delta^{ij} - \frac{v^i v^j}{|v|^2} \right) |v|^{-1}  \, . \label{symmed}
\end{align}
The equilibria of $Q$ (that is, $\mu$ satisfying $Q(\mu,\mu) = 0$) are precisely the Maxwellians
\begin{equation}
\label{eq:parameterizationofMaxwellians}
\mu(\varrho,u,\theta) = \frac{\varrho}{(2 \pi \theta)^{3/2}} e^{-\frac{\abs{v - u}^2}{2\theta}} \, .
\end{equation}
We denote by $\mu_0$ the standard Maxwellian $\mu(1,0,1)$.
Given any distribution $F$, we can extract its mass $\varrho$, momentum $m = \varrho u$, and energy $E = \sfrac{\varrho}{2} (|u|^2 + 3\theta)$ (therefore, also its velocity $u$ and temperature $\theta$) by computing moments:
\begin{equation}\notag
(\varrho,m,E) = \int F (1,v,\sfrac{\abs{v}^2}{2}) \, dv \, .
\end{equation}
These quantities are preserved by collisions, as represented by
\begin{align*}
\int Q(F,G) (1, v, \sfrac{\abs{v}^2}{2}) \, \dee v = 0 \, , \quad \forall F, G \, .
\end{align*}
From~\eqref{def:LandauIntro}, the end states of $F$ should be Maxwellians. Moreover, we integrate the equation~\eqref{def:LandauIntro} against $(1,v,\sfrac{\abs{v}^2}{2})$ on $\R \times \R^3$ to obtain that they should satisfy the Rankine-Hugoniot conditions
\begin{equation}\notag
\begin{aligned}
s_0 \llbracket \varrho \rrbracket &= \llbracket m \rrbracket \\
s_0 \llbracket m \rrbracket &= \llbracket u_1 m + p(1,0,0) \rrbracket \\
s_0 \llbracket E \rrbracket &= \llbracket (u_1 (E+p) ) \rrbracket \, ,
\end{aligned}
\end{equation}
where $\llbracket \cdot \rrbracket = (\cdot) |_{x = -\infty} - (\cdot)|_{x = +\infty}$ is the jump in the quantity $(\cdot)$.

In fact, the small-shock regime can be studied via the classical Chapman-Enskog expansion for hydrodynamics limits, as discussed in more detail below in Section~\ref{sec:outline}. This is because small shocks with jump $\varepsilon = |\llbracket (u,m,E) \rrbracket| \ll 1$ vary on characteristic length scale $\ell = \varepsilon^{-1} \gg 1$,\index{$\varepsilon$} and therefore the Knudsen number $\textrm{Kn} = \textrm{mean free path}/\textrm{macroscopic length scale} \sim \varepsilon \ll 1$ is small. 
In this limit, the 1D compressible Navier-Stokes equations
\begin{align}
	\label{eq:compressibleNSintro}
\begin{cases}
&\partial_x ( \varrho (u_1-s_0) )(t,x) = 0 \, , \\
&\partial_x (\varrho u_1 (u_1-s_0) + p)(t,x) = \partial_x \left( {\frac 43 \mu(\theta)}  \partial_x u_1 \right) \, , \\
&\partial_x (\varrho u_i (u_1-s_0) ) = \partial_x(\mu(\theta) \partial_x u_i) \qquad\qquad\qquad (i=2,3) \, ,  \\
&\partial_x \left( \varrho \left( e + \frac{|u|^2}{2} \right) (u_1-s_0) + p u_1 \right) = \partial_x (\kappa(\theta) \partial_x \theta) \\
&\qquad \qquad \qquad + \partial_x \left( {\frac 43 \mu(\theta)} u_1 \partial_x u_1 \right) + \partial_x \left(\mu(\theta)u_2 \partial_x u_2 \right) + \partial_x \left(\mu(\theta)u_3 \partial_x u_3 \right) \, , 
\end{cases}
\end{align}
appear naturally.\footnote{More precisely, the 3D compressible Navier-Stokes equations in traveling wave form with one-dimensional spatial dependence.} The pressure is given by $p=\rho \theta$, and the viscosity coefficients $\mu$ and $\kappa$ are smooth, positive functions of $\theta > 0$; see Appendix~\ref{sec:chapmanenskogcomputations}. 
We reprove in Appendix~\ref{sec:appendixode} the classical result that for $\varepsilon$ sufficiently small, there exists a unique (up to spatial translations) solution $\left( \varrho_{\rm NS}, u_{\rm NS}, \theta_{\rm NS} \right)$ to~\eqref{eq:compressibleNSintro} with end states satisfying the Rankine-Hugoniot conditions. 
As expected based on experiments and existing results for the cut-off Boltzmann equation, the solution we construct of \eqref{def:LandauIntro}~is a small perturbation of this Navier-Stokes profile. 

Upon applying a Galilean transformation (see Remark~\ref{rmk:galileaninvariance}), we normalize the left end state $\mu_L$\index{$\mu_L$} to be the Maxwellian with $(\varrho_L,u_L,\theta_L) = (1,\varepsilon,1)$ and $0 < \varepsilon \ll 1$. It will be convenient to fix the shock speed $s_0 := \sqrt{\sfrac{5}{3}}$.\footnote{There is some freedom in the choice of normalization; typical choices are (i) fix $\mu_L = \mu_0$, so that the shock speed is $\sqrt{\sfrac{5}{3}} - \varepsilon$, or (ii) fix shock speed zero. There are some technical advantages to our convention.} \index{$s_0$} Then locally there is a unique right end-state $\mu_R$\index{$\mu_R$} with $(\varrho_R,u_R,\theta_R)(\varepsilon)$ satisfying the Rankine-Hugoniot conditions (see the discussion around~\eqref{eq:urparam}) such that the jump between end states is $O(\varepsilon)$.


With this notation, we now state the main theorem. 

\begin{theorem} \label{thm:main}
Let $N \geq 8$ and $0 \leq q_0 < 1$. Let $0 < \varepsilon \ll_{N,q_0} 1$. There exists a smooth shock profile solution $F$ to the Landau equations
\begin{equation}\notag
(v_1 - s_0)\p_x F = Q(F,F)
\end{equation}
satisfying the decomposition
\begin{equation}\notag
F = F_{\rm NS} + \mu_0^{\sfrac{1}{2}} f \, ,
\end{equation}
where $F_{\rm NS}$,\index{$F_{\rm NS}$} defined in~\eqref{def:fNS}, is a `lifting' to the kinetic setting of the unique, up to spatial translations, compressible Navier-Stokes shock profile $(\varrho_{\rm NS},u_{\rm NS},\theta_{\rm NS})$. For any $0 \leq \delta \ll_{N,q_0} 1$, there exists a constant $C_0 = C_0(N,q_0) > 0$ such that for any $\alpha \in \N$ and multi-index $\beta \in \N^3$ with $|\alpha| + |\beta| \leq N$, the remainder $f$ satisfies
\begin{equation}
	\label{eq:estimatetobesatisfied}
\left\| e^{\delta \langle \varepsilon x \rangle^{\sfrac{1}{2}}} \mu_0^{-q_0} \varepsilon^{-\alpha} \p_x^\alpha \p_v^\beta f \right\|_{L^2} \leq C_0 \varepsilon^2 \, .
\end{equation}
The solution is unique, up to spatial translations, in the class of solutions satisfying the decomposition and the estimate~\eqref{eq:estimatetobesatisfied} with $\delta = q_0 = 0$ for all $|\alpha|+|\beta| \leq 8$.
\end{theorem}

\begin{remark}
The approximation $F_{\rm NS}$ is given, to leading order, as a Maxwellian $\mu_{\rm NS} := \mu(\varrho_{\rm NS},u_{\rm NS},\theta_{\rm NS})$, though it also contains a microscopic component $G_{\rm NS}$ which is $O(\varepsilon^2)$. Moreover, we remark that $F_{\rm NS}$ decays exponentially, like $e^{-\delta \brak{\eps x}}$ for some $\delta > 0$, whereas the remainder $f$ is only guaranteed to decay \emph{stretched exponentially}, as in~\eqref{eq:estimatetobesatisfied}. This is a feature shared by the shock profiles in~\cite{CN82} for Boltzmann with hard cut-off potentials. Heuristically, it represents the physical principle that, for potentials softer than `hard spheres,' particles with larger velocities travel farther before equilibrating. Mathematically, the interpretation is that the Chapman-Enskog expansion does not yield a good approximation uniformly in $x$ and $v$. 
\end{remark}

\begin{remark}
It will be evident from the proof that uniqueness holds in the class of solutions satisfying~\eqref{eq:estimatetobesatisfied} with a small multiple of $\varepsilon$ on the right-hand side.
\end{remark}

\subsection*{Comparison with existing mathematical literature}
	\label{sec:comparisonwithmathematicallitearture}


Although the existence and stability of \emph{large} kinetic shock profiles is an important long-term goal of the theory, previous constructions and our own for the Boltzmann and Landau equations concern only small shocks.  

Nicolaenko and Thurber~\cite{NT75} established the existence of weak kinetic shock profiles for the Boltzmann equation with hard sphere potential, that is, with collision kernel
\begin{equation}
	\label{eq:hardspherekernel}
Q(F,G) =  \int_{\mathbb{R}^3}\int_{\mathbb{S}^{2}} \left| (v-v_*) \cdot \sigma \right| 
[ F(v')G(v'_*)-F(v) G(v_*) ] d v_* \, d \sigma \, ,
\end{equation}
where
\begin{align}\notag
	v'=\displaystyle \frac{v+v_*}{2}+\frac{|v-v_*|}{2}\sigma \, , \quad
		v'_*=\displaystyle \frac{v+v_*}{2}-\frac{|v-v_*|}{2}\sigma \, .
\end{align}
Subsequently, Caflisch and Nicolaenko~\cite{CN82} proved existence for the Boltzmann equation with hard cut-off potentials. The method in~\cite{NT75,CN82} (see, also,~\cite{NicolaenkoGeneral,NicolaenkoShockWave}) is a generalized Lyapunov-Schmidt procedure. For~\eqref{eq:hardspherekernel}, there exists a curve of solutions $\phi(s)$ to the generalized eigenvalue problem $(v_1 - s) \tau \phi + \mathsf{L} \phi = 0$, where $s = \sqrt{\sfrac{5}{3}} - \varepsilon$, $|\varepsilon| + |\tau| \ll 1$, and $- \mathsf{L} F = Q(\mu_0,F) + Q(F,\mu_0)$ is the linearized operator~\cite[Corollary 3.10]{NT75}. A key point is to extract the ODE
\begin{equation}\notag
\dot c_0(x) = [\varepsilon \gamma_1 + O(\varepsilon^2)] c_0(x) + [\gamma_2 + O(\varepsilon^2)] c_0(x)^2 \, ,
\end{equation}
which is an approximate Burgers equation governing the coefficient $c_0(x)$ of a well-chosen projection of the perturbation $F$ onto $\phi$, see~\cite[Equation (33)]{NT75} and~\cite[Equation (3.52)]{CN82}. Furthermore, it is verified in~\cite{NT75,CN82} that the above ODE governs, to leading order, weak compressible Navier-Stokes profiles, cf.~\eqref{eq:centermanifoldode}.

In~\cite{LiuYuMicroMacro}, Liu and Yu gave a new approach to the existence of weak shock profiles. The main idea is to prove that an approximate solution, based on the compressible Navier-Stokes solution and Chapman-Enskog procedure, is dynamically stable for perturbations with zero macroscopic contribution: $\int_\R (\varrho,m,E)|_{t = 0} \, dx = 0$. They introduce a macro-micro decomposition and an energy method which carefully extracts the dissipative effects in the macroscopic variables. Because positivity is propagated by the evolution, this method yields positivity of the shock profiles. This method was used by Yu~\cite{YuLimitShock} to analyze inviscid limits to compressible Euler solutions with initial data containing a small shock. These dynamical results were sharpened in~\cite{YuImproved,YuInitialLayers}.

In~\cite{MZ08}, M{\'e}tivier and Zumbrun proved existence of shock profiles for the case of hard spheres by working directly with the steady equations and exploiting more directly the information from the compressible Navier-Stokes equations; see also the related papers on relaxation profiles:~\cite{metivier2009existence,metivier2012existence} and~\cite{CuestaWeakShocksBGKConservation,CuestaWeakShocksBGK}.
One important aspect is that in these cases, the equation can be rewritten in the abstract form $A u' = B(u,u)$, where $A = \langle v \rangle^{-1} (v_1-s)$ and $B(u,u) = \langle v \rangle^{-1} Q(F,F)$ are both \emph{bounded} operators.
This structure is not shared by potentials `softer' than hard spheres or operators with singular potentials; the Landau operator has both issues.

In our work, we build on the general strategy of M\'etivier and Zumbrun, combining it with a more intricate treatment of the collision operator to deal with the central difficulty that the operators $A$ and $B$ are not bounded (due to both moment and derivative loss).   
The moment loss requires one to track more localization in $v$, which complicates even the qualitative proof of existence for the linearized equation (from a Galerkin truncation argument).
To overcome the moment loss, we regularize the collision operator; however this necessitates a functional analysis framework which depends on the regularization parameter itself. These key new difficulties are summarized more precisely in Remark~\ref{rmk:keynewdifficulties}.
Moreover, a significantly more complicated set of norms must be balanced against one another to close the arguments. 
The stretched exponential decay in $x$ is a result of a weak Poincar\'e-type argument which interpolates the loss of moments in the (hypo)coercivity against the available Gaussian moments in $v$.

After the present work, Wynter~\cite{wynter2024shock} gave a construction of weak shock profiles for the Boltzmann equation with a certain class of long-range potentials. His construction involves elements of~\cite{MZ08} and some of the present work. The ODE analysis in~\cite{wynter2024shock} is more explicit than that in Appendix~\ref{sec:appendixode}, which may be instructive.

Standard constructions of weak shocks for viscous systems of conservation laws~\cite{MajdaPego,Pego} proceed by center manifold analysis. In this vein, Liu and Yu~\cite{LiuYuInvariantManifolds} and Pogan and Zumbrun~\cite{PZ18,PZ19,ZumbrunInvariantReview} developed techniques for the study of invariant manifolds for the steady Boltzmann equation. Liu and Yu developed a time-asymptotic method based on the pointwise Green's functions. The method of Pogan and Zumbrun is based on a reduction to a canonical form and, in this form, resolvent estimates, which allow one to construct a bi-semigroup. This approach is evocative of, although not the same as, the analysis of~\cite{NT75}. Both approaches extract the approximate Burgers ODE for the evolution on the center manifold and yield the additional information that the wave speed $\lambda_3 = u+c$, where $c = \sqrt{\sfrac{\gamma p}{\varrho}}$ is the local sound speed, is monotone along the shock profile. Liu and Yu additionally study related half-space problems and the structure of solutions in the transition from evaporation/condensation; see also~\cite{BardosGolseSone,BernhoffGolse}.

We mention the works~\cite{LiuYangYuZhaoRarefaction,XinZengRarefaction,HuangWangYangBoltzmannContact,HuangWangWangYangRiemann,DuanYangYuRarefaction,DuanYangYuContactWaves}, among others, on contact discontinuities and rarefaction waves. These closely related works study the kinetic analogues of basic hydrodynamic flows emanating from the Riemann problem. In these cases, the solutions to Landau that are constructed are self-similar decaying solutions, rather than traveling wave solutions, and the corresponding hydrodynamic solutions admit smooth approximations within the compressible Euler equations themselves. Though by no means easy, these aspects do simplify both the construction of approximate solutions themselves and the stability argument necessary to close a perturbation argument.

Finally, existence and stability of large shocks have been demonstrated for certain discrete velocity models~\cite{CaflischBroadwell,CaflischLiuBroadwell} and BGK models~\cite{GolseShockProfilesPerthameTadmor,CuestaBulletin} (see also~\cite{BenAbdallahSchmeiserFermionic} for an application of the approach in~\cite{GolseShockProfilesPerthameTadmor} and~\cite{BouchutBGKModels} for a discussion of BGK models).

\subsection*{Acknowledgments} DA was supported by National Science Foundation Postdoctoral Fellowship Grant No.\ 2002023, NSF Grant No. 2406947, and the Office of the Vice Chancellor for Research and Graduate Education at the University of Wisconsin--Madison with funding from the Wisconsin Alumni Research Foundation. JB was supported by NSF CAREER grant DMS-1552826, NSF RNMS \#1107444 (Ki-Net), and NSF DMS-2108633. MN was supported by NSF grant DMS-1926686 while employed by the Institute for Advanced Study and by DMS-2307357.

\section{Outline} \label{sec:outline}


\subsection{The formal Chapman-Enskog expansion for the approximate solution}

We follow the general blueprint of M{\'e}tivier and Zumbrun~\cite{MZ08}. The first step in the proof of Theorem \ref{thm:main} is to build an approximate solution based on the classical Chapman-Enskog expansion for hydrodynamic limits; see Section~\ref{sss:approx}.  We first review some basic facts concerning Maxwellian equilibria of the Landau equation.

\subsubsection{Geometry of Maxwellian equilibria}\label{sss:maxwellians}
The set of all Maxwellians $\mu$ is parametrized by mass, velocity, and temperature (written here in units of specific energy, i.e., energy per unit mass) $(\varrho,u,\theta) \in \R_+ \times \R^3 \times \R_+$:
\begin{equation}
    \label{eq:mplusdef}
M_+ = \left\{ \frac{\varrho}{\left(2\pi \theta\right)^{\sfrac 32}} \exp\left( - \frac{|v-u|^2}{2\theta} \right) : (\varrho,u,\theta) \in \R_+ \times \R^3 \times \R_+ \right\} \, .
\end{equation}
To leading order, the approximate solution we construct is locally Maxwellian; that is, $F(x,\cdot) \in M_+$ for each $x \in \mathbb R$. It is hence important to be more precise about the geometry of $M_+$.

Consider the Banach space $\mathbb{H} := \{ F \in L^1_v(\R^3) : F, vF, |v|^2 F \in L^1 \}$ with an appropriate norm.
This is a natural Banach space on which integration against the collision invariants $1, v$, and $\sfrac 12 |v|^2$ is well defined.\index{collision invariants} Let\footnote{$\mathbb{H}_+$ is open in $\mathbb{H}$ (hence, it has the structure of a Banach manifold). This discussion is not particularly sensitive to the exact choices of $\mathbb{H}$ and $\mathbb{H}_+$. We could, for example, ask for exponential decay measured in $L^2$. Notice, however, that in the $L^1$-based topology we employ above, the requirement that $F > 0$ is not open, so we do not impose it.}
\begin{equation}
\mathbb{H}_+ := \left\{ F \in \mathbb{H} : \int F \, \dee v > 0 \, , \int |v|^2 F \, \dee v > 0 \right\} \, .
\end{equation}
Define the linear operator $I : \mathbb{H} \to \R \times \R^3 \times \R$ by\index{$I[\cdot]$}
\begin{equation}\label{op:i}
I[F] := \begin{pmatrix} \varrho \\ m \\ E \end{pmatrix}(F) := \int F \begin{pmatrix} 1 \\ v \\ \sfrac{1}{2} |v|^2 \end{pmatrix} \, \dee v  \, ,
\end{equation}
where $\varrho,m,E$ are the mass, momentum, and energy of the distribution $F$. 
By design, $I|_{\mathbb{H}_+} : \mathbb{H}_+ \to \R_+ \times \R^3 \times \R_+$, and on its image, the change of variables which relates mass, velocity, and temperature to mass, momentum, and energy,
\begin{equation}
    \label{eq:changeofvars}
\R_+ \times \R^3 \times \R_+ \to \R_+ \times \R^3 \times \R_+ : (\varrho,u,\theta) \mapsto (\varrho,\varrho u, \sfrac{1}{2} \varrho (|u|^2 + 3 \theta)) \, ,
\end{equation}
is a diffeomorphism (as may be checked by hand or via the inverse function theorem).
Hence, given $F \in \mathbb{H}_+$, we may write $m = \varrho u$ and $E = \varrho (\sfrac{1}{2}|u|^2 + e)$.

We can now regard $M_+ \subset \mathbb{H}_+$ as a five-dimensional embedded submanifold of $\mathbb{H}_+$, and we may obtain the tangent space at a given Maxwellian $\mu_0$,
\begin{equation}
    \label{eq:characterizationoftangentspace}
T_{\mu_0} M_+ = {\rm span}\left( \mu_0, v \mu_0, |v|^2 \mu_0 \right) \, ,
\end{equation}
by differentiating the parameterization in~\eqref{eq:mplusdef}. (Notice that the spanning vectors we have written on the right-hand side of~\eqref{eq:characterizationoftangentspace} are not, in fact, the obvious derivatives of the parameterization.)

We define a global nonlinear projection $\mu:\mathbb{H}_+ \to M_+$ as follows.
For a given $F\in\mathbb{H}_+$, we compute the hydrodynamic moments $(\varrho,m,E)$ via the linear operator $I$ defined above in \eqref{op:i}, then compute $(\varrho,u,\theta)$ by inverting the change of variables~\eqref{eq:changeofvars}, and finally form the associated Maxwellian:
 \begin{equation}\label{eq:M}
    \mu(F) := \frac{\varrho}{\left(2\pi \theta\right)^{\sfrac 32}} \exp\left( - \frac{|v-u|^2}{2\theta} \right) \, .
 \end{equation}

For a given Maxwellian ${\mu}$, we adopt the standard notation\index{$\sf{L}_{{\mu}}$}
\begin{equation}\label{def:sfl}
\sf{L}_{{\mu}} F := - Q({\mu},F) - Q(F, {\mu})
\end{equation}
for the linearized collision operator, which we consider as an unbounded operator 
on the weighted Hilbert space $L^2({\mu}^{-1} \, \dee v)$ with inner product
\begin{equation}
    \label{eq:innerproduct}
\langle F, G \rangle_{{\mu}} = \int_{\R^3} FG \, {\mu}^{-1} \, \dee v \, .
\end{equation}
This operator is well known to be symmetric and positive semi-definite on this Hilbert space; see Section~\ref{sec:propertiesofLandGamma} for this and other details regarding $\sf{L}_{{\mu}}$. Crucially, its kernel is precisely the tangent space at $M_+$:
\begin{equation}\label{kernie}
\ker \sf{L}_{{\mu}} = T_{{\mu}} M_+ \, .
\end{equation}
Define the orthogonal projection onto $\ker \sf{L}_{{\mu}}$, relative to the inner product $\langle \cdot, \cdot \rangle_{{\mu}}$ in~\eqref{eq:innerproduct}:\index{$\sf{P}_{{\mu}}$}
\begin{equation}\label{def:sfP}
\sf{P}_{{\mu}} : \mathbb{H}\rightarrow \ker \sf{L}_{{\mu}}  \, , \qquad {\ker \sf{L}_{\mu_0} := \mathbb{U}} \, ,
\end{equation}
{where $\mu_0=\mu(1,0,1)$.} We shall also use $\mathbb{V}$ to refer to the infinite-dimensional, purely microscopic range of $\sf{P}^\perp_{{\mu}_0}$.  \index{$\mathbb{V}$}\index{$\mathbb{U}$} It is not difficult to verify that, given $F \in L^2({\mu}^{-1} \, \dee v)$, $\sf{P}_{{\mu}} F$ is simply the unique element of $T_{{\mu}}$ which has the same hydrodynamic moments as $F$. Notably, $\sf{P}_{{\mu}} F = 0$ if and only if $F \in {\rm Mic} := \ker I$.

A crucial takeaway is the following: For each ${\mu}$, the space $L^2({\mu}^{-1} \, \dee v)$ is orthogonally decomposed into $T_{{\mu}} = \ker \sf{L}_{{\mu}}$ and its orthogonal complement. The ``microscopic subspace" ${\rm Mic} \cap L^2({\mu}^{-1} \, \dee v)$ does not depend on ${\mu}$ (except through the intersection with $L^2({\mu}^{-1} \, \dee v)$, which is in a non-technical sense a mild topological restriction). In particular, the images of $(F,G) \mapsto Q(F,G)$ and $\sf{L}_{{\mu}}$ are, up to mild topological restrictions, \emph{purely microscopic} and independent of ${\mu}$. Meanwhile, the ``macroscopic subspace" $T_{{\mu}}$ does depend on ${\mu}$, as does the inner product $\langle \cdot, \cdot \rangle_{{\mu}}$. 

\begin{remark}
	\label{rmk:remarkaboutderivativeofthenonlinearprojection}
One can further verify that the derivative $d{\mu}|_{{\mu}}$ of the global nonlinear projection at ${\mu}$, restricted to $L^2({\mu}^{-1} \, \dee v)$, is exactly the orthogonal projection $\sf{P}_{{\mu}}$. 
\end{remark}


\subsubsection{Construction of the approximate solution}\label{sss:approx}

We ultimately want to solve the traveling wave Landau equation
\begin{align}
(v_1-s_0) \partial_x F = Q(F,F) \label{eq:StandWaveCE}
\end{align}
for small shocks with shock speed $s_0 = \sqrt{\sfrac{5}{3}}$ with Maxwellian end states satisfying the Rankine-Hugoniot conditions: $(\varrho_L,u_L,\theta_L) = (1,\varepsilon,1)$ and $(\varrho_R,u_R,\theta_R) = (1,0,1) + O(\varepsilon)$; see~\eqref{eq:urparam}, \eqref{eq:rightevecs}, \eqref{eq:genuinenonlinearity}, and the discussion preceding Theorem~\ref{thm:main}. We expect that the shocks vary like $O(\varepsilon)$ in amplitude on the length scale $O(\varepsilon^{-1})$; this is true of the small Navier-Stokes shocks. In particular, $(v_1 -s_0) \p_x F$ may be regarded as $O(\varepsilon^2)$, so we expect from~\eqref{eq:StandWaveCE} that $F$ is close to locally Maxwellian.


Let $\mu(F)(t,x,v)$ be the local Maxwellian given by the global nonlinear projection defined above in~\eqref{eq:M}. We decompose $F = \mu(F) + G$, so that $G$ is purely microscopic (i.e., it has vanishing hydrodynamic moments). Computing the hydrodynamic moments by applying $I$ to~\eqref{eq:StandWaveCE}, we obtain (see Appendix~\ref{sec:chapmanenskogcomputations})
\begin{subequations}
\label{eq:compressibleEulerintravelingwaveform}
\begin{align}
&\partial_x ( \varrho (u_1-s_0) )(t,x) = 0 \, , \\
&\partial_x (\varrho u_1 (u_1-s_0) + p)(t,x) = - \int_{\R^3} v_1^2 \partial_x G \,dv \, , \\
&\partial_x (\varrho u_i (u_1-s_0))(t,x) = - \int_{\R^3} v_1 v_i \partial_x G \,dv \qquad\qquad (i=2,3) \, ,  \\
&\partial_x \left( \varrho \left( e + \frac{|u|^2}{2} \right) (u_1-s_0) + p u_1 \right)(t,x) = - \int_{\R^3} \frac{1}{2} v_1 |v|^2 \partial_x G \, dv \, , 
\end{align}
\end{subequations}
where the pressure is given by $p=\varrho \theta$.

We see that if $G = 0$, then these equations reduce to the compressible Euler equations.\footnote{More accurately, the 3D compressible Euler equations, with 1D spatial dependence, in traveling wave form.}
The Chapman-Enskog expansion provides a method for computing the next order corrections, which in this case yield the compressible Navier-Stokes equations. We apply the projection $\sf{P}_{\mu(F)}^\perp$,\index{$\sf{P}_{\mu(F)}^\perp$} onto the microscopic subspace, to the Landau equation~\eqref{eq:StandWaveCE}:
\begin{align}
    \sf{P}_{\mu(F)}^\perp \left[ (v_1-s_0) \partial_x \mu(F) + (v_1-s_0) \partial_x G \right] + \sf{L}_{\mu(F)} G =  Q(G,G) \, , \label{eq:formalMic}
\end{align}
where $\sf{P}_{\mu(F)}^\perp Q = Q$ and $\sf{P}_{\mu(F)}^\perp \sf{L}_{\mu(F)} = \sf{L}_{\mu(F)}$. 
Therefore, \eqref{eq:formalMic} is purely microscopic, and we can invert $\sf{L}_{\mu(F)}$, yielding the leading order approximation 
\begin{equation}
	\label{eq:Gapprox}
G \approx - \sf{L}_{\mu(F)}^{-1} \sf{P}^\perp_{\mu(F)}( v_1 \partial_x \mu(F)) \, ,
\end{equation}
valid up to $O(\varepsilon^3)$ errors, where we used that $\sf{P}^\perp_{\mu(F)} s_0 \p_x \mu(F) = 0$. We insert this approximation back into the right-hand side of~\eqref{eq:compressibleEulerintravelingwaveform} to obtain
the compressible Navier-Stokes equations:
\begin{align}\label{NSE:outline}
\begin{cases}
&\partial_x ( \varrho (u_1-s_0) )(t,x) = 0 \, , \\
&\partial_x (\varrho u_1 (u_1-s_0) + p)(t,x) = \partial_x \left( \frac 43\mu(\theta) \partial_x u_1 \right) \, , \\
&\partial_x (\varrho u_i (u_1-s_0) ) = \partial_x(\mu(\theta) \partial_x u_i) \qquad\qquad\qquad (i=2,3) \, ,  \\
&\partial_x \left( \varrho \left( e + \frac{|u|^2}{2} \right) (u_1-s_0) + p u_1 \right) = \partial_x (\kappa(\theta) \partial_x \theta) \\
&\qquad \qquad \qquad + \partial_x \left(  \frac 43 \mu(\theta)  u_1 \partial_x u_1 \right) + \partial_x \left(\mu(\theta)u_2 \partial_x u_2 \right) + \partial_x \left(\mu(\theta)u_3 \partial_x u_3 \right) \, ,
\end{cases}
\end{align}
valid up to $O(\varepsilon^3)$ errors, where the viscosity coefficients $\mu$ and $\kappa$ are smooth, positive functions of $\theta > 0$; see Appendix~\ref{sec:chapmanenskogcomputations} for details. 

In Appendix~\ref{sec:appendixode}, we reprove that for $\eps$ sufficiently small, there exists a unique (in a certain class, up to translation-in-$x$) solution $\left( \varrho_{\rm NS}, u_{\rm NS}, \theta_{\rm NS} \right)$ to this equation with the boundary conditions
\begin{align*}
\lim_{x \to \mp\infty} \left( \varrho_{\rm NS}, u_{\rm NS}, \theta_{\rm NS} \right) = \left( \varrho_{\sfrac{L}{R}}, u_{\sfrac{L}{R}}, \theta_{\sfrac{L}{R}} \right) \, . 
\end{align*}
A precise statement is contained in Proposition~\ref{pro:smallshocksgeneralhyperbolic}, and we suppose that the particular translation of $(\varrho_{\rm NS}, u_{\rm NS}, \theta_{\rm NS})$ is chosen to satisfy the estimates therein. Associated to the Navier-Stokes solution are the local Maxwellian $\mu_{\rm NS} = \mu(\varrho_{\rm NS}, u_{\rm NS}, \theta_{\rm NS})$, in the notation of~\eqref{eq:parameterizationofMaxwellians}, and the density, momentum, and energy $(\varrho_{\rm NS}, m_{\rm NS}, E_{\rm NS}) = I[\mu_{\rm NS}]$.
We now define our approximate solution $F_{\rm NS}$:\index{$F_{\rm NS}$}\index{$\sf{L}_{\mu_{\rm NS}}^{-1}$}\index{$G_{\rm NS}$}\index{$\sf{P}_{\mu_{\rm NS}}^\perp$}\index{$\sf{P}_{\mu_{\rm NS}}$}\index{$\mu_{\rm NS}$}
\begin{equation}
G_{\rm NS} = - \sf{L}_{\mu_{\rm NS}}^{-1} \sf{P}^\perp_{\mu_{\rm NS}}( v_1 \partial_x \mu_{\rm NS}) \, , \qquad  F_{\rm NS} = \mu_{\rm NS} + G_{\rm NS} \, . \label{def:fNS}
\end{equation}
Let us compute the residual associated to the approximate solution $F_{\rm NS}$:\index{$\mathcal{E}_{\rm NS}$}
\begin{equation}
\mathcal{E}_{\rm NS} := (v_1-s_0) \p_x F_{\rm NS} - Q(F_{\rm NS},F_{\rm NS}) \, .\label{def:residual}
\end{equation}
We have
\begin{equation}\label{pure}
\begin{aligned}
\mathcal{E}_{\rm NS} &= (v_1-s_0) \p_x \mu_{\rm NS} + (v_1-s_0) \p_x G_{\rm NS} + \sf{L}_{\mu_{\rm NS}} G_{\rm NS} - Q(G_{\rm NS},G_{\rm NS}) \\
&= (v_1-s_0) \p_x \mu_{\rm NS} + (v_1-s_0) \p_x G_{\rm NS} - \sf{P}^\perp_{\mu_{\rm NS}}( (v_1-s_0) \partial_x \mu_{\rm NS}) - Q(G_{\rm NS},G_{\rm NS}) \\
&= \sf{P}_{\mu_{\rm NS}}( (v_1-s_0) \p_x \mu_{\rm NS}) - (v_1-s_0) \p_x [\sf{L}_{\mu_{\rm NS}}^{-1} \sf{P}^\perp_{\mu_{\rm NS}}( v_1 \partial_x \mu_{\rm NS})] - Q(G_{\rm NS},G_{\rm NS}) \\
&= - \sf{P}^\perp_{\mu_{\rm NS}} (v_1-s_0) \p_x [\sf{L}_{\mu_{\rm NS}}^{-1} \sf{P}^\perp_{\mu_{\rm NS}}( v_1 \partial_x \mu_{\rm NS})] - Q(G_{\rm NS},G_{\rm NS}) \, ,
\end{aligned}
\end{equation}
where in the last equality we used that $I(\mu_{\rm NS})$ satisfies the compressible Navier-Stokes equations \eqref{NSE:outline}. In particular, $\mathcal{E}_{\rm NS}$ is purely microscopic.
We prove below in Section~\ref{sec:error:CE} that the error is in fact $\mathcal{O}(\eps^3)$ in the relevant norms.

\subsection{Solving for the correction}

At this point, we naturally now seek an exact solution of \eqref{eq:StandWaveCE} which is a small perturbation of $F_{\rm NS}$. In a slight abuse of notation, we shall denote the small perturbation by $F$, so that the solution of the traveling wave equation \eqref{eq:StandWaveCE} is $F_{\rm NS}+F$.  We will eventually use the contraction mapping theorem to solve for $F$. However, it will take some work to see how to set this up.   

We start by looking for solutions $F$ to equations of the type
\begin{equation}
(v_1-s_0) \p_x F + \sf{L}_{\mu_{\rm NS}}F = \mathcal{E} \, , \label{def:LinCE}
\end{equation}
where $\sf{L}_{\mu_{\rm NS}}$ is defined as in~\eqref{def:sfl} and $\mathcal{E}$ is a purely microscopic right-hand side.\index{$\mathcal{E}$}
Eventually, we will 
specifically construct solutions with 
\begin{equation}
 \mathcal{E} = Q(G_{\rm NS},F) + Q(F, G_{\rm NS}) + Q(F,F) - \mathcal{E}_{\rm NS} \, .\label{def:calE}
\end{equation}
The exact linearized operator around a kinetic shock profile will have a non-trivial kernel containing the $x$-derivative of the profile. We will see that the linearized operator on the left-hand side of~\eqref{def:LinCE} also has a one-dimensional kernel.

It will be convenient later (see subsubsections~\ref{sss:micro} and \ref{sss:macro}) to parameterize the macroscopic portion of $F$ using the five-dimensional kernel $\mathbb{U}$ of $\sf{L}_{\mu_0}$ defined in~\eqref{def:sfP}, where $\mu_0 = \mu(1,0,1)$\index{$\mu_0$} is the standard Maxwellian. 
Note that the usefulness of this parameterization relies on the fact that we are constructing small shocks, whose hydrodynamic moments are close to those of $\mu_0$. Notice also that the kernels of $\sf{P}_{\mu_0}$ and $\sf{P}_{\mu_{\rm NS}}$ are the same, as are the ranges of $\sf{P}_{\mu_0}^{\perp}$ and $\sf{P}_{\mu_{\rm NS}}^\perp$.\index{$\mathbb{U}$}\index{$\mathbb{V}$}

Since the linearized collision operator $\sf{L}_{\mu_{0}}$ has a kernel, it is natural to try some version of hypocoercivity 
in order to solve~\eqref{def:LinCE}.
We will use a variant of hypocoercivity commonly used in the context of viscous conservation laws introduced by Kawashima and Shizuta~\cite{KawashimaThesis,ShizutaKawashima1985}. It is essentially equivalent to energy methods used in~\cite{guo2006boltzmann}; see Appendix~\ref{sec:kawashima}. 
However, hypocoercivity techniques such as the Kawashima compensator method only recover control on $\p_x F$; there is a loss in the kernel of $\sf{L}_{\mu_0}$ at low spatial frequencies (see Lemma~\ref{lem:steadyestimate} below). Hence, a separate estimate needs to be made 
to control the kernel of $\sf{L}_{\mu_0}$. In particular, we will need a linearized Chapman-Enskog expansion to extract the linearized Navier-Stokes equations out of~\eqref{def:LinCE}. This will simultaneously make visible the one-dimensional kernel of $(v_1-s_0)\p_x + \sf{L}_{\mu_{\rm NS}}$. \\

At this point we will introduce enough notation to state the main linear theorem. First, it is convenient to write $F = \mu_0^{\sfrac{1}{2}} f$ and $\mathcal{E} = \mu_0^{\sfrac{1}{2}} z$ and conjugate the equation~\eqref{def:LinCE} by $\mu_0^{\sfrac{1}{2}}$. We obtain
\begin{align}
(v_1-s_0) \partial_x f + L_{\rm NS} f = z\, , \label{eq:NLSL}
\end{align}
with the by-now standard notations\index{$L$}\index{$\Gamma$}\index{$L_{\rm NS}$}
\begin{subequations}\label{ells:and:Qs}
\begin{align}
  & L f = - \mu_0^{-\sfrac 12} Q(\mu_0, f \sqrt{\mu_0}) - \mu_0^{-\sfrac 12} Q(f \sqrt{\mu_0}, \mu_0) \label{eq:just:L} \\ 
  & \Gamma[f,g] = \mu_0^{-\sfrac 12} Q(\mu_0^{\sfrac 12} f, \mu_0^{\sfrac 12} g)\label{eq:L:and:gamma} \\
  & L_{\rm NS}(f) = - \Gamma[\mu_0^{-\sfrac 12} \mu_{\rm NS},f] -  \Gamma[f,\mu_0^{-\sfrac 12} \mu_{\rm NS}] \, .
\end{align}
\end{subequations}
We will use $P$ to denote the $L^2_v$-orthogonal projection onto the five-dimensional kernel of $L = -\Gamma[\mu_0^{\sfrac 12},\cdot] - \Gamma[\cdot, \mu_0^{\sfrac 12}] $. This can be expressed in terms of $\sf{P}_{\mu_0}$, defined in~\eqref{def:sfP}, as\index{$P$}\index{$P^\perp$}
\begin{equation}\label{pproj}
P(f) := {\mu_0}^{-\sfrac 12} (\sf{P}_{\mu_0} \left(\sqrt{\mu_0} f)\right) \, , \quad P^\perp = I - P \, .
\end{equation}

We loosely follow two notational conventions: (i) Unconjugated functions ($F$, $\mathcal{E}$) will be uppercase, whereas conjugated functions ($f, z$) will be lowercase, and (ii) unconjugated operators ($\mathsf{L}$, $\mathsf{P}$) will be sans serif, and conjugated operators ($L$, $P$) will be standard.\index{conventions for conjugation}

We now introduce the natural weighted norms for this problem.  We first define the Gaussian weight\index{${\rm w}(\ell, q)$}
\begin{align}
{\rm w }(\ell,q) = \brak{v}^{-\ell} \exp\left(\frac{q}{4}\brak{v}^2\right) \, . \label{def:double:u}  
\end{align}
For the majority of the proof the parameter $q$ will either be $0$ or any fixed $0<q_0<1$; 
hence we suppress the dependence of $\rm w$ on $q$ in most steps. 

We will use the following velocity norm with polynomial weights:  \index{$\normm{\cdot}_{\sigma,\ell,q}$}
\begin{align}
\abs{f}_{\sigma,\ell,q}^2 = \int_{\mathbb R^3} \sigma^{ij} {\rm w}^2(\ell,q) \left(\partial_i f \partial_j f  + v_i v_j f^2 \right) \dee v \, ,  \label{eq:monday:night}
\end{align}
where $\sigma^{ij}(v) = \phi^{ij} \ast \mu_0$.\index{$\sigma$} If $q = 0$, we will abbreviate the norm by $|f|_{\sigma,\ell}$.\index{$|\cdot|_{\sigma,\ell}$} If furthermore $\ell=0$, we simply write $|f|_{\sigma}$.\index{$|\cdot|_\sigma$} It will be convenient to have function spaces corresponding to the norms. Namely, $\mathcal{H}^1_{\ell}$\index{$\mathcal{H}^1_{\ell}$} and $\mathcal{H}^1_{\ell, \rm w}$\index{$\mathcal{H}^1_{\ell, \rm w}$} are Banach spaces consisting precisely of $H^1_{\rm loc}(\R^3)$ functions satisfying\footnote{By a cut-off and mollification procedure, these spaces are equal to the closure of test functions in their norms.}
\begin{subequations}\label{eq:sigma:2}
\begin{align}
  \norm{f}_{\mathcal{H}^1_{{\ell}}}^2 := \int_{\R^3} \langle v \rangle^{-2\ell} \sigma^{ij} \left(\partial_{v_i} f \partial_{v_j} f + v_i v_j f^2 \right)  \, \dee v < +\infty \, , \ \\
    \norm{f}_{\mathcal{H}^1_{\ell,\rm{w}}}^2 := \int_{\R^3} {\rm w}^2(\ell,q_0) \sigma^{ij} \left(\partial_{v_i} f \partial_{v_j} f + v_i v_j f^2 \right) \, \dee v < + \infty \, ,  
\end{align}
\end{subequations}
respectively. When $\ell = 0$, we omit it. The relevance of $\mathcal{H}^1$ and its norm $\| \cdot \|_{\mathcal{H}^1} := |\cdot|_\sigma$ is in the classical lower bound (see \cite{G02} and Section~\ref{sec:error:CE} below)
\begin{align}\label{eq:wed:morning}
\brak{f,Lf} \gtrsim \left| P^\perp f \right|_\sigma^2 \, , \quad \forall f \in \mathcal{H}^1 \, .
\end{align}
For $\beta \in \mathbb N^3$ and $\alpha \in \mathbb N$,\index{$\partial_\beta^\alpha$} we use the standard notation
\begin{equation}\notag
\partial_\beta^{\alpha} := \partial_x^\alpha \partial_v^\beta. 
\end{equation}
We similarly define Sobolev spaces and norms associated to higher derivatives (note the tradeoff between $v$ derivatives and $v$ weights, as is commonly used for the Landau equations, for example, in~\cite{G02}):\index{$\mathcal{H}^{k+1}$}\index{$\mathcal{H}^{k+1}_{\rm w}$}
\begin{align}
\| f \|_{\mathcal{H}^{k+1}} := \sum_{|\beta| \leq k} \| \p_\beta f \|_{\mathcal{H}^1_{|\beta|}} \, \quad\quad \| f \|_{\mathcal{H}^{k+1}_{\rm w}} := \sum_{|\beta|\leq k} \| \p_\beta f \|_{\mathcal{H}^1_{|\beta|,\rm w}} \, .  \notag
\end{align}
We further define appropriate function spaces for right-hand sides: $\mathcal{H}^{-1} = (\mathcal{H}^1)^*$ and\index{$(\mathcal{H}^{1}_\ell)^*$}\index{$(\mathcal{H}^{1}_{\ell, \rm w})^*$}
\begin{equation}\label{def:dual:spaces}
\| z \|_{\mathcal{H}^{-1}_\ell} := \| \langle v \rangle^{-\ell} z \|_{\mathcal{H}^{-1}} \, , \qquad \| z \|_{\mathcal{H}^{-1}_{\ell,\rm w}} := \| {\rm w}(\ell,q_0) z \|_{\mathcal{H}^{-1}} \, .
\end{equation}
Notably, $\| z \|_{\mathcal{H}^{-1}_\ell} \lesssim \| z \|_{\mathcal{H}^{-1}_{\ell,{\rm w}}}$; we prove this in section 3, Lemma~\ref{weddy:night}. \index{$\mathcal{H}^{-1}_{\ell, \rm w}$}\index{$\mathcal{H}^{-1}$}\index{$\mathcal{H}^{-1}_{\ell}$}
For functions $U$ depending on $x$ with values in the finite-dimensional subspace ${\ker}\, {\sf{L}}$ (respectively $\ker L$), we express $U$ in the basis $(1,v,|v|^2/2) \mu_0$ (respectively $(1,v,|v|^2/2) \mu_0^{\sfrac{1}{2}})$) and consider the finite-dimensional norm as the $\ell^2$ norm of the coefficients, for the sake of definiteness. Then we define the Sobolev norm\index{$H^N_\varepsilon$}
\begin{align}\label{xnorm}
\norm{U}_{H^N_\eps} = \sum_{j \leq N} \eps^{-j} \norm{\partial^j U}_{L^2_x} \, .
\end{align}
We further define
\begin{equation}\notag
\| f \|_{H^N_\varepsilon \mathcal{H}^{\pm 1}_\ell} = \sum_{j\leq N} \varepsilon^{-j} \| \partial^j f \|_{\mathcal{H}^{\pm 1}_\ell}
\end{equation}
and analogous function spaces.
Finally, we define the Sobolev spaces and norms used for the full set of $x,v$ variables via\index{$\mathbb X^{N}_{\eps}$}\index{$\mathbb X^{N}_{\eps, \rm w}$}
\begin{align}\notag
\| f \|_{\mathbb X^{N}_{\eps}}  := \sum_{|\beta| + \abs{\alpha} \leq N} \eps^{-\alpha} \| \p_\beta^\alpha f \|_{L^2 \mathcal{H}^1_{|\beta|}} \, , \qquad
\| f \|_{\mathbb X^{N}_{\eps,\rm{w}}}  := \sum_{|\beta| + \abs{\alpha} \leq N} \eps^{-\alpha} \| \p_\beta^\alpha f \|_{L^2\mathcal{H}^1_{|\beta|, \rm w}} \, , 
\end{align}
and their dual-space variants\index{$\mathbb Y^{N}_{\eps}$}\index{$\mathbb Y^{N}_{\eps, \rm w}$}
\begin{align}\notag
\| f \|_{\mathbb Y^{N}_{\eps}}  := \sum_{|\beta| + \abs{\alpha} \leq N} \eps^{-\alpha} \| \p_\beta^\alpha f \|_{L^2\mathcal{H}^{-1}_{|\beta|}} \, , \qquad
\| f \|_{\mathbb Y^{N}_{\eps,\rm{w}}}  := \sum_{|\beta| + \abs{\alpha} \leq N} \eps^{-\alpha} \| \p_\beta^\alpha f \|_{L^2 \mathcal{H}^{-1}_{|\beta|, \rm w}} \, . 
\end{align}
As a corollary of Lemma~\ref{weddy:night} we have $\| f \|_{\mathbb Y^{N}_{\eps}} \lesssim \| f \|_{\mathbb Y^{N}_{\eps,\rm w}}$.

The main linear result is the following. For the proof of this result as a consequence of the variety of linear estimates we prove in section~\ref{secc:micro:linear}, we refer to Remark~\ref{rem:implications}.

\begin{proposition}[Linear estimate]\label{lem:MainLinear}
Let $0 < \varepsilon,\delta \ll 1$ and $N \geq 0$. Let $z \in H^2 \mathcal{H}^{-1}_{-10} \cap \mathbb{Y}^N_{\eps,\rm w}$ be purely microscopic and
satisfy
\begin{align*}
\norm{e^{\delta \brak{\eps x}^{\sfrac 12}} z}_{\mathbb{Y}^N_{\eps}} < \infty  \, .
\end{align*}
Then there exists a unique solution $f$ to~\eqref{eq:NLSL} (that is, $F= \mu_0^{\sfrac 12} f$ solves~\eqref{def:LinCE}) with $f\in H^2 \mathcal{H}^1_{-10} \cap \mathbb{X}^N_{\eps,\rm w}$
satisfying the linear constraint\index{$\ell_\eps$}
\begin{align*}
\ell_\eps( Pf(0) ) = 0 \, ,
\end{align*}
defined in Remark~\ref{rmk:theconstraint} in Appendix~\ref{sec:appendixode}. 
This solution obeys the quantitative estimate
\begin{align*}
\varepsilon \norm{f}_{\mathbb M^N_\eps} \lesssim \norm{z}_{\mathbb{Y}^N_{\eps, \rm w}} + \norm{e^{\delta \brak{\eps x}^{\sfrac 12}} z}_{\mathbb{Y}^N_{\eps}} \, , 
\end{align*}
where\index{$\mathbb{M}_{\varepsilon}^N$}
\begin{align*}
\norm{f}_{\mathbb M^N_\eps} := \norm{e^{\delta \brak{\eps x}^{\sfrac 12} } P f}_{H^N_\eps} + \norm{e^{\delta \brak{\eps x}^{\sfrac 12} } P^{\perp} f}_{\mathbb X^N_\eps} + \norm{P^{\perp} f}_{\mathbb X^N_{\eps,\rm w}}  \, .
\end{align*}
\end{proposition}
\begin{remark}
We use only the $x$-norm defined in~\eqref{xnorm} to measure $Pf$ for the following reason.  Since $P$ is the projection operator defined in~\eqref{pproj}, which for every $x$ projects identically in $v$ onto a five-dimensional space, all $v$-norms are therefore equivalent.  We note also that the estimate could be improved by changing the definition of $\mathbb{M}^N_\eps$ to include an $\eps^{-1}$ prefactor in front of the microscopic terms, i.e. those with $P^\perp f$.  In other words, the $\eps$ loss would only apply to the macroscopic portion $Pf$; see Remark~\ref{rem:improved}.
\end{remark}

We outline the main ideas behind Proposition~\ref{lem:MainLinear} below, which include a linearized Chapman-Enskog expansion, macroscopic and microscopic \emph{a priori} linear estimates, and a fixed point argument which uses the linear estimates to prove Theorem~\ref{thm:main}. 

\subsubsection{Linearized Chapman-Enskog expansion}\label{sss:macro}

It is natural to seek a solution to the linearized Landau equation~\eqref{def:LinCE} within the framework of a linearized Chapman-Enskog expansion.  However, due to the moment loss inherent in the Landau operator~\eqref{symmed}, we must construct solutions to~\eqref{def:LinCE} by replacing the collision operator $Q$ with a regularized operator $Q_\kappa$ where $\kappa\in(0,1]$; in the case $\kappa=0$, $Q_0=Q$.\index{$Q_0$}\index{$Q_\kappa$}  The precise nature of this regularization is discussed further in subsubsection~\ref{sec:existencediscussion}, where we outline the existence argument.  For now we simply note that the regularized operators $Q_\kappa$ and $\sf{L}_{\mu_{\rm NS}}^{\kappa}$ enjoy all the properties necessary to perform a linearized Chapman-Enskog expansion; see subsection~\ref{sec:propertiesofLandGamma} for a delineation of these properties.  We now outline this expansion, using $\sf{L}_{\mu_{\rm NS}}^\kappa$ for $\kappa\in[0,1]$.\index{$\sf{L}^\kappa_{\mu_{\rm NS}}$}

We seek a solution to the equation 
\begin{equation}\label{def:LinCE:reg}
(v_1-s_0) \partial_x F + \sf{L}_{\mu_{\rm NS}}^\kappa F = \mathcal{E} 
\end{equation}
of the form
\begin{equation}
F = \sf{P}_{\mu_{\rm NS}} F + V \, ,\label{def:FLinCH}
\end{equation}
where $V$ is purely microscopic.  
We recognize~\eqref{def:FLinCH} as the linearization of the decomposition $F = \mu(F) + G$ around $F_{\rm NS}$, since $d\mu(F)(\cdot) = \sf{P}_{\mu(F)}(\cdot)$ according to Remark~\ref{rmk:remarkaboutderivativeofthenonlinearprojection}. We introduce $U = \sf{P}_{\mu_0} F$ and observe
\begin{equation}
\sf{P}_{\mu_{\rm NS}} F = \sf{P}_{\mu_{\rm NS}} U \,.  \label{eq:PFNSeqPFNSU}
\end{equation}

Cosmetically, our analysis will look different from that of Section~\ref{sss:approx} because we parameterize the macroscopic part~\eqref{eq:PFNSeqPFNSU} of $F$ by $U \in \mathbb{U}$ instead of by $(\varrho,m,E) \in \R^5$, but the two approaches are equivalent by the linear change of coordinates $U \overset{I}{\mapsto} (\varrho,m,E)$. For example, the nonlinear equations~\eqref{eq:compressibleEulerintravelingwaveform} in Section~\ref{sss:approx}, obtained using the decomposition $F = \mu(F) + G$, are equivalent to
\begin{equation}
	\label{eq:heyitwasequivalentallalong}
\p_x \underbrace{\sf{P}_{\mu_0} (v_1 - s_0) \mu(U) }_{=: H(U)-s_0U} = - \p_x \sf{P}_{\mu_0} (v_1-s_0) G \, ,
\end{equation}
obtained by applying $\sf{P}_{\mu_0}$ to the Landau equation~\eqref{eq:StandWaveCE} and substituting $\mu(U) = \mu(F)$. Notably, $H(U)$ is the Euler flux function written in the $U$ variable.

We now extract the linearized Navier-Stokes equations from the linearized Landau equation~\eqref{def:LinCE}. Applying $\sf{P}_{\mu_0}$ to \eqref{def:LinCE}, we have
\begin{equation}
    \label{eq:firsteq}
\sf{P}_{\mu_0} [(v_1 - s_0) \p_x \sf{P}_{\mu_{\rm NS}} F + v_1 \p_x V ] = 0 \,,
\end{equation}
where we have used that $\mathcal{E}$, $Q$, and $s_0 \p_x V$ are purely microscopic. Again, applying $\sf{P}_{\mu_0}$ is equivalent to applying $I$, as we did in~\eqref{eq:compressibleEulerintravelingwaveform}, but it will give an ODE in the range $\mathbb{U}$ of $\mathsf{P}_{\mu_0}$. 
We will find a leading order approximation, analogous to~\eqref{eq:Gapprox}, for $V$. Applying $\mathsf{P}^\perp_{\mu_{\rm NS}}$ to~\eqref{def:LinCE}, we have
\begin{equation}\notag
\sf{P}_{\mu_{\rm NS}}^\perp [(v_1-s_0) \p_x \sf{P}_{\mu_{\rm NS}} F + (v_1-s_0) \p_x V ] + \sf{L}_{\mu_{\rm NS}}^\kappa V = \mathcal{E} \, ,
\end{equation}
since $\mathsf{L}^\kappa_{\mu_{\rm NS}} \mathsf{P}_{\mu_{\rm NS}} F = 0$. Using again that $\mathcal{E}$ is purely microscopic, we therefore have 
\begin{equation}\notag
V= - \left(\sf{L}_{\mu_{\rm NS}}^\kappa\right)^{-1} \left[\sf{P}_{\mu_{\rm NS}}^\perp [ (v_1-s_0) \p_x \sf{P}_{\mu_{\rm NS}} F + (v_1-s_0) \p_x V] - \mathcal{E}\right] \, .
\end{equation}
Hence, plugging back into~\eqref{eq:firsteq}, we find
\begin{equation}\label{lin:ce}
\sf{P}_{\mu_0} \left[(v_1-s_0) \p_x \sf{P}_{\mu_{\rm NS}} F - v_1 \p_x \left[ \left(\sf{L}_{\mu_{\rm NS}}^\kappa\right)^{-1} \left(\sf{P}_{\mu_{\rm NS}}^\perp [(v_1-s_0) \p_x \sf{P}_{\mu_{\rm NS}} F + (v_1-s_0) \p_x V] - \mathcal{E}]\right) \right] \right] = 0 \, .
\end{equation}
At this point, we recall~\eqref{eq:PFNSeqPFNSU} and 
seek an ODE for $U = \sf{P}_{\mu_0} F$. First, we recognize the gradient of the Euler flux function $H(U)$ defined in~\eqref{eq:heyitwasequivalentallalong}: \index{$\nabla H$}
\begin{equation}
\label{eq:thingtosee}
\sf{P}_{\mu_0} [(v_1-s_0) \sf{P}_{\mu_{\rm NS}} U] = [(\nabla H)(U_{\rm NS}) - s_0\Id] U \, ,
\end{equation}
where $U_{\rm NS} := \sf{P}_{\mu_0} F_{\rm NS}$.\index{$U_{\rm NS}$} To obtain~\eqref{eq:thingtosee}, one again utilizes Remark~\ref{rmk:remarkaboutderivativeofthenonlinearprojection} that the derivative $d\mu|_F$ of the nonlinear projection at $F$ is the linear projection $\sf{P}_{\mu(F)}$.  Next, we seek an expression in $\p_x U$. We have
\begin{equation}\notag
\p_x \sf{P}_{\mu_{\rm NS}} F = (\p_x \sf{P}_{\mu_{\rm NS}}) U + \sf{P}_{\mu_{\rm NS}} \p_x U \, .
\end{equation}
The commutator term, $\p_x \sf{P}_{\mu_{\rm NS}} U = [\partial_x,\sf{P}_{\mu_{\rm NS}}] U$, will eventually be seen to be perturbative in our arguments. 
Define\index{$B_\kappa$}\index{$B$}
\begin{equation}\label{BB}
\sf{P}_{\mu_0} \left[ v_1 \left( \left(\sf{L}_{\mu_{\rm NS}}^\kappa\right)^{-1} (\sf{P}_{\mu_{\rm NS}}^\perp [ v_1 \sf{P}_{\mu_{\rm NS}}])\right) \right] =: B_\kappa(U_{\rm NS}) \, ,
\end{equation}
which is an $x$-dependent linear operator (that is, identifiable with a matrix) on $\mathbb{U}$.  When $\kappa=0$, we shall sometimes abbreviate $B_0$ by simply $B$. The equation \eqref{lin:ce} becomes
\begin{equation}\label{becomes}
\p_x ([(\nabla H)(U_{\rm NS}) - s_0\Id] U) - \p_x (B_\kappa(U_{\rm NS}) \p_x U) = \p_x G \, ,
\end{equation}
where\index{$G$}
\begin{equation}\label{gdef}
G = \sf{P}_{\mu_0}  \left[  v_1 \left( \left(\sf{L}_{\mu_{\rm NS}}^\kappa\right)^{-1} (\sf{P}_{\mu_{\rm NS}}^\perp [ (v_1-s_0) (\p_x \sf{P}_{\mu_{\rm NS}}) U + (v_1-s_0) \p_x V] - \mathcal{E}) \right)  \right]
\end{equation}
refers now to our forcing and error term.
Moreover, by the linear coordinate change\footnote{This is the ``constant linear coordinate change" refered to in~\cite[p. 683]{MZ08}.} $U \mapsto I[U]$, we may rewrite the equations in the variables $(\varrho,m,E) \in \R \times \R^3 \times \R$ instead of $U \in \mathbb{U}$, and the equations become a variant of the linearized compressible Navier-Stokes equations; specifically, without the perturbative $\p_x (B_\kappa'(U_{\rm NS}) U \p_x U_{\rm NS})$.  Finally, we shall actually fix the operator $B_0$ on the left-hand side, and estimate the error coming from $B_\kappa-B_0$ as part of the right-hand side.

\begin{remark}[Key new difficulties]
	\label{rmk:keynewdifficulties}
We now draw special attention to the term
\begin{equation}
	\label{eq:hugeproblem}
\sf{P}_{\mu_0}  \left[  v_1 \left( \left(\sf{L}_{\mu_{\rm NS}} \right)^{-1} (\sf{P}_{\mu_{\rm NS}}^\perp [  (v_1-s_0) \p_x V ] \right)  \right]
\end{equation}
in~\eqref{gdef}, where we set the regularization parameter $\kappa=0$ in $\sf{L}_{\mu_{\rm NS}}^{-1}$. This represents an error term, which must be controlled somehow. $V$ has essentially the same estimates as $\sf{P}^\perp F$, as the two differ by $(\sf{P}_{\mu_{\rm NS}} - \sf{P}) F$. The term $v_1 \p_x V$ \emph{inside} $\sf{L}_{\mu_{\rm NS}}^{-1}$ therefore presents a problem. The linearized Landau operator on $\mathcal{H}^1$ is invertible when the microscopic right-hand side belongs to $\mathcal{H}^{-1}$; this simply does not hold when $\p_x V$ is only estimated in the natural energy space $\mathcal{H}^1$, since (1) $L_{\mu_{\rm NS}}$ loses moments when inverted ($\mathcal{H}^{-1}$ functions may have better localization than $\mathcal{H}^1$ functions), and (2) multiplication by $v_1$ represents a \emph{further} moment loss. Any moment loss will be forgiven once $\sf{P}_{\mu_0}$ is applied on the outside, but it is necessary to make sense of the inside of~\eqref{eq:hugeproblem} first. This leads to a proliferation of technicalities. First, the basic energy estimates should track enough $v$-moments to make sense of~\eqref{eq:hugeproblem}; specifically, we track $\langle v \rangle^{10} f \in H^2 \mathcal{H}^1$ instead of $f \in H^2 \mathcal{H}^1$, which was used in~\cite{MZ08}. At the level of \emph{a priori} estimates, this is enough to close. However, it is also necessary to prove \emph{existence} for~\eqref{def:LinCE:reg}, obtained by a Galerkin approximation procedure in $v$ which moreover exploits the linearized Chapman-Enskog expansion. This procedure is inherently $L^2_v$-based and not well adapted to tracking moments. Our way around this is to prove existence at the level of a regularized collision operator, namely, Landau plus $\kappa$ times a `hard sphere' Landau whose inverse, like that of the Boltzmann hard sphere kernel, \emph{gains} one moment, exactly enough to compensate for the multiplication by $v_1$ inside $(L_{\mu_{\rm NS}}^\kappa)^{-1} v_1 \p_x V$. Existence with the Landau operator is then obtained when $\kappa \to 0^+$.
\end{remark}

\subsubsection{Macroscopic a priori linear estimates} \label{sss:macrolinaprior}

The linear stability\footnote{Specifically, stability within the class of traveling wave solutions, not dynamical stability.} of the traveling wave solutions to the Navier-Stokes equations will be used to estimate the macroscopic variables.
This is only necessary for low wave-numbers in $x$; derivatives of $U$ will be handled at the same time as the microscopic part. 
\begin{proposition}[Macroscopic estimate]
\label{pro:macroscopicestimate}
Let $0 < \varepsilon \ll 1$. In the above notation, for all $G \in L^2 \cap C(\R)$ and $d \in \R$, the linearized macroscopic equation
\begin{equation}\notag
B(U_{\rm NS}) U_x = [(\nabla H)(U_{\rm NS}) - s_0 \Id ] U + G
\end{equation}
has a unique solution $U \in L^2 \cap C(\R)$ satisfying a one-dimensional constraint
\begin{equation}\notag
\ell_\varepsilon(U(0)) = d \, .
\end{equation}
The solution obeys the estimate
\begin{equation}
	\label{eq:Ufromlinearmacroestimate}
\| U \|_{L^2} \lesssim \varepsilon^{-1} \| G \|_{L^2} + |d| \, .
\end{equation}
\end{proposition}

The proof of Proposition~\ref{pro:macroscopicestimate} is involved but essentially contained in \cite{MZ08}. The proof requires detailed knowledge of the shock profile which is obtained from the center manifold analysis in its construction. For completeness, we revisit both the construction of the shock profile and subsequent proof of Proposition~\ref{pro:macroscopicestimate} in Appendix~\ref{sec:appendixode}. We draw special attention to the prefactor $\varepsilon^{-1}$ in~\eqref{eq:Ufromlinearmacroestimate}.

The constraint $\ell_\varepsilon(U(0)) = d$, defined in Remark~\ref{rmk:theconstraint}, can be understood in the following way. 
Suppose we invert $B(U_{\rm NS})$ and view~\eqref{eq:linearizedODE} as a linear ODE $\dot U = m(x) U + \tilde{G}$. Then the unstable subspace $E^u_-(x)$ flowing from $x = -\infty$ and stable subspace $E^s_+(x)$ flowing from $x = +\infty$ have a one-dimensional intersection at $x=0$ which is fixed by the constraint. Technically, $B(U_{\rm NS})$ is not invertible because the density $\varrho$ is not dissipated, and the workaround is explained in Appendix~\ref{sec:appendixode}.

\subsubsection{Microscopic a priori linear estimates}\label{sss:micro}
As explained in the beginning of subsubsection~\ref{sss:macro}, we must treat the linearized equation~\eqref{def:LinCE:reg}, which includes the regularized linear operator $Q_\kappa$ with regularization parameter $\kappa\in[0,1]$.  We view $L^\kappa_{\rm NS}$ (defined analogously to~\eqref{ells:and:Qs}, but with a regularized operator $Q_\kappa$ replacing $Q$) as a perturbation of $L^\kappa$, yielding the linearized, conjugated, regularized equation \index{$Q_\kappa$}\index{$L^\kappa$}\index{$L^\kappa_{\rm NS}$}
\begin{equation}\label{def:LinCE:reg:conj}
 (v_1-s_0)\p_x f + L^\kappa f = \Gamma_\kappa [\mu_0^{-\sfrac{1}{2}} (\mu_{\rm NS} - \mu_0),f] + \Gamma_\kappa [f,\mu_0^{-\sfrac{1}{2}} (\mu_{\rm NS} - \mu_0)] + z \, .
 \end{equation}
Solutions to this equation are estimated in three steps. First is the standard $L^2$ energy estimate
\begin{equation}\notag
\| P^\perp f \|_{L^2 \mathcal{H}^1} \lesssim \varepsilon \| Pf \|_{L^2} + \| z \|_{L^2 \mathcal{H}^{-1}} \, ,
\end{equation}
which fails to be coercive in the kernel of $L$. The second is a `twisted' energy estimate, involving the Kawashima compensator $K$, which is coercive on $\p_x P f$. Together, the standard and twisted estimates yield a `baseline' estimate
\begin{equation}\label{eq:basicmicroestintro}
\left\| \p_x f \right\|_{L^2 \mathcal{H}^1} + \left\| P^\perp f \right\|_{L^2 \mathcal{H}^1} \les \varepsilon \left\| Pf \right\|_{L^2} + \left\| z, \p_x z \right\|_{L^2 \mathcal{H}^{-1}} \, ;
\end{equation}
see Lemma~\ref{lem:basicmicroest}.   The final estimate is an analogue of~\eqref{eq:basicmicroestintro}, but with polynomial weights in velocity.  Roughly speaking, assuming that $\langle v \rangle^{10} z \in L^2 \mathcal{H}^{-1}$ provides bounds for $\langle v \rangle^{10} P^\perp f$ and $\langle v \rangle^{10} \partial_x f$ in $L^2\mathcal{H}^1$. The precise energy estimates are described in Section~\ref{sec:linearestimatesbaseline}, and a version of the Kawashima compensator theory is developed in Appendix~\ref{sec:kawashima}.

Notice that the microscopic estimate is entirely predicated on $\varepsilon \| Pf \|_{L^2}$, whereas the macroscopic estimate depends on $P^\perp f$ through the right-hand side $G$ defined in~\eqref{gdef}. The crucial point in `closing the loop' is that \emph{$G$ depends on $P^\perp f$ through $\p_x P^\perp f$}, which gains a further power of $\varepsilon$. We then, \emph{roughly speaking}, have $\| Pf \| \lesssim \varepsilon^{-1} \| \p_x P^\perp f \| + \| \mathcal{E} \| \lesssim \varepsilon \| Pf \| + \| \mathcal{E} \|$. This is made precise for solutions with $\langle v \rangle^{10} f \in H^2 \mathcal{H}^1$ in Proposition~\ref{pro:basicexistence}.

The unique solutions are subsequently bootstrapped to (i) high regularity in $x$ and $v$, (ii) high regularity in $x$ and $v$ with Gaussian weights in $v$, and (iii) stretched exponential decay in~$x$, unweighted in~$v$.  In particular, we prove in Lemmas~\ref{lem:bs:1} and~\ref{lem:bs:2} that
\begin{equation}\label{est:boot:intro}
\| f \|_{\mathbb{X}^N_{\varepsilon, \rm w}} + \left\| e^{\delta \langle \varepsilon x \rangle^{\sfrac{1}{2}}} f \right\|_{\mathbb{X}^N_\varepsilon} \lesssim \varepsilon^{-1} \left( \| z \|_{\mathbb{Y}^N_{\varepsilon, \rm w}} + \| e^{\delta \langle \varepsilon x \rangle^{\sfrac{1}{2}}} z \|_{\mathbb{Y}^N_\varepsilon} \right) \, .	
\end{equation}

\subsubsection{Existence}
	\label{sec:existencediscussion}

We have now sketched the \emph{a priori} estimates, but it is necessary to demonstrate that solutions to~\eqref{def:LinCE} exist. Somewhat surprisingly, proving existence is fairly non-trivial. Following~\cite{MZ08}, we apply a Galerkin approximation procedure, which includes a linearized Chapman-Enskog expansion. However, we encounter a new difficulty precisely having to do with the moment loss in the Landau collision operator. Namely, to close the macro-micro loop, it is necessary to track a sufficiently large number of moments of $f$; that is, we need an estimate of the type $\langle v \rangle^{10} f \in H^2 \mathcal{H}^1$; see~\eqref{eq:doesn't:close}. This does not play well with the Galerkin approximation, which is adapted only to standard $\mathcal{H}^1$ energy estimates.

To address this difficulty, we prove existence for the linearized equation with a \emph{regularized} collision kernel, which does not suffer from the same loss of moments. We define
\begin{align}
Q_{\rm R}(F_1,F_2)(v) & = \nabla_v \cdot \left( \int_{\R^3} \phi_{\rm R}(v-u) \left( F_1(u)\nabla_v F_2(v) - F_2(v) \nabla_u F_1(u) \right) \, \dee u \right)\notag  \\
\phi_{\rm R}^{ij}(v) &= \left( \delta^{ij} - \frac{v^i v^j}{|v|^2} \right) |v|  \, .  \label{hard:landau}
\end{align}
This is a `harder' Landau collision operator than that used in~\eqref{def:LandauIntro}, and has been studied (for example) in~\cite{G02}. 
The regularized collision operator is then chosen as\index{$Q_\kappa$}\index{$Q_{\rm R}$}\index{$\kappa$}
\begin{equation}\label{Qkappadef}
Q_\kappa = Q + \kappa Q_{\rm R} \, .
\end{equation}
We write $\sf{L}_{\mu_{\rm NS}}^\kappa = \sf{L}_{\mu_{\rm NS}} + \kappa \sf{L}_{\mu_{\rm NS}}^{\rm R}$ for the regularized, linearized operator, and $L_{\rm NS}^{\kappa}$ for the regularized, linearized, conjugated operator.  In general, any notation pertaining to $Q$ can be appended with a $\kappa$, signifying usage of $Q_\kappa$ instead of $Q$.

The regularized operator has different mapping properties than the non-regularized operator. We define\index{$\sigma_{\rm R}$}\index{$\sigma_\kappa$}
\begin{equation}\label{twofiddytwo}
\sigma_{\rm R} = \phi_{\rm R} \ast \mu \, , \quad \sigma_\kappa = \phi_\kappa \ast \mu_0 = \sigma + \kappa \sigma_{\rm R} \, .
\end{equation}
We define the following polynomial weighted norms:  \index{$\mathcal{H}^1_{\rm R}$}\index{$\mathcal{H}^1_\kappa$}
\begin{equation}\label{H1kappa}
|f|_{\sigma_{\rm R},\ell}^2 = \int_{\R^3} \langle v \rangle^{2\ell}  \sigma_{\rm R}^{ij} \left( \partial_i f \partial_j f + v_i v_j f^2 \right) \, \dee v  \, , \quad |f|_{\sigma_\kappa,\ell}^2 = |f|_{\sigma,\ell}^2 + \kappa |f|_{\sigma_{\rm R},-\ell}^2 \, .
\end{equation}
Note the difference between the $\ell=0$ and $\ell\neq 0$ cases in the definition of the $|\cdot|_{\sigma_\kappa,\ell}$ norm, which is due to the difference in moment gain/loss between $\sigma$ and $\sigma_{\rm R}$ and roughly follows the convention in~\cite{G02}. \index{$\normm{\cdot}_{\sigma_\kappa, \ell}$} \index{$\normm{\cdot}_{\sigma_{\rm R}, \ell}$}
The linearized collision operator is coercive in the norm $|f|_{\sigma_\kappa, 0}:=|f|_{\sigma_\kappa}$, which is analogous to the norm $|\cdot|_{\sigma,0,0}$ defined in~\eqref{eq:monday:night}.  We associate to these norms the function spaces consisting of $H^1_{\rm loc}$ functions with finite norm, as in~\eqref{eq:sigma:2}.  We define $\mathcal{H}^{-1}_{\ell, \kappa}$ to be the associated dual space, as in~\eqref{def:dual:spaces}; specifically,\index{$\mathcal{H}^{-1}_{\ell,\kappa}$}
\begin{equation}\label{newduals}
 \| z \|_{\mathcal{H}^{-1}_{\ell, \kappa}} = \| \langle v \rangle^{-\ell} z \|_{\mathcal{H}^{-1}_\kappa} \, .
\end{equation}
We will not need any of the exponentially weighted spaces corresponding to $L^\kappa$ for $\kappa>0$, which is why we only consider $q=0$ above.\index{$\mathcal{H}^{-1}_{\rm R}$}\index{$\mathcal{H}^1_\kappa$}

To prove existence of solutions to~\eqref{def:LinCE}, we first study solutions to the regularized linearized equation
\begin{equation}\notag
(v_1 - s_0) \p_x f + L_{\rm NS}^\kappa f = z \, .
\end{equation}
Solutions to this equation obey the same \emph{a priori} estimates as solutions to~\eqref{def:LinCE}, but existence is easier. In Section~\ref{sekk:linear}, we prove existence for this equation by a non-trivial Galerkin approximation procedure in the space $H^2 \mathcal{H}^1_\kappa$.  We note that passing to the limit in $\kappa$ requires a proof that the smallness parameter $\varepsilon$ (the size of the shock) can be taken uniformly positive, irrespective of $\kappa\rightarrow 0$.  We achieve this in subsection~\ref{sec:methodofcontinuity} using the method of continuity.

\subsubsection{The proof of Theorem \ref{thm:main}} \label{sss:outproof}

Since the linear problem has now been solved by Proposition~\ref{lem:MainLinear}, it remains to estimate $\mathcal{E}$. Recalling the definitions~\eqref{def:calE} and~\eqref{def:residual}, we have the following estimates.

\begin{proposition}
	\label{pro:remainderest}
The following estimate holds on the remainder:
\begin{align}
\norm{e^{\delta \brak{\eps x}^{\sfrac 12}} \mu_0^{-\sfrac 12} \mathcal{E}_{\rm NS}}_{\mathbb Y^N_\eps} + \norm{\mu_0^{-\sfrac 12} \mathcal{E}_{\rm NS}}_{\mathbb Y^N_{\eps, \rm w}} \lesssim \eps^3 \, .   \label{rem:est:one}
\end{align}
Furthermore, if we let $f=\mu_0^{-\sfrac 12} F \in \mathbb X^N_{\eps, \rm w}$, then we have the estimates 
\begin{align*}
\norm{e^{\delta \brak{\eps x}^{\sfrac 12}} \Gamma[g_{\rm NS},f] }_{\mathbb Y^N_\eps} + \norm{\mu_0^{-\sfrac 12} \Gamma[g_{\rm NS},f] }_{\mathbb Y^N_{\eps, \rm w}} &\lesssim \eps^2 \norm{f}_{\mathbb M^N_\eps} \\
\norm{e^{\delta \brak{\eps x}^{\sfrac 12}} \Gamma[f,f] }_{\mathbb Y^N_\eps} + \norm{\mu_0^{-\sfrac 12} \Gamma[f,f] }_{\mathbb Y^N_{\eps, \rm w}} &\lesssim \norm{f}_{\mathbb M^N_\eps}^2 \, . 
\end{align*}
\end{proposition}
We prove this proposition in section~\ref{sec:error:CE}. From here, it is straightforward to use a fixed point argument to complete the proof of Theorem~\ref{thm:main}; see section~\ref{sec:fixed:point}. 

\section{Error estimates for Chapman-Enskog expansion}\label{sec:error:CE}


\subsection{The $\sigma$ norms}
In this section, our main references are~\cite{G02,GS08}.

For the sake of generality, we introduce a parameter $\theta \in [0,2]$ and the weight functions\index{${\rm w}$}\index{${\rm w}_{\rm R}$}
\begin{equation}\label{waits}
{\rm w}(\ell,q,\theta)=\langle v \rangle^{-\ell} e^{\frac q 4 |v|^{\theta}} \, , \qquad {\rm w}_{\rm R}(\ell) = \langle v \rangle^\ell
\end{equation} 
where $\ell \in \R$ and $(q,\theta) \in ([0,+\infty) \times [0,2)) \cup ((0,1) \times \{ 2 \})$; see \cite{GS08}.  Note the difference in the powers of $\langle v \rangle$ between the weight corresponding to $Q$ and the weight corresponding to $Q_{\rm R}$. The norm $|\cdot|_{\sigma,\ell,q,\theta}$ is defined as in~\eqref{eq:monday:night}.  We recall the norms $|\cdot|_{\sigma_{\rm R}, \ell}$ and $|\cdot|_{\sigma_{\kappa}, \ell}$ from~\eqref{H1kappa}.  The first estimate below is contained in~\cite[Lemma~5]{GS08}, the second is contained in~\cite[Corollary~1]{G02} (although adjusted to include both an upper and lower bound, as in~\cite{GS08}), and the third is an immediate consequence of the first two and~\eqref{twofiddytwo}.  In each estimate, $|\cdot|_2$ refers to the unweighted, standard $L^2$ norm in the variable $v$.


\begin{lemma}[Norm control]\label{l:norm:control}
Given $v \neq 0$, we write $P_v = v \otimes v |v|^{-2}$. 
There exist implicit constants such that, for any $f, g \in H^1_{\rm loc}(\R^3)$ with $|g|_{\sigma,\ell,q,\theta} < +\infty$ and $|f|_{\sigma_{\rm R},\ell} < +\infty$, we have the norm equivalence:
\begin{subequations}\notag
\begin{align*}
|g|_{\sigma,\ell,q,\theta}^2 &\approxeq \left| {\rm w}(\ell, q, \theta) [1+|v|]^{-\sfrac{3}{2}} \{ P_v \partial_i g \} \right|_{2}^2 + \left| {\rm w}(\ell, q, \theta) [1+|v|]^{-\sfrac{1}{2}} \{[I-P_v] \partial_i g \} \right|_{2}^2 \notag\\
&\qquad  + \left| {\rm w}(\ell, q, \theta) [1+|v|]^{-\sfrac{1}{2}} g \right|_{2}^2 \, , \\
|f|_{\sigma_{\rm R},\ell}^2 &\approxeq \left| {\rm w}_{\rm R}(\ell) [1+|v|]^{- \sfrac 12} \{ P_v \partial_i f \} \right|_{2}^2 + \left|{\rm w}_{\rm R}(\ell) [1+|v|]^{\sfrac 12} \{[I-P_v] \partial_i f \} \right|_{2}^2  + \left| {\rm w}_{\rm R}(\ell) [1+|v|]^{\sfrac 12} f \right|_{2}^2  \\
|f|^2_{\sigma_\kappa, \ell} &\approxeq |f|^2_{\sigma, \ell, 0,0} + \kappa |f|^2_{\sigma_{\rm R}, - \ell }  \, . 
\end{align*}
\end{subequations}
\end{lemma}

\begin{corollary}[Control by weighted dual norms]\label{weddy:night}
We have that $\|z\|_{\mathcal{H}^{-1}_\ell} \les \| z \|_{\mathcal{H}^{-1}_{\ell, \rm w}}$, and $\| z \|_{\mathcal{H}^{-1}_{\theta,\kappa}} \les \| z \|_{\mathcal{H}^{-1}_{\theta-m,\kappa}}$ for $\kappa\in[0,1]$, $\theta\in\mathbb{R}$, and $m\geq 0$.
\end{corollary}
\begin{proof}
We prove the first estimate by writing
\begin{align*}
\| z \|_{\mathcal{H}^{-1}_\ell} &= \| z \langle v \rangle^{-\ell} \|_{\mathcal{H}^{-1}}\\
&= \sup_{\| f \|_{\mathcal{H}^1}=1} \left \langle f , z \langle v \rangle^{-\ell} \right \rangle_{(\mathcal{H}^1, \mathcal{H}^{-1})} \\
&= \sup_{\| f \|_{\mathcal{H}^1}=1} \left \langle \exp\left(- \frac q4 \langle v \rangle^2 \right) f , z {\rm w} \right \rangle_{(\mathcal{H}^1, \mathcal{H}^{-1})} \\
&\leq \sup_{\| f \|_{\mathcal{H}^1}=1} \left\| f \exp\left(- \frac q4 \langle v \rangle^2 \right)  \right\|_{\mathcal{H}^1} \| z \|_{\mathcal{H}^{-1}_{\ell, \rm w}} \, .
\end{align*}
The result will then follow from the inequality 
$$\left\| f \exp\left(- \frac q4 \langle v \rangle^2 \right)  \right\|_{\mathcal{H}^1} \les \| f \|_{\mathcal{H}^1}  \, . $$ 
This claim however follows from direct computation, \cite[Lemma~3]{G02}, which shows that $|\sigma^{ij}(v)| \les \langle v \rangle^{-1}$ and that $\sigma^{ij}(v) v^i v^j f^2 = \lambda_1(v)|v|^2 f^2$, where $\lambda_1(v)\sim \langle v \rangle^{-3}$ asymptotically as $|v|\rightarrow \infty$, the fact that $\exp(- \sfrac q4 \langle v \rangle^2 ) \les \langle v \rangle^{\ell'}$ for any $\ell'$ (albeit with an implicit constant depending on $\ell'$), and Lemma~\ref{l:norm:control}.  The proof of the second inequality uses~\eqref{newduals} and is similar, and we omit further details.
\end{proof}


\subsection{Properties of $L$ and $\Gamma$}
\label{sec:propertiesofLandGamma}

We first note that the regularized operator $Q_\kappa$ defined \index{$Q_\kappa$}in~\eqref{Qkappadef} for $\kappa\in[0,1]$ preserves mass, momentum and energy, i.e. $I[Q_\kappa]=0$ for $I[\cdot]$ defined in~\eqref{op:i}.  Indeed, these conservation laws follow from the divergence-form structure of $Q_\kappa$, integration by parts, the evenness of $\phi^{ij}$, $\phi^{ij}_{\rm R}$, and thus $\phi^{ij}+\kappa\phi^{ij}_{\rm R}$, and the identity $\phi^{ij}(v) v_i = \phi^{ij}_{\rm R}(v) v_i =0$.  Next, the H-theorem and the fact that the manifold of Maxwellians $M_+ \subset \mathbb{H}_+$ described in subsubsection~\ref{sss:maxwellians} characterizes all equilibria of $Q_\kappa$ in $\mathbb{H}_+$ follow using the aforementioned properties in conjunction with the fact that $\phi^{ij}$, $\phi^{ij}_{\rm R}$, and thus $\phi^{ij}+\kappa\phi^{ij}_{\rm R}$ are projection operators in $v$ multiplied by positive functions of $v$.  Now considering the linearized operator $\sf{L}_\mu^\kappa$ for general Maxwellians $\mu$, we check the properties~\eqref{kernie}--\eqref{def:sfP}.  Considering $\sf{L}_\mu^\kappa$ as an unbounded operator on the same Hilbert space as in~\eqref{eq:innerproduct}, we claim that it is symmetric and positive-definite, which follows immediately from the same properties for $\sf{L}_\mu$ and $\sf{L}_\mu^{\rm R}$.  As before, the kernel of $\sf{L}_\mu^\kappa$ is therefore the tangent space to $M_+$, and so the projection operators $\sf{P}_\mu$ and $\sf{P}_\mu^\perp$ do not depend on $\kappa$.    Thus we have all the properties needed to carry out the linearized Chapman-Enskog expansion in subsubsection~\ref{sss:macro}.

We \index{$L^{\rm R}$}\index{$L^\kappa$}decompose the conjugated linearized operators $L^{\rm R}$ (defined analogously to~\eqref{ells:and:Qs} but with $Q$ replaced by $Q_{\rm R}$ as defined in~\eqref{hard:landau}) and $L^\kappa$ for $\kappa\in[0,1]$ (defined by $L^\kappa=L+\kappa L^{\rm R}$ for $\kappa\in[0,1]$), as
\begin{equation}\notag L^\bullet = -A_\bullet - K_\bullet
\end{equation}
for $\bullet=\kappa, {\rm R}$.  Following~\cite[Eq. (47)-(48)]{GS08}, for $\bullet=\kappa, {\rm R}$, we set $\sigma_\bullet^i = \sigma_\bullet^{ij}v_j/2$ and
\begin{equation}\label{some:IDs}
\begin{aligned}
A_\bullet g &= \p_i [\sigma_\bullet^{ij} \p_j g ] - \sigma_\bullet^{ij} \frac{v_i}{2} \frac{v_j}{2} + \p_i \sigma_\bullet^i g \\
K_\bullet g &= - \mu^{-\sfrac{1}{2}} \p_i ( \mu [\phi_\bullet^{ij} \ast (\mu^{\sfrac{1}{2}} [\p_j g + \frac{v_j}{2} g] )]) \, .
\end{aligned}
\end{equation}
It is useful to recognize that $A_\bullet g = \mu^{-\sfrac{1}{2}} \div [ (\sigma_\bullet \mu) \nabla (\mu^{-\sfrac{1}{2}} g) ]$, which, in particular, reveals its one-dimensional kernel $\R \cdot \mu^{\sfrac{1}{2}}$. We write $A_{\bullet,1} = \p_i [\sigma_\bullet^{ij} \p_j g ] - \sfrac14 \sigma_\bullet^{ij} v_i v_j$, which is evidently coercive, and $A_{\bullet, 2} = \p_i \sigma^i_\bullet g$. Additionally, we recognize that $K_\bullet g = -\mu^{-\sfrac{1}{2}} \nabla (\mu [\phi_\bullet \ast (\mu^{\sfrac{1}{2}} g)])$.

For $g \in \mathcal{H}^1$, \cite[Lemma~5]{G02} gives that
\begin{equation}\notag
|P^\perp g|_{\sigma}^2 \lesssim \langle Lg,g \rangle \, .
\end{equation}
In fact, \cite[Lemma~5]{G02} proves the analogous coercivity estimate for harder Landau operators, including $Q_{\rm R}$.  Thus for $g \in \mathcal{H}^1_\kappa$, defined in~\eqref{H1kappa} for $\kappa\in(0,1]$ and understood to be the usual $\mathcal{H}^1$ 
 when $\kappa=0$, we have that there exist constants $C_{\rm R}$ and $C_1$ such that
\begin{equation}
\label{eq:coercivityestimate}
|P^\perp g |_{\sigma_{\rm R}}^2 \leq C_{\rm R} \langle L^{\rm R} g, g \rangle  \, , \qquad  | P^\perp g |_{\sigma_\kappa}^2 \leq C_1 \langle L^\kappa g,g \rangle \, .
\end{equation}
We therefore see that for $\bullet=\kappa, {\rm R}$, the operator $L^\bullet : \mathcal{H}^1_\bullet \to \mathcal{H}^{-1}_\bullet$ maps $\mathcal{H}^1_\bullet$ to $\mathcal{H}^{-1}_\bullet$ boundedly.  In fact, \cite[Lemmas~5,6]{G02} give that $A_\bullet, K_\bullet : \mathcal{H}^1_\bullet \to \mathcal{H}^{-1}_\bullet$ map $\mathcal{H}^1_\bullet$ to $\mathcal{H}^{-1}_\bullet$ boundedly.  Moreover, given $h \in \mathcal{H}^{-1}_{\bullet, \rm mic} := \mathcal{H}^{-1}_\bullet \cap \{ Pf = 0 \}$, there exists a unique solution $f \in \mathcal{H}^1_{\bullet, \rm mic} := \mathcal{H}^1_\bullet \cap \{ Pf = 0 \}$ to $L_\kappa f = h$. This may be proven via the Lax-Milgram theorem in the Hilbert space $\mathcal{H}^1_{\bullet, \rm mic}$ with $B(f,g) = \langle L^\bullet f, g\rangle$ and the coercivity estimate~\eqref{eq:coercivityestimate}. 

Frequently, it will be necessary to understand the behavior of $L^\bullet$ in more general function spaces, so as to obtain coercivity estimates more general than~\eqref{eq:coercivityestimate}.  In the next lemma, we first recall~\cite[Lemma 9]{GS08} in estimates~\eqref{GS:Lemma:9a} and~\eqref{GS:Lemma:9b}, slightly adapt~\cite[Lemma~6]{G02} in~\eqref{new:co}, and combine the estimates to obtain general coercivity estimates for $L^\kappa$ in~\eqref{new:co:kappa}.  The proofs of~\eqref{new:co} and~\eqref{new:co:kappa} are contained in Appendix~\ref{sec:proofofnonlinearestimate}.

\begin{lemma}[Coercivity of $L, L^{\rm R}, L^\kappa$]\label{lem:propertiesofL}
The following coercivity estimates hold.
\begin{enumerate}
\item Let $0\leq q<1$, $l\geq 0$, $|\beta|>0$, $\ell=|\beta|-l$, and ${\rm w}(\ell, q)$ defined as in~\eqref{def:double:u}. Then for $\eta>0$ small enough there exists $C(\eta)>0$ such that for any $g$ for which the norms on the right-hand side are finite, we have
\begin{align}\label{GS:Lemma:9a}
\left\langle {\rm w}^2(\ell, q) \p_\beta[L g] , \p_\beta g \right\rangle \geq \left| \p_\beta g \right|^2_{\sigma,\ell,q} - \eta \sum_{|\beta_1|=|\beta|} \left| \p_{\beta_1} g \right|_{\sigma,\ell,q}^2 - C(\eta) \sum_{|\beta_1|<|\beta|} \left| \p_{\beta_1} g \right|_{\sigma,|\beta_1|-l,q}^2 \, .
\end{align}
If $|\beta|=0$, we have
\begin{align}\label{GS:Lemma:9b}
\left\langle {\rm w}^2(\ell, q) [L g] , g \right\rangle \geq \delta_q^2 \left| g \right|_{\sigma,\ell,q} - C(\eta) \left| {\rm w}(\ell,0) \overline{\chi}_{C(\eta)} g \right|_2^2 \, ,
\end{align}
where $\delta_q = 1 - q^2 - \eta >0$ for $\eta>0$ small enough and $\overline\chi_R(|v|) $ a smooth cutoff function which vanishes for $|v|\geq 2R$ and is equal to $1$ for $|v|\leq R$.

\item Let $\theta > 0$ be given.  Then there exist $\delta>0$ and $C$ such that, for any $g$ such that the norms on the right-hand side are finite, 
\begin{align}\label{new:co}
\left \langle \langle v \rangle^{2\theta} L^{\rm R} g, g \right \rangle  \geq \delta^2 \left| g \right|^2_{\sigma_{\rm R}, \theta} - C |g|_{\sigma_{\rm R}, \theta - \sfrac 12}^2 \, .
\end{align}
\item Let $\kappa \in [0,1]$ and $\vartheta<0$ be given and $|\cdot|_{\sigma_\kappa, \vartheta}$ defined as in~\eqref{H1kappa}.  Then there exist $\delta>0$ and $C$ such that, for any $g$ such that the norms on the right-hand side are finite, 
\begin{equation}\label{new:co:kappa}
\left \langle \langle v \rangle^{-2\vartheta} L^\kappa g, g \right \rangle \geq \delta^2 | g |_{\sigma_\kappa,\vartheta}^2 - C | g |^2_{\sigma_\kappa, \vartheta + \sfrac 12}
\end{equation}
\end{enumerate}
\end{lemma}

\begin{remark}
	\label{rmk:technicalbootstrappingremark}
We quickly review one way to bootstrap solutions $f \in \mathcal{H}^1_{\rm mic}$ to the linearized equation $Lf = h \in \mathcal{H}^{-1}$ into the appropriate function spaces to apply energy estimates like Lemma~\ref{lem:propertiesofL}. This reasoning is not often made explicit, so we mention it here for completeness. This strategy will be called upon in Lemmas~\ref{lemma:boot:v} and~\ref{lem:bs:1}, but only in a qualitative way.

Recall that $-L = A+K$ and that $A$ is decomposed into $A_1$ and $A_2$, defined below~\eqref{some:IDs}, where $A_1$ is coercive. Let $\chi_R = \chi(\cdot/R)$ be a smooth cut-off function with $\chi \equiv 1$ on $B_1$ and $\chi \equiv 0$ outside $B_2$.
Since $|\p_\beta \sigma|(z) \lesssim_\beta \langle z \rangle^{-1-|\beta|}$, we decompose $A_2 = (1-\chi_R) A_2 + \chi_R A_2$. Then $\tilde{A} = A_1 + (1-\chi_R) A_2$ is a small perturbation of $A_1$ when $R \gg 1$, and we write $-\tilde{A} f = h - \chi_R A_2f + Kf$. It is not difficult to justify the mapping properties of the perturbed operator $\tilde{A}$ on weighted spaces (for example, one can replace $\sigma$ by $\sigma + \delta I$, justify all estimates, and subsequently allow $\delta \to 0^+$). Then one may utilize Lemma~7 (see, additionally, the comment beneath) and Lemma~8 in~\cite{GS08} to prove that $Kf$ belongs to the relevant dual spaces (notably, it gains one derivative and belongs to nearly optimal Gaussian weighted spaces with $\theta = 2$ and $q \in (0,1)$). Finally, the same strategy can be used to bootstrap solutions $f \in L^2 \mathcal{H}^1$ to the inhomogeneous problem $(v_1 - s_0) \p_x f + Lf = h \in L^2 \mathcal{H}^{-1}$. Analogous results hold in the space $\mathcal{H}^1_{\kappa}$ for the $\kappa$-regularized collision kernel.
\end{remark}

The following lemma is slightly adapted from \cite[Lemma~10]{GS08}; see also \cite[Lemmas~6.4, 6.5]{DuanYangYuContactWaves}.  The proof is contained in Appendix~\ref{sec:proofofnonlinearestimate}.

\begin{lemma}[Nonlinear estimates]\label{lem:lemma10:new}
Let $\alpha,\beta$ be given, $0\leq \theta \leq 2$, and $\ell\geq 0$. 
Then, for any $q \geq 0$ (or $q \in (0,1)$ when $\theta = 2$), we have
\begin{align}\notag
\left\langle {\rm w}^2(\ell,q,\theta) \partial_\beta^\alpha \Gamma[g_1,g_2] , \partial_\beta^\alpha g_3 \right\rangle &\lesssim \sum_{\substack{|\alpha_1|+|\beta_1| \leq N \\ \bar{\beta}\leq \beta_1 \leq \beta}} \left|\partial_{\bar\beta}^{\alpha_1} g_1\right|_{{\sigma},\ell} \, \left|\partial_{\beta-\beta_1}^{\alpha-\alpha_1} g_2\right|_{{\sigma},\ell,q,\theta} \left| \partial_\beta^\alpha g_3 \right|_{\sigma,\ell,q,\theta} \, .
\end{align}
For $\Gamma_{\rm R}$, we further restrict the values of $q=\theta=0$ (so that no exponential weights are considered).  Then 
\begin{align}\notag
\left\langle {\rm w}^2(\ell,0,0) \partial_\beta^\alpha \Gamma_{\rm R}[g_1,g_2] , \partial_\beta^\alpha g_3 \right\rangle & \lesssim \sum_{\substack{|\alpha_1|+|\beta_1| \leq N \\ \bar{\beta}\leq \beta_1 \leq \beta}} \left|\partial_{\bar\beta}^{\alpha_1} g_1\right|_{{\sigma_{\rm R}},-\ell} \, \left|\partial_{\beta-\beta_1}^{\alpha-\alpha_1} g_2\right|_{{\sigma_{\rm R}},-\ell} \left| \partial_\beta^\alpha g_3 \right|_{\sigma_{\rm R},-\ell} \, ,
\end{align}
for all $g_1$, $g_2$, $g_3$ whose norms on the right-hand side are finite.
\end{lemma}

Before we estimate $g_{\rm NS} = \mu_0^{-\sfrac{1}{2}} G_{\rm NS}$, we make two comments. First, the above estimates were centered at the standard Maxwellian $\mu_0$. However, a change of variables $v = (w-u)/\theta$ yields analogous properties for the collision operator linearized about and conjugated with respect to any Maxwellian $\mu(\varrho,u,\theta)$. Second, when $\theta > 1$, then $\mu(1,u,\theta)$ simply does not belong to the optimal Gaussian weighted spaces for the standard Maxwellian $\mu_0$; that is, $| \mu_0^{-\sfrac{1}{2}} \mu(1,u,\theta) |_{\sigma,0,q,2} = +\infty$ when $|q - 1| \ll 1$. We will therefore fix $q_0 \in (0,1)$ and subsequently choose $\varepsilon$.


\begin{lemma}
\label{lem:mugests}
Let $0 < q_0 < q_1 < 1$. Let $0 < \varepsilon \ll 1$ be small enough to guarantee that 
$$\frac{q_1+1}{4} < \frac{1}{2\theta_{\rm NS}}  \, . $$
Then, for all $\alpha \in \N$ and multi-indices $\beta \in \N^3$, we have
\begin{equation}\notag
\begin{aligned}
| \p^\alpha_\beta [\mu_0^{-\sfrac{1}{2}} (\mu_{\rm NS} - \mu_0)]|_{\sigma,0,q_0,2} &\lesssim \varepsilon^{\alpha+1} \\
| \p^\alpha_\beta g_{\rm NS} |_{\sigma,0,q_0,2} &\lesssim \varepsilon^{\alpha+2} \, .
\end{aligned}
\end{equation}
\end{lemma}


\begin{proof}[Proof of Lemma~\ref{lem:mugests}]
Estimates on $\mu_{\rm NS} = \mu(\varrho_{\rm NS},u_{\rm NS},\theta_{\rm NS})$ follow directly from the estimates on the compressible Navier-Stokes shock profile; see Proposition~\ref{pro:smallshocksgeneralhyperbolic}. Recall from~\eqref{def:fNS} that $G_{\rm NS}$ is the purely microscopic solution, in a suitable space, to 
$$Q(\mu_{\rm NS},G_{\rm NS}) + Q(G_{\rm NS},\mu_{\rm NS}) = \sf{P}^\perp_{\mu_{\rm NS}}( v_1 \partial_x \mu_{\rm NS}) \, . $$
 From the variant of Lemma~\ref{lem:propertiesofL} centered at $\mu_{\rm NS}$, we have that $G_{\rm NS}$ exists and satisfies estimates on $\p_\beta G_{\rm NS}$ described in the surrounding discussion. To estimate derivatives, we differentiate the equation to obtain
\begin{align}
&Q(\mu_{\rm NS},\p^\alpha G_{\rm NS}) + Q(\p^\alpha G_{\rm NS},\mu_{\rm NS}) \notag\\
&= - \sum_{\alpha'=1}^{\alpha-1}  C(\alpha',\alpha) \left[ Q(\p_x^{\alpha'} \mu_{\rm NS}, \p_x^{\alpha-\alpha'} G_{\rm NS}) + Q(\p_x^{\alpha-\alpha'} G_{\rm NS},\p_x^{\alpha'} \mu_{\rm NS}) \right] + \p_x^\alpha [ \sf{P}^\perp_{\mu_{\rm NS}}( v_1 \partial_x \mu_{\rm NS}) ]\label{eq:differentiateequationwoot}
\end{align}
and apply the same estimates with a new right-hand side, which is estimated according to the nonlinear estimates in Lemma~\ref{lem:lemma10:new}. (The calculation~\eqref{eq:differentiateequationwoot} is at first formal but may be justified by mollifying the equation, estimating the mollified equation with a new commutator term on the right-hand side, and sending the mollification parameter to zero. A version of this is carried out in the proof of Lemma~\ref{lem:basicmicroest}.) Notably, the operator $\sf{P}^\perp_{\mu_{\rm NS}} = I - \sf{P}_{\mu_{\rm NS}}$ is explicit and gains $\varepsilon$ whenever it is differentiated in $x$.
\end{proof}

Frequently, it will be necessary to invert the linearized collision operator at a Maxwellian which is nearby the standard Maxwellian.
\begin{lemma}[Invertibility at nearby Maxwellia]
\label{lem:inverttofindg}
Let $\mu_1 = \mu(\varrho_1,u_1,\theta_1)$ be a Maxwellian near the standard Maxwellian $\mu_0$; that is, $|(\varrho_1-1,u_1,\theta_1-1)| \ll 1$. Then the linearized operator
\begin{equation}\notag
L_{\mu_1} : f \mapsto - \Gamma\left[\mu_0^{-\sfrac 12}\mu_1, f\right] - \Gamma\left[f,\mu_0^{-\sfrac 12}\mu_1 \right]
\end{equation}
maps $\mathcal{H}^1$ to $\mathcal{H}^{-1}$. Moreover, $L_{\mu_1}$, restricted to the (conjugated) microscopic subspace $\mathcal{H}^1_{\rm mic}$, as a map to the (conjugated) microscopic subspace $\mathcal{H}^{-1}_{\rm mic}$, is invertible and satisfies
\begin{equation}\label{invert:est}
\| L_{\mu_1}^{-1} \|_{\mathcal{H}^{-1}_{\rm mic} \to \mathcal{H}^1_{\rm mic}} \les 1 \, .
\end{equation}
Finally, we have that $L_{\mu_1}$, $L_{\mu_1} \circ P^\perp$, and $L_{\mu_1} \circ P_{\rm NS}^\perp$ are equal on $\mathcal{H}^1_{\rm mic}$, where $P^\perp$ is defined immediately below~\eqref{eq:wed:morning} and $P_{\rm NS}^\perp$ is defined analogously but using $\mu_{\rm NS}$ as defined above~\eqref{def:fNS}.
\end{lemma}
\begin{proof}
The mapping property follows from the nonlinear estimate for $\Gamma$ from Lemma~\ref{lem:lemma10:new}. The invertibility and the estimate~\eqref{invert:est} follow from considering $L_{\mu_1} = L -2 \Gamma[\mu_0^{-\sfrac 12}(\mu_1-\mu_0), \cdot]$ to be a small perturbation of $L$, which is bounded from below on $\mathcal{H}^1_{\rm mic}$ and satisfies the estimate~\eqref{invert:est}. Finally, the claim regarding the restrictions is true because $P^\perp$ and $P_{\rm NS}^\perp$ restricted to the microscopic subspace are the identity.
\end{proof}

With the results of this section in hand, we can prove the estimate for the remainder given in Proposition~\ref{pro:remainderest}.

\begin{proof}[Proof of Proposition~\ref{pro:remainderest}]
We utilize the conjugated version of~\eqref{pure}:
\begin{equation}
\mu_0^{-\sfrac{1}{2}} \mathcal{E}_{\rm NS} := - {P}^\perp (v_1-s) \p_x [L_{{\rm NS}}^{-1} {P}^\perp ( (v_1-s) \partial_x \mu_0^{-\sfrac{1}{2}} \mu_{\rm NS})] - \Gamma(g_{\rm NS},g_{\rm NS}) \, .\notag
\end{equation}
We complete the proof by referring to the estimates on $\mu_0^{-\sfrac{1}{2}} \mu_{\rm NS}$ and $g_{\rm NS}$ in Lemma~\ref{lem:mugests}, the coercivity estimates in Lemma~\ref{lem:propertiesofL}, the nonlinear estimates in Lemma~\ref{lem:lemma10:new}, and invertibility properties of $L_{\rm NS}$ in Lemma~\ref{lem:inverttofindg}.
\end{proof}

\section{Linear estimates}\label{secc:micro:linear}


\subsection{Statement of estimates}
	\label{sec:preliminaries}

We now work with the conjugated linearized equation \eqref{def:LinCE:reg:conj}, where $z$ is a purely microscopic (conjugated) right-hand side (see~\eqref{pure}). We will use $P$ to denote the $L^2$-orthogonal projection onto the five-dimensional kernel of $L^\kappa$ for $\kappa\in[0,1]$ (which does not depend on $\kappa$), defined as in~\eqref{ells:and:Qs} by 
$$L^\kappa(\cdot) = -\Gamma_\kappa[\mu_0^{\sfrac 12},\cdot] -\Gamma_\kappa[\cdot, \mu_0^{\sfrac 12}] = - \mu_0^{-\sfrac 12} Q_\kappa (\mu_0, \mu_0^{\sfrac 12}\cdot) - \mu_0^{-\sfrac 12} Q_\kappa (\mu_0^{\sfrac 12}\cdot,\mu_0) \, . $$
We also recall the equality~\eqref{pproj}, which relates $P$ and $\sf{P}_{\mu_0}$. With these notations, the goal of this section is to prove the following two propositions.

\begin{proposition}[Linear estimates: Baseline]
\label{pro:basicexistence}
Let $0 < \varepsilon, \kappa \ll 1$.
Let $z \in H^2 \mathcal{H}_{-10,{\kappa}}^{-1}$ and $d \in \R$. There exists a unique solution $f \in H^2 \mathcal{H}^1_{-10,\kappa}$ to \eqref{def:LinCE:reg:conj} satisfying the one-dimensional constraint\index{$\ell_\varepsilon$}
\begin{equation}\notag
\ell_\varepsilon( Pf(0) ) = d
\end{equation}
defined in Proposition~\ref{pro:ODEpropappendix} and Remark~\ref{rmk:theconstraint}. The solution satisfies
\begin{equation}\label{apriori:est}
\| f \|_{H^2_\varepsilon \mathcal{H}^1_{-10,\kappa}} \les \varepsilon^{-1} \| z \|_{H^2_\varepsilon \mathcal{H}_{-10,\kappa}^{-1}} + |d| \, . 
\end{equation}
\end{proposition}
\noindent The proof of Proposition~\ref{pro:basicexistence} will be split into three parts. First, we \emph{assume}  in section~\ref{secc:micro:linear} that $f\in H^2 \mathcal{H}^1$ is a solution and prove a version of the \emph{a priori} estimate~\eqref{apriori:est}, but with $-10$ replaced by $0$.  We then immediately use this result to bootstrap polynomial velocity moments up to order $10$ so as to verify~\eqref{apriori:est}.  We finally prove the existence of a unique solution in section~\ref{sekk:linear}.  

We also prove the following weighted, higher order estimates in subsection~\ref{ss:le:boot}.  For these estimates, we need only consider $\kappa=0$ and $d=0$, so that we are studying the original linearized equation~\eqref{eq:NLSL}.

\begin{proposition}[Linear estimate: Bootstrapping]\label{p:boots}
Let $N \geq 2$ and $0 < \varepsilon \ll_N 1$. Suppose that $z \in \mathbb{Y}^N_{\varepsilon,\rm w}$ and $e^{\delta\langle \varepsilon x \rangle^{\sfrac 12} } z \in \mathbb{Y}^N_{\varepsilon}$. If $f \in H^2 \mathcal{H}^1_{-10}$ is the unique solution to~\eqref{eq:NLSL} with $\ell_\varepsilon(Pf(0)) = 0$ as in Proposition~\ref{pro:basicexistence}, then
\begin{subequations}\label{p:boots:est}
\begin{align}
\label{v:weighted}
\| f \|_{\mathbb{X}^N_{\varepsilon, \rm w}} &\les \varepsilon^{-1} \| z \|_{\mathbb{Y}^N_{\varepsilon, \rm w}} \,  , \\
\label{x:weighted}
\| e^{\delta \langle \varepsilon x \rangle^{\sfrac 12}} f \|_{\mathbb{X}^N_{\varepsilon}} &\les \varepsilon^{-1} \left( \|  e^{\delta \langle \varepsilon x \rangle^{\sfrac 12}} z \|_{\mathbb{Y}^N_{\varepsilon}} + \| z \|_{\mathbb{Y}^N_{\varepsilon, \rm w}} \right) \, .
\end{align}
\end{subequations}
\end{proposition}
\begin{remark}[Proof of Proposition~\ref{lem:MainLinear}]\label{rem:implications}
Proposition~\ref{lem:MainLinear} is an immediate consequence of Proposition~\ref{p:boots} and the fact that $P^\perp f$ and $Pf$ are linearly independent, so that $\| f \|_{\mathbb{M}^N_\eps}$ is controlled from above by $\| f \|_{\mathbb{X}^N_{\eps,\rm w}} + \| e^{\delta \langle \eps x \rangle^{\sfrac 12}} f \|_{\mathbb{X}^N_{\eps}}$.
\end{remark}

\subsection{Linear estimates: Baseline}
\label{sec:linearestimatesbaseline}


We first present the basic microscopic \emph{a priori} estimate.
\begin{lemma}[Basic \emph{a priori} microscopic estimate]
\label{lem:basicmicroest}
Let $z \in H^1 \mathcal{H}_{-10,\kappa}^{-1}$ be given and assume that $f \in L^2 \mathcal{H}^1_\kappa$ is a solution of \eqref{def:LinCE:reg:conj}. Then $f \in H^1 \mathcal{H}^1_{-10,\kappa}$, and
\begin{equation}
\label{eq:basicmicroest}
 \left\| \p_x f \right\|^2_{L^2 \mathcal{H}^1_{-10,\kappa}} + \left\| P^\perp f \right\|^2_{L^2 \mathcal{H}^1_{-10,\kappa}} \les \varepsilon^2 \left\| Pf \right\|^2_{L^2} + \left\| z, \p_x z \right\|^2_{L^2 \mathcal{H}_{-10,\kappa}^{-1}}   \, .
\end{equation}
\end{lemma}
\noindent Before proving Lemma~\ref{lem:basicmicroest}, we set the notations
\begin{equation}\notag
M_5 = \sqrt{\mu_0} \: {\rm span} \left\{ 1,v_1,v_2,v_3,|v|^2 \right\} 
\end{equation}
and
\begin{equation}\notag
M_9 = \sqrt{\mu_0} \: {\rm span} \; \left\{ 1,v_1,v_2,v_3,|v|^2,v_1^2,v_1v_2,v_1v_3,|v|^2v_1 \right\} \, .
\end{equation}
Projection onto the first space is already denoted by $P$ (corresponding to the kernel of the linearized operator, cf.~\eqref{pproj}), while projection onto the second will be denoted by $P_{\leq 9}$. These finite-dimensional spaces allow us to state the following Kawashima compensator lemma.  For the proof of Lemma~\ref{lem:kawa:comp}, we refer to Appendix~\ref{sec:kawashima}, in particular~\eqref{kawa:useful}.
\begin{lemma}[Kawashima compensator]\label{lem:kawa:comp}
There exists an $L^2$--antisymmetric operator $K : M_9 \to M_9$ and a constant $C_K$ which, for all $g \in M_9$, satisfy
\begin{equation}
\| Pg \|^2_{L^2_v} \leq C_K \left[ \la [K P_{\leq 9},v_1] g,g \ra + \left\| \la v \ra^{-100} P^\perp g \right\|^2_{L^2_v} \right] \, .
\end{equation}
\end{lemma}
\noindent Our energy estimate will contain $\la [K P_{\leq 9},v_1] \p_x P_{\leq 9} f, \p_x P_{\leq 9} f \ra$; that is, the Kawashima compensator is applied to the equation for $f$ and paired with $\p_x P_{\leq 9} f$. Notice that 
$$ \la [K P_{\leq 9},v_1]g,g \ra = 2\la K P_{\leq 9} v_1 g,g \ra \textnormal{  for all  } g\in M_9 $$
 by symmetry and antisymmetry.

\begin{proof}[Proof of Lemma~\ref{lem:basicmicroest}]
To begin, we suppose that $f \in H^1 \mathcal{H}^1_\kappa$ and derive the \emph{a priori} estimate
\begin{equation}\label{apriori:easy}
\left\| \p_x f \right\|^2_{L^2 \mathcal{H}^1_{0,\kappa}} + \left\| P^\perp f \right\|^2_{L^2 \mathcal{H}^1_{0,\kappa}} \les \varepsilon^2 \left\| Pf \right\|^2_{L^2} + \left\| z, \p_x z \right\|^2_{L^2 \mathcal{H}_{0,\kappa}^{-1}} \, .
\end{equation}
Afterwards, we bootstrap moments in $v$ to prove~\eqref{eq:basicmicroest}. Finally, we justify the computations when $f \in L^2 \mathcal{H}^1_\kappa$.

Recall~\eqref{def:LinCE:reg:conj}:
\begin{equation}\label{wed:morning:equation}
(v_1-s_0) \p_x f + L^\kappa f = - \Gamma_\kappa[\mu_0^{-\sfrac{1}{2}} (\mu_{\rm NS} - \mu_0),f] - \Gamma_\kappa[f,\mu_0^{-\sfrac{1}{2}} (\mu_{\rm NS} - \mu_0)] + z \, .
\end{equation}
Notably, we have from Lemma~\ref{lem:lemma10:new}, the fact that $P \Gamma_\kappa[\cdot,\cdot] = 0$ (which follows from the properties of $Q_\kappa$ specified in the beginning of subsection~\ref{sec:propertiesofLandGamma}), and~\eqref{H1kappa} that, for a given $\delta>0$, there exists $C(\delta)>0$ such that
\begin{align}
|\la \Gamma_\kappa&[f,\mu_0^{-\sfrac{1}{2}} (\mu_{\rm NS} - \mu_0)], f \ra_{x,v} | + |\la \Gamma_\kappa[\mu_0^{-\sfrac{1}{2}} (\mu_{\rm NS} - \mu_0),f], f \ra_{x,v} | \notag \\
& \les |\la \Gamma[f,\mu_0^{-\sfrac{1}{2}} (\mu_{\rm NS} - \mu_0)], f \ra_{x,v} | + |\la \Gamma[\mu_0^{-\sfrac{1}{2}} (\mu_{\rm NS} - \mu_0),f], f \ra_{x,v} | \notag \\
&\qquad + \kappa |\la \Gamma_{\rm R}[f,\mu_0^{-\sfrac{1}{2}} (\mu_{\rm NS} - \mu_0)], f \ra_{x,v} | + \kappa|\la \Gamma_{\rm R}[\mu_0^{-\sfrac{1}{2}} (\mu_{\rm NS} - \mu_0),f], f \ra_{x,v} | \notag \\
& \les \varepsilon \| P^\perp f \|_{L^2\mathcal{H}^1}^2 + \varepsilon \| P^\perp f \|_{L^2\mathcal{H}^1} \| Pf \| + \varepsilon \kappa \| P^\perp f \|_{L^2 \mathcal{H}^1_{\rm R}}^2 + \varepsilon \kappa \| P^\perp f \|_{L^2 \mathcal{H}^1_{\rm R}} \| P f \| \notag \\
& \leq C(\delta) \varepsilon^2 \| Pf \|^2 + \delta \| P^\perp f \|_{L^2\mathcal{H}^1_\kappa}^2  \label{desert:cowboy}
\end{align}
since for all $x$, $f_{\rm NS}(x,\cdot) - \mu_0(\cdot) = O(\varepsilon)$ in strong $v$-topologies from Lemma~\ref{lem:mugests}.  We do not specify the norm on $\| P f \|$ since $Pf$ belongs to a finite-dimensional subspace.  We now multiply by $f$, integrate in $x$ and $v$, and apply the above estimate and the coercivity estimate from~\eqref{eq:coercivityestimate} to obtain
\begin{equation}\notag
\frac{1}{C_1}\| P^\perp f \|_{L^2 \mathcal{H}^1_\kappa}^2 \leq \delta \| P^\perp f \|_{L^2 \mathcal{H}^1_\kappa}^2 + C(\delta) \varepsilon^2 \| P f \|^2 + |\la z, P^\perp f \ra| \, ,
\end{equation}
which after splitting $|\langle z , P^\perp f \rangle|$ and choosing $\delta$ appropriately depending on $C_1$ implies that there exists $C_{\rm I}$ such that
\begin{equation}
\label{eq:I}
\tag{I}
\| P^\perp f \|_{L^2 \mathcal{H}^1_\kappa}^2 \leq C_{\rm I} \left[ \varepsilon^2 \| P f \|^2 + \| z \|^2_{L^2 \mathcal{H}^{-1}_{0,\kappa}} \right] \, .
\end{equation}
Next, we apply $\p_x$ to the equation, multiply by $\p_x f$, and integrate. Since $\p_x(\mu_0^{-\sfrac{1}{2}} (\mu_{\rm NS} - \mu_0)) = O(\varepsilon^2)$ in the $\Gamma_\kappa$ term, we have that there exists $\tilde C$ such that
\begin{align*}
\frac{1}{C_1}\| P^\perp \p_x f \|_{L^2 \mathcal{H}^1_\kappa}^2 &\leq \delta \| P^\perp \partial_x f \|_{L^2 H^1_\kappa}^2 + C(\delta) \varepsilon^2 \| P \partial_x f \|^2 \\
&\qquad  + | \langle \partial_x z , P^\perp \partial_x f \rangle | + \tilde C\varepsilon^2 \| \partial_x f \|_{L^2 H^1_\kappa} \| f \|_{L^2 H^1_\kappa}  \, .
\end{align*}
Splitting the term with $z$ again, using the triangle inequality on $f$ and $\partial_x f$ in the final term and then splitting using Cauchy-Schwarz, and using an appropriate choice of $\delta$ to absorb terms with $\partial_x f$, we have that there exists $C_{\rm II}$ such that
\begin{equation}
\label{eq:II}
\tag{II}
\| P^\perp \p_x f \|_{L^2 \mathcal{H}^1_\kappa}^2 \leq C_{\rm II} \left[ \varepsilon^2 \| P \p_x f \|_{L^2 \mathcal{H}^1_\kappa}^2 + \| \p_x z \|^2_{L^2 \mathcal{H}^{-1}_{0,\kappa}}  +  \varepsilon^2 \| Pf \|^2 + \varepsilon^2 \| P^\perp f \|_{L^2 \mathcal{H}^1_\kappa}\right]  \, .
\end{equation}

Finally, we apply $KP_{\leq 9}$ to the rewritten equation, multiply by $P_{\leq 9}\p_x f$, and integrate. We obtain
\begin{align*}
\la &K P_{\leq 9} v_1 \p_x P_{\leq 9} f, P_{\leq 9} \p_x f \ra + \la KP_{\leq 9} L^\kappa f , P_{\leq 9} \p_x f \ra \\
&\qquad= - \langle K P_{\leq 9} v_1 \p_x P_{>9} f, P_{\leq 9} \p_x f \rangle + \langle K P_{\leq 9} \Gamma_\kappa[f_{\rm NS} - \mu_0^{\sfrac 12}, f] , P_{\leq 9} \p_x f \rangle \\
&\qquad \qquad + \langle KP_{\leq 9}\Gamma_\kappa[f,f_{\rm NS} - \mu_0^{\sfrac 12}] , P_{\leq 9} \p_x f \rangle + \la KP_{\leq 9} z, P_{\leq 9} \p_x f \ra \, .
\end{align*}
The first term on the left-hand side above will provide the needed coercivity from Lemma \ref{lem:kawa:comp}.
To treat the second term on the left-hand side, we pass $K P_{\leq 9} $ onto the second slot by duality.  Then using the strong decay of the output $P_{\leq 9} K P_{\leq 9} \p_x f$ combined with Lemma~\ref{lem:lemma10:new}, this term is bounded by 
$$  \| P^\perp f \|_{L^2 \mathcal{H}^1} \| \partial_x f \|_{L^2 \mathcal{H}^1} + \kappa \| P^\perp f \|_{L^2 \mathcal{H}^1_{\rm R}}  \| \partial_x f \|_{L^2 \mathcal{H}^1_{\rm R}} \, .$$  The first term on the right-hand side may be treated similarly; by using that $P_{\leq 9} \partial_x f\in L^2 \mathcal{H}^{-1}_\kappa$ since it belongs to a finite-dimensional subspace for which all norms are equivalent, we may bound
$$  \left|  \langle K P_{\leq 9} v_1 \p_x P_{>9} f, P_{\leq 9} \p_x f \rangle  \right|  =  \left| \langle \p_x P_{>9} f, v_1 P_{\leq 9} K P_{\leq 9} \p_x f \rangle  \right| \les \| \p_x P^\perp f \|_{L^2 \mathcal{H}^1_\kappa}  \| P_{\leq 9} \p_x f \|_{L^2 \mathcal{H}^1_\kappa} \, . $$
The second term on the right-hand side may be bounded by 
$$  \left| \left \langle K P_{\leq 9} \Gamma_\kappa \left[ f_{\rm NS} - \mu_0^{\sfrac 12} , f \right] , P_{\leq 9} \partial_x f \right \rangle \right|  \leq  \varepsilon^2 C(\delta) \| Pf \|^2 + C(\delta) \| P^\perp f \|_{L^2 \mathcal{H}^1_\kappa}^2 + \delta \| \partial_x f \|_{L^2 \mathcal{H}^1_\kappa}^2   $$
in a similar manner as~\eqref{desert:cowboy}, with a similar bound holding for the third term.  Next, the fourth term is bounded by
$C(\delta) \| z \|^2_{L^2 \mathcal{H}^{-1}_{0,\kappa}} + \delta \| \p_x f \|^2_{L^2 \mathcal{H}^1_\kappa}$.  Finally, applying the Kawashima compensator estimate and using Lemma~\ref{l:norm:control} to bound the last term in Lemma~\ref{lem:kawa:comp} with the usual $L^2 \mathcal{H}^1_\kappa$ norm, there exists $\overline C$ such that
\begin{align*}
\frac{\| P\p_x f \|_{L^2_{x,v}}^2}{\overline C C_K} &\leq \| P^\perp f \|_{L^2 \mathcal{H}^1} \| \partial_x f \|_{L^2 \mathcal{H}^1} + \kappa \| P^\perp f \|_{L^2 \mathcal{H}^1_{\rm R}}  \| \partial_x f \|_{L^2 \mathcal{H}^1_{\rm R}} \\
&\qquad  + \| \p_x P^\perp f \|_{L^2 \mathcal{H}^1_\kappa}  \| P_{\leq 9} \p_x f \|_{L^2 \mathcal{H}^1_\kappa} + \varepsilon^2 C(\delta) \| Pf \|^2 \\
&\qquad \qquad  + C(\delta) \| P^\perp f \|_{L^2 \mathcal{H}^1_\kappa}^2 + 2\delta \| \partial_x f \|_{L^2 \mathcal{H}^1_\kappa}^2 + C(\delta) \| z \|_{L^2 \mathcal{H}^{-1}_{0,\kappa}}^2  \, .
\end{align*}
We may condense the above estimate by splitting via Cauchy-Schwarz and then combining the first two terms using~\eqref{H1kappa}, splitting the third term using Cauchy-Schwarz and the triangle inequality, splitting the fifth term using the triangle inequality, and then absorbing onto the left-hand side; this yields $C_{\rm III}$ such that
\begin{align}
\| P\p_x f \|_{L^2_{x,v}}^2 & \leq C_{\rm III} \left[ \| P^\perp f \|_{L^2 \mathcal{H}^1_\kappa}^2 + \| \p_x P^\perp f \|_{L^2 \mathcal{H}^1_\kappa}^2 + \varepsilon^2 \| P f \|^2 + \| z \|^2_{L^2 \mathcal{H}^{-1}_{0,\kappa}} \right]  \, .  \label{eq:III}
\tag{III}
\end{align}

We now sum~\eqref{eq:I},~\eqref{eq:II}, and a constant $A_{\rm III}$ multiplied by~\eqref{eq:III}, obtaining 
\begin{align*}
\| P^\perp f \|^2_{L^2\mathcal{H}^1_\kappa} &+  \| P^\perp  \partial_x f \|^2_{L^2\mathcal{H}^1_\kappa} + A_{\rm III} \| P \partial_x f \|^2 \\
&\leq \left( C_{\rm I} + C_{\rm II} + A_{\rm III}C_{\rm III} \right) \left( \varepsilon^2 \| Pf \|^2 + \| z , \partial_x z \|_{L^2\mathcal{H}^1_\kappa}^2 \right) \\
&\qquad +  C_{\rm II} \varepsilon^2 \left(  \| P \partial_x f \|^2 + \| P^\perp f \|_{L^2\mathcal{H}^1_\kappa}^2 \right) + A_{\rm III} C_{\rm III} \left( \| P^\perp f \|_{L^2\mathcal{H}^1_\kappa} + \| P^\perp \partial_x f \|_{L^2\mathcal{H}^1_\kappa}^2 \right) \, .
\end{align*}
Now choose $A_{\rm III}$ such that $A_{\rm III}C_{\rm III} < 1$ and $\varepsilon$ sufficiently small such that the remaining terms on the last line of the right-hand side can be absorbed onto the left-hand side.  This completes the proof of~\eqref{apriori:easy} for $f \in H^1 \mathcal{H}^{1}_\kappa$.

We now prove~\eqref{eq:basicmicroest} by induction on the number of moments.  We assume inductively that for $-9.5 \leq \theta \leq 0$, we have proven that
\begin{equation}\label{apriori:inductive}
\left\| \p_x f \right\|^2_{L^2 \mathcal{H}^1_{\theta,\kappa}} + \left\| P^\perp f \right\|^2_{L^2 \mathcal{H}^1_{\theta,\kappa}} \les \varepsilon^2 \left\| Pf \right\|^2 + \left\| z, \p_x z \right\|^2_{L^2 \mathcal{H}_{\theta,\kappa}^{-1}} \, .
\end{equation}
Our claim is that the above estimate holds with $\theta$ replaced by $\theta-\sfrac 12$.  We first note that $L^\kappa f = L^\kappa P^\perp f$ by~\eqref{pproj}.  Using this observation, we test~\eqref{wed:morning:equation} with $(P^\perp f) \langle v \rangle^{-2(\theta-\sfrac 12)}$.  Applying Lemma~\ref{lem:lemma10:new} and Cauchy-Schwarz, we bound
\begin{equation}\notag
\left| \left\langle \Gamma_\kappa[f,\mu_0^{-\sfrac{1}{2}} (\mu_{\rm NS} - \mu_0)], (P^\perp f) \langle v \rangle^{-2(\theta-\sfrac 12)} \right \rangle_{x,v} \right| \leq (\gamma + C\varepsilon) \| P^\perp f \|^2_{L^2 \mathcal{H}^1_{\theta-\sfrac 12, \kappa}} + C(\gamma) \varepsilon^2 \| P f \|^2 \, ,
\end{equation}
where $\gamma$ will be chosen sufficiently small shortly.  A similar bound holds for the other term with $\Gamma_\kappa$.   For the term $(v_1-s_0)\partial_x f$, we have that
\begin{align*}
\left| \left \langle (v_1-s_0) \p_x f , (P^\perp f) \langle v \rangle^{-2(\theta-\sfrac 12)} \right \rangle_{x,v} \right| & = \left| \left \langle (v_1-s_0) \p_x P f , (P^\perp f) \langle v \rangle^{-2(\theta-\sfrac 12)} \right \rangle_{x,v} \right| \\
& \les \left\| \partial_x Pf \right\| \left\| P^\perp f \right\|_{L^2 \mathcal{H}^1_{\theta,\kappa}} \\
&\les \varepsilon^2 \| Pf \|^2 + \| z , \p_x z \|_{L^2 \mathcal{H}_{\theta,\kappa}^{-1}}^2 \, .
\end{align*}
Now applying these estimates in conjunction with~\eqref{new:co:kappa}, Corollary~\ref{weddy:night} and the inductive assumption, we obtain that
\begin{align}
\delta^2 \| P^\perp f \|^2_{L^2 \mathcal{H}^1_{\theta-\sfrac 12, \kappa}} &\leq \left \langle \langle v \rangle^{-2(\theta-\sfrac 12)} L^\kappa P^\perp f, P^\perp f \right \rangle + C \| P^\perp f \|^2_{L^2 \mathcal{H}^1_{\theta, \kappa}}  \notag \\
&\les  \varepsilon \| Pf \|^2  + \left| \left \langle z, (P^\perp f) \langle v \rangle^{-2(\theta-\sfrac 12)} \right \rangle \right| + \left\|  z , \partial_x z \right\|_{L^2 \mathcal{H}^{-1}_{\theta,\kappa}}^2 \notag \\
\implies \| P^\perp f \|^2_{L^2 \mathcal{H}^1_{\theta-\sfrac 12, \kappa}}  &\les \varepsilon \| P f \|^2 + \left\|  z , \partial_x z \right\|_{L^2 \mathcal{H}^{-1}_{\theta-\sfrac 12,\kappa}}^2 \, .  \label{eq:IV}
\tag{IV}
\end{align}
In order to prove a similar estimate for $\partial_x f$, we apply $\partial_x$ to~\eqref{wed:morning:equation}, test with $\partial_x f \langle v \rangle^{-2(\theta-\sfrac 12)}$, and apply Lemmas~\ref{lem:propertiesofL} and~\ref{lem:lemma10:new}, obtaining that
\begin{align}
\|\partial_x f \|^2_{L^2 \mathcal{H}^1_{\theta-\sfrac 12, \kappa}}  &\leq \delta^{-2} \left \langle L^\kappa \partial_x f , \partial_x f \langle v \rangle^{-2(\theta-\sfrac 12)}  \right \rangle + C \| \partial_x f \|_{L^2 \mathcal{H}^1_{\theta, \kappa}}^2 \notag\\
&\les \left| \left\langle \partial_x \Gamma_\kappa[f,\mu_0^{-\sfrac{1}{2}} (\mu_{\rm NS} - \mu_0)], \partial_ x f \langle v \rangle^{-2(\theta-\sfrac 12)} \right \rangle \right| \\
&\qquad  + \left| \left\langle \partial_x \Gamma_\kappa[\mu_0^{-\sfrac{1}{2}} (\mu_{\rm NS} - \mu_0),f], \partial_ x f \langle v \rangle^{-2(\theta-\sfrac 12)} \right \rangle \right| \notag\\
&\qquad \qquad + \left| \left \langle \partial_x z , \partial_x f \langle v \rangle^{-2(\theta-\sfrac 12)} \right \rangle \right| + \| \partial_x f \|_{L^2 \mathcal{H}^1_{\theta, \kappa}}^2 \notag \\
&\les \varepsilon \| \partial_x f \|_{L^2 \mathcal{H}^1_{\theta-\sfrac 12, \kappa}} \left(\| \partial_x f \|_{L^2 \mathcal{H}^1_{\theta-\sfrac 12, \kappa}} + \varepsilon \| f \|_{L^2 \mathcal{H}^1_{\theta-\sfrac 12, \kappa}} \right) \notag\\
&\qquad + \| \partial_x z \|_{L^2 \mathcal{H}^{-1}_{\theta-\sfrac 12,\kappa}} \| \partial_x f \|_{L^2 \mathcal{H}^{1}_{\theta-\sfrac 12,\kappa}} + \varepsilon^2 \| P f \|^2 + \| z , \p_x z \|_{L^2 \mathcal{H}^{-1}_{\theta,\kappa}}^2 \notag \\
&\implies \|\partial_x f \|^2_{L^2 \mathcal{H}^1_{\theta-\sfrac 12, \kappa}} \leq C \left( \varepsilon^2 \| Pf \|^2 + \varepsilon^2 \| P^\perp f \|_{L^2\mathcal{H}^1_{\theta-\sfrac 12, \kappa}} \right) + \| z , \p_x z \|_{L^2\mathcal{H}^{-1}_{\theta-\sfrac 12,\kappa}}^2 \, .\label{eq:V}
\tag{V}
\end{align}
Summing~\eqref{eq:IV} and~\eqref{eq:V} and using the smallness of $\varepsilon$, we find that
\begin{equation}
\|\partial_x f \|^2_{L^2 \mathcal{H}^1_{\theta-\sfrac 12, \kappa}} + \|P^\perp f \|^2_{L^2 \mathcal{H}^1_{\theta-\sfrac 12, \kappa}} \les \varepsilon^2\| Pf \|^2 + \| z , \partial_x z \|_{L^2\mathcal{H}^{-1}_{\theta-\sfrac 12, \kappa}}^2 \, ,\notag
\end{equation}
concluding the proof.

We now justify the above computations for $f \in L^2 \mathcal{H}^{1}_\kappa$. We mollify the equation in $x$ by convolving against a suitable test function $\psi_\delta(\cdot) = \delta^{-1} \psi(\cdot \delta^{-1})$, where $\delta > 0$. We adopt the notation $f^\delta = \psi_\delta \ast f^\delta$. Then
\begin{align*}
(v_1-s)\p_x f^\delta + L^\kappa f^\delta
&= z^\delta + \underbrace{ \psi_\delta \ast  \left(  \Gamma\left[\mu_0^{-\sfrac12 }(\mu_{\rm NS}-\mu_0), f \right] +   \Gamma\left[f,\mu_0^{-\sfrac12}(\mu_{\rm NS}-\mu_0)\right]  \right)}_{{\rm C}_\delta f} \\
&\qquad  - \underbrace{ \Gamma \left[ \mu_0^{-\sfrac12}(\mu_{\rm NS}-\mu_0), f^\delta \right] -  \Gamma \left[f^\delta , \mu_0^{-\sfrac12}(\mu_{\rm NS}-\mu_0)\right]}_{\tilde C_\delta f} \, .
\end{align*}
The first commutator term can be rewritten as
\begin{equation}\notag
{\rm C}_\delta f(x,\cdot) = \int_\R  \psi_\delta(h) \Gamma\left[ \mu_0^{-\sfrac 12} (\mu_{\rm NS}(x+h,\cdot) - \mu_{\rm NS}(x,\cdot) ),f(x+h,\cdot) \right] \, \dee h \, .
\end{equation}
Crucially,
\begin{equation}\notag
\delta^{-1} \| {\rm C}_\delta f \|_{L^2 \mathcal{H}^{-1}_{0,\kappa}} + \| \p_x {\rm C}_\delta f \|_{L^2 \mathcal{H}^{-1}_{0,\kappa}} \les \| f \|_{L^2 \mathcal{H}^{{-1}}_{0,\kappa}} \, ,
\end{equation}
with dependence on the Lipschitz norm of the `coefficients' $\mu_{\rm NS}$; the proof of this uses duality and Lemma~\ref{lem:lemma10:new}.  The second commutator term $\tilde C_\delta f$ can be handled similarly. Then we perform the same energy estimates on the equation for $f^\delta$. Upon passing $\delta \to 0^+$, we obtain $f \in H^1 \mathcal{H}_\kappa^{{1}}$ from lower semicontinuity.
\end{proof}


We further bootstrap the solution by one degree of $x$-regularity. (The same procedure will be used when we track many $x$-derivatives.)

\begin{lemma}
	\label{lem:bootstrappingonedegreeofreg}
Let $z \in H^2 \mathcal{H}^{-1}_{-10,\kappa}$ be given and assume that $f \in H^1 \mathcal{H}^1_\kappa$ is a solution of~\eqref{def:LinCE:reg:conj}. Then $f \in H^2 \mathcal{H}^1_{-10,\kappa}$, and
\begin{equation}
\label{eq:basicmicroestbootstrappedonce}
\left\| \p_x^2 f \right\|^2_{L^2 \mathcal{H}^1_{-10,\kappa}} + \left\| P^\perp \partial_x f \right\|^2_{L^2 \mathcal{H}^1_{-10,\kappa}} \les \varepsilon^4 \left\| Pf \right\|^2_{L^2} + \left\| \varepsilon z, \p_x z,  \p_x^2 z \right\|^2_{L^2 \mathcal{H}^{-1}_{-10,\kappa}} \, .
\end{equation}
\end{lemma}
\begin{proof}
We differentiate~\eqref{def:LinCE:reg:conj} in $x$: 
\begin{align*}\notag
&(v_1-s_0) \partial_x (\partial_x f) + L^\kappa \partial_x f = \\
&\quad + \p_x \Gamma_\kappa\left[\mu_0^{-\sfrac 12} (\mu_{\rm NS} - \mu_0 ), f\right] + \p_x \Gamma_\kappa\left[f,\mu_0^{-\sfrac 12} (\mu_{\rm NS} - \mu_0 )\right] + \partial_x z \, .
\end{align*}
We apply Lemma~\ref{lem:basicmicroest} to $\partial_x f \in L^2 \mathcal{H}^1_{-10,\kappa}$ to obtain that $\partial_x f \in H^1 \mathcal{H}^1_{-10,\kappa}$ and
\begin{equation}
\label{eq:baselinetopartialf}
\begin{aligned}
&\| \p_x (\partial_x f) \|^2_{L^2 \mathcal{H}^1_{-10,\kappa}} + \| P^\perp \partial_x f \|^2_{L^2 \mathcal{H}^1_{-10,\kappa}} \les \varepsilon^2 \| P\partial_x f \|^2_{L^2} \\
&\, + \left\| \partial_x \Gamma[\mu_0^{-\sfrac 12}(\mu_{\rm NS} - \mu_0),f] \right\|^2_{L^2 \mathcal{H}^{-1}_{-10,\kappa}} + \left\| \partial_x \Gamma[f, \mu_0^{-\sfrac 12}(\mu_{\rm NS} - \mu_0)] \right\|^2_{L^2 \mathcal{H}^{-1}_{-10,\kappa}} + \| \p_x z, \p_x^2 z \|^2_{L^2 \mathcal{H}^{-1}_{-10,\kappa}} \, .
\end{aligned}
\end{equation}
Lemma~\ref{lem:basicmicroest} for $f$ grants us the following bound for the first term on the right-hand side:
\begin{equation}
\label{eq:baselinetopartialf1}
\varepsilon^2 \| P(\partial_x f) \|^2_{L^2} \les \varepsilon^4 \| Pf \|_{L^2} + \varepsilon^2 \| z, \p_x z \|_{L^2 \mathcal{H}^{-1}_{-10,\kappa}} \, .
\end{equation}
Similarly, Lemmas~\ref{lem:basicmicroest} and~\ref{lem:lemma10:new} give that
\begin{equation}
\label{eq:baselinetopartialf2}
\begin{aligned}
 &\left\| \partial_x \Gamma[\mu_0^{-\sfrac 12}(\mu_{\rm NS} - \mu_0),f] \right\|^2_{L^2 \mathcal{H}^{-1}_{-10,\kappa}} + \left\| \partial_x \Gamma[f, \mu_0^{-\sfrac 12}(\mu_{\rm NS} - \mu_0)] \right\|^2_{L^2 \mathcal{H}^{-1}_{-10,\kappa}} \\
 &\quad \les \varepsilon^2 \| f \|_{L^2 \mathcal{H}^1_{-10,\kappa}} + \varepsilon \| \partial_x f \|_{L^2 \mathcal{H}^{1}_{-10,\kappa}} \\
 &\quad \les \varepsilon^2 \| Pf \|_{L^2} + \varepsilon \| z, \p_x z \|_{L^2 \mathcal{H}^{-1}_{-10,\kappa}} \, .
\end{aligned}
\end{equation}
We combine~\eqref{eq:baselinetopartialf}--\eqref{eq:baselinetopartialf2} to complete the estimate.
\end{proof}

We next prove the \emph{a priori} estimate. To relate the macroscopic and microscopic estimate, we will use the linearized Chapman-Enskog expansion from section~\ref{sec:outline}.

\begin{proof}[Proof of Proposition~\ref{pro:basicexistence}, \emph{a priori} estimate]
We assume that $f \in H^2 \mathcal{H}^1_{-10,\kappa}$ solves~\eqref{def:LinCE:reg:conj} with $z \in H^2 \mathcal{H}^{-1}_{-10,\kappa}$ and work towards a proof of \eqref{apriori:est}.  We first note that $f$ satisfies the microscopic estimates~\eqref{eq:basicmicroest} and~\eqref{eq:basicmicroestbootstrappedonce}, which depend on the macroscopic variable $Pf$.  We write $F = \sqrt{\mu_0} f$ and $\mathcal{E} = \sqrt{\mu_0} z$, so that $F$ solves the (unconjugated) linearized equation~\eqref{def:LinCE:reg} with right-hand side $\mathcal{E}$. Define $U = \sf{P}_{\mu_0} F$ and write the decomposition
\begin{equation}\notag
F = \sf{P}_{\mu_{\rm NS}} U + V \, ,
\end{equation}
which serves to define the purely microscopic $V$. From~\eqref{becomes}, we have that
\begin{equation}\notag
\p_x ([(\nabla H)(U_{\rm NS}) - sI] U) - \p_x (B_\kappa(U_{\rm NS}) \p_x U) = \p_x \left( G + (B_0-B_\kappa)(U_{\rm NS}) \partial_x U \right) \, ,
\end{equation}
where $H$, $B_\kappa$, and $G$ have already been defined in subsubsection~\ref{sss:macro}.  As already mentioned, by a constant linear change of coordinates $U \mapsto I[U]$, we may rewrite the equation in the variables $(\rho,m,E) \in \R \times \R^3 \times \R$ instead of $U \in \mathbb{U}$, and the equation becomes a variant of the linearized compressible Navier-Stokes equations, but without the term $\p_x (B_\kappa'(U_{\rm NS}) U \p_x U_{\rm NS})$, which is perturbative in our arguments. 

The macroscopic estimate from~Proposition~\ref{pro:macroscopicestimate} yields that
\begin{equation}\notag
\| U \|_{L^2} \lesssim \varepsilon^{-1} \left( \| G \|_{L^2} + \|  (B_0-B_\kappa)(U_{\rm NS}) \partial_x U \|_{L^2} \right) + |d| \, .
\end{equation}
By the definition~\eqref{gdef} of $G$ and the fact that $P v_1$ is bounded on $\mathcal{H}^1_{\kappa}$, we have that
\begin{equation}\notag
\| G \|_{L^2} \lesssim \varepsilon^2 \| U \|_{L^2} + \| \mu_{0}^{-\sfrac 12} (v_1-s_0) \p_x V \|_{L^2 \mathcal{H}^{-1}_\kappa} + \| \mu_{0}^{-\sfrac 12} \mathcal{E} \|_{L^2 \mathcal{H}^{-1}_\kappa} \, ,
\end{equation}
where we have used that $\| \p_x \sf{P}_{\mu_{\rm NS}} \| \les \varepsilon^2$ from~\eqref{eq:decay} and Lemma~\ref{lem:inverttofindg}.  In order to estimate $\mu_0^{-\sfrac 12} (v_1-s_0) \p_x V$, we use Lemma~\ref{l:norm:control} and Lemma~\ref{lem:bootstrappingonedegreeofreg}, which gives that
\begin{align}
\| \mu_0^{-\sfrac 12} (v_1-s_0) \partial_x V \|_{L^2 \mathcal{H}^{-1}_\kappa} &\les \|\mu_0^{-\sfrac 12} \langle v \rangle^{5} (v_1-s_0) \partial_x V\|_{L^2} \notag \\
&\les \| \mu_0^{-\sfrac{1}{2}} \p_x V \|_{L^2 \mathcal{H}^1_{-10,\kappa}} \notag \\
&\les  \| \p_x \mathsf{P}^\perp_{\mu_0} \mathsf{P}_{\mu_{\rm NS}} U \| + \| P^\perp \mu_0^{-\sfrac 12} \partial_x F \|_{L^2 \mathcal{H}^1_{-10,\kappa}} \notag \\
&\les \varepsilon^2 \| U \|_{L^2} + \| \varepsilon \mu_0^{-\sfrac 12} \mathcal{E} , \p_x \mu_0^{-\sfrac 12} \mathcal{E} , \p_x^2 \mu_0^{-\sfrac 12} \mathcal{E} \|_{L^2\mathcal{H}^1_{-10,\kappa}} \, .  \label{eq:doesn't:close}
\end{align}
Finally, note that the a priori estimate~\eqref{eq:basicmicroest} for $\partial_x f$ and~\eqref{op:est} bounds
$$  \|  (B_0-B_\kappa)(U_{\rm NS}) \partial_x U \|_{L^2} \les \kappa \| \p_x U \|_{L^2}  \les \kappa\varepsilon \| U \|_{L^2} + \kappa \| \mu_0^{-\sfrac 12} \mathcal{E} , \p_x \mu_0^{-\sfrac 12} \mathcal{E} \|_{L^2\mathcal{H}^1_{-10,\kappa}} \, . $$

In summary,
\begin{equation}\notag
\| U \|_{L^2}  \les (\varepsilon+\kappa) \| U \|_{L^2} + \varepsilon^{-1} \| \mu_{0}^{-\sfrac 12} \mathcal{E} \|_{H^2 \mathcal{H}^{-1}_{-10,\kappa}} + |d| \qquad \implies \qquad
\| U \|_{L^2} \lesssim \varepsilon^{-1} \| z \|_{H^2 \mathcal{H}^{-1}_{-10,\kappa}} + |d|\, .
\end{equation}
Note that the microscopic estimate crucially allowed us to control $\p_x f$, and in particular, $\p_x U = \p_x \sf{P}_{\mu_0} F$, so that the above estimate closes. Inserting this macroscopic bound into the microscopic estimate~\eqref{eq:basicmicroest} and~\eqref{eq:basicmicroestbootstrappedonce} completes the proof.
\end{proof}
\begin{remark}\label{rem:improved}
We note that the macroscopic piece in fact satisfies a bound with constant $O(\varepsilon^{-1})$, whereas the microscopic piece satisfies a bound with constant $O(1)$. 
\end{remark}

\subsection{Linear estimates: Bootstrapping}\label{ss:le:boot}
\subsubsection{Unweighted estimates}

\begin{lemma}[Bootstrapping $x$-derivatives]\label{bootstrapping:x}
Let $N \geq 2$ and $z \in H^2 \mathcal{H}^{-1}_{-10} \cap H^N \mathcal{H}^{-1}$. If $f \in H^2 \mathcal{H}^1_{-10}$ is the unique solution to~\eqref{eq:NLSL} with $\ell_\varepsilon(Pf(0)) = 0$ as in Proposition~\ref{pro:basicexistence}, then $f \in H^N_x \mathcal{H}^1$, and
\begin{equation}\notag
\| f \|_{H^N_\varepsilon \mathcal{H}^1} \les \varepsilon^{-1} (\| z \|_{H^2_\varepsilon \mathcal{H}^{-1}_{-10}} +  \| z \|_{H^N_\varepsilon \mathcal{H}^{-1}} ) \, .
\end{equation}
\end{lemma}
\begin{proof}
This will be a consequence of Proposition~\ref{pro:basicexistence} and bootstrapping via Lemma~\ref{lem:basicmicroest}. The case $N=2$ is already verified in Proposition~\ref{pro:basicexistence}. Suppose the statement is already known already for $N \geq 2$. To prove the $N+1$ statement, we simply apply $N$ derivatives to the equation:
\begin{equation}\notag
(v_1 - s_0) \p_x^N f + L_{\rm NS} \p_x^N f = \sum_{k=1}^N c(k,N) \left(  \Gamma [\mu_0^{-\sfrac 12} \p_x^k \mu_{\rm NS}, \p_x^{N-k} f] + \Gamma [\p_x^{N-k} f, \mu_0^{-\sfrac 12} \p_x^k \mu_{\rm NS}] \right)  + \p_x^N z \, .
\end{equation}
Since
\begin{equation}\notag
\| \Gamma [\mu_0^{-\sfrac 12} \p_x^k \mu_{\rm NS}, \p_x^{N-k} f] \|_{L^2 \mathcal{H}^{-1}} \lesssim \varepsilon^{k+1} \| \p_x^{N-k} f \|_{L^2 \mathcal{H}^1} \, ,
\end{equation}
and similarly for the other term with $\Gamma$, the induction hypothesis and Lemma~\ref{lem:basicmicroest} (more specifically, the beginning of the proof, containing the estimates without 10 moments) yield the statement for $N+1$.
\end{proof}

\begin{lemma}[Bootstrapping $v$-derivatives]\label{lemma:boot:v}
Let $N \geq 2$, and suppose that $z \in H^2 \mathcal{H}^{-1}_{-10} \cap \mathbb{Y}_\varepsilon^N$. If $f \in H^2 \mathcal{H}^1_{-10}$ is the unique solution to~\eqref{eq:NLSL} with $\ell_\varepsilon(Pf(0)) = 0$ as in Proposition~\ref{pro:basicexistence}, then
\begin{equation}\notag
 \| f \|_{\mathbb{X}^N_\varepsilon} \lesssim \varepsilon^{-1} (\| z \|_{H^2_\varepsilon \mathcal{H}^{-1}_{-10}} + \| z \|_{\mathbb{Y}^N_\varepsilon} ) \, .
\end{equation}
\end{lemma}

\begin{remark}
Technically, you can pass derivatives in $v_2$ and $v_3$ through the term $v_1 \p_x f$, so it should be possible to obtain better localization in $v_2$ and $v_3$. For convenience, we use the isotropic weights in $v$.
\end{remark}

\begin{proof}
We begin by assuming that $f$ qualitatively belongs to the appropriate spaces; that is, the left-hand side of~\eqref{eq:bootstrappingvderiv} is assumed to be finite. We mention how to justify the estimates afterward.

Let $N\geq 2$.  Notice that when $|\beta|=0$, we have a satisfactory bound for $\partial^\alpha f$ for all $0 \leq \alpha \leq N$ from Lemma~\ref{bootstrapping:x}.  We therefore treat all $\alpha + \abs{\beta} \leq N $ by inducting on $|\beta|$; for the moment we assume additionally that $\beta \cdot e_1 \geq 1$, which is the hardest case.  We apply $\p^\alpha_\beta$ to the equations, multiply by ${\rm w}(|\beta|,0,0)(v) = \langle v \rangle^{-|\beta|}$, pair with ${\rm w} \p^\alpha_\beta f$, and integrate by parts in $x$ in the $(v_1-s_0)\p_x $ term, obtaining that
\begin{align*}
&\brak{ {\rm w}^2 \partial_\beta^\alpha f, \partial_x \partial_{\beta - e_1}^\alpha f} + \brak{{\rm w}^2 \partial_{\beta}^\alpha f, \partial_{\beta}^\alpha L f} \\
& \, =  \brak{{\rm w}^2 \partial_{\beta}^\alpha f, \partial_{\beta}^\alpha \Gamma [\mu_0^{-\sfrac 12}(\mu_{\rm NS} - \mu_0),f]} + \brak{\partial_{\beta}^\alpha \Gamma [\mu_0^{-\sfrac 12}(\mu_{\rm NS} - \mu_0),f],{\rm w}^2 \partial_{\beta}^\alpha f} + \brak{{\rm w}^2 \partial_{\beta}^\alpha f, \partial_{\beta}^\alpha z } \, . 
\end{align*}
Applying Lemma~\ref{lem:propertiesofL} and recalling~\eqref{eq:sigma:2}, we then have that
\begin{align*}
\| \p_\beta^\alpha f \|_{\mathcal{H}^1_\ell}^2 &\leq \eta \sum_{|\beta_1|=|\beta|} \| \p_{\beta_1}^\alpha f \|_{\mathcal{H}^1_\ell}^2  + C(\eta) \sum_{|\beta_1|<|\beta|} \| \p_{\beta_1}^\alpha f \|_{\mathcal{H}^1_{\ell-(|\beta|-|\beta_1|)}}^2  \\ 
&\qquad +  \brak{{\rm w}^2 \partial_{\beta}^\alpha f, \partial_{\beta}^\alpha \Gamma [\mu_0^{-\sfrac 12}(\mu_{\rm NS} - \mu_0),f]} + \brak{\partial_{\beta}^\alpha \Gamma [\mu_0^{-\sfrac 12}(\mu_{\rm NS} - \mu_0),f],{\rm w}^2 \partial_{\beta}^\alpha f}\\
&\qquad  + \brak{{\rm w}^2 \partial_{\beta}^\alpha f, \partial_{\beta}^\alpha z } - \brak{ {\rm w}^2 \partial_\beta^\alpha f, \partial_{x} \partial_{\beta - e_1}^\alpha f}  \, .
\end{align*}
The transport term is the reason for the negative polynomial weights (see~\cite{G02}), and may be estimated using Lemma~\ref{l:norm:control} and~\eqref{eq:sigma:2} by
\begin{align*}
\brak{{\rm w}^2 \partial_\beta^\alpha f, \partial^{x} \partial_{\beta - e_1}^\alpha f} \lesssim \norm{\partial_\beta^\alpha f}_{\mathcal{H}^1_\ell} \norm{\partial^{\alpha+1}_{\beta-e_1} f}_{\mathcal{H}^1_{\ell-1}} \, .
\end{align*}
The nonlinear term can be estimated using Lemmas~\ref{lem:lemma10:new} and~\ref{lem:mugests}. Repeating the above argument in the case $\beta\cdot e_1=0$, for which the transport term vanishes, produces an identical estimate.  Then inducting on $\abs{\beta}$ and using that the base case (with $|\beta|=0$ and $N$ fixed) is covered by Proposition~\ref{pro:basicexistence}, we ultimately obtain the desired velocity-weighted estimate.

We now mention how to prove the \emph{qualitative} assumptions which justify the energy estimates. We employ a variation on the strategy in Remark~\ref{rmk:technicalbootstrappingremark}. Suppose that $f, \p_x f \in L^2 \mathcal{H}^1$ and $z \in \mathbb{Y}^1_\varepsilon$, and we wish to demonstrate that $\p_\beta f \in L^2 \mathcal{H}^1_1$ for $|\beta| = 1$. We write
\begin{equation}
	\label{eq:fulllefthandsideinvertible}
\underbrace{(v_1 - s_0) \p_x f - \tilde{A} f}_{=: Tf} + \Gamma [\mu_0^{-\sfrac 12}  \mu_{\rm NS}, f] + \Gamma [ f, \mu_0^{-\sfrac 12} \mu_{\rm NS}] = \chi_R A_2f + Kf + z
\end{equation}
where $\tilde{A} = A_1 + (1- \chi_R) A_2$. The operator $T$ on the left-hand side maps $\mathbb{X}^1_\varepsilon$ into $\mathbb{Y}^1_\varepsilon$ and is invertible, with $\| T^{-1} \|_{\mathbb{Y}^1_\varepsilon \to \mathbb{X}^1_\varepsilon} \lesssim 1$; the $\Gamma$ terms map $\mathbb{X}^1_\varepsilon$ into $\mathbb{Y}^1_\varepsilon$ and have operator norm $\lesssim \varepsilon$. Therefore, the full left-hand side of~\eqref{eq:fulllefthandsideinvertible} is invertible, and the right-hand side belongs to $\mathbb{Y}^1_\varepsilon$, so that $f$ qualitatively belongs to $\mathbb{X}^1_\varepsilon$. This strategy bootstraps by 1 in the upper index of the $\mathbb{X}$ spaces each iteration.
\end{proof}

\subsubsection{Strong velocity weights}

\begin{lemma}\label{lem:bs:1}
Let $N \geq 2$, and suppose that $z \in \mathbb{Y}^N_{\varepsilon,\rm w}$. If $f \in H^2 \mathcal{H}^1_{-10}$ is the unique solution to~\eqref{eq:NLSL} with $\ell_\varepsilon(Pf(0)) = 0$ as in Proposition~\ref{pro:basicexistence}, then
\begin{equation}
	\label{eq:bootstrappingvderiv}
\| f \|_{\mathbb{X}^N_{\varepsilon, \rm w}} \les \varepsilon^{-1} \| z \|_{\mathbb{Y}^N_{\varepsilon, \rm w}}  \, .
\end{equation}
\end{lemma}

\begin{proof}
First, we estimate $x$-derivatives.  We apply $\partial^\alpha$ to the equation, multiply by $\rm w$, and pair with ${\rm w} \partial^\alpha f$, obtaining that
\begin{align*}
\left \langle {\rm w}^2 \p^\alpha f , L(\p^\alpha f) \right \rangle &= \left \langle {\rm w}^2 \p^\alpha z , \p^\alpha f \right \rangle \\
&\qquad + \left \langle {\rm w}^2 \p^\alpha \Gamma [(\mu_{\rm NS} - \mu_0) \mu_0^{-\sfrac 12}, f] , \p^\alpha f \right \rangle + \left \langle \p^\alpha f , {\rm w}^2 \p^\alpha \Gamma [(\mu_{\rm NS} - \mu_0) \mu_0^{-\sfrac 12}, f] \right \rangle \, . 
\end{align*}
Applying~\eqref{GS:Lemma:9b} from Lemma~\ref{lem:propertiesofL}, we find that
\begin{align}
\| \p^\alpha f \|_{\mathcal{H}^1_{0, \rm w}}^2 &\les \| \p^\alpha f \|_{\mathcal{H}^1}^2 + \left \langle {\rm w}^2 \p^\alpha z , \p^\alpha f \right \rangle \notag\\
&\qquad + \left \langle {\rm w}^2 \p^\alpha \Gamma [(\mu_{\rm NS} - \mu_0) \mu_0^{-\sfrac 12}, f] , \p^\alpha f \right \rangle + \left \langle \p^\alpha f,  {\rm w}^2 \p^\alpha \Gamma [(\mu_{\rm NS} - \mu_0) \mu_0^{-\sfrac 12}, f] \right \rangle \, . \label{coming:back}
\end{align}
In order to estimate the second term on the right-hand side, we claim for the moment that $\| {\rm w}\p^\alpha f \|_{\mathcal{H}^1_\ell} \les \| \p^\alpha f \|_{\mathcal{H}^1_{\ell,  \rm w}}$, in which case this term may be bounded by an implicit constant multiplied by $\| z \|_{\mathcal{H}^{-1}_{0, \rm w}} \| f \|_{\mathcal{H}^1_{0, \rm w}}$.  In order to prove the claim, we compute:
\begin{align*}
\| {\rm w} \p^\alpha f \|_{\mathcal{H}^1_\ell}^2  &= \int_{\mathbb{R}^3} \sigma^{ij} {\rm w}^2 \bigg{[} v_i v_j  (\p^\alpha f)^2 \left( 1+ \frac{q^2}{4} \right) \\
&\qquad  + \p_i \p^\alpha f \p_j \p^\alpha f + \p_i \p^\alpha f \p^\alpha f \frac{q v_j}{2} + \p_j \p^\alpha f \p^\alpha f \frac{q v_i}{2} \bigg{]} \, \dee v \, .
\end{align*}
The first two terms (those for which no derivatives have landed on $\p^\alpha f$, or all derivatives have landed on $\p^\alpha f$) are easily controlled directly by $\| \p^\alpha f \|_{\mathcal{H}^1_{\ell, \rm w}}$.  To control the cross terms, we have that
\begin{align*}
\int_{\mathbb{R}^3} \sigma^{ij} {\rm w}^2  \p_j \p^\alpha f \p^\alpha f \frac{q v_i}{2} \, \dee v  \les \left( \int_{\mathbb{R}^3} \left[ \sigma^{ij} v_j \langle v \rangle^{\sfrac 12} \p_i \p^\alpha f {\rm w} \right]^2 \dee v \right)^{\sfrac 12} \left( \int_{\mathbb{R}^3} \left[ \langle v \rangle^{-\sfrac 12} \p^\alpha f {\rm w} \right]^2 \, \dee v \right)^{\sfrac 12} \, .
\end{align*}
The second term in the product above can be controlled using $\| f \|_{\mathcal{H}^1_{\ell, \rm w}}$.  We split the first term using $\p_i \p^\alpha f = \frac{v_i v_k}{|v|^2} \p_k \p^\alpha f + \left( \delta_{ik} - \frac{v_i v_k}{|v|^2} \right) \p_k \p^\alpha f$.  For the portion coming from the first term, we use \cite[Lemma 3]{G02}, which asserts that $\sigma^{ij} v_i v_j \sim |v|^{-1}$, and Lemma~\ref{l:norm:control}, to bound
\begin{align*}
\int_{\mathbb{R}^3} \left[ \sigma^{ij} v_j \langle v \rangle^{\sfrac 12} \frac{v_i v_k}{|v|^2} \p_k \p^\alpha f {\rm w} \right]^2 \dee v \les \int_{\mathbb{R}^3} \left[ \langle v \rangle^{-\sfrac 32} \p_k \p^\alpha f {\rm w} \right]^2 \dee v \les \| f \|_{\mathcal{H}^1_{\ell \rm w}} \, .
\end{align*}
For the portion coming from the second term, we again use \cite[Lemma 3]{G02}, which asserts that $|\sigma^{ij} v_j| \les \langle v \rangle^{-1}$, and Lemma~\ref{l:norm:control}, to bound 
\begin{align*}
\int_{\mathbb{R}^3} \left[ \sigma^{ij} v_j \langle v \rangle^{\sfrac 12} \left( \delta_{ik} - \frac{v_i v_k}{|v|^2} \right) \p_k \p^\alpha f {\rm w} \right]^2 \dee v \les \int_{\mathbb{R}^3} \left[ \langle v \rangle^{-\sfrac 12} \left( I - P_v \right) \p_k \p^\alpha f {\rm w} \right]^2 \dee v \les \| f \|_{\mathcal{H}^1_{\ell, \rm w}} \, .
\end{align*}
Finally, coming back to~\eqref{coming:back}, we apply Lemma~\ref{bootstrapping:x} and Corollary~\ref{weddy:night} to the first term, Lemma~\ref{lem:lemma10:new} to the second term, and the above computations to the third term, deducing that
\begin{align*}
\| \p^\alpha f \|_{\mathcal{H}^1_{0, \rm w}}^2 \les \eps^{-2} \| \p^\alpha z \|_{\mathcal{H}^{-1}_{0, \rm w}}^2 + \sum_{\alpha'<\alpha} \| \p^{\alpha'} f \|_{\mathcal{H}^1_{0, \rm w}}^2 \, .
\end{align*}

Now that we have bounds for all $x$-derivatives, we repeat the exact same proof as in the previous lemma; the presence of the Gaussian weights changes the definition of $\rm w$ but no other aspect of the remainder of the proof.
\end{proof}

\subsubsection{Strong spatial weights}

In this section, we obtain the stretched exponential decay estimates in $x$. 


\begin{lemma}\label{lem:bs:2}
Let $N \geq 2$ and $0 < \delta \ll 1$. Suppose that $z \in \mathbb{Y}^N_{\varepsilon, \rm w}$ and $e^{\delta \langle \varepsilon x \rangle^{\sfrac{1}{2}}} z \in \mathbb{Y}^N_\varepsilon$. If $f \in H^2 \mathcal{H}^1_{-10}$ is the above unique solution, then~\eqref{x:weighted} holds; that is,
\begin{equation}\notag
\left\| e^{\delta \langle \varepsilon x \rangle^{\sfrac{1}{2}}} f \right\|_{\mathbb{X}^N_\varepsilon} \lesssim \varepsilon^{-1} \left( \| z \|_{\mathbb{Y}^N_{\varepsilon, \rm w}} + \| e^{\delta \langle \varepsilon x \rangle^{\sfrac{1}{2}}} z \|_{\mathbb{Y}^N_\varepsilon} \right) \, .
\end{equation}
\end{lemma}

\begin{proof}
We consider the equation satisfied by $e^{\omega_\varepsilon} f$, where
\begin{equation}\notag
\omega(x) = \delta \left( 2 \langle x \rangle^{\frac{1}{2}-\kappa} - \langle x \rangle^{\frac{1}{2}} \right) \, ,
\end{equation}
$\omega_\varepsilon(x) = \omega(\varepsilon x)$, and we suppress the dependence of the notation on $0 < \kappa \ll 1$. We have
\begin{equation}\notag
(v_1 - s_0)\p_x (e^{\omega_\varepsilon} f) + L_{\rm NS} (e^{\omega_\varepsilon} f) = e^{\omega_\varepsilon} z + (v_1 - s_0) (\p_x \omega_\varepsilon) e^{\omega_\varepsilon} f \, .
\end{equation}
Notably, $\omega_\varepsilon f$ satisfies the one-dimensional constraint at $x = 0$. Moreover, the right-hand side is purely microscopic, since applying $P$ to the original equation yields $\p_x [P (v_1-s_0) f] = 0$, giving that $P (v_1 - s_0) f \in L^2$ is constant in $x$, i.e., $P (v_1 - s_0) f = 0$. Therefore, Lemma~\ref{lemma:boot:v} gives that $\omega_\varepsilon f$ satisfies the unweighted estimates
\begin{equation}
	\label{eq:unweightedestimates}
 \| e^{\omega_\varepsilon} f \|_{\mathbb{X}^N_\varepsilon} \lesssim  \| e^{\omega_\varepsilon} z \|_{\mathbb{Y}^N_\varepsilon} + \| [ (v_1 - s_0) (\p_x \omega_\varepsilon) e^{\omega_\varepsilon} f]  \|_{_{\mathbb{Y}^N_\varepsilon}} \, .
\end{equation}
We wish to estimate the second term on the right-hand side uniformly as $\kappa \to 0^+$.

Consider the case $|\alpha| = |\beta| = 0$, which demonstrates the reason for the exponent $\sfrac{1}{2}$ in the stretched exponential weight. We have the property
\begin{equation}\notag
|\p_x \omega_\varepsilon(x)| = \varepsilon | (\p_x \omega)(\varepsilon x) | \lesssim \varepsilon \delta \langle \varepsilon x \rangle^{-\sfrac 12} \, .
\end{equation}
We therefore wish to estimate
\begin{equation}\notag
\varepsilon \delta \| \langle \varepsilon x \rangle^{-\sfrac 12} (v_1 - s_0) e^{\omega_\varepsilon} f \|_{L^2\mathcal{H}^{-1}} \les \varepsilon \delta \| \langle \varepsilon x \rangle^{-\sfrac 12} \langle v \rangle^{\sfrac 12} (v_1 - s_0) e^{\omega_\varepsilon} f \|_{L^2_{x,v}} \, ,
\end{equation}
where we have applied Lemma~\ref{l:norm:control} in order to control the dual norm.  We write $f = Pf + P^\perp f$. We clearly have
\begin{equation}\notag
\varepsilon \delta \| \langle \varepsilon x \rangle^{-\sfrac 12} \langle v \rangle^{\sfrac 12} (v_1 - s_0) e^{\omega_\varepsilon} Pf \|_{L^2_{x,v}} \lesssim \varepsilon \delta \| e^{\omega_\varepsilon} f \|_{L^2 \mathcal{H}^1} \, .
\end{equation}
The $P^\perp f$ term is more subtle. Recall that $\| \langle v \rangle^{-\sfrac 12} f \|_{L^2_v} \lesssim \| f \|_{\mathcal{H}^1}$. Hence, when $\langle \varepsilon x \rangle^{\sfrac 12} \geq \langle v \rangle^2$, we have
\begin{equation}\notag
\varepsilon \delta \| \mathbf{1}_{\{ \langle \varepsilon x \rangle^{\sfrac 12} \geq \langle v \rangle^2 \}} \langle \varepsilon x \rangle^{-\sfrac 12} \langle v \rangle^{\sfrac 12} (v_1 - s_0) e^{\omega_\varepsilon} P^\perp f \|_{L^2_{x,v}} \lesssim \varepsilon \delta \| e^{\omega_\varepsilon} f \|_{L^2 \mathcal{H}^1} \, .
\end{equation}
In the region $\langle \varepsilon x \rangle^{\sfrac 12} \leq \langle v \rangle^2$, we have $\langle \varepsilon x \rangle^{-\sfrac 12} e^{\delta \langle \varepsilon x \rangle^{\sfrac{1}{2}}} \lesssim \langle v \rangle^{-2} e^{q_0 \langle v \rangle^2/4}$, if $\delta$ is chosen sufficiently small. Therefore,
\begin{equation}\notag
\varepsilon \delta \| \mathbf{1}_{\{ \langle \varepsilon x \rangle^{\sfrac 12} \leq \langle v \rangle^2 \}} \langle \varepsilon x \rangle^{-\sfrac 12} \langle v \rangle^{\sfrac 12} (v_1 - s_0) e^{\omega_\varepsilon} P^\perp f \|_{L^2} \lesssim \varepsilon \delta \| f \|_{L^2 \mathcal{H}^1_{\rm w}} \, .
\end{equation}

We now consider the general case and estimate
\begin{equation}
	\label{eq:thingtobeboundedrightnow}
\varepsilon^\alpha \| \p_x^\alpha (\p_x \omega_\varepsilon) \langle v \rangle^{\sfrac 12} \langle v \rangle^{-|\beta|} \p_\beta [(v_1 - s_0) e^{\omega_\varepsilon} f] \|_{L^2}
\end{equation}
with $|\alpha| + |\beta| \leq N$. Since
\begin{equation}\notag
|\p_x^k \omega_\varepsilon| \lesssim_k \varepsilon^k \delta ( \langle \varepsilon x \rangle^{\sfrac 12-\kappa} ) \, ,
\end{equation}
\eqref{eq:thingtobeboundedrightnow} is bounded by
\begin{equation}\notag
C \sum_{\alpha'=0}^\alpha \varepsilon^{\alpha'+1} \delta \| \langle \varepsilon x \rangle^{-\sfrac 12} \langle v \rangle^{\sfrac 32-|\beta|} e^{\omega_\varepsilon} (|\p^{\alpha'}_{\beta-e_1} f| + |\p^{\alpha'}_\beta f|) \|_{L^2} \, .
\end{equation}
By the same splitting procedure as above, this term is bounded by
\begin{equation}\notag
C \varepsilon \delta \left[ \sum_{|\beta'|\leq |\beta|} \| e^{\omega_\varepsilon} f \|_{H^{|\alpha|}_\varepsilon \mathcal{H}^{1}_{|\beta'|}} + \sum_{|\beta'|\leq |\beta|} \| f \|_{L^2 \mathcal{H}^1_{|\beta|,w}} \right] \, .
\end{equation}
If we suppose that $\varepsilon \leq 1$ and choose $0 < \delta \ll 1$, then the $e^{\omega_\varepsilon f}$ terms can be absorbed into the left-hand side of~\eqref{eq:unweightedestimates}. Finally, we send $\kappa \to 0^+$. \end{proof}

\section{Linear existence}\label{sekk:linear}


Following~\cite{MZ08}, we prove existence by Galerkin approximation. However, as discussed in subsubsection~\ref{sec:existencediscussion}, the Galerkin approximation is most natural in the space $H^2 \mathcal{H}^1$, which is not evidently strong enough to close the \emph{a priori} estimates in the Landau setting, where we track $10$ additional moments (see~\eqref{eq:doesn't:close}). Therefore, the entirety of this section treats the regularized collision operator $Q_\kappa$ with $0 < \kappa \ll 1$. 

The initial goal in this section is to prove existence for the $\kappa$-regularized linearized equation:\index{$\kappa$}\index{$Q_\kappa$}
\begin{proposition}
\label{pro:basicexistenceforkappareg}
Let $0 < \kappa \ll 1$ and $0 < \varepsilon \ll_\kappa 1$. Let $z \in H^2 \mathcal{H}^{-1}_\kappa$ be purely microscopic and $d \in \R$. Then there exists a unique solution $f \in H^2 \mathcal{H}^1_\kappa$ to the $\kappa$-regularized linearized equation
\begin{equation}\notag
(v_1 - s_0) \p_x f + L_{\rm NS}^\kappa f = z
\end{equation}
with one-dimensional constraint
\begin{equation}\notag
\ell_\varepsilon({Pf(0)}) = d \, .
\end{equation}
\end{proposition}

\emph{A priori} the smallness condition on $\varepsilon$ may depend on the regularization parameter $\kappa$, but we remove this restriction in Section~\ref{sec:methodofcontinuity}. Then we take $\kappa \to 0^+$ to complete the proof of Proposition~\ref{lem:MainLinear}.

\subsection{Quantitative Galerkin approximation}

Recall the \emph{Hermite functions}
\begin{equation}
	\label{eq:Hermitefns}
H_\alpha := c_\alpha (-1)^{|\alpha|} \p^\alpha e^{-\frac{|v|^2}{4}} \, ,
\end{equation}
indexed by the multi-index $\alpha \in \N_{\geq 0}^3$, where $c_\alpha > 0$ ensures $L^2$-normalization. The Hermite functions form an orthonormal basis of $L^2$: Consider the operator $T := - \Delta_v + |v|^2/4$ on $D(T) := H^2 \cap \{ f \in L^2 : |v|^2 f \in L^2 \} \subset L^2(\R^3)$. It is well known that, in the above functional framework, $T$ is closed, densely defined, and self-adjoint, and it has compact resolvent. It has an associated $L^2$ orthonormal basis $(H_\alpha)$ of eigenfunctions with corresponding eigenvalues $\lambda_\alpha := (3 + 2|\alpha|)/2$. See, for example,~\cite[p. 119]{ZworskiBook}. Let $\Pi_{\leq N}$ be the $L^2$-orthogonal projection onto ${\rm span} \, \{ H_\alpha : 0 \leq |\alpha| \leq N \}$ and $\Pi_{> N} = I - \Pi_{\leq N}$. Notably, the image of $\Pi_{\leq 3}$ contains $M_9$. These projections will be applied in the conjugated variables, which is the reason for the particular choice of Gaussian in~\eqref{eq:Hermitefns}. The projections are well adapted to the linearized operator at $\mu_0$.

We consider the approximate system
\begin{equation}
	\label{eq:approximatesystem}
-\eta \p_x^2 f + \Pi_{\leq N} (v_1-s_0) \Pi_{\leq N} \p_x f + \Pi_{\leq N} L_{\rm NS}^\kappa \Pi_{\leq N} f = \Pi_{\leq N} z \, ,
\end{equation}
with one-dimensional phase condition
\begin{equation}
	\label{eq:phaseconditionapp}
\ell_\varepsilon(f(0)) = d \, .
\end{equation}
This is the Galerkin truncation of~\eqref{def:LinCE:reg:conj} (or equivalently the $\kappa$-regularization of~\eqref{eq:NLSL}) with an additional elliptic term $-\eta \p_x^2 f$. Let us give names to the truncated linear operators:
\begin{equation}\notag
A_N := -\eta \p_x^2 f + \Pi_{\leq N} (v_1-s_0) \Pi_{\leq N} \, , \quad  L_{{\rm NS},N}^\kappa  := \Pi_{\leq N}  L_{\rm NS}^\kappa  \Pi_{\leq N} \, .
\end{equation}
Additionally, the right-hand side is $z_N := \Pi_{\leq N} z$. In what follows, we suppose $z \in H^2 \mathcal{H}_{{\kappa}}^{-1}$ is purely microscopic.

We begin with the \emph{a priori} estimates for the Galerkin-truncated system.
\begin{proposition}
	\label{pro:aprioriapproximatestimates}
Let $0 < \kappa \ll 1$, $0 < \varepsilon \ll_\kappa 1$, $N \geq 10$, and $\eta \geq 0$. Suppose that $f \in H^1 \mathcal{H}^1$ is a solution of the approximate system~\eqref{eq:approximatesystem}-\eqref{eq:phaseconditionapp} with $f = \Pi_{\leq N} f$. Then $f \in H^2 \mathcal{H}^1_\kappa$ and
\begin{equation}\notag
\| f \|_{H^2_\varepsilon \mathcal{H}^1_\kappa} \lesssim_\kappa \varepsilon^{-1} \| z \|_{H^2_\varepsilon \mathcal{H}^{-1}_\kappa} \, .
\end{equation}
\end{proposition}

Before proving~\ref{pro:aprioriapproximatestimates}, we present the following analogue of Lemma~\ref{lem:inverttofindg}.

\begin{lemma}
\label{lem:invertibilitytofindgn}
Let $\mu_1 = \mu(\varrho_1,u_1,\theta_1)$ be a Maxwellian near the standard Maxwellian $\mu_0$; that is, $|(\varrho_1-1,u_1,\theta_1-1)| \ll 1$. Let $N \geq 10$ and $\kappa \geq 0$. Then the linearized operator
\begin{equation}\notag
L_{\mu_1,N}^\kappa : f \mapsto - \Pi_{\leq N} \Gamma_\kappa\left[\mu_0^{-\sfrac 12}\mu_1, f \right] - \Pi_{\leq N} \Gamma_\kappa\left[f, \mu_0^{-\sfrac 12}\mu_1 \right]
\end{equation}
maps $\mathcal{H}^1_{\kappa,N} := \mathcal{H}^1_{\kappa,N} \cap {\rm image} (\Pi_{\leq N}) $ to $\mathcal{H}^{-1}_{\kappa,N} := \mathcal{H}^{-1}_{\kappa,N} \cap {\rm image} (\Pi_{\leq N}) $. Moreover, this operator, restricted to the (conjugated) microscopic subspace $\mathcal{H}^1_{{\rm mic},\kappa,N}$, 
\begin{equation}
	\label{eq:mymappingprop}
L_{\mu_1,N}^\kappa : \mathcal{H}^1_{{\rm mic},\kappa,N} \to \mathcal{H}^{-1}_{{\rm mic},\kappa,N} \, ,
\end{equation}
and considered as a map to the (conjugated) microscopic subspace $\mathcal{H}^{-1}_{{\rm mic},\kappa,N}$, is invertible and satisfies
\begin{equation}\label{invert:estNkappa}
\| (L_{\mu_1,N}^\kappa)^{-1} \|_{\mathcal{H}^{-1}_{{\rm mic},\kappa,N} \to \mathcal{H}^1_{{\rm mic},\kappa,N}} \les 1 \, .
\end{equation}
\end{lemma}
\begin{proof}[Proof of Lemma~\ref{lem:invertibilitytofindgn}]
The mapping property~\eqref{eq:mymappingprop} follows from the mapping properties of $\Gamma_\kappa[\mu_0^{-\sfrac 12}\mu_1,\cdot] + \Gamma_\kappa[\cdot,\mu_0^{-\sfrac 12}\mu_1]$ on the smooth, well-localized Hermite functions.  To complete the proof, we prove that $L_{\mu_1,N}^\kappa$ is bounded below with uniform estimates when $\mu_1$ is near $\mu_0$. 
That is, when $f \in \mathcal{H}^1_{{\rm mic},\kappa,N}$, we have
\begin{align*}
\langle L_{\mu_1,N}^\kappa f, f \rangle &= \langle L_{\mu_1}^\kappa f, f \rangle \\
&= \langle L_{\mu_0}^\kappa f, f \rangle -  \langle \Gamma_\kappa[\mu_0^{-\sfrac 12}(\mu_1-\mu_0),f], f \rangle - \langle f, \Gamma_\kappa[\mu_0^{-\sfrac 12}(\mu_1-\mu_0),f] \rangle \\
&\geq C^{-1} \| f \|_{\mathcal{H}^1_\kappa}^2 - o_{(\varrho_1,u_1,\theta_1) \to (1,0,1)} (1) \| f \|_{\mathcal{H}^1_\kappa}^2 \, ,
\end{align*}
with constants independent of $\kappa$. The operators in question are finite dimensional, so this furnishes invertibility and the estimates.
\end{proof}

\begin{proof}[Proof of Proposition~\ref{pro:aprioriapproximatestimates}]
\emph{1. Microscopic estimates}.

We begin with the analogue of the basic microscopic estimate~\eqref{eq:basicmicroest}. The key points are that (i) the energy estimates do not ``see" the projections $\Pi_{\leq N}$, and (ii) the $-\eta \p_x^2 f$ term provides a new advantageous contribution, except in (iii) the Kawashima energy estimate, but its error can be estimated using the new advantageous terms. First, we multiply the equation by $f$ and integrate by parts, using that $f = \Pi_{\leq N} f$. We have
\begin{equation}
	\label{eq:firstmicroestapp}
\eta \int |\p_x f|^2 \, \dee v \, \dee x + \int L_{\rm NS}^\kappa f \cdot f \, \dee v \, \dee x = \int z f \, \dee v \, \dee x \, .
\end{equation}
When we pass one spatial derivative through the equation and multiply by $\p_x f$, we obtain
\begin{align}
	\eta \int &|\p_x^2 f |^2 \, \dee v \, \dee x + \int L_{\rm NS}^\kappa \p_x f \cdot \p_x f \, \dee v \, \dee x - \int \Gamma_\kappa[\mu_0^{-\sfrac 12} \p_x \mu_{\rm NS}, f] \p_x f \, \dee v \, \dee x \notag \\
& - \int \Gamma_\kappa[f,\mu_0^{-\sfrac 12} \p_x \mu_{\rm NS}] \p_x f \, \dee v \, \dee x = \int \p_x z \p_x f \, \dee v \, \dee x \, . \label{eq:secondmicroestapp}
\end{align}
When we apply $K$ to the equation and multiply by $\p_x f$, we obtain
\begin{equation}
	\label{eq:thirdmicroestapp}
- \eta \int \p_x^2 K f \p_x f \, \dee v \, \dee x + \int (K v_1 \p_x f) \p_x f + \int KL_{\rm NS}^\kappa f \cdot \p_x f \, \dee v \, \dee x = \int K z \p_x f \, \dee v \, \dee x \, .
\end{equation}
Notice the absence of projections $\Pi_{\leq N}$ in the above equalities. The desired energy estimate is obtained by summing~\eqref{eq:firstmicroestapp},~\eqref{eq:secondmicroestapp}, and a small multiple of~\eqref{eq:thirdmicroestapp}. Therefore, the term $- (2C_{III})^{-1} \eta \int \p_x^2 K f \p_x f $ can be absorbed by the terms $\eta \int |\p_x f|^2$ and $\eta \int |\p_x^2 f|^2$ via Young's inequality to obtain, exactly as before,
\begin{equation}
	\label{eq:basicmicroestrepeated}
\| \p_x f \|_{L^2 \mathcal{H}^1_\kappa} + \| P^\perp f \|_{L^2 \mathcal{H}^1_\kappa} \les \varepsilon \| Pf \|_{L^2} + \| z, \p_x z \|_{L^2 \mathcal{H}^{-1}_\kappa} \, .
\end{equation}
Subsequently, the estimate
\begin{equation}
\label{eq:basicmicroestbootstrappedoncerepeated}
\| \p_x^2 f \|_{L^2 \mathcal{H}^1_\kappa} + \| P^\perp \partial_x f \|_{L^2 \mathcal{H}^1_\kappa} \les \varepsilon^2 \| Pf \|_{L^2} + \| \varepsilon z, \p_x z,  \p_x^2 z \|_{L^2 \mathcal{H}^{-1}_\kappa}
\end{equation}
is obtained by differentiating the Galerkin truncated equation once and applying~\eqref{eq:basicmicroestrepeated}.  Crucially, the implicit constants in \eqref{eq:basicmicroestrepeated}--\eqref{eq:basicmicroestbootstrappedoncerepeated} do not depend on $\kappa$.
\medskip


\emph{2. Macroscopic estimates}. We now require the analogue of the linearized Chapman-Enskog expansion. Write $F = \sqrt{\mu_0} f$ and $\mathcal{E} = \sqrt{\mu_0} z$. We write $\sf{L}_N^\kappa = \sqrt{\mu_0} L_{\mu_0,N}^\kappa \sqrt{\mu_0}^{-1}$ and $\sf{L}_{{\rm NS},N}^\kappa = \sqrt{\mu_0} L_{{\rm NS},N}^\kappa \sqrt{\mu_0}^{-1}$. It will be convenient to write $\sf{L}_{{\rm NS},N}^\kappa$ as a block matrix
\begin{equation}\notag
\begin{bmatrix}
0 & 0 \\
\sf{L}_{21,N}^\kappa & \sf{L}_{22,N}^\kappa 
\end{bmatrix}\, ,
\end{equation}
where $\sf{L}_{21,N}^\kappa = \sf{P}_{\mu_0}^\perp \sf{L}_{{\rm NS},N}^\kappa \sf{P}_{\mu_0}$ and $\sf{L}_{22,N}^\kappa = \sf{P}_{\mu_0}^\perp \sf{L}_{{\rm NS},N}^\kappa \sf{P}_{\mu_0}^\perp$. (This notation suppresses the $x$-dependence of the operators.)
 Again, we write $F = U + \sf{P}^\perp_{\mu_0} F$, where $U = \sf{P}_{\mu_0} F$. The finite-dimensional analogue of the projection $\sf{P}_{\rm NS}$ onto the tangent space at $\mu_{\rm NS}$ is
\begin{equation}\notag
\sf{P}_{{\rm NS},N}^\kappa = \sf{P}_{\mu_0} - (\sf{L}_{22,N}^\kappa)^{-1} \sf{L}_{21,N}^{\kappa} \sf{P}_{\mu_0} \, .
\end{equation}
This is the projection onto the kernel of the operator $\sf{L}^\kappa_{{\rm NS},N}$. The complementary projection $\sf{P}^{\perp,\kappa}_{{\rm NS},N} := I - \sf{P}_{{\rm NS},N}^\kappa$ is a projection onto the microscopic subspace. The linearized Chapman-Enskog ansatz is
\begin{equation}\notag
F = \sf{P}_{{\rm NS},N}^\kappa U + V \, , 
\end{equation}
which we take as the definition of $V$. Plugging the ansatz into the equation, we obtain
\begin{equation}
	\label{eq:galerkinchapman1}
-\eta \p_x^2 \sf{P}_{{\rm NS},N}^\kappa U - \eta \p_x^2 V + \sf{\Pi}_{\leq N} (v_1-s_0) \p_x \sf{P}_{{\rm NS},N}^\kappa U + \sf{\Pi}_{\leq N}  (v_1-s_0) \p_x V + \sf{L}_{{\rm NS},N}^\kappa V = \sf{\Pi}_{\leq N} \mathcal{E} \, ,
\end{equation}
where we used that $\sf{L}_{{\rm NS},N}^\kappa \sf{P}_{{\rm NS},N}^\kappa = 0$. We apply the projection $\sf{P}_{{\rm NS},N}^{\perp,\kappa}$ and invert $\sf{L}_{{\rm NS},N}^\kappa$ to obtain an expression for $V$:
\begin{equation}
	\label{eq:galerkinchapman2}
V = - (\sf{L}_{{\rm NS},N}^\kappa)^{-1} \sf{P}_{{\rm NS},N}^{\perp,\kappa} (\sf{\Pi}_{\leq N} (v_1-s_0) \p_x \sf{P}_{{\rm NS},N}^\kappa U + \sf{\Pi}_{\leq N} (v_1-s_0) \p_x V - \sf{\Pi}_{\leq N}  \mathcal{E} + \eta h) \, ,
\end{equation}
where
\begin{equation}
	\label{eq:expressionforh}
h = (\p_x^2 \sf{P}_{{\rm NS},N}^\kappa) U + 2 \p_x \sf{P}_{{\rm NS},N}^\kappa \p_x U + \p_x^2 V \, .
\end{equation}
We substitute the expression back into the equations projected by $\sf{P}_{\mu_0}$, obtaining
\begin{equation}
	\label{eq:substituteintome}
\sf{P}_{\mu_0} (v_1-s_0) \p_x \sf{P}_{{\rm NS},N}^\kappa U + \sf{P}_{\mu_0} (v_1-s_0) \p_x V = 0 \, .
\end{equation}
Let\footnote{While we do not prove that $\nabla H_N^\kappa$ is the gradient of a flux function $H_N^\kappa$, we adopt this notation nonetheless.}
\begin{align*}
\nabla H_N^\kappa &:= \sf{P}_{\mu_0} v_1 \sf{P}_{{\rm NS},N}^\kappa \, , \qquad 
B_N^\kappa  := \sf{P}_{\mu_0} v_1 (\sf{L}_{{\rm NS},N}^\kappa)^{-1} \sf{P}_{{\rm NS},N}^{\perp,\kappa} \sf{\Pi}_{\leq N}  v_1  \sf{P}_{{\rm NS},N}^\kappa  \, ,
\end{align*}
and
\begin{equation}
	\label{eq:gnexpression}
g_N^\kappa = \sf{P}_{\mu_0} v_1 (\sf{L}_{{\rm NS},N}^\kappa)^{-1} \sf{P}_{{\rm NS},N}^{\perp,\kappa} (\sf{\Pi}_{\leq N}  (v_1-s_0) (\p_x \sf{P}_{{\rm NS},N}^\kappa) U + \sf{\Pi}_{\leq N}  (v_1-s_0) \p_x V - \sf{\Pi}_{\leq N} E + \eta h) \, .
\end{equation}
Then the equation~\eqref{eq:substituteintome}, integrated once, becomes
\begin{equation}\notag
(\nabla H_N^\kappa - s_0{\rm Id}) U =  B_N^\kappa \p_x U + g_N^\kappa \, .
\end{equation}
For now, we take for granted the following property of the coefficients:
\begin{equation}
	\label{eq:convergenceofcoefficients}
|\nabla H_N^\kappa - (\nabla H)(U_{\rm NS})| + |B_N^\kappa - B_\kappa(U_{\rm NS})| = o_{N \to +\infty}(1)  \, ,
\end{equation}
uniformly in $x$, with convergence possibly depending on $\kappa$. This will be justified below at the end of this subsection. Then we have the further comparison $|B_\kappa(U_{\rm NS}) - B(U_{\rm NS})| \lesssim \kappa$.
We consider $U$ as solving
\begin{equation}\notag
[(\nabla H)(U_{\rm NS}) - s_0{\rm Id}] U = B(U_{\rm NS}) \p_x U + (\nabla H(U_{\rm NS})-\nabla H_N^\kappa) U + (B_N^\kappa - B(U_{\rm NS})) \p_x U + g_N^\kappa \, ,
\end{equation}
with the desired phase condition. Then by the macroscopic estimate in Proposition~\ref{pro:macroscopicestimate}, we have
\begin{equation}\notag
\| U \|_{L^2} \lesssim \varepsilon^{-1} o_{N\to +\infty}(1) \| U \|_{L^2} + \varepsilon^{-1} (o_{N \to +\infty}(1) + \kappa) \| \p_x U \|_{L^2} + \varepsilon^{-1} \| g_N^\kappa \|_{L^2} \, .
\end{equation}

It remains to `close the loop'. Examining the expression~\eqref{eq:gnexpression} for $g_N^\kappa$, we see that it will be necessary to estimate quantities like $\|\mu_0^{-\sfrac 12} (v_1-s) (\p_x \sf{P}_{{\rm NS},N}^\kappa) U \|_{L^2 \mathcal{H}^{-1}_{\kappa}}$. More specifically, we require 
\begin{align}
\label{eq:estimateonderivsofcoefficients}
&\| (v_1-s_0) P^\perp P_{{\rm NS},N}^\kappa \|_{\mathcal{H}^1_\kappa \to \mathcal{H}^{-1}_\kappa} \lesssim_\kappa \varepsilon \, , \quad \|  (v_1-s_0) \p_x P_{{\rm NS},N}^\kappa, \varepsilon^{-1} (v_1-s_0) \p_x^2 P_{{\rm NS},N}^\kappa \|_{\mathcal{H}^1_\kappa \to \mathcal{H}^{-1}_\kappa} \lesssim_\kappa \varepsilon^2 \, , \\
	\label{eq:yetanotherthingtojustify}
&\qquad \qquad \qquad \| \partial^2_x P_{{\rm NS},N}^\kappa \|_{\mathcal{H}^1_\kappa \to \mathcal{H}^{-1}_\kappa} \lesssim_\kappa \varepsilon^3 \, ,
\end{align}
whose justification we temporarily delay. Note that the implicit constants may depend on $\kappa$. By Lemma~\ref{lem:invertibilitytofindgn}, we have 
\begin{equation}\notag
\| g_N^\kappa \|_{\mathcal{H}^1_\kappa} \lesssim \varepsilon^2 \| U \| + \| \sf{\Pi}_{\leq N} (v_1-s_0) \p_x V \|_{\mathcal{H}^{-1}_\kappa} + \| \sf{\Pi}_{\leq N} \mathcal{E} \|_{\mathcal{H}^{-1}_\kappa} + \eta \| h \|_{\mathcal{H}^{-1}_\kappa} \, ,
\end{equation}
where we use the estimate on $\p_x P_{{\rm NS},N}^\kappa$ from~\eqref{eq:estimateonderivsofcoefficients}. \emph{We emphasize the decisive use of the following property of the $\kappa$-regularized function spaces, which follows from Lemma~\ref{l:norm:control}:}
\begin{equation}\notag
\| f, v f \|_{\mathcal{H}^{-1}_\kappa} \lesssim_\kappa \| f \|_{\mathcal{H}^1_\kappa} \, .
\end{equation}
We required such an estimate for terms like $\| (v_1-s_0) \p_x V \|_{\mathcal{H}^{-1}_\kappa}$, just as in~\eqref{eq:doesn't:close}. 
Like before, we have
\begin{equation}\notag
P^\perp \p_x F = P^\perp \p_x (P_{{\rm NS},N}^\kappa U) + \p_x V \, ,
\end{equation}
and, by the estimates on $P^\perp_{\mu_0} P_{{\rm NS},N}^\kappa$ and $P^\perp_{\mu_0} (\p_x P_{{\rm NS},N}^\kappa)$ from~\eqref{eq:estimateonderivsofcoefficients}, we have
\begin{equation}\notag
\| (v_1-s_0) \mu_0^{-\sfrac 12} \p_x V \|_{L^2 \mathcal{H}^{-1}_\kappa} \lesssim_\kappa \varepsilon^2 \| U \|_{L^2} + \varepsilon \| \p_x U \|_{L^2} + \| \mu_0^{-\sfrac 12} (v_1-s_0) P^\perp \partial_x F \|_{L^2 \mathcal{H}^{-1}_\kappa} \, .
\end{equation}
Meanwhile, from the expression~\eqref{eq:expressionforh} for $h$, we have
\begin{equation}\notag
\| h \|_{\mathcal{H}^{-1}_\kappa} \lesssim_\kappa \varepsilon^3 \| U \| + \varepsilon^2 \| \p_x U \| + \| \p_x^2 V \|_{\mathcal{H}^{-1}_\kappa} \, ,
\end{equation}
by the estimate~\eqref{eq:yetanotherthingtojustify}. Of course,
\begin{equation}\notag
P^\perp \p_x^2 F = P^\perp \p_x^2 (P_{{\rm NS},N}^\kappa U) + \p_x^2 V \, ,
\end{equation}
so that
\begin{equation}\notag
\| (v_1-s_0) \mu_0^{-\sfrac 12} \p_x^2 V \|_{L^2 \mathcal{H}^{-1}_\kappa} \lesssim_\kappa \varepsilon^3 \| U \|_{L^2} + \varepsilon^2 \| \p_x U \|_{L^2} + \varepsilon \| \p_x^2 U \|_{L^2} + \| \mu_0^{-\sfrac 12} (v_1-s_0) P^\perp \partial_x^2 F \|_{L^2 \mathcal{H}^{-1}_\kappa} \, .
\end{equation}
Combining the estimates is enough to close the loop.
\end{proof}

We now proceed with the analysis of the coefficients~\eqref{eq:convergenceofcoefficients}, together with~\eqref{eq:estimateonderivsofcoefficients} and \eqref{eq:yetanotherthingtojustify}. We necessarily utilize a property of our Galerkin basis; namely, the decay and regularity of a function is reflected in the decay of its coefficients in the Galerkin basis. Given $f \in L^2$, denote by $(f_\alpha)$ the coefficients in the Hermite basis. Then
\begin{equation}\notag
f,Tf,\hdots, T^k f \in L^2 \iff \sum_\alpha \langle \lambda_\alpha \rangle^{2k} |f_\alpha|^2 < +\infty \, ,
\end{equation}
with equivalence of norms:
\begin{equation}\notag
\| f, Tf, \hdots, T^k f \|_{L^2} \approx_k \left( \sum_\alpha \langle \lambda_\alpha \rangle^{2k} |f_\alpha|^2 \right)^{\sfrac{1}{2}} \, .
\end{equation}

The key will be to estimate
\begin{equation}	\label{eq:estimatingdifferentderivatives}
((\mathsf{L}_{22,N}^\kappa)^{-1} \sf{\Pi}_{\leq N} - \mathsf{L}_{22}^{-1}) \mathcal{E}
\end{equation}
on well-localized, high regularity, purely microscopic functions $\mathcal{E} = \mu_0^{\sfrac 12} z$.

\begin{lemma}	\label{lem:estimatingdifferentderivatives}
For $0 < \varepsilon \ll 1$, $0 \leq \kappa \leq 1$, and purely microscopic $z$ satisfying the normalization
\begin{equation}\notag
\| \mu_0^{-\sfrac 18} \p_\beta z \|_{L^2} \leq 1 \, , \quad \forall |\beta| \leq 1000 \, ,
\end{equation}
we have
\begin{equation}\notag
\| \mu_0^{-\sfrac 12} ((\mathsf{L}_{22,N}^\kappa)^{-1} \sf{\Pi}_{\leq N} - (\mathsf{L}_{22}^\kappa)^{-1}) \mathcal{E} \|_{H^2_\varepsilon \mathcal{H}^1_\kappa} \lesssim_\kappa N^{-50} \, .
\end{equation}
\end{lemma}

A typical example is $z = P^\perp L_{\rm NS}^\kappa g$ with $g \in \ker L$.

\begin{proof}[Proof of Lemma~\ref{lem:estimatingdifferentderivatives}]
Let $F_N = (\mathsf{L}_{22,N}^\kappa)^{-1} \sf{\Pi}_{\leq N} \mathcal{E}$ and $F = (\mathsf{L}^\kappa_{22})^{-1} \mathcal{E}$, both of which are purely microscopic. Then
\begin{equation}
	\label{eq:nseqdifferentiateme}
L_{\rm NS}^\kappa f = z \, , \quad \Pi_{\leq N} L_{\rm NS}^\kappa f_N = \Pi_{\leq N} z \, .
\end{equation}
Therefore,
\begin{equation}\notag
\Pi_{\leq N} L_{\rm NS}^\kappa (\Pi_{\leq N} f - f_N) = - \Pi_{\leq N} L_{\rm NS}^\kappa \Pi_{> N} f \, .
\end{equation}
The estimates on the non-truncated problem yield $\| \mu_0^{-\sfrac{1}{16}} \p_\beta f \|_{L^2} \lesssim 1$ for all $|\beta| \leq 1001$. Hence,
\begin{equation}\notag
\| \Pi_{\leq N} f - f_N \|_{\mathcal{H}^1} \lesssim \| \Pi_{>N} f \|_{\mathcal{H}^1} \lesssim N^{-100} \, ,
\end{equation}
due to the decay of the Galerkin coefficients of $f$. In particular, we have
\begin{equation}
	\label{eq:reasoningissimilartome}
\| \langle v \rangle^{10} (f - f_N) \|_{\mathcal{H}^1} \lesssim \| \langle v \rangle^{10} (\Pi_{\leq N} f - f_N) \|_{\mathcal{H}^1} + \| \langle v \rangle^{10} \Pi_{>N} f \|_{\mathcal{H}^1} \lesssim N^{-50} \, .
\end{equation}

We next demonstrate how to estimate spatial derivatives of~\eqref{eq:estimatingdifferentderivatives} when $\mathcal{E}$ is fixed. The desired estimates are on $\langle v \rangle^{10} \p_x (f-f_N)$ in $L^2$. We differentiate~\eqref{eq:nseqdifferentiateme} once to obtain
\begin{equation}\notag
L \p_x f = 2\Gamma_\kappa(\mu^{-1/2} \p_x \mu_{\rm NS}, f) \, , \quad \Pi_{\leq N} L f_N = 2\Pi_{\leq N} \Gamma_\kappa(\mu^{-1/2} \p_x \mu_{\rm NS}, f_N) \, .
\end{equation}
The relevant equation is
\begin{equation}\notag
\Pi_{\leq N} L \p_x (\Pi_{\leq N} f-f_N) = 2\Gamma_\kappa(\mu^{-\sfrac 12} \p_x \mu_{\rm NS}, \Pi_{\leq N} f - f_N) + 2\Gamma_\kappa(\mu^{-\sfrac 12} \p_x \mu_{\rm NS}, \Pi_{> N} f) - \Pi_{\leq N} L \Pi_{> N} \p_x f \, .
\end{equation}
The estimates on the non-truncated problem yield $\| \mu_0^{-\sfrac{1}{16}} \p_\beta \p_x f \|_{L^2} \lesssim \varepsilon^2$ for all $|\beta| \leq 1001$. This yields
\begin{equation}\notag
\| \p_x (\Pi_{\leq N} f - f_N) \|_{\mathcal{H}^1} \lesssim \varepsilon^2 \| \Pi_{>N} f, \p_x f \|_{\mathcal{H}^1} \lesssim N^{-100} \, ,
\end{equation}
\begin{equation}\notag
\| \langle v \rangle^{10} \p_x (f - f_N) \|_{\mathcal{H}^1} \lesssim \varepsilon^2 N^{-50} \, ,
\end{equation}
by reasoning similar to~\eqref{eq:reasoningissimilartome}. Analogous reasoning for the second derivative $\p_x^2 (f-f_N)$ yields
\begin{equation}\notag
\| \langle v \rangle^{10} \p_x^2 (f - f_N) \|_{\mathcal{H}^1} \lesssim \varepsilon^3 N^{-50} \, .
\end{equation}
\end{proof}

\begin{proof}[Proof of~\eqref{eq:estimateonderivsofcoefficients} and~\eqref{eq:yetanotherthingtojustify}]
Recall that $\sf{P}_{{\rm NS},N}^\kappa = (I - (\sf{L}_{22,N}^\kappa)^{-1} \sf{\Pi}_{\leq N} \sf{L}_{{\rm NS}}^\kappa) \sf{P}_{\mu_0}$ and, therefore, 
$$\sf{P}_{\mu_0}^\perp \sf{P}_{{\rm NS},N}^\kappa = - (\sf{L}_{22,N}^\kappa)^{-1} \sf{\Pi}_{\leq N} \sf{L}_{{\rm NS}}^\kappa \sf{P}_{\mu_0} \, . $$
 If $(\sf{L}_{22,N}^\kappa)^{-1} \sf{\Pi}_{\leq N}$ were replaced by $\sf{L}_{22}^{-1}$, then the desired estimates~\eqref{eq:estimateonderivsofcoefficients} and~\eqref{eq:yetanotherthingtojustify} would be evident. Moreover, Lemma~\ref{lem:estimatingdifferentderivatives} contains the necessary error estimates between the two operators.
\end{proof}

Finally, we explain the convergence of the coefficients.
\begin{proof}[Proof of Equation~\ref{eq:convergenceofcoefficients} (Convergence of coefficients)]
First, we address
\begin{equation}\notag
\nabla H_N^\kappa - \nabla H(U_{\rm NS}) = \sf{P}_{\mu_0} v_1 (\sf{P}_{{\rm NS},N}^\kappa - \sf{P}_{\mu_{\rm NS}}) \, ,
\end{equation}
which converges to zero uniformly in $x$ by Lemma~\ref{lem:estimatingdifferentderivatives}. 
Next, we address
\begin{equation}\notag
\begin{aligned}
B_N^\kappa - B = \sf{P}_{\mu_0} v_1 [ (\sf{L}_{{\rm NS},N}^\kappa)^{-1} \sf{P}_{{\rm NS},N}^{\perp,\kappa} (\sf{\Pi}_{\leq N}  v_1  \sf{P}_{{\rm NS},N}^\kappa) - (\sf{L}_{{\rm NS}}^\kappa)^{-1} \sf{P}_{{\rm NS}}^{\perp,\kappa} ( v_1  \sf{P}_{\mu_{\rm NS}}) ] \, .
\end{aligned}
\end{equation}
We divide the inside bracket into the terms
\begin{equation}\notag
(\sf{L}_{{\rm NS},N}^\kappa)^{-1} \sf{P}^{\perp,\kappa}_{{\rm NS},N} \sf{\Pi}_{\leq N} v_1 (\sf{P}_{{\rm NS},N}^\kappa - \sf{P}_{\mu_{\rm NS}}) \, ,
\end{equation}
\begin{equation}\notag
((\sf{L}_{{\rm NS},N}^\kappa)^{-1} \sf{\Pi}_{\leq N} - L^{-1}_{\rm NS}) \sf{P}^\perp_{\mu_{\rm NS}} v_1 \sf{P}_{\mu_{\rm NS}} \, ,
\end{equation}
and
\begin{equation}\notag
(\sf{L}_{{\rm NS},N}^\kappa)^{-1} \sf{\Pi}_{\leq N} (\sf{P}^{\perp,\kappa}_{{\rm NS},N} - \sf{P}^\perp_{\mu_{\rm NS}}) v_1 \sf{P}_{\mu_{\rm NS}} \, ,
\end{equation}
each of which converges to zero uniformly in $x$ by Lemma~\ref{lem:estimatingdifferentderivatives}.
\end{proof}

\subsection{Qualitative Galerkin approximation}

In this section, we demonstrate existence for the Galerkin-truncated system:
\begin{proposition}
	\label{pro:existenceforgalerkintruncated}
Let $0 < \kappa \ll \varepsilon \ll_\kappa 1$, $N \geq 10$, and $\eta \geq 0$. There exists a solution $f \in H^1 \mathcal{H}_\kappa^1$ to the approximate system~\eqref{eq:approximatesystem}-\eqref{eq:phaseconditionapp}.
\end{proposition}
Proposition~\ref{pro:existenceforgalerkintruncated} will follow immediately from Lemmas~\ref{lem:hyperbolicity},~\ref{lem:largeetaholyshit}, and~\ref{lem:solvabilityoflinearODEs} below. With this in hand, we will complete the proof of Proposition~\ref{pro:basicexistenceforkappareg}.

We rewrite the approximate system as a $5+2r$-dimensional first-order ODE in the image of $\Pi_{\leq N}$, where $r = r(N)$ is the dimension of the finite-dimensional microscopic subspace, namely, the image of the projection $\Pi^\perp \Pi_{\leq N}$. The unknown is
\begin{equation}\notag
\mathcal{U} = 
\begin{bmatrix}
u \\ v \\ w
\end{bmatrix}
=
\begin{bmatrix}
PF \\ P^\perp F \\ \p_x P^\perp F
\end{bmatrix} \, .
\end{equation}
After $x$-integrating the $u''$ equation, we have
\begin{equation}\notag
\begin{aligned}
- \eta u' + A_{11} u + A_{12} v &= 0 \\
v' &= w \\
- \eta w' + A_{21} u' + A_{22} w + L_{21,N}^\kappa u + L_{22,N}^\kappa v &= z_N \, ,
\end{aligned}
\end{equation}
where $A_{11} = P (v_1-s_0) P$, $A_{12} = P (v_1-s_0) P^\perp$, etc. We substitute the $u'$ equation into the $w'$ equation to obtain the ODE
\begin{equation}\notag
\mathcal{U}' = \mathbb{A}(x) \mathcal{U} + \begin{bmatrix} 0 & 0 & z_N \end{bmatrix} \, ,
\end{equation}
where
\begin{equation}\notag
\mathbb{A} = 
\eta^{-1} \begin{bmatrix}
A_{11} & A_{12} & 0 \\
0 & & \eta I \\
\eta^{-1} A_{21}A_{11} - L_{21,N}^\kappa & \eta^{-1} A_{21} A_{12} - L_{22,N}^\kappa & A_{22}
\end{bmatrix} \, .
\end{equation}
We write $\mathbb{A}_{L,R}$ to denote the asymptotic operators as $x \to \mp \infty$. We shall require the following lemma on $\mathbb{A}_{L,R}$.
\begin{lemma}
	\label{lem:hyperbolicity}
	For $0 < \kappa \ll 1$, $0 < \varepsilon \ll_\kappa 1$, $\eta > 0$, and $N \gg_{\kappa,\varepsilon,\eta} 1$, the asymptotic matrices $\mathbb{A}_{L,R}$ are hyperbolic, that is, $\sigma(\mathbb{A}_{L,R}) \cap i\R = \emptyset$.
\end{lemma}

\begin{proof}[Proof of Lemma~\ref{lem:hyperbolicity}]
For the sake of contradiction, we suppose that one of the asymptotic matrices has an imaginary eigenvalue $i \tau$, $\tau \in \R$, with non-trivial eigenvector $\mathcal{U}$. We drop the $L,R$ notation. Substituting $w = i\tau v$ and writing $F = \mu_0^{1/2} (u+v)$, we have
\begin{equation}
	\label{eq:illrefthislateR}
\begin{aligned}
\sf{P}_{\mu_0} (v_1 - s_0) F &= i \tau \eta \sf{P}_{\mu_0}F \\
(i\tau \sf{\Pi}_{\leq N}(v_1 - s_0) + \sf{L}_{{\rm NS},N}^\kappa + \tau^2 \eta) F &= 0 \, .
\end{aligned}
\end{equation}
The first equation is redundant when $\tau \neq 0$ but useful when $\tau = 0$. 

When $\tau = 0$, we have that $\sf{L}^\kappa_{{\rm NS},N} F = 0$, that is, $F = \sf{P}^\kappa_{{\rm NS},N} U$ ($U = \sf{P}_{\mu_0}F$). The first equation yields
\begin{equation}\notag
(\nabla H_N^\kappa - s_0I) U = \sf{P}_{\mu_0} (v_1 - s_0) \sf{P}_{{\rm NS},N}^\kappa U = 0 \, .
\end{equation}
However, for $N$ sufficiently large, $\nabla H_N^\kappa - s_0I$ has trivial kernel, which is the desired contradiction.

From now on, assume $\tau \neq 0$. First, we repeat the Kawashima-type energy estimate from the quantitative Galerkin approximation. We write $F = \sqrt{\mu_0} f$ and obtain
\begin{equation}\notag
|\tau| \| f \|_{\mathcal{H}^1_\kappa} + \eta \tau^2 \| f \|_{L^2} + \| P^\perp f \|_{\mathcal{H}^1_\kappa} \lesssim \varepsilon \| Pf \| \, .
\end{equation}
When $|\tau| + \eta \tau^2 \geq C_0 \varepsilon$ for $C_0 \gg 1$, we obtain the desired contradiction $f = 0$.

We now consider $|\tau| + \eta \tau^2 \leq C_0 \varepsilon$. To simplify computations, it will be convenient to normalize $\| U \| = 1$. We repeat the Chapman-Enskog argument, which will be relatively simple because derivatives are replaced by $i\tau$ and we normalize $U$. We write $F = U + \sf{P}^\perp F = \sf{P}_{{\rm NS},N}^\kappa U + V$. Then
\begin{equation}\notag
\eta \tau^2 F + i \tau \sf{\Pi}_{\leq N} (v_1-s_0) \sf{P}_{{\rm NS},N}^\kappa U + i \tau \sf{\Pi}_{\leq N} (v_1-s_0) V + \sf{L}_{{\rm NS},N}^\kappa V = 0 \, ,
\end{equation}
cf. \eqref{eq:galerkinchapman1}. Again, we apply $\sf{P}_{{\rm NS},N}^{\perp,\kappa}$ and invert $\sf{L}_{{\rm NS},N}^\kappa$ to obtain
\begin{equation}\notag
V = - (\sf{L}_{{\rm NS},N}^\kappa)^{-1} \sf{P}_{{\rm NS},N}^{\perp,\kappa} (i \tau \sf{\Pi}_{\leq N} (v_1-s_0) \sf{P}_{{\rm NS},N}^\kappa U + i \tau \sf{\Pi}_{\leq N}  (v_1-s_0) V + \eta \tau^2 V) \, ,
\end{equation}
cf. \eqref{eq:galerkinchapman2}. We insert this into the macroscopic equation in~\eqref{eq:illrefthislateR} to obtain
\begin{equation}\notag
(\nabla H_N^\kappa - s_0I) U - i\tau B_N^\kappa U = i \tau \eta U + g_N^\kappa
\end{equation}
where
\begin{equation}\notag
g_N^\kappa = P (v_1 - s_0) (\sf{L}_{{\rm NS},N}^\kappa)^{-1} [ P_{{\rm NS},N}^{\perp,\kappa} (i \tau \sf{\Pi}_{\leq N}  (v_1-s_0) V) + \eta \tau^2 V ] = O (|\tau| + \eta \tau^2 ) \| V \|_{\mathcal{H}^1_\kappa} = O(\varepsilon^2) \, ,
\end{equation}
and the coefficients $\nabla H_N^\kappa$ and $B_N^\kappa$ are evaluated at the asymptotic endpoint.
We utilize that the microscopic estimate yields $\| V \|_{\mathcal{H}^1_\kappa} \lesssim_\kappa \varepsilon \| U \|$. Hence, 
\begin{equation}\notag
(\nabla H_N^\kappa - s_0I) U - i \tau B_N^\kappa U - i \tau \eta U = O(\varepsilon^2) \, .
\end{equation}
We moreover write
\begin{equation}\notag
(\nabla H(U_{\rm NS}) - s_0I) U - i \tau B U - i \tau \eta U = (\nabla H(U_{\rm NS}) - \nabla H_N^\kappa) U + i \tau (B_N^\kappa - B) U + O(\varepsilon^2) \, .
\end{equation}
We wish to invert $\nabla H(U_{\rm NS}) - s_0I - i \tau B - i \tau \eta I$ with $\| (\nabla H(U_{\rm NS}) - s_0I - i \tau B - i \tau \eta I)^{-1} \| \lesssim \varepsilon^{-1}$. If this is possible, then
\begin{equation}\notag
\| U \| \lesssim \varepsilon^{-1} o_{N \to +\infty}(1) \| U \| + \varepsilon^{-1} (o_{N \to +\infty}(1) + \kappa) |\tau| \| U \| + \varepsilon \, ,
\end{equation}
which contradicts that $\| U \| = 1$.

If $|\tau| + |\tau| \eta \leq 2c_0$ with $c_0 \ll 1$, then we consider the full inverse as a perturbation of $(\nabla H - s_0I)^{-1}$ satisfying $\| (\nabla H - s_0I)^{-1} \| \lesssim \varepsilon^{-1}$.

Otherwise, we rely on symmetrizability and the Kawashima theory in Remark~\ref{rmk:symandkawashima}. The compressible Navier-Stokes are symmetrizable at every non-vacuum state, not merely at $(u_L,s_L)$, and the symmetric positive-definite symmetrizer $A^0$ varies continuously. The desired estimate will be obtained as a consequence of estimates for the inverse of the symmetrized operator $A^0( i (\nabla H - s_0 I) + \tau (B + \eta I) )$,
and, in particular, energy estimates for the equation
\begin{equation}\notag
A^0( i(\nabla H - s_0I) + \tau (B+\eta I) ) U = G \, ,
\end{equation}
for unknown $U$ and right-hand side $G$.
If $|\tau| \eta \geq c_0 \varepsilon$, then a standard energy estimate yields the conclusion, since $\tau \eta A^0$ is coercive. The remaining case is $c_0 \varepsilon \leq |\tau| \leq C_0 \varepsilon$. $A^0 (\nabla H - s_0I)$ satisfies the Kawashima condition with respect to $A^0 B$ because the condition is invariant under small symmetric perturbations of the matrices, the constants in the estimates vary in a controlled way, and we verified the condition at $(u_L,s_L)$. The Kawashima-type energy estimates in Lemma~\ref{lem:steadyestimate} yield that the matrix is invertible with inverse size $O(\varepsilon^{-1})$.\footnote{Notice that $\tau$ multiplies the damping term $B + \eta I$, not the oscillation term $i(\nabla H - s_0I)$.} \end{proof}



\begin{remark}
For $|\tau| \eta \ll \varepsilon$, Lemma~\ref{lem:hyperbolicity} could also be closed by considering $- i \tau \eta I$ as a small perturbation of $\nabla H - s_0I - i \tau B$ and applying the techniques in Proposition~\ref{pro:ODEpropappendix}. However, for the convenient topological argument in Lemma~\ref{lem:largeetaholyshit}, it will be useful to have hyperbolicity of the endpoint matrices $\mathbb{A}_{L,R}$ for arbitrary $\eta$. We re-invoke symmetrizability specifically to handle the case when $|\tau| \eta$ is not small with respect to $\varepsilon$.
\end{remark}

\begin{lemma}
	\label{lem:largeetaholyshit}
	For $0 < \kappa \ll 1$, $0 < \varepsilon \ll_\kappa$, $\eta > 0$, and $N \gg_{\kappa,\varepsilon,\eta} 1$, the asymptotic matrices $\mathbb{A}_{L,R}$ satisfy ${\rm dim} \, E^u(\mathbb{A}_{L,R}) + {\rm dim} \, E^s(\mathbb{A}_{R,L}) = 5+2r+1$, where $E^s$ and $E^u$ denote the stable and unstable subspaces, respectively.
\end{lemma}

\begin{proof}
We again drop the $L,R$ notation. The eigenvalues of $\mathbb{A}$ are continuous in the parameter $\eta$ and, since $\mathbb{A}$ is hyperbolic for every $\eta$ by Lemma~\ref{lem:hyperbolicity}, they cannot cross the imaginary axis. Therefore, it is paradoxically enough to determine the dimensions of $E^u$ and $E^s$ when $\eta \gg 1$. We make the ``elliptic" change of variables $(u,v,w) \mapsto (u,v,\tilde{w})$, where $\tilde{w} = \eta^{\sfrac 12} w$. The conjugated matrix $\tilde{\mathbb{A}}$ acting in the new variables is given by
\begin{equation}
	\label{eq:thisisthespectrumIwannaknow}
\eta^{1/2} \tilde{\mathbb{A}} = 
\begin{bmatrix}
0 & 0 & 0 \\
0 & 0 & I \\
- L_{21} & - L_{22} & 0 
\end{bmatrix} +
\eta^{-1/2} \begin{bmatrix} A_{11} & A_{12} & 0 \\
0 & 0 & 0 \\
\eta^{-\sfrac 12} A_{21} A_{11} & \eta^{-\sfrac 12} A_{21} A_{12} & 0
\end{bmatrix} =: T_0 + \delta T_1 + \delta^2 T_2 \, ,
\end{equation}
where $\delta = \eta^{-\sfrac 12}$. Thus we encounter a finite-dimensional spectral perturbation problem.

First, we analyze the spectrum of $T_0$. We will prove that
\begin{equation}\notag
{\rm dim} \, E^u(T_0) = {\rm dim} \, E^s(T_0) = r \text{ and } {\rm dim} \, \ker T_0 = 5 \, .
\end{equation}
We begin with the non-zero eigenvalues. The $(v,\tilde{w})$ subspace is invariant under $T_0$; its restriction to this subspace is $\tilde{T}_0 = \begin{bmatrix} 0 & I \\ - L_{22,N}^\kappa & 0 \end{bmatrix}$. The matrix $- L_{22,N}^\kappa$ is bounded below: $\| f \|_{L^2}^2 \lesssim - \langle L_{22,N}^\kappa f, f \rangle$ for microscopic $f$. 
Hence, its spectrum belongs to $\{ {\rm{Re}} \,  \lambda > 0 \}$. 
We observe that $\tilde{T}_0^2 = \begin{bmatrix}
-L_{22,N}^\kappa & 0 \\
0 & - L_{22,N}^\kappa
\end{bmatrix}$, and therefore $\tilde{T}_0$ is hyperbolic, having no eigenvalues on the imaginary axis. The matrix $\tilde{T}_0$ has the further property that it conjugates to $-\tilde{T}_0$:
\begin{equation}
	\label{eq:conjugationofT}
\begin{bmatrix}
I & 0 \\
0 & -I
\end{bmatrix} \begin{bmatrix} 0 & I \\ - L_{22,N}^\kappa & 0 \end{bmatrix} \begin{bmatrix}
I & 0 \\
0 & -I
\end{bmatrix} = - \begin{bmatrix} 0 & I \\ - L_{22,N}^\kappa & 0 \end{bmatrix} \, .
\end{equation}
This is motivated by the observation that $\lambda$ is an eigenvalue of $\tilde{T}_0$ with eigenvector $(v,w)$ if and only if $v = \lambda w$ and $-L_{22,N}^\kappa v = \lambda^2 v$. In particular,~\eqref{eq:conjugationofT} implies that eigenvalues of $\tilde{T}_0$ arise only in $\pm$ pairs, and the eigenspaces satisfy $E_{\lambda} = \begin{bmatrix}
I & 0 \\
0 & -I
\end{bmatrix} E_{-\lambda}$ for each eigenvalue $\lambda$.

The operator $T_0$ has a five-dimensional kernel, namely, $\ker T_0 = \{ (u,v,0) : u+v = P^\kappa_{{\rm NS},N} u \}$. This accounts for the remaining eigenvalues of $T_0$.

The projection $\begin{bmatrix}
P^\kappa_{{\rm NS},N} &  \\
 & 0
\end{bmatrix}$ has range $\ker T_0$ and kernel $\{ (0,v,\tilde{w}) \}$, which is a complementary invariant subspace to $\ker T_0$. Therefore, $P_{{\rm NS},N}^\kappa$ is the spectral projection onto $\ker T_0$, and $I - P_{{\rm NS},N}^\kappa$ is the complementary spectral projection.

We now return to the spectrum of~\eqref{eq:thisisthespectrumIwannaknow}. Since zero is semisimple, the analytic perturbation theory yields that the perturbed eigenvalues $\lambda_k$, $1 \leq k \leq 5$ (counted with algebraic multiplicity), satisfy $\lambda_k = \delta \mu_k + o(\delta)$, where $\mu_k$ is the $k$th eigenvalue of $P_{{\rm NS},N}^\kappa T_1|_{\ker T_0}$. See~\cite[p. 265-269]{Baumgartel}. If we parameterize $\ker T_0$ by $u \mapsto P_{{\rm NS},N} u$, where $u \in \ker T_0$, then this operator is exactly $\mu_0^{-\sfrac 12} (\nabla H_N - s_0I) \mu_0^{\sfrac 12} = P (v_1 - s_0) P_{{\rm NS},N}^\kappa$, where $\nabla H_N$ is evaluated at the asymptotic endpoint. Recall that $|\nabla H_N^\kappa - \nabla H(U_{\rm NS})| = o(1)$ as $N \to +\infty$ and $\kappa \rightarrow 0$ and we recognize $\nabla H(U_{\rm NS})$ as the Euler flux function in suitable coordinates. This implies that, at the left endpoint, the null eigenvalues of $T_0$ become four stable eigenvalues and one unstable eigenvalue of $\mathbb{A}_L$. At the right endpoint, the null eigenvalues of $T_0$ become five stable eigenvalues and no unstable eigenvalues of $\mathbb{A}_R$. (The shock speed $s_0$ is slower than the fastest characteristic on the left, faster than the fastest characteristic on the right.) This completes the proof.
\end{proof}

\begin{proof}[Proof of Proposition~\ref{pro:existenceforgalerkintruncated}]
Fix $0 < \kappa \ll 1$ and $0 < \varepsilon \ll_\kappa 1$. For $0 < \eta \ll 1$ and $N \gg_{\kappa,\eta} 1$, the conditions of Lemma~\ref{lem:solvabilityoflinearODEs} (solvability) below are verified by Lemma~\ref{lem:hyperbolicity} (hyperbolicity), Lemma~\ref{lem:largeetaholyshit} (dimension counting), and Proposition~\ref{pro:aprioriapproximatestimates} (\emph{a priori} estimates and uniqueness).
\end{proof}

\begin{proof}[Proof of Proposition~\ref{pro:basicexistenceforkappareg}]
Let $f^{(N,\eta)}$ denote the solution to the approximate (conjugated) system~\eqref{eq:approximatesystem} with constraint~\eqref{eq:phaseconditionapp}. We consider $N \to +\infty$ and $\eta \to 0^+$. By \emph{a priori} estimates and weak compactness, there exists a subsequence, which we do not label, such that $f^{(N,\eta)} \rightharpoonup f$ in $H^2 \mathcal{H}^1$. Subsequently, we have that $(v_1-s_0)\p_x f^{(N,\eta)}$ converges to $(v_1-s_0)\p_x f$, $\Pi_{\leq N} E$ converges to $E$, and $\eta \p_x^2 f^{(N,\eta)}$ converges to zero, in the sense of distributions.

To verify convergence of $L_{{\rm NS},N}^\kappa f^{(N,\eta)}$, we consider $\langle L_{{\rm NS},N}^\kappa f^{(N,\eta)}, \psi \rangle$, where $\psi = c(x) \phi(v)$, $c \in C^\infty_0(\R)$ and $\phi$ belongs to our Galerkin basis. For $N \gg 1$ depending on $\phi$, this expression is simply $\langle L_{\rm NS}^\kappa f^{(N,\eta)}, \psi \rangle$. After integrating by parts in the collision operator to move a $v$ derivative onto $\phi$, see Section~\ref{sec:propertiesofLandGamma}, it is relatively straightforward to pass to the limit using the weak convergence.

Finally, since $Pf^{(N,\eta)} \in H^2_x$ is uniformly bounded, we can pass to a further subsequence to ensure that $\ell_\varepsilon(Pf^{(N,\eta)}(0)) \to \ell_\varepsilon(Pf(0))$, which verifies the one-dimensional constraint.
\end{proof}

Finally, we state the abstract ODE lemma which we are invoking.

\begin{lemma}[Linear ODE solvability]
\label{lem:solvabilityoflinearODEs}
Consider the ODE
\begin{equation}\notag
\dot u = A(x) u + f
\end{equation}
where $u : \R \to \R^n$, $A \in L^\infty(\R;\R^{n \times n})$, and $f \in L^\infty(\R;\R^n)$, supplemented by conditions
\begin{equation}\notag
\ell_1 \cdot u(0) = c_1, \hdots, \ell_m \cdot u(0) = c_m \, ,
\end{equation}
where $\ell_1, \hdots, \ell_m \in \R^n$ are linearly independent, and $\vec{c} \in \R^m$.

 Suppose that
 \begin{enumerate}
 \item $A$ is continuous and exponentially asymptotic to hyperbolic matrices $A_{\pm}$ as $x \to \pm \infty$,
 \item\label{item:dimensionality} $\dim E^u_- + \dim E^s_+ = n+m$, where $E^u_-$ is the unstable subspace of $A_-$ and $E^s_+$ is the stable subspace of $A_+$, and
 \item $u=0$ is the unique bounded solution to the above problem when $f=0$ and $\vec{c} = 0$.
 \end{enumerate}
Then, to each $f \in L^\infty(\R;\R^n)$ and $\vec{c} \in \R^m$, there exists a unique bounded solution $u$ to the problem.
 \end{lemma}

See Appendix~\ref{sec:analysislinearizedmacro}, \cite[p. 136-142]{HenryBook}, or \cite[Chapter 3]{KapitulaPromislow} for related arguments.


\subsection{Extending the range of $\varepsilon$ and taking $\kappa \to 0^+$}
\label{sec:methodofcontinuity}

We now show how to use the method of continuity to extend the range of $\varepsilon$:
\begin{proposition}
	\label{pro:methodofcontinuity}
For all $0 < \kappa \ll \varepsilon \ll 1$, purely microscopic $z$ satisfying $z \in H^2 \mathcal{H}^{-1}_{-10,\kappa}$, and $d \in \R$, there exists a unique solution $f \in H^2 \mathcal{H}^1_{-10,\kappa}$ to the $\kappa$-regularized system satisfying $\ell_\varepsilon(Pf(0)) = d$.
\end{proposition}

The key point is the existence for the enlarged range of $\varepsilon$. Uniqueness is already guaranteed by the \emph{a priori} estimates.

To prove Proposition~\ref{pro:methodofcontinuity}, we require

\begin{lemma}[Method of continuity]
\label{lem:methodofcontinuity}
Let $X$ and $Y$ be Banach spaces. Let $A(t) : [0,1] \to B(X,Y)$ be a continuous family of bounded linear operators. Suppose the following \emph{a priori} estimate: There exists $C > 0$ such that for all $x \in X$ and $t \in [0,1]$, we have
\begin{equation}
	\label{eq:methodofcont}
\| x \|_X \leq C \| A(t) x \|_Y \, .
\end{equation}
Then the following are equivalent:
\begin{enumerate}
\item There exists $t_0$ such that $A(t_0)$ is surjective.
\item For all $t \in [0,1]$, $A(t)$ is surjective.
\end{enumerate}
\end{lemma}
See, e.g.,~\cite[Theorem~5.2]{gilbarg1977elliptic}.

\begin{proof}
($\impliedby$) is obvious, so we prove ($\implies$). The \emph{a priori} estimate~\eqref{eq:methodofcont} grants that $A(t_0)$ is injective, so the assumption that it is surjective implies that it is boundedly invertible and $\| A(t_0)^{-1} \|_{Y \to X} \leq C$, again by~\eqref{eq:methodofcont}. Hence, $A(t)$ is invertible for $|t - t_0| \ll_{C,D} 1$, where $D = \max_{t \in [0,1]} \| A(t) \|_{X \to Y} < +\infty$. We repeat this argument $O_{C,D}(1)$ times to demonstrate that $A(t)$ is invertible for all $t \in [0,1]$.
\end{proof}

\begin{proof}[Proof of Proposition~\ref{pro:methodofcontinuity}]
Let 
$$ X = \{ f : f \in H^2 \mathcal{H}^1_{-10,\kappa} \, , (v_1 - s_0) \p_x f \in H^2 \mathcal{H}^{-1}_{-10,\kappa} \, , P (v_1 - s_0) \p_x f = 0\}$$ and 
$$Y = \{ z \in H^2_x \mathcal{H}^{-1}_\kappa \cap {\rm Mic} : H^2 \mathcal{H}^{-1}_{-10,\kappa} \} \times \R$$
 with appropriate norms. For $0 < \varepsilon \ll 1$, we define
\begin{equation}\notag
A(\varepsilon)f := ( (v_1 - s_0)\p_x f + L_{\rm NS}^\kappa f , \ell_\varepsilon(Pf) ) \, .
\end{equation}
Then, from Proposition~\ref{pro:smallshocksgeneralhyperbolic} and Remark~\ref{rmk:theconstraint}, $A(\varepsilon) : X \to Y$ is bounded, continuous in the parameter $\varepsilon$, and satisfies the \emph{a priori} estimate $\| f \|_X \leq C \| A(\varepsilon) f \|_Y$ from Proposition~\ref{pro:basicexistence}, which we proved in subsection~\ref{sec:linearestimatesbaseline}. Moreover, for $0 < \varepsilon \ll_\kappa 1$, $A(\varepsilon)$ is invertible. Therefore, the method of continuity therefore yields invertibility in the $\kappa$-independent range where the \emph{a priori} estimates hold.
\end{proof}

Finally, we complete the proof of Proposition~\ref{pro:basicexistence}.

\begin{proof}[Proposition~\ref{pro:basicexistence}]
Fix $0 < \varepsilon \ll 1$, purely microscopic $z$ satisfying $\langle v \rangle^{20} z \in H^{10} \mathcal{H}^1$ (in particular, it is uniformly bounded in the $\kappa$-regularized spaces for $0 < \kappa \leq 1$), and $d \in \R$. Existence for more general $z$ will be available \emph{a posteriori} by approximation. For $0 < \kappa \ll \varepsilon \ll 1$, Proposition~\ref{pro:methodofcontinuity} furnishes a solution $f^\kappa$ to the $\kappa$-regularized system with right-hand side $z$ and $\ell_\varepsilon(Pf^\kappa(0)) = d$. The \emph{a priori} estimates from Proposition~\ref{pro:basicexistence} grant
\begin{equation}\notag
\sup_{0 < \kappa \ll 1} \| f^\kappa \|_{H^2 \mathcal{H}^1_{-10}} + \kappa^{1/2} \| f^\kappa \|_{H^2 \mathcal{H}^1_{\rm R}} < +\infty \, ,
\end{equation}
and therefore there exists a sequence, which we do not label, satisfying $f^\kappa \rightharpoonup f$ in $H^2 \mathcal{H}^1_{-10}$ as $\kappa \to 0^+$. By compact embedding, $Pf^\kappa(0) \to Pf(0)$, so $\ell_\varepsilon(Pf^\kappa(0)) \to \ell_\varepsilon(Pf(0))$, and the value of the constraint is kept. Evidently $(v_1-s_0) \p_x f^\kappa$ converges to $(v_1-s_0) \p_x f$ in the sense of distributions, so it remains to verify the convergence of $L_{\rm NS}^\kappa f^\kappa$ to $L_{\rm NS} f$. Let $\phi \in C^\infty_0(\R \times \R^3)$ be a test function. We have
\begin{equation}\notag
\langle L_{\rm NS}^\kappa f^\kappa, \phi \rangle = \langle L_{\rm NS} f^\kappa, \phi \rangle + \kappa \langle L^{\rm R}_{\rm NS} f^\kappa, \phi \rangle \, .
\end{equation}
Like when we took $N \to +\infty$ in the Galerkin approximation, the term $\langle L_{\rm NS} f^\kappa, \phi \rangle$ is handled by integrating by parts in the collision operator to move a $v$ derivative onto $\phi$, see Section~\ref{sec:propertiesofLandGamma}, after which one may pass to the limit $\langle L_{\rm NS} f, \phi \rangle$ using the weak convergence. Meanwhile, the $\kappa$ term converges to zero:
\begin{equation}\notag
\kappa |\langle L^{\rm R}_{\rm NS} f^\kappa, \phi \rangle| \lesssim \kappa \| f^\kappa \|_{L^2 \mathcal{H}^1_{\rm R}} \| \phi \|_{L^2 \mathcal{H}^1_{\rm R}} = O(\kappa^{1/2}) \, .
\end{equation}
This completes the proof.
\end{proof}

\section{Fixed point argument}\label{sec:fixed:point}

\begin{proof}[Proof of Theorem~\ref{thm:main}] \emph{1. Existence and uniqueness with $\ell_\varepsilon(Pf(0)) = 0$.} In sections~\ref{secc:micro:linear} and~\ref{sekk:linear}, we constructed and estimated solutions, $f$, to the conjugated, linearized equation~\eqref{eq:NLSL} for the residual. In order to solve the nonlinear problem, we will solve the conjugated version of~\eqref{def:LinCE}; namely,
\begin{align*}
(v_1-s_0) \p_x f + L_{\rm NS} f = 2\Gamma\left[ g_{\rm NS},f\right] + {\Gamma}[f,f] - \mu^{-\sfrac 12}_0\mathcal{E}_{\rm NS} \, , 
\end{align*}
where $g_{\rm NS}=\mu_0^{-\sfrac 12} G_{\rm NS}$ is defined using~\eqref{def:fNS}.  
Using Proposition~\ref{pro:remainderest}, we can estimate the purely microscopic right-hand side
\begin{align*}
z[h] := 2\Gamma\left[ g_{\rm NS}, h \right] + \Gamma[h,h] - \mu^{-\sfrac 12}_0 \mathcal{E}_{\rm NS} \, ;
\end{align*}
in particular, the right-hand side satisfies the assmptions of Proposition~\ref{p:boots}. We therefore define the mapping $h \mapsto \mathcal{A}[h]$ by applying Proposition~\ref{pro:basicexistence}, so that $\mathcal{A}[h]$ is the unique (in $H^2\mathcal{H}^1_{-10}$) solution to the linear problem
\begin{align*}
\left((v_1-s_0) \p_x  + L_{\rm NS}\right) \mathcal{A}[h] = 2\Gamma\left[ g_{\rm NS},h \right] + {\Gamma}[h,h] - \mu_0^{-\sfrac 12}\mathcal{E}_{\rm NS}  = z[h] \, ,
\end{align*}
subject to the constraint
\begin{align*}
\ell_\eps \left( P \mathcal{A}[h](0) \right) = 0 \, . 
\end{align*}
Recalling the definition of $\mathbb{M}^N_\eps$ from Proposition~\ref{lem:MainLinear} and applying Proposition~\ref{p:boots}, we have that 
\begin{align}
\eps \norm{\mathcal{A}[h]}_{\mathbb M^N_{\eps}} & \lesssim \norm{z[h]}_{\mathbb Y^N_{\eps, \rm w}} + \norm{e^{\delta\langle \eps x \rangle^{\sfrac 12}} z[h]}_{\mathbb Y^N_\eps} \, . \label{ineq:AqFinal}
\end{align}
Define the closed ball
\begin{align*}
\mathcal{B}_\eps := \set{ h \in \mathbb M^N_\eps: \norm{h}_{\mathbb M^N_\eps} \leq M \eps^2} \, ,
\end{align*}
where $M$ depends on the (possibly large) implicit constants from~\eqref{rem:est:one} and~\eqref{p:boots:est}.  By \eqref{ineq:AqFinal} and Proposition~\ref{pro:remainderest}, we have that
\begin{align*}
\mathcal{A} : \mathcal{B}_\eps \mapsto \mathcal{B}_\eps
\end{align*}
for $\eps$ sufficiently small depending only on $N$ and universal constants.  To show that it is a contraction, for $h_1,h_2$ we evaluate
\begin{align*}
 \left((v_1-s_0) \p_x  + L_{\rm NS}\right)\left(\mathcal{A}[h_1] - \mathcal{A}[h_2]\right) = z[h_1] - z[h_2],
\end{align*}
and note that
\begin{align*}
z[h_1] - z[h_2]= 2\Gamma\left[g_{\rm NS},h_1 - h_2\right] + \Gamma[h_1 - h_2,h_1] + \Gamma[h_2,h_2-h_1] \, . 
\end{align*}
By a variation of Proposition~\ref{pro:remainderest}, we have
\begin{align*}
\norm{z[h_1] - z[h_2]}_{\mathbb Y^N_{\eps, \rm w}} + \norm{e^{\delta\langle \eps x \rangle^{\sfrac 12}}\left(z[h_1] - z[h_2]\right)}_{\mathbb Y^N_\eps}  \lesssim \left(\eps^2 + \norm{h_1}_{\mathbb M^N_\eps} + \norm{h_2}_{\mathbb M^N_\eps}\right) \norm{h_1 - h_2}_{\mathbb M^N_\eps} \, . 
\end{align*}
Therefore, the contraction mapping theorem implies that there exists a unique $h_\ast \in \mathcal{B}_\eps$ with 
\begin{align*}
\mathcal{A}[h_\ast] = h_\ast \, . 
\end{align*}
This nearly concludes the proof of Theorem~\ref{thm:main}; the only thing left to justify is that~\eqref{eq:estimatetobesatisfied} holds with weights \emph{simultaneously} in $x$ and $v$, since $\mathbb{M}^N_\varepsilon$ weighted $x$ and $v$ separately.  However, one may prove estimates in $\mathbb{M}^N_\varepsilon$ with $q_0$ replaced by a slightly larger value $q_0 < \tilde q_0 <1$, and then interpolate the $x$ weighted bounds with the slightly stronger $v$ weighted bounds to obtain the simultaneously weighted bounds in Theorem~\ref{thm:main}.

\emph{2. Uniqueness up to translation.} First, we remark that the above contraction argument holds in the ball $\{f \in \mathbb{X}^N_\eps : \| f \|_{\mathbb{X}^N_{\eps,\rm w}} \leq c_1 \varepsilon \}$ for sufficiently small $c_1$.

Suppose that $F$ is a solution whose error $f = \mu_0^{-\sfrac{1}{2}} (F - F_{\rm NS})$ satisfies
\begin{equation}
\| f \|_{\mathbb{X}^N_{\eps,\rm w}}  \leq c_0 \varepsilon
\end{equation}
for $c_0$ to be determined below.
Then $|\ell_\varepsilon(Pf(0))| \lesssim c_0 \varepsilon$. We wish to prove that a translate of $F$ satisfies the one-dimensional constraint and is therefore the solution already constructed. Let $h \in \R$. By the lower bound on $\ell_\varepsilon(P \p_x \mu_0^{-\sfrac{1}{2}}\mu_{\rm NS})$ in Remark~\ref{rmk:Ihavelowerboundsonell}, there exists $|h| \ll 1/\varepsilon$ (provided that $c_0 \ll 1$) such that $\ell_\varepsilon(P \mu_0^{-\sfrac{1}{2}} (F(0) - F_{\rm NS}(h))) = 0$; that is, a translation of $F_{\rm NS}$ can correct the constraint up to a small multiple of $\varepsilon$. We therefore estimate
\begin{equation}
\begin{aligned}
\| F(\cdot-h) - F_{\rm NS} \|_{\mathbb{X}^N_{\eps,\rm w}}  &\leq \| F - F_{\rm NS} \|_{\mathbb{X}^N_{\eps,\rm w}}  + \| F_{\rm NS} - F_{\rm NS}(\cdot+h) \|_{\mathbb{X}^N_{\eps,\rm w}}  \\
&\leq c_0 \varepsilon + O(\varepsilon^2) + \| \mu_{\rm NS} - \mu_{\rm NS}(\cdot+h) \|_{\mathbb{X}^N_{\eps,\rm w}}  \\
&\leq c_0 \varepsilon + O(\varepsilon^2) + C h \varepsilon^2 \leq c_1 \varepsilon \, ,
\end{aligned}
\end{equation}
since
\begin{equation}
\| \mu_{\rm NS} - \mu_{\rm NS}(\cdot-h) \|_{\mathbb{X}^N_{\eps,\rm w}}  \leq h \| \p_x \mu_{\rm NS} \|_{\mathbb{X}^N_{\eps,\rm w}}  \lesssim h \varepsilon^2 \, .
\end{equation}
In conclusion, $F(\cdot-h)$ is the solution already constructed.
 \end{proof}

\appendix

\section{Chapman-Enskog computations}
\label{sec:chapmanenskogcomputations}

In this appendix, we justify the derivation of the three-dimensional compressible Navier-Stokes equations~\eqref{eq:compressibleNSintro} in traveling wave form with one-dimensional spatial dependence from the traveling wave Landau equation~\eqref{def:LandauIntro}.  The computations in this section are essentially classical, and we refer the reader to \cite[Appendix A]{DuanYangYuContactWaves} for a recent treatment of this derivation, and to the standard references~\cite{BardosGolseLevermore1991, BGL2, CCbook, guo2006boltzmann, UkaiYang}.  We however require a slightly non-standard version of these computations in order to carry out the linearized Chapman-Enskog expansion for the regularized Landau operator $Q_\kappa=Q+\kappa Q_{\rm R}$ in subsubsection~\ref{sss:macro}.  The main novelty is that the computation of the viscosity coefficients in the linearized Chapman-Enskog expansion, cf.~\eqref{lin:ce}, \eqref{BB}, and~\eqref{becomes}, requires inversion of the operator $L_{\mu_{\rm NS}}^\kappa$.  Crucially, the viscosity coefficients in the linearized expansion in subsubsection~\ref{sss:macro} are computed according to an identical procedure as that which produces the full compressible Navier-Stokes equations in~\eqref{NSE:outline}, save for the substitution of $L^\kappa_{\mu_{\rm NS}}$ for $L_{\mu_{\rm NS}}$.  Since the case $\kappa=0$ corresponds to the standard Landau operator, cf.~\eqref{Qkappadef}, we treat in this section general $\kappa\in[0,1]$.  We first derive the compressible Navier-Stokes equations, thus validating~\eqref{NSE:outline} and obtaining along the way a formula for the viscosity coefficients~\eqref{BB} for the linearized, regularized equation, before deriving a stability result comparing $B_\kappa$ to $B_0$ for $\kappa \ll 1$.  
\medskip

\noindent\texttt{Derivation of~\eqref{NSE:outline} and viscosity coefficients in~\eqref{BB}}.    Recall the definition of $\mu$ from subsubsection~\ref{sss:maxwellians}, and the associated macroscopic space spanned by 
\begin{align}\notag
    \chi_0(v) &= \frac{1}{\sqrt{\varrho}}\mu \, , \qquad
    \chi_i(v) = \frac{v_i-u_i}{\sqrt{\varrho \theta}} \mu \, \, \qquad (i=1,2,3) \, ,  \qquad 
    \chi_4(v) = \frac{1}{\sqrt{6\varrho}} \left( \frac{|v-u|^2}{\theta} - 3 \right) \mu \, ,
\end{align}
where $\langle \cdot, \cdot \rangle$ is the $L^2_v$ inner product.  We recall the projection operators $\sf{P}_{\mu} (h)$ and $\sf{P}_{\mu}^\perp (h) = h - \sf{P}_{\mu} (h)$ defined in~\eqref{def:sfP}.   Note crucially that these projection operators and macroscopic spaces play precisely the same role in the regularized equation 
\begin{equation}\label{eq:reg:appp}
(v_1-s_0)\p_x F = Q_\kappa(F,F)
\end{equation}
and its linearization, given in~\eqref{def:LinCE:reg}, as in the standard Landau equation and its linearization. Supposing that we have a solution $F$ to the standing wave equation~\eqref{eq:reg:appp} (which includes~\eqref{def:LandauIntro} as the case $\kappa=0$), we decompose $F(t,x,v) = \mu(F)(t,x,v) + f(t,x,v)$, where
$$ \sf{P}_{\mu(F)}(F(t,x,\cdot))(v) = \mu(F)(t,x,v) \, , \qquad f(t,x,v) = \sf{P}^\perp_{\mu(F)} (F(t,x,\cdot)) (v) \, . $$
Substituting the decomposition above into~\eqref{eq:reg:appp}, we obtain
\begin{align}
    (v_1-s_0) \partial_x \mu(F) + (v_1-s_0) \partial_x f = -\sf{L}^\kappa_{\mu(F)} f + Q_\kappa(f,f) \, , \label{eq:sw:subbed}
\end{align}
where $\sf{L}^\kappa_{\mu(F)}$ is defined in~\eqref{def:sfl}. We now multiply \eqref{eq:sw:subbed} by $(1,v,\sfrac 12 |v|^2)$ and integrate in $v$ to obtain the five macroscopic equations
\begin{align}\label{deriving}
\begin{cases}
&\partial_x ( \varrho (u_1-s_0) )(t,x) = 0 \, , \\
&\partial_x (\varrho u_1 (u_1-s_0) + p)(t,x) = - \int_{\R^3} v_1^2 \partial_x f(t,x,v) \,dv \, , \\
&\partial_x (\varrho u_i (u_1-s_0) )(t,x) = - \int_{\R^3} v_1 v_i \partial_x f(t,x,v) \,dv \qquad\qquad (i=2,3) \, ,  \\
&\partial_x \left( \varrho \left( e + \frac{|u|^2}{2} \right)(u_1-s_0) + p u_1 \right)(t,x) = - \int_{\R^3} \frac{1}{2} v_1 |v|^2 \partial_x f(t,x,v) \, dv \, ,  
\end{cases}
\end{align}
where the pressure $p=\varrho \theta=(\gamma-1)\varrho e$. To justify this, we first note that for all five computations, the integral of the right-hand side of \eqref{eq:sw:subbed} against the collision invariants vanishes. Then the first equation is immediate by the definitions of $\varrho(t,x)$ and $\varrho(t,x) u_1(t,x)$. For the second equation, we use that 
\begin{align*}
\int_{\R^3} v_1(v_1-s_0) \mu(F)
&= \int_{\R^3} v_1 \theta(t,x) \underbrace{\frac{(v_1-u_1)}{\theta} \mu(F)}_{=-\partial_{v_1}\mu(F)}+\int_{\R^3} v_1 (u_1-s_0) \mu(F) \\
&= \int_{\R^3}  \theta \mu(F)) + \int_{\R^3} v_1 (u_1-s_0) \mu(F)  \\
&= \theta\varrho + (u_1-s_0) (u_1 \varrho) \, .
\end{align*}
For the third and fourth equations, we make very similar computations, except that no pressure appears since we have now $v_k-u_k(t,x)$ for $k=2,3$ on the second line, and then the integral vanishes after integrating by parts. For the fifth equation, we write that
\begin{align*}
    \int_{\R^3} \frac{|v|^2}{2} (v_1-s_0) \mu(F) &= (u_1-s_0) \int_{\R^3} \frac{|v|^2}{2} \mu(F) + \theta \int_{\R^3} \frac{|v|^2}{2} \frac{(v_1-u_1)}{\theta} \mu(F) \\
    &= (u_1-s_0) \varrho \left( e + \frac{1}{2} |u|^2 \right) + \theta \varrho u_1 \, ,
\end{align*}
where we have integrated by parts again to obtain $\theta \varrho u_1$. This finishes the derivation of~\eqref{deriving}.

We now apply the projection operator $\sf{P}^\perp_{\mu(F)}$ to~\eqref{eq:sw:subbed} to obtain that
\begin{align*}
    \sf{P}^\perp_{\mu(F)} \left[ (v_1-s_0) \partial_x \mu(F) + (v_1-s_0) \partial_x f \right] = \sf{P}^\perp_{\mu(F)} \left[  -\sf{L}^\kappa_{\mu(F)} f + Q_\kappa(f,f) \right] \, .
\end{align*}
We expect that $f\ll \mu(F)$, and so to leading order, we may guess that 
$$  f \approx - \left(\sf{L}^\kappa_{\mu(F)}\right)^{-1} \sf{P}^\perp_{\mu(F)}( (v_1-s_0) \partial_x \mu(F)) \, .  $$
Now, we first compute 
\begin{equation}\label{first:compute}
\int_{\R^3}v_1 v_i \partial_x \left( \left(\sf{L}^\kappa_{\mu(F)}\right)^{-1} \sf{P}^\perp_{\mu(F)} ( (v_1-s_0) \partial_x \mu(F)) \right) \, , \qquad  i=1,2,3 \, . 
\end{equation}
We first expand the argument of $(\sf{L}^\kappa_{\mu(F)})^{-1} \sf{P}^\perp_{\mu(F)}$, obtaining that
\begin{align}
    (v_1-s_0) \partial_x \mu(F) &= (v_1-s_0) \left[ \frac{\partial_x \varrho}{(2\pi \theta)^{\sfrac 32}} - \frac 32 \frac{2\pi \varrho  \partial_x \theta}{(2\pi \theta)^{\sfrac 52}} \right] \exp\left( - \frac{|v-u|^2}{2\theta} \right)  \notag\\
    &\qquad + (v_1-s_0) \mu(F) \left[ \frac{(v_k-u_k)\partial_x u_k}{\theta} + \frac{|v-u|^2}{\theta^2} \partial_x \theta \right]  \, . \label{eq:dxM}
\end{align}
Since the first term is a multiple (depending on $(t,x)$) of $(v_1-s_0)\mu(F)$, it belongs to the kernel of $\sf{P}^\perp_{\mu(F)}$, and we may ignore it.  In addition, the portion of the second term with $-s_0$ is also in the kernel, and so we may ignore it. We focus now on expanding the rest of the expression (the part of the second term with $v_1$) in terms of  the Burnett functions~\cite[Appendix]{DuanYangYuRarefaction}
\begin{align*}
\hat A^\kappa_j(v) &= \frac{|v|^2-5}{2} v_j \, , \quad A^\kappa_j \left( \frac{v-u}{\sqrt{\theta}} \right) = \left(\sf{L}^\kappa_{\mu(F)}\right)^{-1} \left[ \hat A^\kappa_j \left( \frac{v-u}{\sqrt{\theta}} \right) \mu(F) \right] \\
\hat B^\kappa_{ij}(v) &= v_i v_j - \frac{1}{3} \delta_{ij} |v|^2 \,, \qquad B^\kappa_{ij} \left( \frac{v-u}{\sqrt{\theta}} \right) = \left(\sf{L}^\kappa_{\mu(F)}\right)^{-1} \left[ \hat B^\kappa_{ij} \left( \frac{v-u}{\sqrt{\theta}} \right) \mu(F) \right] \, .
\end{align*}

We first note that the $v_1 v_i$ on the outside of~\eqref{first:compute} can be expanded as
\begin{align*}
    v_1 v_i = (v_1 - u_1)(v_i-u_i) + \textnormal{kernel} = \theta \hat B^\kappa_{i1} \left( \frac{v-u}{\sqrt{\theta}} \right) + \textnormal{kernel} \, ,
\end{align*}
where by ``kernel'' we mean terms which, when multiplied by $\mu(F)$, belong to the kernel of $\sf{P}^\perp_{\mu(F)}$ and so can be ignored using the self-adjointness of $(\sf{L}^\kappa_{\mu(F)})^{-1}$.  Now using~\cite[Lemma 6.1]{DuanYangYuRarefaction}, which asserts that $\langle \hat A^\kappa_i , B^\kappa_{jk} \rangle=0$ for any $i,j,k$ and the self-adjointness of $(\sf{L}^\kappa_{\mu(F)})^{-1}$ again, we do not need to consider the portion of $v_1 \partial_x \mu(F)$ (inside the argument of $\sf{L}^\kappa_{\mu(F)} \sf{P}^\perp_{\mu(F)}$) which is parallel to $\hat A_1^\kappa$. The second part of the second term in~\eqref{eq:dxM} can therefore be ignored, since it is a linear combination of $\hat A_1^\kappa$ and elements of the kernel of $\sf{P}^\perp_{\mu(F)}$. Then for the first part, we can subtract $u_1$ for free and expand
\begin{align*}
    \frac{(v_k-u_k)\partial_x u_k}{\theta} (v_1-u_1) = \frac{\partial_x u_k}{\theta} \left( \hat B^\kappa_{k1} (v-u) + \frac 13 \delta_{k1} |v-u|^2 \right) = \partial_x u_k \hat B^\kappa_{k1} \left( \frac{v-u}{\sqrt{\theta}} \right) + \textnormal{kernel} \, .
\end{align*}
We finally obtain that
\begin{align*}
  \int_{R^3} v_1 v_i \left(\sf{L}^\kappa_{\mu(F)}\right)^{-1}(\sf{P}^\perp_{\mu(F)}( (v_1-s_0) \partial_x \mu(F))) &=  \int_{\R^3} \theta \hat B^\kappa_{i1} \left( \frac{v-u}{\sqrt{\theta}} \right) \partial_x u_k \left(\sf{L}^\kappa_{\mu(F)}\right)^{-1}\left[ \hat B^\kappa_{k1} \left( \frac{v-u}{\sqrt{\theta}} \right) \right] \\
    &= \partial_x u_k \theta \int_{\R^3} \hat B^\kappa_{i1} \left( \frac{v-u}{\sqrt{\theta}} \right) B^\kappa_{k1} \left( \frac{v-u}{\sqrt{\theta}} \right) \, .
\end{align*}
Using that $\langle \hat B^\kappa_{ii} , B^\kappa_{ii} \rangle$ is positive and independent of $i$, and $\langle \hat B^\kappa_{ij} , B^\kappa_{kl} \rangle$ is zero unless either $(i,j)=(k,l)$, $(i,j)=(l,k)$, or $i=j$ and $k=l$ (see again~\cite[Lemma~6.1]{DuanYangYuRarefaction}), we have that the only relevant terms are those with $\langle \hat B^\kappa_{i1} , B^\kappa_{i1} \rangle$, i.e. $i=k$. At this point {we set}
$$ \mu^\kappa_i(\theta) =  \theta \int_{\R^3} \hat B^\kappa_{i1} \left( \frac{v-u}{\sqrt{\theta}} \right) B^\kappa_{i1} \left( \frac{v-u}{\sqrt{\theta}} \right)  $$
for $i=1,2,3$, where we are \emph{not} summing over the repeated index $i$ in the above formula.  Finally, using that $\mu_2=\mu_3$ and $\langle \hat B^\kappa_{11}, B^\kappa_{11}\rangle = 2 \langle \hat B^\kappa_{12}, B^\kappa_{12}\rangle + \langle \hat B^\kappa_{11}, B^\kappa_{22} \rangle$, where all terms involved are positive, and $\langle \hat B^\kappa_{21}, B^\kappa_{21} \rangle = \langle \hat B^\kappa_{12}, B^\kappa_{12} \rangle$, we define $\mu^\kappa(\theta)$   
$$  \mu^\kappa(\theta)=  \mu^\kappa_2(\theta) = \mu^\kappa_3(\theta) \, . $$
In order to see the relation between $\mu_2^\kappa$ and $\mu_1^\kappa$, namely $\mu_1^\kappa=\sfrac 43 \cdot \mu_2^\kappa$, we compute the integrals in polar coordinates and use that there exists a scalar function $b$ such that $B_{ij}^\kappa(v) = b(|v|^2) \hat B_{ij}^\kappa(v)$; see for example~\cite[eq~(3.67)]{Golse}.  This concludes the derivation of the density and momentum balances for compressible Navier-Stokes, and thus the derivation of~\eqref{NSE:outline} from the $\kappa=0$ case.  

Finally, we must still compute the viscosity in the energy balance, namely
\begin{equation}\label{last:compute}
 \int_{\R^3} \frac{v_1}{2} |v|^2 \partial_x \left( \left(\sf{L}^\kappa_{\mu(F)}\right)^{-1} \sf{P}^\perp_{\mu(F)} (v_1 \partial_x \mu(F) ) \right) \, . 
\end{equation}
We first expand
\begin{align*}
    \frac 12 v_1 |v|^2 &= \frac 12 (v_1 - u_1) |v|^2 + \frac 12 u_1 |v|^2 \\
    &= \frac 12 (v_1 - u_1) |v-u|^2  - \frac 12 (v_1 - u_1) \left[ |u|^2 - 2\langle u ,v \rangle \right] + \textnormal{kernel} \\
    &= \frac 12 (v_1 - u_1) |v-u|^2 + v_1 \langle u, v\rangle + \textnormal{kernel} \, .
\end{align*}
As usual, by ``kernel'' we mean terms which are in the kernel of $\sf{P}^\perp_{\mu(F)}$.  From the term that has $v_1\langle u,v\rangle$, we pick up terms identical to the ones we just calculated above, just with an extra factor of $u$.  So we focus our efforts on the first term, which is a linear combination of $\hat A^\kappa_1(v-u)$ and elements of the kernel, which we ignore.  Going back to \eqref{eq:dxM}, we can just consider the term with $|v-u|^2v_1$, which will be a linear combination of $\hat A^\kappa_1$ and elements of the kernel; all other terms involving $\hat B^\kappa_{\bullet}$ can be ignored due to the orthogonality condition $\langle \hat A^\kappa_\bullet, B^\kappa_\bullet \rangle=0$. Counting factors of $\theta$, we see that we are short a factor of $\theta$ in the denominator, so we pick up a $\theta$ in the numerator and then set 
$$ \kappa(\theta) = \theta \int_{\R^3} \hat A^\kappa_1 \left( \frac{v-u}{\sqrt{\theta}} \right) A^\kappa_1  \left( \frac{v-u}{\sqrt{\theta}} \right) \, , $$
which is positive; see again~\cite[section~6]{DuanYangYuRarefaction}.  Finally, we deduce that to leading order,
\begin{align}
\begin{cases}
&\partial_x ( \varrho (u_1-s_0) )(t,x) = 0 \, , \notag\\
&\partial_x (\varrho u_1 (u_1-s_0) + p)(t,x) = \partial_x \left( \frac 43 \mu^\kappa(\theta)\partial_x u_1 \right) \, , \notag\\
&\partial_x (\varrho u_i (u_1-s_0)) = \partial_x(\mu^\kappa(\theta) \partial_x u_i) \qquad\qquad\qquad (i=2,3) \, , \notag \\
&\partial_x \left( \varrho \left( e + \frac{|u|^2}{2} \right) (u_1-s_0) + p u_1 \right) = \partial_x (\kappa^\kappa(\theta) \partial_x \theta) \\
&\qquad \qquad \qquad + \partial_x \left(  \frac 43 \mu^\kappa(\theta) u_1 \partial_x u_1 \right) + \partial_x \left(\mu^\kappa(\theta)u_2 \partial_x u_2 \right) + \partial_x \left(\mu^\kappa(\theta)u_3 \partial_x u_3 \right) \, , \notag 
\end{cases}
\end{align}
where the viscosity coefficients $\mu^\kappa$ and $\kappa^\kappa$ are smooth positive functions of $\theta$ given by the formulas above.
\medskip

\noindent\texttt{Stability of viscosity coefficients for $\kappa\ll 1$.}  From \eqref{first:compute} and~\eqref{last:compute}, which crucially are equivalent to~\eqref{BB}, we have that
\begin{align*}
(B_\kappa - B_0)(U_{\rm NS}) = \sf{P}_{\mu_0} \left[ v_1 \left( \left[ \sf{L}_{\mu_{\rm NS}}^{-1} - \left(\sf{L}_{\mu_{\rm NS}}^\kappa\right)^{-1} \right] (\sf{P}_{\mu_{\rm NS}}^\perp [ (v_1-s_0) \sf{P}_{\mu_{\rm NS}}])\right) \right] \, .
\end{align*}
We are interested in computing the error in operator norm.  Note that $B_\kappa$ and $B_0$ are \emph{linear} operators on the finite-dimensional vector space $\mathbb{U}$ defined in~\eqref{def:sfP}, and so there is no need to specify the norms in any of the following computations.  From the standard expansion of $(\sf{L}_{\mu_{\rm NS}}^\kappa)^{-1}$ in a Neumann series for $\kappa$ sufficiently small, we have that
\begin{align*}
-\sf{L}_{\mu_{\rm NS}}^{-1} + \left(\sf{L}_{\mu_{\rm NS}}^\kappa\right)^{-1}  &= -\sf{L}_{\mu_{\rm NS}}^{-1} + \left(\sf{L}_{\mu_{\rm NS}} \left[ \Id + \underbrace{\kappa \sf{L}_{\mu_{\rm NS}}^{-1} \sf{L}_{\mu_{\rm NS}}^{\rm R}}_{:=\mathcal{L}} \right]\right)^{-1} \\
&= -\sf{L}_{\mu_{\rm NS}}^{-1}  + \left(\sum_{j=0}^\infty \mathcal{L}^j (-1)^j \right) \sf{L}_{\mu_{\rm NS}}^{-1} \\
&= \left(\sum_{j=1}^\infty \mathcal{L}^j (-1)^j \right) \sf{L}_{\mu_{\rm NS}}^{-1} \, .
\end{align*}
From this, we see that 
\begin{equation}\label{op:est}
\| B_\kappa - B_0 \| \les \kappa
\end{equation}
 for a $\kappa$-independent implicit constant.

\section{Construction and analysis of the Navier-Stokes shocks}
\label{sec:appendixode}


\subsection{The compressible Navier-Stokes equations}
We consider mass, momentum, and energy densities:
\begin{equation}\notag
(\varrho,\varrho u, \varrho(\sfrac{1}{2}|u|^2 + e)) = (\varrho,m,E) \, .
\end{equation}
These are the ``conservative variables.'' We require the equation of state
\begin{equation}\notag
p = \varrho \theta \, ,
\end{equation}
where the absolute temperature $\theta$ is proportional to internal energy (per unit mass) $e$:
\begin{equation}\notag
\theta = (\gamma - 1) e \, .
\end{equation}
Here $\gamma > 1$ is the adiabatic exponent; $\gamma=\sfrac 53$ for monatomic gas. We further require the stress tensor
\begin{equation}\notag
\tau = 2\mu D + \lambda \Id \div u  \, ,
\end{equation}
where
\begin{equation}\notag
D = \frac{1}{2} (\nabla u + \nabla u^T)
\end{equation}
is the deformation tensor. In general, we allow the viscosity coefficients $\mu, \lambda > 0$  and thermal conductivity $\kappa > 0$ to be smooth functions of $\varrho$ and $\theta$. For the compressible Navier-Stokes equations for monatomic gas derived from the Landau equations with one-dimensional spatial dependence, we saw in the previous appendix that $\lambda$, $\mu$, and $\kappa$ are smooth positive functions of $\theta$; in particular none of these coefficients exhibit $\varrho$ dependence.

The compressible Navier-Stokes equations are
\begin{equation}
\begin{aligned}
\p_t \varrho + \div(\varrho u) &= 0 \\
\p_t (\varrho u) + \div (\varrho u \otimes u + p I ) &= \div \tau \\
\p_t (\varrho (e+\sfrac{1}{2}|u|^2)) + \div (\varrho u(e+\sfrac{1}{2}|u|^2) + up) &= \div (\tau \cdot u) + \div \left( \kappa \nabla \theta \right) \, ,
\end{aligned}\notag
\end{equation}
and we express them in the variables $(\varrho,m,E)$ as
\begin{equation}
	\label{eq:fullthreedim}
\begin{aligned}
\p_t \varrho + \div m &= 0 \\
\p_t m + \div \left( \frac{m \otimes m}{\varrho} + p I \right) &= \div \tau \\
\p_t E + \div \left( \frac{m}{\varrho} (E+p) \right) &= \div \left( \frac 1 \varrho \tau \cdot m \right) + \div \left( \kappa \nabla \theta \right) \, ,
\end{aligned} 
\end{equation}
where 
\begin{align*}
\tau(\varrho,m,E) = \mu  \left( \nabla \left( \frac m \rho \right) + \nabla \left( \frac m \rho \right)^T \right) + \lambda \Id \div \left( \frac m \varrho \right)  \, &, \qquad  \theta(\varrho,m,E)=(\gamma-1)\left(\frac{E}{\varrho}-\frac{|m|^2}{2\varrho^2}\right) \, , \\
p = \varrho \theta = \varrho (\gamma-1)&\left(\frac{E}{\varrho}-\frac{|m|^2}{2\varrho^2}\right) \, .
\end{align*}
We assume that the solutions are translation-invariant in $x_2$ and $x_3$; that is, solutions are functions of $x_1 = x$ only. The equations become
\begin{equation}
	\label{eq:Euleronespatial}
\begin{aligned}
\p_t \varrho + \p_x m_1 &= 0 \\
\p_t m + \p_x \left( \frac{m m_1}{\varrho} + p e_1 \right) &= \p_x \tau_{\bullet 1} \\
\p_t E + \p_x \left( \frac{m_1}{\varrho} (E+p) \right) &= \p_x \left( \frac 1 \varrho \tau \cdot m \right)_1 + \p_x \left(\kappa \p_x (\gamma-1) \left( \frac E \varrho - \frac{|m|^2}{2\varrho^2} \right) \right) \, ,
\end{aligned} 
\end{equation}
where $\tau_{\bullet 1}$ is the first column of $\tau$, and $(\sfrac 1 \varrho \tau \cdot m)_1$ is the first entry of the matrix-vector product $\sfrac 1 \varrho \tau \cdot m$.  We may write the equations~\eqref{eq:Euleronespatial} succinctly in the form
\begin{equation}\notag
\p_t U + \p_x F(U) = \p_x (B(U) \p_x U) \, ,
\end{equation}
where $U = (\varrho,m,E)$,
\begin{equation}\notag
F(U) = \begin{pmatrix} m_1 \\ \frac{m m_1}{\varrho} + p e_1  \\ \frac{m_1}{\varrho} (E+p) \end{pmatrix}
\end{equation}
is the flux, and the dissipation matrix $B(U)$ is defined by
\begingroup
\renewcommand*{\arraystretch}{1.8}
\begin{equation}
  \label{eq:B:computed}
\begin{bmatrix}
0 & 0 & 0 & 0 & 0 \\
\frac{- (2\mu + \lambda) m_1}{\varrho^2} & \frac{(2\mu + \lambda)}{\varrho} & 0 & 0 & 0 \\
\frac{-2\mu m_2}{\varrho^2} & 0 & \frac{2\mu}{\varrho} & 0 & 0 \\
\frac{-2\mu m_3}{\varrho^2} & 0 & 0 & \frac{2\mu}{\varrho} & 0  \\
\substack{\frac{ |m|^2 \left( - 2\mu + \kappa(\gamma-1) \right)}{\varrho^3} \\ - \frac{\lambda m_1^2}{\varrho^3} - \frac{\kappa(\gamma-1)}{\varrho^2} E }& \frac{m_1}{\varrho^2}\left( (2\mu+\lambda) - {\kappa(\gamma-1)} \right) & \frac{m_2}{\varrho^2}\left( 2\mu - {\kappa(\gamma-1)} \right) & \frac{m_3}{\varrho^2}\left( 2\mu - {\kappa(\gamma-1)} \right) & \frac{\kappa(\gamma-1)}{\varrho}
\end{bmatrix} \, .
\end{equation}
\endgroup

Setting the right-hand side to zero in~\eqref{eq:Euleronespatial}, one obtains the three-dimensional Euler equations with one-dimensional spatial dependence.  These equations are in fact hyperbolic, a property inherited from the full three-dimensional model~\eqref{eq:fullthreedim}~\cite{markfelder}, since $F$ is simply the flux function in the first spatial variable. The standard one-dimensional Euler equations, i.e.,~\eqref{eq:Euleronespatial} with $m_2 = m_3 = 0$, have two genuinely non-linear eigenvalues, $u \pm c$, and one linearly degenerate eigenvalue, $u$, where $c = \sqrt{\sfrac{\gamma p}{\varrho}}$ is the sound speed.
 In the fully one-dimensional setting, the wave speeds $u+c$ and $u-c$ are associated with shocks and $u$ with contact discontinuities. In our setting, the eigenvalue $u$ is a triple eigenvalue, with two additional directions associated to the $m_2$ and $m_3$ variables.

The differential of the flux function is
\begingroup
\renewcommand*{\arraystretch}{1.8}
\begin{equation}\notag
\nabla F =
  \left[\begin{array}{@{}c|ccc|c@{}}
   0  & 1 & 0 & 0  & 0 \\ \hline
    -\frac{m_1^2}{\varrho^2} + \p_\varrho p & \frac{2m_1}{\varrho} + \p_{m_1} p & \p_{m_2} p& \p_{m_3} p& \p_E p \\
-\frac{m_1 m_2}{\varrho^2} & \frac{m_2}{\varrho}  & \frac{m_1}{\varrho} & 0 & 0 \\
-\frac{m_1 m_3}{\varrho^2} & \frac{m_3}{\varrho} & 0 & \frac{m_1}{\varrho} & 0 \\ \hline
-\frac{m_1}{\varrho^2}(E+p) + \frac{m_1}{\varrho} \p_\varrho p & \frac{1}{\varrho} (E+p) + \frac{m_1}{\varrho} \p_{m_1} p & \frac{m_1}{\varrho} \p_{m_2} p& \frac{m_1}{\varrho} \p_{m_3} p & \frac{m_1}{\varrho} (1+\p_E p)
  \end{array}\right] \, .
\end{equation}
\endgroup
For our analysis, it will be enough to compute most quantities at $(\varrho,u,\theta) = (1,0,0,0,1)$, which corresponds to $(\varrho,m,E) = (1, 0, 0, 0,(\gamma-1)^{-1}) =: U_L$. The sound speed $c = \sqrt{\gamma}$, $\p_\varrho p = 0$, $p=1$, and $\partial_E p = \gamma-1$. Therefore,
\begin{equation}
  \label{eq:fluxevaluated}
(\nabla F)(U_L) =
  \begin{bmatrix}
  0 & (1,0,0) & 0 \\
  0_{3\times 1} & 0_{3\times3} & (\gamma-1,0,0)^T \\
  0 & (\sfrac{\gamma}{(\gamma-1)},0,0) & 0
  \end{bmatrix}
\end{equation}
The right eigenvectors $r_p$ corresponding to $-c$, $0$, and $c$ at $U_L$ may be chosen to be, respectively,
\begin{equation}
  \label{eq:rightevecs}
(1,-c,0,0,\sfrac{\gamma}{(\gamma-1)}) \, , \qquad   e_1 \, , e_3 \, , e_4 \, , \qquad r_5 := - (1,c,0,0,\sfrac{\gamma}{(\gamma-1)}) \, .
\end{equation}
It can be verified directly that $c = \sqrt{\sfrac{\gamma p}{\varrho}}$ is genuinely nonlinear in a neighborhood of $U_L$:
\begin{align}
\Lambda &:= \left(\nabla_{\varrho, m , E}\left(\frac{m_1}{\varrho}+c\right)\right)(U_L) \cdot r_5 \notag\\
&= \left( - 2^{-1} \gamma^{\sfrac 12}, 0, 0, 0, 2^{-1} \gamma^{\sfrac 12}(\gamma-1) \right) \cdot \left( -1, -c, 0, 0, \frac{-\gamma}{\gamma-1} \right)^T \notag\\
&= \frac 12 \gamma^{\sfrac 12}(1-\gamma) < 0 \, .  \label{eq:genuinenonlinearity}
\end{align}
The left eigenvectors $l_p$ are
\begin{equation}\notag
\frac{\gamma-1}{2\gamma} \left(0,-\frac{c}{\gamma-1},0,0,1\right) \, , \quad \left(1,0,0,0-\frac{(\gamma-1)}{\gamma}\right) \, , e_3 \, , e_4 \, , \quad l_5 := - \frac{\gamma-1}{2\gamma} \left(0,\frac{c}{\gamma-1},0,0,1\right) \, .
\end{equation}
Finally, we compute the dissipation matrix $B(U_L)$ using~\eqref{eq:B:computed}, obtaining
\begin{equation}
  \label{eq:dissipationevaluated}
B(U_L) =  \left[\begin{array}{@{}c|ccc|c@{}}
0 & & & & \\
\hline
     & 2\mu+\lambda &  &  &  \\
& &  \mu & & \\
& & & \mu & \\ \hline
-\kappa & &  & & (\gamma-1) \kappa 
  \end{array}\right] \, .
\end{equation}


Suppose that $U(x,t)$ is a solution of such a system. Then $(T_s U)(x-st,t)$ is also a solution. Indeed, we calculate
\begin{equation}\notag
\p_t (T_s U(x-st,t)) = A_s (\p_t U)(x-st,t) - s A_s (\p_x U)(x-st,t)
\end{equation}
\begin{align*}
\p_x F(T_s U(x-st,t)) &= (\nabla F)(U(x-st,t)) A_s (\p_x U)(x-st,t)\\
&= A_s [(\nabla F)(U(x-st)) + sI] (\p_x U)(x-st,t)
\end{align*}
\begin{equation}\notag
\p_x B(T_s U(x-st,t)) \p_x T_s U(x-st,t) = A_s \p_x B(U(x-st,t)) (\p_x U)(x-st,t) \, .
\end{equation}
In our setting, the corresponding linearized equation is
\begin{equation}\notag
- s_L \p_x U + (\nabla F)(U_{\rm NS}^{(\varrho_L,u_L,\theta_L,s_L)}) \p_x U = \p_x B(U_{\rm NS}^{(\varrho_L,u_L,\theta_L,s_L)}) \p_x U \, .
\end{equation}
Then $V = T_{-\varepsilon} U$ satisfies the equation
\begin{equation}\notag
- (s_L-\varepsilon) \p_x V + (\nabla F)(U_{\rm NS}^{(\varrho_L,u_L-\varepsilon,\theta_L,s_L-\varepsilon)}) \p_x V = \p_x B(U_{\rm NS}^{(\varrho_L,u_L-\varepsilon,\theta_L,s_L-\varepsilon)}) \p_x V \, .
\end{equation}

\begin{remark}[Galilean invariance]
  \label{rmk:galileaninvariance}
For any $s \in \R$, we consider the Galilean boost $T_s : \R^5 \to \R^5$ defined by
\begin{equation}\notag
T_s (\varrho,m,E) = (\varrho,m+\varrho s e_1, E + (m_1+\varrho s) s) \, ,
\end{equation}
corresponding to $(\varrho,u,\theta) \mapsto (\varrho,u+s,\theta)$ in the non-conservative variables. Then $(T_s)_{s \in \R}$ is a group of linear transformations satisfying
\begin{itemize}
\item (flux transformation) $T_s (\nabla F)(U) = ((\nabla F)(T_s U) - sI) T_s$, and
\item (viscosity transformation) $T_s B(U) = B(T_s U) T_s $.
\end{itemize}
Consequently, $U(x,t)$ is a solution of the compressible Euler/Navier-Stokes equations if and only if $(T_s U)(x-st,t)$ is a solution. Likewise, the Rankine-Hugoniot conditions are Galilean invariant.

For any $U_L \in \R^+ \times \R^3 \times \R^+$, suppose that $U_{\rm NS}[U_L,s_L]$ is a viscous shock with left end-state $U_L$ traveling with speed $s_L$. Then $U_{\rm NS}[T_s U_L,s_L+s] := T_s U_{\rm NS}^{(\varrho_L,u_L,\theta_L,s_L)}$ is a viscous shock with left end-state $T_s U_L$ and speed $s_L+s$. If $U$ is a solution to the linearized equations
\begin{equation}\notag
[(\nabla F)(U_{\rm NS}[U_L,s_L]) - s_L I ] U = B(U_{\rm NS}[U_L,s_L]) U_x \, ,
\end{equation}
then $T_s U$ is a solution to the linearized equations
\begin{equation}\notag
[(\nabla F)(U_{\rm NS}[T_s U_L,s_L+s]) - (s_L + s)I ] U = B(U_{\rm NS}[T_s U_L,s_L+s]) U_x \, .
\end{equation}

In the analysis below, it is convenient to fix the left end-state as $(\varrho,u,\theta) = (1,0,1)$ and keep the shock speed $s = \sqrt{\sfrac{5}{3}} - \varepsilon$ in the ODE. In the remainder of the paper, it is more convenient to keep the shock speed $s = s_0 := \sqrt{\sfrac{5}{3}}$ and vary the left end-state as $(\varrho,u,\theta) = (1,\varepsilon,1)$. The Galilean invariance means that the ODE results described below will still be applicable.
\end{remark}

\subsection{Existence of small shocks}
From now on, we work with general viscous conservation laws satisfying certain properties, which we must verify for the compressible Navier-Stokes equations introduced above.  Let $u_L \in \R^n$ and $f = (f^1,\hdots,f^n)$ be a smooth, hyperbolic flux defined in an open neighborhood of $u_L$.
In a sufficiently small neighborhood of $u_L$, we may enumerate the (real) eigenvalues $\lambda_1(u) , \hdots, \lambda_n(u)$ of $\nabla f(u_L)$ and consider associated left and right eigenvectors $l_1(u), \hdots, l_n(u)$ and $r_1(u), \hdots, r_n(u)$, respectively, satisfying $l_i(u) r_j(u) = \delta_{ij}$. The Rankine-Hugoniot conditions are
\begin{equation}
	\label{eq:rankinehugoniot}
s \llbracket u \rrbracket = \llbracket f \rrbracket \, ,
\end{equation}
where $s$ is the shock speed, $\llbracket u \rrbracket = u_L - u_R$ and $\llbracket f \rrbracket = f(u_L) - f(u_R)$.  We assume that $\lambda_p(u_L)$ is a simple eigenvalue.  Then in a sufficiently small neighborhood of fixed $u_L$, the unique solutions $(u_L,u_R,s)$ to~\eqref{eq:rankinehugoniot} consist of the trivial solutions $(u_L,u_L,s)$, $s \in \R$, and the points $u_R$ on the \emph{Hugoniot curve}, which may be smoothly parameterized in $\mu$ to satisfy
\begin{equation}
  \label{eq:urparam}
u_R(\mu) = u_L + \mu r_p(u_L) + O(\mu^2) \, , \quad s(\mu) = \lambda_p(u_L) + \frac{\mu}{2} (\nabla \lambda_p)(u_L) \cdot r_p(u_L) + O(\mu^2) \, ;
\end{equation}
see~\cite[Chapter 6, Theorem 1.1]{lefloch} or~\cite[Section~8.2]{dafermos}. We say that the $p^{\rm th}$ characteristic field is \emph{genuinely nonlinear} in a neighborhood of $u_L$ if $(\nabla \lambda_p)(u) \cdot r_p(u) \neq 0$ in that neighborhood (see~\eqref{eq:genuinenonlinearity}). Without loss of generality, we assume that the sign is negative, so that 
\begin{equation}\label{def:Lambda}
\Lambda = (\nabla \lambda_p)(u_L) \cdot r_p (u_L)<0 \, . 
\end{equation}
In this case, we may parameterize the curve $(u_R,s)$ in $\varepsilon:=\lambda_p(u_L)-s=s_L-s=-\sfrac{\mu\Lambda}{2} + O(\varepsilon^2)$, where we adopt the notation $s_L = \lambda_p(u_L)$. 

\begin{proposition}[Existence of small shock profiles]
\label{pro:smallshocksgeneralhyperbolic}
Let $u_L \in \R^n$ and $f$ be a smooth, hyperbolic flux defined in a neighborhood of $u_L$. Suppose that $s_L=\lambda_p(u_L)$ is a simple eigenvalue of $\nabla f(u_L)$, and assume that the $p^{\rm th}$ characteristic field is genuinely nonlinear and satisfies~\eqref{def:Lambda}. Define
\begin{equation}\notag
A := (\nabla f)(u_L) - s_L \Id \, ,
\end{equation}
where $\Id$ is the $n \times n$ identity matrix. Let $B(u)$ be a smooth $n\times n$ matrix-valued function defined in a neighborhood of $u_L$ with block decomposition
\begin{equation}
A = \begin{bmatrix}
A_{11} & A_{12} \\
A_{21} & A_{22}
\end{bmatrix} \, , \quad
B = \begin{bmatrix}
0 & 0 \\
B_{21} & B_{22} 
\end{bmatrix} \, ,  \label{eq:no:one}
\end{equation}
corresponding to the decompositions
\begin{equation}\notag
u = (u^I,u^{II}) \in \R^{n-m} \times \R^m \, , \quad f = (f^I,f^{II}) \in \R^{n-m} \times \R^m
\end{equation}
of $u$ and $f$ into two components.  Suppose the following conditions are satisfied at $u = u_L$:
\begin{enumerate}
\item\label{cond2} The square matrix $\begin{bmatrix} A_{11} \, A_{12} \\ B_{21} \, B_{22} \end{bmatrix}$, evaluated at $u_L$, has full rank.
\item\label{condinvertibility}  Consider the $(n-m)\times n$ matrix
\begin{equation}
\begin{bmatrix} A_{11} \, A_{12} \end{bmatrix} = \nabla f^I - s_L \begin{bmatrix}  \Id_{(n-m)\times (n-m)} \, 0_{(n-m)\times m} \end{bmatrix} =: \nabla f^I - s_L \tilde \Id \notag 
\end{equation}
obtained by taking the first $n-m$ rows of the $n\times n$ matrix $A$.  We assume that for any $\xi \in \R\setminus\{0\}$,
\begin{equation}
  \label{eq:invertibilityassumption}
(-\xi^2 B(u_L) + i \xi  A)|_{\mathbb{C}\cdot \ker (\nabla f^I - s_L \tilde \Id)} \text{ is one-to-one} \, ,
\end{equation}
where $\mathbb{C}\cdot \ker (\nabla f^I - s_L \tilde \Id)\cong \mathbb{C}^m$ is the complexification of $ \ker (\nabla f^I - s_L \tilde \Id)\cong \R^m$.
\item\label{cond3}  We assume that
\begin{equation}
	\label{eq:assumption}
\alpha := (l_p \cdot B r_p)(u_L) > 0 \, .
\end{equation}
\end{enumerate}
Then for $0 < \varepsilon \ll 1$, there exists a family of solutions $u^\varepsilon : \R \to \R^n$ to the viscous shock ODE
\begin{equation}\label{vs:ODE}
f(u)_x - s(\varepsilon) u_x = (B(u) u_x)_x \, ,
\end{equation}
satisfying $u^\varepsilon \to u_L$ as $x \to -\infty$ and $u^\varepsilon \to u_R(\varepsilon)$ as $x \to +\infty$, where we adopt the parameterization in~\eqref{eq:urparam}.   In addition, there exists $M\geq 1$ such that these solutions satisfy, for all $k \geq 0$, the decay estimates
\begin{subequations}\label{eq:decay}
\begin{align}
\label{eq:decaytoleftendstate}
|\p_x^k (u_\varepsilon - u_L)| \leq C_k \varepsilon^{k+1} e^{-\varepsilon |x|/M} \, , \quad x \leq 0 \, , \\
\label{eq:decaytorightendstate}
|\p_x^k (u_\varepsilon - u_R(\varepsilon))| \leq C_k \varepsilon^{k+1} e^{-\varepsilon |x|/M} \, , \quad x \geq 0 \, .
\end{align}
\end{subequations}
Finally, $u^\varepsilon$ is unique, up to translations, among solutions close to $u_L$, and the family of solutions $\{u^\varepsilon\}_\varepsilon$ varies continuously with respect to $\varepsilon$ in the $C^k$ topology for any $k \geq 0$.
\end{proposition}

Proposition~\ref{pro:smallshocksgeneralhyperbolic} might be regarded as an elaboration on~\cite{MajdaPego,Pego}, which are only missing the decay estimates~\eqref{eq:decay}, obtained by expanding the ODE on the center manifold to second order. See also~\cite{FreistuhlerReview} and~\cite[Appendix A]{mascia2002pointwise}. For the compressible Navier-Stokes equations, existence and uniqueness of \emph{large} shock profiles were obtained in~\cite{gilbarg1951existence,GP53}, also without the decay estimates~\eqref{eq:decay}.

Before proving Proposition~\ref{pro:smallshocksgeneralhyperbolic}, we comment on the various assumptions and verify them for the compressible Navier-Stokes equations~\eqref{eq:Euleronespatial} in conservative variables at $u_L = (1,0,0,0,(\gamma-1)^{-1})$, $s_L = \sqrt{\gamma}$.

\begin{remark}[Non-degeneracy of traveling wave ODE]\label{remark2}
Assumption~\ref{cond2} is equivalent to
\begin{equation}
\textnormal{$\nabla f^I-s_L \tilde \Id$ is full rank, and $B|_{\ker (\nabla f^I- s_L \tilde \Id)}$ is one-to-one.}   \label{needed}
\end{equation}
To see the equivalence, split the domain $\R^n = (\ker \nabla f^I-s_L \tilde \Id)\oplus Q$, where $Q$ is any complementary subspace. Then $\nabla(f^I-s_L \tilde \Id)$ is one-to-one on $Q$. Now rewrite the block matrix $\begin{bmatrix} A_{11} \, A_{12} \\ B_{21} \, B_{22} \end{bmatrix}$ as a block matrix according to the decomposition $Q \oplus \ker \nabla f^I-s_L \tilde \Id$ in the domain and $\R^m \oplus \R^{n-m}$ in the target, where the copy of $\R^m$ corresponds to the image of $Q$ under $B$. It becomes
\begin{equation}\notag
\begin{bmatrix}
\nabla(f^I- s_L \tilde \Id)|_Q  & 0 \\
B|_Q & B|_{\ker \nabla (f^I-s_L \tilde \Id)}
\end{bmatrix} \, ,
\end{equation}
which clarifies the equivalence.

The relevance of Assumption~\ref{cond2} is the following:
$\nabla f^I - s_L \tilde \Id$ full rank yields, by the implicit function theorem, that the equation $f^I(u) - f^I(u_L) - s (u^I-u_L^I) = 0$ describes a smooth $m$-dimensional manifold locally near $(u,s) = (u_L,s_L)$. Its tangent space at $(u_L,s_L)$ is $\ker \nabla f^I$. Since $B|_{\ker \nabla f^I} : {\rm ker} \, \nabla f^I \to {\rm ran} B$ is invertible, the traveling wave ODE~\eqref{vs:ODE} can be a realized as a smooth ODE on this manifold, as we will see below.  In the case of the compressible Navier-Stokes equations, Assumption~\ref{cond2} may be verified using~\eqref{eq:fluxevaluated},~\eqref{eq:dissipationevaluated}, $s_L = \sqrt{\gamma}$, and a block decomposition in which $m=4$, so that $\begin{bmatrix} A_{11} \, A_{12} \end{bmatrix}= \begin{bmatrix} -s_L \, 1 \, 0 \, 0 \, 0 \end{bmatrix}$.
\end{remark}

\begin{remark}[Symmetrizability and Kawashima condition]
\label{rmk:symandkawashima} We now remark upon the conditions in Assumption~\ref{condinvertibility} by making the following two claims, which we then prove.
\begin{enumerate}
\item\label{claim1} Suppose that the system is symmetrizable at $(u_L,s_L)$ in the sense that there exists $A^0 \in \R^{n \times n}$ symmetric and positive-definite satisfying
\begin{equation}
A^1 := \nabla f(u_L) A^0 \, \text{ is symmetric, and  } B^{11} := B(u_L) A^0 \text{ symmetric and non-negative } \, .  \notag
\end{equation}
Suppose furthermore that the Kawashima condition
\begin{equation}\notag
\{ \text{eigenvectors of } A^1 - s_L A^0 \} \cap \{ \ker B^{11} \} = \{ 0 \}
\end{equation}
is satisfied. Then Assumption~\ref{condinvertibility} is satisfied.
\item\label{claim2} The compressible Navier-Stokes equations are symmetrizable at $(u_L,s_L)$ and satisfy the Kawashima condition.
\end{enumerate}

\end{remark}
\begin{proof}[Proof of Remark~\ref{rmk:symandkawashima}]  We first prove claim~\ref{claim1}.  It is immediate that Assumption~\ref{condinvertibility} will follow from the invertibility of $-\xi^2 B(u_L) + i\xi A$, which itself follows from the invertibility of $(-\xi^2 B(u_L) + i\xi A)A^0$ due to the assumption that $A^0$ is real, symmetric, and positive-definite (and hence invertible).  But the invertibility of $(-\xi^2 B(u_L) + i\xi A)A^0$ is precisely a result of the assumption that $A A^0$ satisfies the Kawashima condition in the sense of Lemma~\ref{lem:kawashimalem} and an application of Lemma~\ref{lem:steadyestimate}.

To prove claim~\ref{claim2} in general, we recall that the existence of a viscosity-compatible convex entropy guarantees symmetrizability~\cite[Proposition~1.19]{ZumbrunNotes}. For the compressible Navier-Stokes equations, the entropy is $\eta = -\varrho S$, with corresponding entropy flux $-u \varrho S$, and its viscosity-compatibility and convexity are verified in~\cite[Remark 1.24, Example 1.25]{ZumbrunNotes} and~\cite[Section~4]{KS88}. In this setting, we may choose $A^0 = (\nabla^2 \eta)^{-1}$.   From~\cite[Section~4]{KS88}, we have that
\begin{equation}\notag
A_0 = \begin{bmatrix}
1 & 0 & 0 & 0 & (\gamma-1)^{-1} \\
0 & 1 & 0 & 0 & 0 \\
0 & 0 & 1 & 0 & 0 \\
0 & 0 & 0 & 1 & 0 \\
(\gamma-1)^{-1} & 0 & 0 & 0 & \frac{\gamma}{(\gamma-1)^2} 
\end{bmatrix} \, ,
\end{equation}
which is easily seen to be symmetric and positive definite. Then referring to~\eqref{eq:fluxevaluated}, we can compute $A A_0$ directly, obtaining
\begin{equation}
A A_0 = \begin{bmatrix}
-s_L & 1 & 0 & 0 & \frac{-s_L}{\gamma-1} \\
1 & -s_L & 0 & 0 & \frac{\gamma}{\gamma-1} \\
0 & 0 & -s_L & 0 & 0 \\
0 & 0 & 0 & -s_L & 0 \\
\frac{-s_L}{\gamma-1} & \frac{\gamma}{\gamma-1} & 0 & 0 & \frac{-s_L \gamma}{(\gamma-1)^2} 
\end{bmatrix} \, .  \notag
\end{equation}
Next, referring to~\eqref{eq:dissipationevaluated}, we can compute $B(u_L) A_0$, obtaining
\begin{equation}
B(u_L) A_0 = \begin{bmatrix}
0 & 0 & 0 & 0 & 0 \\
0 & 2\mu+\lambda & 0 & 0 & 0 \\
0 & 0 & \mu & 0 & 0 \\
0 & 0 & 0 & \mu & 0 \\
0 & 0 & 0 & 0 & \kappa
\end{bmatrix} \, .  \notag
\end{equation}
Then the kernel of $B(u_L) A_0 = \textnormal{span}\{e_1\}$, and computing $A A_0 e_1$, we obtain $(-s_L, 1, 0, 0, \sfrac{-s_L}{\gamma-1})$, so $e_1$ is not an eigenvector of $A A_0$.  Therefore the Kawashima condition is satisfied, and thus Assumption~\ref{condinvertibility} is satisfied by~\ref{claim1}.
\end{proof}

\begin{remark}
Assumption~\ref{cond3} ensures that the one-dimensional ODE arising in the center manifold reduction below flows in the correct direction. In the case of the compressible Navier-Stokes equations,
\begin{equation}
\alpha := (l_p \cdot B r_p)(u_L) = \frac{1}{2} \left[ (2\mu + \gamma) + \frac{\kappa(\gamma-1)^2}{\gamma} \right] > 0 \, .  \notag
\end{equation}
\end{remark}

\newcommand{\ovs}{\overline{s}}
\begin{proof}[Proof of Proposition~\ref{pro:smallshocksgeneralhyperbolic}]
The $x$-integrated shock profile equation is
\begin{equation}
\label{eq:integratedshockode}
f(u) - f(u_L) - s(u-u_L) = B(u) u_x \, .
\end{equation}
From~\eqref{eq:no:one}, the equation in the $I$ component is
\begin{equation}
\label{eq:equationtodiff}
f^I(u) - f^I(u_L) - s(u^I-u_L^I) = 0 \, ,
\end{equation}
which we may consider as a functional equation $F(u, s)=0$ for unknown pairs $(u,s)$.  More precisely, we have that $F: \R^{n-m} \times \R^{m}\times \R\rightarrow \R^{n-m}$, and $u=(u_1, u_2)$ with 
$$u_2 \in X_2 := \ker (\nabla f^I(u_L) - s_L \tilde \Id) \cong \R^m$$
 belonging to the $m$-dimensional kernel and $u_1 \in X_2^\perp \cong \R^{n-m}$ belonging to the perpendicular space.  By assumption~\ref{cond2} and~\eqref{needed}, $\nabla f^I(u_L) - s_L \tilde \Id : \R^{n}\rightarrow \R^{n-m}$ has full rank, and is therefore invertible off of its kernel $X_2$.  Therefore, the $(n-m)\times (n-m)$ Jacobian matrix $\frac{\partial F}{\partial u_1}$ is non-singular, and so the implicit function theorem yields a unique, smooth, local parameterization of the solutions $(u,s)=(u_1,u_2,s)$ to~\eqref{eq:equationtodiff} as
\begin{equation}\label{eq:u1u2:pram}
(u_1,u_2,s) = \left( u_1(u_{2},s),u_2,s \right) \, , \qquad u_1(u_{L,2}, s_L) = u_{L,1} \, ,
\end{equation}
where $u_L=(u_{L,1}, u_{L,2})$. After performing a translation by $u_{L,2}$ in the $u_2$ component of $u_1(u_2,s)$, we abuse notation slightly by notating our new parametrization still by $u_1$, which now satisfies $u_1(0,s_L) = u_{L,1}$.  We claim now that $\p_s u_1(0,s_L) = 0$, which follows from the identity
$$  \p_s u_1(0,s_L) =  - \left[ \frac{\p F}{\p u_1} (u_{L,1}, u_{L,2}, s) \right]^{-1} \cdot \frac{\partial F}{\partial s} (u_{L,1}, u_{L,2}, s) \, . $$
Indeed the second matrix (actually an $(n-m)\times 1$ vector) above is easily computed to be $(u^I-u^I_L)|_{u_L}$, which gives $0$.  Finally, we note that $\frac{\p u_1}{\p u_2}(0,s_L)=0$ by applying a similar formula as above, and using that since $u_2\in X_2$, $\frac{\p F}{\p u_2}(u_L,s)=0$.

Now recall from~\eqref{needed} that $B(u_L)|_{X_2}$ is injective, and recall from~\eqref{eq:no:one} that $\textnormal{range} \, B \subseteq \R^m$, where $\R^m$ corresponds to the $II$ component of $\R^n$. Therefore $\textnormal{range} \, B(u_L) \cong \R^m$, where $\R^m$ is the $II$ component of $\R^n$, and by the smoothness of $B$ in a neighborhood of $u_L$, we furthermore have that $\textnormal{range} \, B(u) \cong \R^m$ for $|u-u_L|\ll 1$. Applying these considerations, taking the $II$ component of equation~\eqref{eq:integratedshockode} by applying the $m\times n$ projection matrix $P_{II}^{j\ell}: \R^n \rightarrow X_2 \cong \R^m$, and using~\eqref{eq:no:one} and $\frac{\p u_1}{ \p u_2}(0,s_L)=0$, we find that
\begin{align}
\notag
P_{II}^{j\ell} &\left[f^\ell(u_1(u_2,s), u_2) - f^\ell(u_{L,1}, u_{L,2}) - s\left( [u_1(u_2,s),u_2]^\ell - u_L^\ell \right) \right] \notag\\
&= P_{II}^{j\ell} B^{\ell k}(u_1(u_2,s), u_2) \p_x u^k \notag \\
&= B^{jk}(u_1(u_2,s), u_2) \begin{pmatrix}
\frac{\p u_1^k}{\p u_2^\ell}(u_2,s) \p_x u_2^\ell \\ \p_x u_2^k
\end{pmatrix} \notag \\
&= B^{jk}(u_1(u_2,s), u_2) \p_x u_2^k  \, ,	\label{eq:odeforu2}
\end{align}
which is an ODE for $u_2$. Since $B^{jk}$ is a nonsingular $m\times m$ matrix with (range equal to the $II$ subspace) and inverse which we denote by $(B^{-1})^{ij}$ (with domain $II$),~\eqref{eq:odeforu2} is equivalent to
\begin{align}
\dot u_2^i &= (B^{-1})^{ij}(u_1(u_2,s),u_2) P_{II}^{j\ell}\left[f^\ell(u_1(u_2,s), u_2) - f^\ell(u_{L}) - s\left( [u_1(u_2,s),u_2]^\ell - u_L^\ell \right) \right] \, ,   \label{eq:newodeforu2}
\end{align}
where $\dot u_2$ signifies differentiation in $x$. Steady states of~\eqref{eq:newodeforu2} satisfy the Rankine-Hugoniot conditions, and we have already classified steady states in a neighborhood of $(u_L,s_L)$. Obviously, $u = u_L$ is a steady state regardless of $s$; otherwise, the steady states are $u_R(\varepsilon)$ as in~\eqref{eq:urparam}.

We now consider the system of ODEs for $(u_2,s)$ which consists of~\eqref{eq:newodeforu2} together with the equation $\dot s = 0$. Computing the linearization of this system of ODEs at the equilibrium $(u_L,s_L)$ gives
\begin{equation}   \label{system}
\begin{aligned} 
\dot u_2^i &= \left( B^{-1} \right)^{ij}(u_L) P_{II}^{jl} \left[ \p_{u_2^q} f^\ell(u_L) u_2^q - s_L u_2^\ell \right] \, , \\
\dot s &= 0 \, ;
\end{aligned}
\end{equation}
we note that in the above computation, we have used that $\p_{u_2} u_1(0,s_L)=0$, and $\p_s u_1(0, s_L)=0$.  Finally, recall that the vector $r_p$ is in the kernel of $A$, since it is the eigenvector corresponding to the $p^{\rm th}$ eigenvalue $s_L$ of $\nabla f (u_L)$; thus $r_p$ also belongs to $X_2$.  Therefore, $(u_2,s)=(r_p,s)$ is in the kernel of the above linearized operator.

We wish to consider the center manifold associated to the system of ODEs for the augmented vector field $(\dot u_2, \dot s)$ at $(u,s) = (u_L,s_L)$.  We first prove that its linearization has no non-zero eigenvalues on the imaginary axis.
\begin{lemma}
\label{lem:spectralassertion}
The spectrum of the linearization satisfies
\begin{equation}
	\label{eq:spectralassertion}
\sigma \left[ \left( B^{-1} \right)^{kj}(u_L) P_{II}^{jl} \left[ \p_{u_2^q} f^\ell(u_L)  - s_L \delta^{\ell q} \right] \right] \cap i\R = \{ 0 \} \, ,
\end{equation}
and zero is an algebraically simple eigenvalue.
\end{lemma}
\begin{proof}
The conclusion~\eqref{eq:spectralassertion} fails to hold for matrices of the form $B^{-1} C$ precisely when there exists $v\in \C^m = \C \cdot \R^m$ and $\xi \in \R\setminus\{0\}$ such that $B^{-1}C v= i \xi v$, which is equivalent to $Cv = B i \xi v$, or to a non-trivial kernel for the matrix $-\xi^2 B + i \xi C$.  Thus the assertion~\eqref{eq:spectralassertion} is equivalent to the injectivity of
\begin{equation}
 -\xi^2  B(u_L) |_{\mathbb{C} \cdot X_2} + i \xi  A|_{\mathbb{C} \cdot X_2}  \notag
\end{equation}
on $\mathbb{C}\cdot X_2$, 
which is the content of assumption~\ref{condinvertibility}.  In order to prove that zero is a simple eigenvalue, we first note that its geometric multiplicity is one by the assumption that $s_L$ is a simple eigenvalue of $\nabla f$.  In order to prove that its algebraic multiplicity is one, suppose towards a contradiction that $v_2\in X_2$ is a generalized eigenvector satisfying 
$$\left[ B^{-1} P_{II} (\p_{u_2} f - s_L \Id_{m\times m}) \right] v_2 = r_p \, . $$
Then 
\begin{equation}
P_{II} A v_2 = B r_p \qquad \implies \qquad A v_2 = B r_p \, ,  \notag
\end{equation}
where the implication is a consequence of the fact that $P_I A$ has full rank, and so $X_2 = \ker P_I A$ must be mapped to $II=I^\perp$.  
Now pairing with the left eigenvector $l_p$, we obtain the contradiction
\begin{equation}\notag
l_p A v_2 = 0 = l_p \cdot B r_p \overset{\eqref{eq:assumption}}{\neq} 0 \, .
\end{equation}
\end{proof}

Returning to the proof of Proposition~\ref{pro:smallshocksgeneralhyperbolic}, we have that the center subspace associated to the zero eigenvalue is two-dimensional.   Therefore, we invoke the center manifold theorem (see for example~\cite{HIbook}) to obtain a two-dimensional center manifold for~\eqref{eq:newodeforu2}, augmented with the equation $\dot s=0$. This center manifold is foliated by one-dimensional invariant submanifolds $\{ s = \text{constant} \}$. 
We parameterize these invariant submanifolds via
\begin{equation}\label{eq:u2:pram}
u_2 = u_2(\eta,s) = \eta  r_p(u_L) + h(\eta,s) \, , \quad h(\eta,s): \R \times B_\delta(s_L) \subset \R^2 \rightarrow X_2 \, , \quad h(\eta,s)=O(\eta^2+|s-s_L|^2) \, ,
\end{equation}
where $\delta$ is sufficiently small, and $h(0,s)=0$ for all $s$, so that $(0,s)$ corresponds to the steady solution $(u_2,s)=(u_{L,2},s)$ of~\eqref{eq:newodeforu2}.  From now on, $r_p$ and $l_p$ will always be evaluated at $u_L$, and so we suppress the dependency.  Tangency to the center subspace means that $h(0,s_L) = 0$ and $\nabla_{\eta,s} h(0,s_L) = 0$. Applying the chain rule to $\dot u_2$ and using that $\dot s=0$, the ODE~\eqref{eq:newodeforu2} for $u_2(x)$ becomes an ODE for $\eta(x)$:
\begin{align*}
[r_p^i + \p_\eta h^i(\eta,s)] \dot \eta& = (B^{-1})^{ij}(u_1(\eta r_p + h(\eta, s),s),\eta r_p + h(\eta, s)) P_{II}^{j\ell} \\
&\qquad \cdot \bigg{[} f^\ell(u_1(\eta r_p + h(\eta, s),s),\eta r_p + h(\eta, s)) - f^\ell(u_L) \\
&\qquad \qquad - s \left( [u_1(\eta r_p + h(\eta, s),s),\eta r_p + h(\eta, s)]^\ell - u_L^\ell \right) \bigg{]} \, .
\end{align*}
It will be crucial that $\eta$ is \emph{increasing} along the flow between the two steady states $u_L$ and $u_R$.  We know that $\eta = 0$ corresponds to $u_2=0$, and thus to $u_L$.  We set $\eta = \bar{\eta}(s) := - 2\varepsilon/\Lambda + O(\varepsilon^2)$ to be the value of $\eta$ corresponding to $u_R$, where we have obtained this formula by recalling the parametrization for the curve $(u_R,s)$ in terms of $\varepsilon=s_L-s$ from underneath~\eqref{def:Lambda}. Notably,~\eqref{def:Lambda} implies that $\bar{\eta}(s) > 0$ when $0 < \varepsilon \ll 1$, where positive $\varepsilon$ corresponds to shock speeds satisfying the Lax shock inequalities.

To determine the direction, we rewrite
\begin{equation}\label{manifold:ODE:2}
B^{mi}(\eta,s) [r_p^i + \p_\eta h^i(\eta,s)] \dot \eta = P^{m\ell}_{II} \left[ f^{\ell}(\eta,s) - f^{\ell}(u_L) - s\left( u^{\ell}(\eta,s) - u_L^{\ell} \right) \right] \, ,
\end{equation}
and left-multiply by $l_p$. Since $B(0,s_L)=B(u_L)$ and $\p_\eta h(0,s_L)=0$, we obtain 
\begin{equation}\notag
[ l_p^m B^{mi}(u_L) r_p^i + d(\eta,s) ] \dot \eta = \underbrace{l_p^m P^{m\ell}_{II} \left[ f^{\ell}(\eta,s) - f^{\ell}(u_L) - s\left( u^{\ell}(\eta,s) - u_L^{\ell} \right) \right]}_{G(\eta,s)} \, ,
\end{equation}
where 
\begin{align}
d(\eta, s) &= l_p^m B^{mi}(\eta,s)\left[ r_p^i + \partial_\eta h^i(\eta,s) \right] - l_p^m B^{mi}(u_L) r_p^i \notag \\
&= l_p^m B^{mi}(0,s) r_p^i + \int_0^\eta l_p^m \partial_{\tilde \eta} B^{mi}(\tilde \eta, s) \, \dee \tilde \eta - l_p^m B^{mi}(u_L) r_p^i \notag \\
&\qquad + l_p^m B^{mi}(\eta, s) \left[ \partial_\eta h^i(0,s) + \int_0^\eta \partial_{\tilde \eta}^2 h^i(\tilde \eta, s) \, \dee \tilde \eta \right] \notag \\
&= \int_{s_L}^s l_p^m \partial_{\tilde s}B^{mi}(0, \tilde s) r_p^i \, \dee \tilde s + \int_0^\eta l_p^m \partial_{\tilde \eta} B^{mi}(\tilde \eta, s) \, \dee \tilde \eta \notag \\
&\qquad + l_p^m B^{mi}(\eta, s) \left[ \int_{s_L}^s \partial_{\tilde s} \partial_\eta h^i(0, \tilde s) \, \dee \tilde s + \int_0^\eta \partial_{\tilde \eta}^2 h^i(\tilde \eta, s) \, \dee \tilde \eta \right]  r_p^i \notag \\
&= O(|s-s_L|+|\eta|)
\end{align}
and $\alpha := l_p \cdot B(u_L) r_p > 0$ by~\eqref{eq:assumption}. Hence,
\begin{equation}\label{eq:eta:def}
\dot \eta = F(\eta,s) \, , \qquad F(\eta,s) = \frac{G(\eta,s)}{\alpha + d(\eta,s)} \, .
\end{equation}

To simply show the existence of heteroclinic orbits, without quantitative estimates, it is enough to show that $F > 0$ when $0 < \eta < \bar{\eta}(s)$. For example, it would be enough to show that $\p_\eta F(0,s) > 0$ for $s < s_L$.  We will obtain more precise information by expanding to \emph{second order}. We know that $\p_s \bar{\eta} < 0$ by recalling that $\bar{\eta}(s)=\frac{-2\varepsilon}{\Lambda} + O(\varepsilon^2)=\frac{-2(s_L-s)}{\Lambda} + O((s_L-s)^2)$, differentiating in $s$, and using~\eqref{def:Lambda}. Then we claim that
\begin{equation}\notag
G(0,s) = 0 = G(\bar{\eta}(s),s) \, , \quad \nabla_{\eta,s} G(0,s_L) = 0 \, .
\end{equation}
In fact the first equality is a consequence of our choice of parametrization $h$, which satisfies $h(0,s)=0$ for all $s\in B_\delta(s_L)$, the second equality for $G$ is a consequence of the Rankine-Hugoniot conditions, and the vanishing of $\nabla_{\eta, s}G(0,s_L)$ is a consequence of 
\begin{subequations}
\begin{align}
\p_\eta G &= l_p^m P_{II}^{m\ell} \left[ \nabla_k f^\ell(\eta, s) \frac{\p u^k}{\p \eta} - s \frac{\p u^\ell}{\p \eta} \right] \, , \\
\p_s G &= l_p^m P_{II}^{m\ell} \left[ \nabla_k f^\ell(\eta, s) \frac{\p u^k}{\p s} - s \frac{\p u^\ell}{\p s} - (u^\ell - u^\ell_L) \right] \\
\frac{\p u_2}{\p \eta}\bigg{|}_{(\eta,s)=(0,s_L)} &= r_p \, , \quad \frac{\p u_1}{\p u_2}\bigg{|}_{(u_2,s)=(0,s_L)}=0  \, , \quad \frac{\p u_2}{\p s}\bigg{|}_{(\eta,s)=(0,s_L)} = 0 \, .
\end{align}\label{comps:for:G}
\end{subequations}
Therefore, $F(0,s_L)$ and $\nabla_{\eta,s} F(0,s_L)$ also vanish. Consequently, 
$$\nabla_{\eta,s}^2 F(0,s_L) = \alpha^{-1} \nabla_{\eta, s}^2 G(0,s_L) \, , $$
so we compute the second derivatives of $G$. Evidently, $(\p_s^k G)(0,s_L) = 0$ for all $k$, since differentiating $G$ any number of times along the families of steady states yields zero; equivalently, $G(0,s)$ depends on $s$ through $u_2(0,s)$, which however vanishes for all $s$ by our choice of $h$. Next, we compute $\p_s \p_\eta G$ by using~\eqref{comps:for:G}, obtaining
\begin{align*}
\p_s \p_\eta G = l_p^m P_{II}^{m\ell} \left[ \nabla_{kj}f^\ell \frac{\p u^j}{\p s} \frac{\p u^k}{\p \eta} + \nabla_k f^\ell \frac{\p^2 u^k}{\p s \p \eta} - \frac{\p u^\ell}{\p \eta} - s \frac{\p^2 u^\ell}{\p s \p \eta} \right] \, .
\end{align*}
At $(\eta,s)=(0,s_L)$, we use~\eqref{comps:for:G} to simplify this to 
$$l_p^m P_{II}^{m\ell} \left( \left[ \nabla_k f^\ell(u_L) - s_L \delta_k^\ell \right] \frac{\p^2 u^k}{\p s \p \eta}(0,s_L) - r_p^\ell \right) \, .  $$
However, at $(\eta,s)=(0,s_L)$, we have that
\begin{align*}
  \frac{\p^2 u^k}{\p s \p \eta} \bigg{|}_{(0,s_L)} = \frac{\p}{\p s} \left[ \frac{\p u^k_1}{\p u_2^j} \frac{\p u_2^j}{\p \eta} + \frac{\p u_2^k}{\p \eta} \right]  \bigg{|}_{(0,s_L)} &= \left[ \frac{\p^2 u^k_1}{\p u_2^j \p u_2^m} \frac{\p u_2^j}{\p \eta}\frac{\p u_2^m}{\p s} +  \frac{\p u^k_1}{\p u_2^j} \frac{\p^2 u_2^j}{\p \eta \p s}+ \frac{\p^2 u_2^k}{\p \eta \p s} \right]  \bigg{|}_{(0,s_L)}\\
&=  \frac{\p^2 u_2^k}{\p \eta \p s} \bigg{|}_{(0,s_L)} \in X_2 \, .
\end{align*}
Then since $X_2=\ker(\nabla f - s_L \Id)$, we deduce that $\p_s \p_\eta G(0,s_L)=-l_p^m P_{II}^{m \ell} r_p^\ell$.  Note, however, that the left kernel of $A$ is equal to the span of $l_p$ by the assumption that $s_L$ is a simple eigenvalue and the fact that row rank is equal to column rank.  Therefore, $l_p$ is perpendicular to all the columns of $A$, and therefore perpendicular to the range of $A$, which contains $I$ by~\eqref{needed}.  Finally, we deduce then that $\p_s \p_\eta G(0,s_L)=-l_p^m P_{II}^{m \ell} r_p^\ell= -l_p^\ell r_p^\ell=-1$.
Finally,
\begin{align*}
\p_\eta^2 G = l_p^m P_{II}^{m\ell} \left[ \nabla_{kj} f^\ell \frac{\p u^k}{\p\eta}\frac{\p u^j}{\p\eta} + \nabla_k f^\ell \frac{\p^2 u^k}{\p\eta^2} - s \delta_k^\ell \frac{\p^2 u^k}{\p\eta^2} \right] \, .
\end{align*}
Notice that since $l_p^m P_{II}^{m\ell}=l_p^\ell$, the second and third terms above vanish.  Focusing then on the first term evaluated at $(\eta,s)=(0,s_L)$, we have
\begin{align*}
l_p^\ell \nabla_{kj} f^\ell(u_L) r_p^k r_p^j = l_p^\ell \partial_j \lambda_p(u_L) r_p^\ell r_p^j = \Lambda <0
\end{align*}
by genuine nonlinearity.  Finally, we deduce that $\p_\eta^2 G(0,s_L)=\Lambda$.

We now summarize the asymptotics by Taylor expanding $F$ to second order around $(\eta,s)=(0,s_L)$, obtaining
\begin{align}
F(\eta,s) &= \frac{\alpha^{-1}}{2} \begin{bmatrix} 0 & -1 \\ -1 & \Lambda \end{bmatrix} : \begin{bmatrix} (s-s_L)^2 & (s-s_L)\eta \\ (s-s_L)\eta & \eta^2 \end{bmatrix} + O(|\eta|^3 + \varepsilon^3) \notag\\
&={\alpha^{-1}} \left(\varepsilon \eta + \frac{\Lambda}{2} \eta^2 \right) + O(|\eta|^3 + \varepsilon^3) \, .  
\label{eq:basicexpression}
\end{align}
Differentiating, we obtain
\begin{equation}\notag
\p_\eta F(\eta,s) = {\alpha^{-1}} \left(\varepsilon + \Lambda \eta \right) + O(|\eta|^2 + \varepsilon^2) \, .
\end{equation}
Hence, crucially,
\begin{equation}
  \label{eq:spec1}
\p_\eta F(0,s) = {\alpha^{-1}}\varepsilon + O(\varepsilon^2) =: m^u \, ,
\end{equation}
and
\begin{equation}
\label{eq:spec2}
\p_\eta F(\bar{\eta},s) = -\alpha^{-1}\varepsilon + O(\varepsilon^2) =: m^s \, .
\end{equation}

At this point we have constructed a trajectory $\eta(x):\R\rightarrow [0, -\sfrac{2\varepsilon}{\Lambda}+O(\varepsilon^2)]$ which corresponds to a solution $u_2$ of~\eqref{eq:newodeforu2}, and thus a solution to the integrated shock profile equation~\eqref{eq:integratedshockode}.  We note that $\eta$ is unique up to translations, which implies uniqueness of $u^\varepsilon$ (as stated in the Proposition), in a neighborhood of $u_L$, up to translations.  Then, fixing a translate, for example, by choosing $\eta(0) = \bar{\eta}(s)/2$, we obtain quantitative estimates for $\eta(x)$ by considering the unstable and stable manifolds at the equilibria $F(0,s)$ and $F(\bar{\eta}(s),s)$, respectively, and using that $\eta$ converges to these equilibria exponentially; more precisely, we have that
\begin{equation}
		\label{eq:firstexptoendstates}
|\eta(x)| \lesssim \varepsilon e^{c\varepsilon x} \, ,\quad  x < 0 \, ,
\end{equation}
\begin{equation}
		\label{eq:secondexptoendstates}
|\eta(x) - \bar{\eta}| \lesssim \varepsilon e^{-c\varepsilon x} \, ,\quad x > 0 \, .
\end{equation}
These estimates in turn imply estimates for $u_2=u_2(\eta,s)=\eta r_p + h(\eta,s)$, which in turn provides the estimates in~\eqref{eq:decay} for $k=0$. The derivative estimates~\eqref{eq:decay} may then be obtained by bootstrapping: To control $\p_x^k u^\varepsilon$, it will again be enough to estimate $\p_x^k \eta$ and appeal to the parameterizations~\eqref{eq:u1u2:pram} and~\eqref{eq:u2:pram}. For $k=1$, we use that $\dot \eta = F(\eta,s)$, the above expansion of $F$, and the exponential convergence to end states~\eqref{eq:firstexptoendstates}-\eqref{eq:secondexptoendstates}. For $k=2$, we differentiate the equation to obtain $\ddot \eta = F_\eta(\eta,s) \dot \eta$ and use that $F_\eta(\eta,s)$ is $O(\varepsilon)$. The estimates for arbitrary $k \geq 2$ can be obtained by continuing in this fashion.

The continuity in $\varepsilon$ of the family $\{u^\varepsilon\}_\varepsilon$ in $C^k$ topologies can be shown as follows.  First, recall that $\varepsilon=s_L-s$, and so continuity with respect to $\varepsilon$ is equivalent to continuity with respect to $s$.  From~\eqref{eq:urparam}, the solutions of the Rankine-Hugoniot conditions are continuous with respect to $s$.  Next, we see from \eqref{eq:eta:def} that $\eta$ varies continuously with respect to $s$.  This in turn implies continuity of $u_2$ in the $C^0$ topology with respect to $s$ from~\eqref{eq:u2:pram} and the center manifold theorem. This then implies continuity with respect to $s$ of solutions $u=(u_1,u_2,s)$ to~\eqref{eq:integratedshockode} from~\eqref{eq:u1u2:pram} and the implicit function theorem, concluding the proof of continuity in the $C^0$ topology.  Continuity in higher norms follows from interpolation and the uniform estimates~\eqref{eq:decay}.
\end{proof}

\begin{corollary}[Spectral information]\label{cor:spectralknowledge}
For $0 < \varepsilon \ll 1$, the spectrum of
\begin{equation}
	\label{eq:matrix}
\left( B^{-1} \right)^{ij}(u) P_{II}^{j\ell} (\p_k f^\ell(u) - s(\varepsilon)\delta_k^\ell) : X_2 \to X_2
\end{equation}
evaluated at $u=u_L$ and $u=u_R$ consists of two parts:
\begin{enumerate}
\item $o_{\varepsilon \to 0^+}(1)$ perturbations of the $m-1$ non-central eigenvalues of
\begin{equation}\notag
\left( B^{-1} \right)^{ij}(u_L) P_{II}^{j\ell} A^{\ell k}
\end{equation}
\item An algebraically simple eigenvalue satisfying $\mu_L(\varepsilon) = \alpha^{-1}\varepsilon + O(\varepsilon^2)$ at $u_L$ and $\mu_R(\varepsilon) = -\alpha^{-1}\varepsilon + O(\varepsilon^2)$ at $u_R$. 
\end{enumerate}
In particular,~\eqref{eq:matrix} evaluated at $u_L$ and $u_R(\varepsilon)$ is hyperbolic and the dimensions of its unstable subspace evaluated at $u_L$ and its the stable subspace evaluated at $u_R$ sum to $n+1$.
\end{corollary}
\begin{proof}
(1) The continuity in $\varepsilon$ of the eigenvalues follows from standard finite-dimensional perturbation theory. (2) The central eigenvalue is algebraically simple, and so it varies smoothly with respect to smooth perturbations. The perturbed central eigenvalue is associated to the direction along the one-dimensional invariant manifold. We rewrote the ODE on the center manifold in the $\eta$-parameterization, that is, $u_2 = u_2(\eta,s)$. Since the spectrum of the linearization \emph{at an equilibrium} is independent of the parameterization of the ODE, the central eigenvalues are those from~\eqref{eq:spec1} and~\eqref{eq:spec2}.  Equivalently, one may compute the central eigenvalue directly and observe the claimed asymptotics; this computation is identical to the analysis of the ODE for $\eta$.
\end{proof}

\begin{remark}
The ODE~\eqref{manifold:ODE:2} on the center manifold is approximately
\begin{equation}
  \label{eq:centermanifoldode}
[ \alpha + O(|s-s_L|+|\eta|) ] \dot \eta = - \eta (s-s_L) + \frac{\Lambda}{2} \eta^2 + O(|\eta|^3 + |s-s_L|^3) \, .
\end{equation}
We recognize this as a perturbation of the profile equation for traveling wave solutions to viscous Burgers. Some of the above analysis could be simplified by making a normal form reduction to the logistic equation $\dot y = \varepsilon y (1-\sfrac{y}{\varepsilon})$. We intend to discuss this elsewhere.
\end{remark}

\subsection{Analysis of the linearized macroscopic equations}
\label{sec:analysislinearizedmacro}

Fix $0 < \varepsilon \ll 1$ and a shock profile $\bar{u}(x) = u^\varepsilon(x)$ with speed $s=s(\varepsilon)$, constructed using the assumptions and conclusions of Proposition~\ref{pro:smallshocksgeneralhyperbolic}. We have crucially that~\eqref{eq:linearizedODE} is equivalent to solving the macroscopic portion~\eqref{becomes} of the linearized Landau equation.
\begin{proposition}
  \label{pro:ODEpropappendix}
For all $g \in L^2 \cap C(\R)$ and $d \in \R$, the linearized macroscopic equation
\begin{equation}
  \label{eq:linearizedODE}
B(\bar{u}) u_x = (\nabla f(\bar{u}) - s \Id ) u + g
\end{equation}
has a unique solution $u \in L^2 \cap C(\R)$ satisfying a one-dimensional constraint
\begin{equation}
  \label{eq:constraint}
l_\varepsilon(u(0)) = d \, .
\end{equation}
The solution obeys the estimate
\begin{equation}\notag
\| u \|_{L^2} \lesssim \varepsilon^{-1} \| g \|_{L^2} + |d|  \, .
\end{equation}
\end{proposition}

Our main reference for this proposition is~\cite[p. 695-696]{MZ08}.

\begin{remark}
  \label{rmk:theconstraint}
The assumption that $g \in C(\R)$ is for convenience. For $g \in L^2$, the solution $u$ decomposes into a part in $L^2$ and a continuous part, and the one-dimensional constraint is imposed on the continuous part. The constraint~\eqref{eq:constraint} is that $z^c(0) = d \eta$, where $z^c$ is defined in~\eqref{eq:zsplit} and $\eta$ is defined above~\eqref{zcuse}. We may choose the shock profiles to be continuous-in-$\varepsilon$ in the $L^\infty$ topology (for example, the choice of translate in the proof of Proposition~\ref{pro:smallshocksgeneralhyperbolic} will do), so that $l_\varepsilon$ is continuous-in-$\varepsilon$. The functional $\ell_\varepsilon$ in Proposition~\ref{pro:macroscopicestimate} is more precisely
\begin{equation}\notag
\ell_\varepsilon = l_\varepsilon \circ T_{-\varepsilon} \circ I \, ,
\end{equation}
where $T_{-\varepsilon}$ is the Galilean transformation from Remark~\ref{rmk:galileaninvariance} and $I$, defined in~\eqref{op:i}, extracts hydrodynamic moments.
\end{remark}

\begin{proof}
The equation~\eqref{eq:linearizedODE} in the $I$ component is
\begin{equation}
  \label{eq:solvemyfirstequation}
0 = (\nabla f^I(\bar{u}) - s \tilde \Id) u + P_I g \, ,
\end{equation}
where $\tilde \Id$ is as in~\eqref{eq:invertibilityassumption}. Towards a parametrization of the solutions $u$ to this equation in terms of $u_2 \in X_2$, we first recall from the proof of Proposition~\ref{pro:smallshocksgeneralhyperbolic} that $\nabla f^I(u_L)-s_L\tilde \Id$ is invertible on $X_1$.  We claim that in fact $\nabla f^I(u_L)-s_L \tilde \Id$ is invertible on $\ker(B(u_L))$.  To see this, we claim that we can choose $Q=\ker(B(u_L))$ in Remark~\ref{remark2}.  If this were not the case, then the intersection of $\ker(B(u_L))$ and $\ker(\nabla f^I(u_L)-s_L\tilde\Id)$ would contain a non-zero vector, which would then lie in the kernel of the operator in assumption~\ref{cond2} from Propositon~\ref{pro:smallshocksgeneralhyperbolic}, which however was assumed to be full rank (invertible).  Now due to the smallness of $0<\varepsilon\ll 1$, we therefore have that $\nabla f^I(\bar{u})-s\tilde\Id$ is injective on $\ker(B(\bar{u}))$.  We thus define a collection of linear isomorphisms $q_1(\bar{u}): \ker(B(u_L))\rightarrow \ker(B(\bar{u}))$ parametrized by $\bar{u}$.  In a similar fashion, we claim that $B(\bar{u})$ is invertible on $\ker(\nabla f^I(\bar{u})-s\tilde\Id)$.  This follows again from the smallness of $\varepsilon$ and~\eqref{needed}.   Thus there exists a family of linear isomorphisms $q_2(\bar{u},s): \ker(\nabla f^I(u_L)-s_L\tilde\Id)=X_2 \rightarrow \R^n$, parametrized by $\bar{u}$, which map $X_2$ to $\ker(\nabla f^I(\bar{u})-s \tilde \Id)$.  

We therefore can parametrize all solutions of~\eqref{eq:solvemyfirstequation} as
$$  {u} =  q_1(\bar{u},s) u_1 + q_2(\bar{u},s)u_2  \, , \qquad u_1 \in X_1\, , \quad u_2 \in X_2 \, , $$
where we may solve for $u_1$ from~\eqref{eq:solvemyfirstequation}, obtaining
\begin{equation}
  \label{eq:u1def}
u_1 = - \left[(\nabla f^I(\bar{u})-s\tilde \Id) q_1(\bar{u},s) \right]^{-1} P_I g \, .
\end{equation}
Next, we note that our choice of parametrization and the estimates~\eqref{eq:decay} ensure that
\begin{equation}
B(\bar{u}) (q_1(\bar{u}) u_1)_x = B(\bar{u}) \p_x (q(\bar{u})) u_1 = O(\varepsilon^2)|P_I g| \, .  \notag
\end{equation}
In particular, no derivatives of $u_1$ enter the ODE for $u_2$.  It remains to solve for $\tilde{u} = u - q_1(\bar{u})u_1=q_2(\bar{u})u_2$. Let
\begin{equation}
\label{eq:tildegdef}
\tilde{g} = P_{II} g - B(\bar{u}) (q_1(\bar{u},s) u_1)_x \, .
\end{equation}
We have
\begin{equation}\notag
B(\bar{u}) \p_x\tilde{u} = P_{II}(\nabla f(\bar{u}) - sI) \tilde{u} + \tilde{g} \, ,
\end{equation}
which after writing out
$$  \p_x \tilde u(x) = \frac{\p q_2}{\p \bar{u}}(x) \p_x \bar{u}(x) u_2(x) + q_2(\bar{u}(x)) \p_x u_2(x)  $$
and using the injectivity of $B(\bar{u})$ on $\textnormal{ran} \, q_2(\bar{u}(x))$ (which follows again from the injectivity of $B(u_L)$ on $X_2$ and a small choice of $\varepsilon$) gives an ODE for $u_2$:
\begin{equation}\notag
\p_x u_2 = \left( B(\bar{u}) q_2(\bar{u}) \right)^{-1} P_{II} \left[(\nabla f(\bar{u}) - s\Id) q_2(\bar{u}) u_2 + \tilde{g} - B(\bar{u}) \frac{\p q_2}{\p \bar{u}} \p_x \bar{u} u_2 \right] \, .
\end{equation}
For brevity, we rewrite
\begin{equation}\notag
w = u_2 \, , \quad h = \left( B(\bar{u}) q_2(\bar{u}) \right)^{-1} \tilde{g} \, ,
\end{equation}
\begin{equation}\notag
m := \left( B(\bar{u}) q_2(\bar{u}) \right)^{-1} P_{II} \left(\nabla f(\bar{u}) - s \Id - B(\bar{u})\frac{\p q_2}{\p \bar{u}}\p_x \bar{u} \right)  \, .
\end{equation}
Then the ODE is
\begin{equation}\label{w:ODE}
w_x = m(x) w + h \, .
\end{equation}


We now summarize the spectral information on $m$. Notably, $m(x)$ satisfies the asymptotics
\begin{equation}\notag
|\partial_x^k (m - m_{\pm})| \lesssim \varepsilon^{k+1} e^{-\varepsilon |x|/M} \, , \quad \forall k \geq 0 \, ,
\end{equation}
by the estimates~\eqref{eq:decaytoleftendstate}--\eqref{eq:decaytorightendstate} in Proposition~\ref{pro:smallshocksgeneralhyperbolic}. The end states $m_\pm$ (hence, $m(x)$ itself) are $O(\varepsilon)$ perturbations of a fixed $m_0$, namely, the linearized operator in~\eqref{system}; that is, $m_{\pm} = m_0 + O(\varepsilon)$.  Furthermore, the endstates $m_{\pm}$ satisfy the spectral conclusions detailed in Corollary~\ref{cor:spectralknowledge}, assuming a sufficiently small choice of $\varepsilon$ relative to $\alpha^{-1}$.  
The fixed $m_0$ has $k$ stable eigenvalues (counting algebraic multiplicity), $n-k-1$ unstable eigenvalues, and a simple zero eigenvalue. Hence, for sufficiently small $\varepsilon$, there exists $\theta>0$ such that $m$ has spectrum consisting of $k$ ``strongly" stable eigenvalues $\lambda$ with ${\rm Re} \, k > \theta$, $n-k-1$ ``strongly" unstable eigenvalues with ${\rm Re} \, k < -\theta$, and a simple eigenvalue $\mu(x) = O(\varepsilon)$. We know from Corollary~\ref{cor:spectralknowledge} that $\mu(m_{\pm}) = \mp {\alpha^{-1}} \varepsilon + O(\varepsilon^2)$, and $|\mu(x) - \mu(m_\pm)|$ decays exponentially with the same rates as $m$ itself.  

We denote the spectral projections $P^s(m(x))$ onto the strongly stable subspace of $m(x)$, $P^u(m(x))$ onto the strongly unstable subspace, and $P^c(m(x))$ onto the central eigenspace. By standard perturbation theory, namely, smooth dependence of the resolvents, these projections onto the group eigenspaces depend smoothly on the argument $m(x)$ in a neighborhood of $m_0$. Since $m$ is fixed, we abbreviate $P^s(x) = P^s(m(x))$, $P^s_0 = P^s(m_0)$, etc.

We block diagonalize $m(x)$ along its strongly stable, strongly unstable, and center subspaces. Set
\begin{equation}
	\label{eq:Qdefapp}
Q(x) := P^s(x) P^s_0 + P^c(x) P^c_0 + P^u(x) P^u_0 \, ,
\end{equation}
and set
\begin{equation}\notag
\begin{aligned}
p(x) &= Q^{-1}(x) m(x) Q(x) \\
&= \underbrace{Q^{-1}(x) m(x) P^s(x) Q(x)}_{m^s(x)} + \underbrace{Q^{-1}(x) m(x) P^c(x) Q(x)}_{m^c(x)} + \underbrace{Q^{-1}(x) m(x) P^u(x) Q(x)}_{m^u(x)} \, .
\end{aligned}
\end{equation}
Notice that $Q(x)$ is uniformly-in-$\varepsilon$ invertible, since it is an $O(\varepsilon)$ perturbation of the identity, and $Q'(x) = O(\varepsilon^2)$. The maps $m^{\bullet}(x)$ are adapted to the fixed spectral subspaces $E^\bullet_0$ of $m_0$ for $\bullet=s,c,m$. Additionally, $m^s(x)$ is an $O(\varepsilon)$ perturbation of $m^s_0$, which has negative spectrum, and $m^u(x)$ is an $O(\varepsilon)$ perturbation of $m^u_0$, which has positive spectrum. Finally, $m^c(x) v = \mu(x) v$, where $v$ spans $E^c_0$, and $\mu(x)$ satisfies the aforementioned asymptotics. If we make the change of variables
\begin{equation}
	\label{eq:zdefwithQ}
z = Q^{-1} w \, ,
\end{equation}
then the ODE~\eqref{w:ODE} becomes
\begin{equation}\notag
z'(x) = p(x) z(x) + \underbrace{Q^{-1}(x) h(x) - (Q^{-1})'(x) Q(x) z(x)}_{:=\tilde{h}} \, .
\end{equation}
By our block diagonalization, bounded solutions $z(x)$ may be decomposed as
\begin{equation}
  \label{eq:zsplit}
z(x) = \underbrace{P_0^s z(x)}_{:=z^s} + \underbrace{P^c_0 z(x)}_{:=z^c} + \underbrace{P^u_0 z(x)}_{:=z^u} \, ,
\end{equation}
where
$$ (z^\bullet)'(x) = m^\bullet(x) z^\bullet(x) + P^{\bullet}_0 \tilde h(x) \, , \qquad \bullet=s,c,u \, . $$
Furthermore, there exists propagators $\Phi^s$ and $\Phi^u$ such that
\begin{equation}\notag
z^s(x) = \int_{-\infty}^x \Phi^s(x,y) \tilde{h}(y) \, dy \, , 
\end{equation}
\begin{equation}\notag
z^u(x) = -\int_{x}^\infty \Phi^u(x,y) \tilde{h}(y) \, dy \, .
\end{equation}
Fix $\eta$ with $|\eta| = 1$ such that ${\rm image} \, P^c_0 = {\rm span} \, \eta$. For the central eigenvalue, we integrate starting from $x=0$ to obtain
\begin{equation}\label{zcuse}
z^c(x) =  d \eta  + \int_0^x e^{\int_y^x \mu(y') \, dy'} \tilde{h}(y) \, dy \, .
\end{equation}
The constraint $l_\varepsilon(u(0)) = d$ is precisely that $z^c(0)  = d\eta$.
Due to the fact that $m^s(x)$ has strictly negative spectrum, and $m^u(x)$ has strictly positive spectrum, there exist $M_1$ and $M_2$ such that propagators satisfy
\begin{equation}\notag
\| \Phi^s(x,y) \| \lesssim e^{-(x-y)/M_1} \, , \quad x > y \, ,
\end{equation}
\begin{equation}\notag
\| \Phi^u(x,y) \| \lesssim e^{-(y-x)/M_1 } \, , \quad y < x \, ,
\end{equation}
and
\begin{equation}\notag
\left| e^{\int_y^x \mu(y') \, dy'} \right| \lesssim e^{-\varepsilon |x-y|/M_2} \, , \quad x, y \in \R \, .
\end{equation}
Notice that $\tilde{h}$ contains an $O(\varepsilon^2)$ linear term in $z$, which, due to its size, is considered lower order and can be added on perturbatively by a contraction argument.   In conclusion, Duhamel's formula yields
\begin{equation}\notag
\| u_2 \|_{L^2} \lesssim \varepsilon^{-1} \| \tilde{g} \|_{L^2} + |d|  \, ,
\end{equation}
and, recalling the definitions~\eqref{eq:u1def} and~\eqref{eq:tildegdef} of $u_1$ and $\tilde{g}$, respectively,
\begin{equation}\notag
\| u \|_{L^2} \lesssim \varepsilon^{-1} \| g \|_{L^2} + |d| \, .
\end{equation}
\end{proof}

\begin{remark}
	\label{rmk:Ihavelowerboundsonell}
The functional $l_\varepsilon$ is nearly a projection onto the subspace spanned by the right eigenvector~$r_p$:
\begin{equation}
l_\varepsilon = |r_p| (l_p+O(\varepsilon)) \cdot \, ,
\end{equation}
as can be seen from~\eqref{eq:Qdefapp},~\eqref{eq:zdefwithQ}, and~\eqref{zcuse}. From the ODE~\eqref{eq:eta:def} for the center manifold parameter $\eta$, the shocks in Proposition~\ref{pro:smallshocksgeneralhyperbolic} therefore satisfy
\begin{equation}
l_\varepsilon(\p_x u(x)) \gtrsim (1+O(\varepsilon)) \varepsilon e^{-c\varepsilon x} \, .
\end{equation}
This will be important in verifying uniqueness up to translation of the kinetic shocks in Section~\ref{sec:fixed:point}.
\end{remark}

\section{Kawashima compensators}\label{sec:kawashima}

\subsection{Basic construction}
Let $\C^m = X \oplus Y$ (orthogonal decomposition) and let $A : \C^m \to \C^m$ be a Hermitian matrix. Let $1 \leq n < m$ and assume that $\dim X = n$. We say that $A$ satisfies the \emph{Kawashima condition} if
\begin{equation}
    \label{eq:A}
    \text{for each non-zero } x \in X, \; A^k x \not\in X \text{ for some } k \in \{ 1, \hdots, n \} \, .
\end{equation}
This is equivalent to
\begin{equation}
\label{eq:B}
    \{ \text{eigenvectors of $A$} \} \cap X = \{ 0 \} \, .
    \end{equation}
\begin{proof}[Proof of equivalence]
($\neg\eqref{eq:B} \implies \neg \eqref{eq:A}$) If there exists a non-zero eigenvector $u$ of $A$ in $X$, then $A^k u \in X$ for all $k$.

($\neg\eqref{eq:A} \implies \neg \eqref{eq:B}$)  Let $x$ be such that $A^k x \in X$ for all $k\in \{1,\dots,n\}$.  Since $X$ is $n$-dimensional and we have assumed that $A^k x \in X$ for $k\in\{1,\dots,n\}$, there exists a smallest positive integer $d\in \{1,\dots,n\}$ for which the set of $d+1$ vectors $\{x, Ax, \dots, A^d x\}$ is a linearly dependent set.  Since $d$ is the smallest such integer, any linear relation with
$$  c_0 x + c_1 Ax + \dots + c_d A^d x = 0 \, , \qquad c_i \textnormal{ not all zero} $$
must be such that $c_d\neq 0$ to avoid contradicting the minimality of $d$.  Fix such a linear relation, which without loss of generality may be chosen such that $c_d=1$. Consider the operator 
$$  A|_{\textnormal{span}\{x, Ax, \dots, A^{d-1}x\}} \, .  $$
Then since $A^dx$ is expressible as a linear combination of $x,Ax, \dots, A^{d-1}x$, we have that $A|_{\textnormal{span}\{x, Ax, \dots, A^{d-1}x\}}$ is a linear operator with domain and range $\textnormal{span}\{x, Ax, \dots, A^{d-1}x\}$.  This linear operator must then have a non-zero eigenvector $x_0 \in X$, showing also that $A:\mathbb{C}^m\rightarrow \mathbb{C}^m$ has a non-zero eigenvector in $X$.
\end{proof}

The standard set-up is to suppose that $L : \R^m \to \R^m$ is a negative semi-definite Hermitian matrix which induces the subspaces $X = \ker L$, and $Y = (\ker L)^{\perp}$, and that $A$ satisfies the Kawashima condition with $X=\ker L$. We now assume throughout that $A, L$ satisfy these properties.

\begin{lemma}
    \label{lem:kawashimalem}
If the Kawashima condition is satisfied, then there exists a \emph{Kawashima compensator}, that is, a skew-adjoint operator $K$ satisfying
\begin{equation}
    \label{eq:Kawashimaconditionestimate}
    \| P_X u \|^2 \lesssim {\rm Re} \langle KA u, u \rangle + \| P_Y u \|^2 \, .
\end{equation}
\end{lemma}
\noindent Notice that ${\rm Re} \langle KA u, u \rangle = \frac{1}{2} \langle [K,A] u , u \rangle$ by (skew-)adjointness. The model equations to keep in mind are
\begin{equation}
\label{eq:equationtoestimate}
\begin{aligned}
\dot u &= i\tau  A u + Lu  \\
u(\cdot,0) &= u_0
\end{aligned}
\end{equation}
and
\begin{equation}
    \label{eq:steadyequation}
 i\tau A u + Lu + f = 0  \, ,
\end{equation}
 where $\tau \in \R$ is a non-zero parameter. 
Solutions to~\eqref{eq:equationtoestimate} and~\eqref{eq:steadyequation} can be estimated via a combination of two energy estimates:
\begin{enumerate}
\item the standard estimate in which we pair the equation with $u$, and
\item a `twisted' estimate in which we apply $iK$ to the equation, pair with $u$, and take ${\rm Re}$.
\end{enumerate}

\begin{lemma}[Steady case]
    \label{lem:steadyestimate}
    Let $(u,f)$ be a solution to~\eqref{eq:steadyequation}, and assume that the Kawashima condition is satisfied. Then
    \begin{equation}
    \| u \| \lesssim \max(1,|\tau|^{-2}) \| f \| \, . 
    \end{equation}
    In particular, $i\tau A + L$ is invertible. 
\end{lemma}
\begin{proof}
    The standard energy estimate yields
    \begin{equation}\label{eq:standard:one}
    \| P_Y u \|^2 \lesssim \| P_Y f \|^2 +  \| P_X f \| \| P_X u \|  \, .
    \end{equation}
    The `twisted estimate' yields the identity
    \begin{equation}
        \label{eq:twisted}
    - \tau {\rm Re} \langle KA u, u \rangle + {\rm Re} \langle iKLu,u \rangle = - {\rm Re} \langle iK f, u \rangle \, ,
    \end{equation}
    which gives
    \begin{equation}
        \label{eq:twistedwithtau}
    |\tau| \| P_X u \|^2 \overset{\eqref{eq:Kawashimaconditionestimate}}{\lesssim} \| P_Y u \| \| u \| + \| f \| \| u \| + |\tau| \| P_Y u \|^2 \, .
    \end{equation}
    If $|\tau| \geq 1$, we sum~\eqref{eq:standard:one} and $|\tau|^{-1} \times$~\eqref{eq:twistedwithtau}:
    \begin{equation}
        \label{eq:fromhere}
    \| u \|^2_{L^2} \lesssim \| P_Y f \|^2 +  \| P_X f \| \| P_X u \| + |\tau|^{-1} \| P_Y u \| \| P_X u \| + |\tau|^{-1} \| f \| \| u \| \, ,
    \end{equation}
    after using $\|P_Y u \| \| u \| \lesssim \| P_Y u \|^2 + \| P_Y u \| \|P_X u \|$ and the estimate~\eqref{eq:standard:one}. We then split products in~\eqref{eq:twistedwithtau} using Young's inequality, and absorb small $\| P u \|^2_{L^2}$ terms into the left-hand side to close the estimate:
    \begin{equation}\notag
    \| P u \|^2_{L^2} \lesssim \| f \|^2
    \end{equation}
    Therefore, we focus on $0 < |\tau| \leq 1$. We have
    \begin{equation}   
  \notag
    \| P_Y u \|^2 + \tau^2 \| P_X u \|^2 \lesssim \| P_Y f \|^2 + \| P_X f \| \| P_X u \| + |\tau| \| P_Y u \| \| P_X u \| + |\tau| \| f \| \| u \| \, ,
    \end{equation}
    which, after we split products and absorb small terms, begets
    \begin{equation}
        \label{eq:betterestimate}
    \| P_Y u \|^2 + \tau^2 \| P_X u \|^2 \lesssim \| P_Y f \|^2 + \tau^{-2} \| P_X f \|^2 \, .
    \end{equation}
\end{proof}

\begin{remark}
Notice that we actually have the more precise estimate~\eqref{eq:betterestimate}. In particular, when $P_X f = 0$, we have the improvement
\begin{equation}\notag
\| u \| \lesssim \max(1,|\tau|^{-1}) \| f \| \, ,
\end{equation}
which is akin to the situation in the steady Landau estimates in Lemma~\ref{lem:basicmicroest} (the right-hand side is purely macroscopic, that is, belongs to the damped subspace).
\end{remark}

We may similarly use the Kawashima compensator to recover coercivity in the time-dependent problem~\eqref{eq:equationtoestimate}.  In fact, one may choose $0 < \delta \ll (1+|\tau|^{-1})^{-1}$ such that
\begin{equation}
    \label{eq:hypocoercivityfunctional}
\Phi[u] := \frac{1}{2} \|u\|^2 + \frac{\delta}{2} ({\rm sgn} \, \tau) {\rm Re} \langle iK u, u \rangle
\end{equation}
 is a hypocoercivity functional.

\begin{lemma}[Time-dependent case]
    Let $u$ be a solution to~\eqref{eq:equationtoestimate}, and assume that the Kawashima condition is satisfied. Then
    \begin{equation}\notag
     \|u(t)\| \lesssim e^{-\mu \min(1,|\tau|^2) t} \|u_0\| \, , 
    \end{equation}
    where $\mu > 0$ is independent of $u_0$.
\end{lemma}

\begin{proof}
The standard energy estimate yields
\begin{equation}
    \label{eq:standardtimedep}
\frac{d}{dt} \frac{1}{2} \|u\|^2 \leq - c_1 \|P_Y u\|^2 \, .
\end{equation}
Next, we compute
\begin{equation}\notag
\frac{d}{dt} \frac{1}{2} \langle iKu,u\rangle = \frac{1}{2} \langle iK(i\tau A + L)u,u \rangle + \frac{1}{2} \langle iKu, (i\tau A+L)u \rangle \, .
\end{equation}
The right-hand side is
\begin{equation}\notag
\begin{aligned}
- \underbrace{\frac{1}{2} \tau \langle (KA-AK)u,u \rangle}_{\tau {\rm Re} \langle KA u, u \rangle} + \frac{1}{2} \langle iKLu,u \rangle + \frac{1}{2} \langle iKu,Lu \rangle
\end{aligned}
\end{equation}
Therefore, we obtain
\begin{equation}
\begin{aligned}
    \label{eq:twistedtimedepwithtau}
\frac{d}{dt} \frac{1}{2} ({\rm sgn} \, \tau) {\rm Re} \langle iKu,u\rangle &\leq - |\tau| {\rm Re} \langle KA u, u \rangle + C_1 \|P_Y u\| \left( \|P_X u\| + \| P_Y u \| \right) \\
&\leq - c_2 |\tau| \|P_X u\|^2 + C_2 (1+|\tau|^{-1}) \|P_Y u\|^2 \, .
\end{aligned}
\end{equation}
Finally, we define $\Phi[u]$ as in~\eqref{eq:hypocoercivityfunctional} and compute its time derivative (equivalently, add~\eqref{eq:standardtimedep} and $\delta \times~\eqref{eq:twistedtimedepwithtau}$). We have
\begin{equation}\notag
\frac{d}{dt} \Phi[u] \leq - c_3 \Phi[u] \, ,
\end{equation}
provided that $\delta \leq \sfrac{1}{2} \| iK \|_{\rm HS}$ (Hilbert-Schmidt norm) to ensure that $\Phi[u] \approx \sfrac{1}{2} \|u\|^2$ and $\delta \leq c_1 (2C_2)^{-1} (1+|\tau|^{-1})^{-1}$ to ensure that the last term on the right-hand side of $\delta \times \eqref{eq:twistedtimedepwithtau}$ can be absorbed into the $-c_1 \|P_Y u\|^2$ term on the right-hand side of~\eqref{eq:standardtimedep}. In particular, $c_3 \gtrsim \min( \delta c_2 |\tau|,c_1)$. When $0 < |\tau| \leq 1$, $c_3 \gtrsim |\tau|^2$.
\end{proof}


\begin{proof}[Proof of Lemma~\ref{lem:kawashimalem}]
To motivate the construction, we begin with a special case in which $Ax \not\in X$ for all nonzero $x \in X$.

\noindent\emph{Special case}.  Consider the operator $P_Y A|_X$. Our assumption yields that $P_Y A|_X$ is injective; indeed if $P_Y A(x_1) = P_Y A(x_2)$ for two nonzero vectors $x_1 \neq x_2$, then $A(x_1-x_2)\in X$, a contradiction. Therefore we have that $P_Y A|_X$ is a bijection between $X$ and its image $Z := {\rm im}(P_Y A|_X)$. We recall the decomposition $\C^m = X \oplus Y$ and rewrite $A$ as a block matrix
\begin{equation}\notag
    A = \begin{bmatrix}
    A_{11} & A_{12}  \\
    A_{21} & A_{22}
    \end{bmatrix} \, ,
\end{equation}
where $A_{11} = P_X A|_X \equiv 0$ by our assumption, $A_{21} = P_Y A|_X$, etc. Notably, by our earlier claim that $A_{21} : X \to Z$ is invertible, and the fact that $A_{21} = A_{12}^*$ since $A$ is Hermitian and the projections are orthogonal, we have that
\begin{equation}
\label{eq:positivedef}
A_{12} A_{21} \geq \text{const.} \times {\rm Id}_{X \to X} 
\end{equation} is Hermitian positive definite. Using the same block decomposition as above, we define the Kawashima compensator
\begin{equation}\notag
    K = \begin{bmatrix}
    0 & A_{12} \\
    - A_{21} & 0
    \end{bmatrix} \, .
\end{equation}
The idea is that $K$ acts as a rotation between $X$ and $Y$ which interchanges the undamped region $X$ with the damped region $Y$. Recalling that $A_{11} \equiv 0$, the commutator $[K,A]$ is explicitly computed as
\begin{equation}\notag
    KA-AK = \begin{bmatrix}
    2A_{12}A_{21} & A_{12}A_{22} \\
    - A_{22}A_{21} & - 2 A_{21}A_{12} 
      \end{bmatrix} \, .
\end{equation}
Then
\begin{equation}\notag
\begin{aligned}
{\rm Re} \langle KA u, u \rangle &= \frac{1}{2} \langle [K,A] u, u \rangle \\
&\geq \langle A_{12} A_{21} u, u \rangle - |\langle A_{22} u, A_{21} u \rangle| - |\langle A_{21} A_{12} u, u \rangle| \, .
\end{aligned}
\end{equation}
Finally, we split the cross terms using Young's inequality with $\varepsilon$ and combine the above estimate with~\eqref{eq:positivedef} to obtain~\eqref{eq:Kawashimaconditionestimate}.
\smallskip

\noindent\emph{General case}. We define the subspaces
\begin{equation}\notag
F_k := \{ x \in \mathbb{C}^m : x, Ax, \hdots, A^kx \in X \} \, , \quad k \geq 0 \, .
\end{equation}
With the convention $F_{-1} = \mathbb{C}^n$, they satisfy the nesting property
\begin{equation}\notag
\cdots \subset F_{n} \subset \cdots \subset F_1 \subset F_0 = X \subset \mathbb{C}^n = F_{-1} \, .
\end{equation}
We note that for $k \geq 0$, we can characterize $F_k$ as the preimage $A|_X^{-1}(F_{k-1})$.  Furthermore, if $1 \leq k_0 \leq n$ is the least integer such that Condition~\eqref{eq:A} is satisfied for all $x\in X$, then $F_{k} = \{ 0 \}$ for all $k \geq k_0$.

For $k \geq 0$, we define $E_k$ to be the orthogonal complement of $F_k$ in $F_{k-1}$, so that
\begin{equation}
E_k \oplus F_k = F_{k-1} \, .  \label{billys:decomp}
\end{equation}
Using the convention that $E_0 = Y$, we have as a consequence the decomposition
\begin{equation}\notag
    E_n \oplus \cdots \oplus E_{1} = X \, , \qquad   E_n \oplus \cdots \oplus E_{1} \oplus E_0 = X \oplus Y = \mathbb{C}^m \, .
\end{equation}
In addition, $E_{k_0} = F_{k_0 - 1}$, and $E_k = F_k = \{ 0 \}$ for $k > k_0$.
We now crucially observe that whenever $k\leq k_0$ and $x \in E_k \subset F_{k-1}$, we have $x,Ax,\hdots,A^{k-1} x \in X$ but $A^k x \not\in X$. More precisely, while $Ax \in F_{k-1}$ for each $x\in  F_k$, we have that $Ax \notin F_{k-1}$ for each $x \in E_k$, although $Ax \in F_{k-2}$.  As a consequence of these observations, we have that
\begin{equation}\notag
{\rm ker} \, P_{E_{k-1}} A|_{E_k} = \{ 0 \} \, .
\end{equation}
Indeed if $x_1, x_2 \in E_{k}$ were such that $P_{E_{k-1}}Ax_1 = P_{E_{k-1}}Ax_1$, then $P_{E_{k-1}}A(x_1-x_2)=0 \in F_{k-1}$, contradicting the above claim that $Ax \notin F_{k-1}$ for $x\in E_k$. We have now extracted the required coercivity on the subspace $E_k$.  The special case corresponds to $k_0=1$ and $X=E_1$.

Recalling that $E_{k-1}\subseteq F_{k-1}^\perp$, the desired compensator at level $1 \leq k \leq k_0$ is
\begin{equation}\notag
K_k = P_{F^\perp_{k-1}} A P_{E_k} - P_{E_k} A P_{F_{k-1}^\perp} \, .
\end{equation}
For convenience we write $K_k$ as a block matrix in the subspaces $F_k, E_k, F_{k-1}^\perp$, which from \eqref{billys:decomp} form an orthogonal decomposition of $\mathbb{C}^m$.  Set $A_{11} = P|_{F_k} A P_{F_k}$, $A_{12} = P|_{F_k} A P_{E_{k}}$, etc., and recall that $P_{F_{k-1}^\perp} A P_{F_k} \equiv 0$.  Then $A$ and $K_k$ may be written as
\begin{equation}\notag
A = \begin{bmatrix}
     A_{11} & A_{12} & 0 \\
     A_{21} & A_{22} & A_{23} \\
     0 & A_{32} & A_{33}
    \end{bmatrix} \, , \qquad   K_k = \begin{bmatrix}
     0 & 0 & 0 \\
     0 & 0 & A_{23} \\
     0 & -A_{32} & 0
    \end{bmatrix} \, .
\end{equation}
Computing the commutator leads to
\begin{equation}\notag
K_kA-AK_k = \begin{bmatrix}
     0 & 0 & -A_{12} A_{23} \\
     0 & 2A_{23} A_{32} & A_{23} A_{33} - A_{22}A_{23} \\
     -A_{32} A_{21} & -A_{32} A_{22} + A_{33} A_{32} & -2 A_{32}A_{23}
    \end{bmatrix} \, .
\end{equation}
Note crucially that $A_{23} A_{32}$ is coercive on $E_k$, although we will still need to control the contributions from the rest of the commutator.  We compute $\langle [K_k, A]u , u \rangle$ and repeatedly use Cauchy-Schwarz with $\varepsilon$; specifically, we choose $\varepsilon_k \ll1$ and $C_k \geq \varepsilon_k^{-1}$ and use $\varepsilon_k\| P_{F_k} u \|^2$ to help bound $\langle A_{32} A_{21} u , u \rangle$ and $\langle A_{12} A_{23} u , u \rangle$, and $C_k\| P_{F_{k-1}^\perp} u \|^2$ to help bound $\langle (A_{23} A_{33} - A_{22} A_{23}) u , u \rangle$, $\langle (A_{32} A_{22} - A_{33} A_{32}) u , u \rangle$, and $\langle A_{32} A_{23} u , u \rangle$. As a consequence, we find that for $\delta_k, \varepsilon_k$ sufficiently small and $C_k$ sufficiently large, 
\begin{align}
\delta_k  \left\| P_{E_k} \right\|^2 \leq {\rm Re} \left \langle K_k A u, u \right \rangle + C_k \left\| P_{F_{k-1}^\perp} u \right\|^2 + \varepsilon_k \left\| P_{F_k} u \right\|^2 \, . \label{hypoo}
\end{align}

Finally, the full Kawashima compensator is defined by
\begin{equation}\notag
K = \sum_{i=1}^{k_0} \gamma_i K_i \, ,
\end{equation}
where $\gamma_k$, $C_k$, $\varepsilon_k$, and $\delta_k$ are chosen via backwards induction as follows.  Beginning at level $k_0$ and recalling that $F_{k_0}=\{0\}$ and $F_{k_0-1} = E_{k_0}$, we set $\gamma_{k_0}=1$ so that
\begin{align}
\delta_{k_0} \left\| P_{E_{k_0}} u \right\|^2 &\leq {\rm Re} \left \langle K_{k_0} A u, u \right \rangle + C_{k_0} \left\| P_{E_{k_0}^\perp} u \right\|^2 \notag \\
&= {\rm Re} \left \langle K_{k_0} A u, u \right \rangle + C_{k_0} \left\| P_{E_{k_0-1}} u \right\|^2 + C_{k_0} \left\| P_{(E_{k_0}\oplus E_{k_0-1})^\perp} u \right\|^2 \, . \label{hypo:yay}
\end{align}
Now we may choose $\gamma_{k_0-1}$ large enough and $\varepsilon_{k_0-1}$ small enough so that we can absorb the term with $P_{E_{k_0-1}} u$ above and the term with $P_{F_{k_0}-1}u = P_{E_{k_0}}u$ from \eqref{hypoo} with $k=k_0-1$, obtaining the inequality
\begin{align*}
\frac{\delta_{k_0}}{2} \left\| P_{E_{k_0}} u \right\|^2 + \delta_{k_0-1} \left\| P_{E_{k_0-1} }u \right\|^2 &\leq {\rm Re} \left \langle (K_{k_0} + \gamma_{k_0-1} K_{k_0-1}) A u , u \right \rangle \\
&\qquad + C_{k_0} \left\| P_{(E_{k_0}\oplus E_{k_0-1})^\perp} u \right\|^2 + C_{k_0-1} \left\| P_{F_{k_0-2}^\perp} u \right\|^2 \, . 
\end{align*}
Now recalling that $F_{k_0-2} = E_{k_0-1} \oplus F_{k_0-1} = E_{k_0-1} \oplus E_{k_0}$ and re-labeling $\tilde E_{k_0-1} = E_{k_0}\oplus E_{k_0-1}$, $\tilde \delta_{k_0-1} = \min(\delta_{k_0}/2, \delta_{k_0-1}) $, $\tilde K_{k_0-1} = K_{k_0} + \gamma_{k_0-1} K_{k_0-1}$, and $\tilde C_{k_0-1} = 2\max(C_{k_0}, C_{k_0-1})$, we have that
\begin{align*}
\tilde \delta_{k_0-1} \left\| P_{\tilde E_{ k_0-1} }u \right\|^2 &\leq {\rm Re} \left \langle \tilde K_{k_0-1} A u , u \right \rangle + \tilde C_{k_0-1} \left\| P_{\tilde E_{k_0-1}^\perp} u \right\|^2 \, . 
\end{align*}
But this inequality is identical in character to \eqref{hypo:yay}, and so it is clear that we may continue this process inductively. We then effectively reduce to the special case and obtain that
$$  \delta \left\| P_X u \right\|^2 \leq {\rm Re} \left \langle KA u , u \right \rangle + C \left\| P_Y u \right\|^2  \, ,$$
and thus concluding the proof of Lemma~\ref{lem:kawashimalem}.
\end{proof}

\begin{remark}
Alternatively, one can define the compensator at level $k$ as follows. Let $k>1$  and consider the operator
\begin{equation}\notag
\tilde K_k = P_Y A^k P_{E_k} - P_{E_k} A^k P_Y \, .
\end{equation}
If one were allowed to take the commutator of $\tilde K_k$ with $A^k$, then the construction of $K_k$ would be identical to the construction of $K_1$. However, we may define $K_k= \tilde K_k A^{k-1} + A \tilde K_k A^{k-2} + \dots + A^{k-2} \tilde K_k A + A^{k-1} \tilde K_k$, so that by direct computation, 
\begin{equation}\notag
\langle \tilde K_k , A^k \rangle = \langle K_k, A \rangle \, .
\end{equation}
This construction of $K_k$ demonstrates the direct exchange of coercivity between $E_k$ and $Y$, rather than the indirect exchange between $E_k$ and $Y$ via $E_{k-1}, \dots, E_1$.
\end{remark}

\begin{remark}
In \cite{BZ}, Beauchard and Zuazua gave a similar construction and moreover provide several equivalent characterizations of the Kawashima condition; for example, it is equivalent to the Kalman rank condition familiar from control theory.
\end{remark}

\subsection{Example: Oscillators with partial damping}

Consider the damped harmonic oscillator
\begin{equation}\notag
\ddot x + \kappa \dot x + x = 0 \, ,
\end{equation}
where the friction coefficient $\kappa = 1$ for convenience of exposition. We rewrite the second-order ODE as a first order system
\begin{equation}\notag
\begin{aligned}
\dot x &= y \\
\dot y &= - x - y
\end{aligned} \, .
\end{equation}
With $u = (x,y)$, the ODE can be rewritten in the form~\eqref{eq:equationtoestimate} with
\begin{equation}\notag
A = \begin{bmatrix} 0 & -i \\
i & 0 \end{bmatrix} \, , \quad L = \begin{bmatrix} 0 & 0 \\
0 & - 1 \end{bmatrix} \, .
\end{equation}
The Kawashima compensator is
\begin{equation}\notag
 K = \begin{bmatrix}
 0 & -i \\
 -i & 0
 \end{bmatrix} \, ,
 \end{equation} and the hypocoercivity functional is
\begin{equation}\notag
\Phi = \frac{1}{2} |x|^2 + \frac{1}{2} |y|^2 - \delta xy \, ,
\end{equation}
where $0 < \delta \leq 1/2$. One can analogously consider, for $n \geq 3$, oscillator chains of the type
$$
\dot x_1 = x_2 \, , \qquad \dot x_k = -x_{k-1} + x_k \, , \, 2 \leq k \leq n-1 \, , \qquad \dot x_n = - x_{n-1} - x_n \, .
$$

\subsection{Example: Linearized Landau equation}

For $\kappa\in[0,1]$, consider the linearized (and regularized, when $\kappa\in(0,1]$) Landau equation
\begin{equation}
\label{eq:linearizedLandau}
\p_t f + v \cdot \nabla_x f = L^\kappa f \, ,
\end{equation}
where $L^\kappa = \frac{2}{\sqrt{\mu}} Q_\kappa(\mu,\sqrt{\mu} f)$, $Q_\kappa$ is the collision operator defined in~\eqref{Qkappadef}, and $\mu = e^{-\frac{|v|^2}{2}}$. It is well known that $L$ and $L_{\rm R}$, and thus $L^\kappa$, with appropriate densely-defined domain $D(L) \subset L^2$ are closed, self-adjoint, and positive semi-definite.  We furthermore showed in subsection~\ref{sec:propertiesofLandGamma} that the kernel is precisely
\begin{equation}\notag
X := \ker L^\kappa = {\rm span}(\sqrt{\mu},v \sqrt{\mu}, |v|^2 \sqrt{\mu}) \, .
\end{equation}
Then taking the union of $X$ with its image under multiplication by $v_1,v_2,v_3$, we obtain the subspace consisting of thirteen moments,
\begin{equation}\notag
M_{13} = {\rm span}(\sqrt{\mu},v_i \sqrt{\mu}, v_i v_j \sqrt{\mu}, |v|^2 v_i \sqrt{\mu}) \, ,
\end{equation}
where $1 \leq i,j \leq 3$. The equation~\eqref{eq:linearizedLandau} does not preserve the subspace $M_{13}$, so we cannot reduce the problem to the ODE system~\eqref{eq:equationtoestimate}. Nonetheless, the Kawashima approach will be effective. It will be convenient to consider the Fourier transform in the $x$-direction and, without loss of generality, the Fourier variable $\vec{k} = k \vec{e}_1$. With this reduction, we will consider only nine moments,
\begin{equation}\notag
M_{9} = {\rm span}(\sqrt{\mu},v_i \sqrt{\mu}, v_1 v_i \sqrt{\mu}, |v|^2 \sqrt{\mu}, |v|^2 v_1 \sqrt{\mu}) \, ,
\end{equation}
where $1 \leq i \leq 3$. Let $Y$ be the orthogonal complement of $X$ in $M_{9}$. Our goal is to understand the (projected) multiplication operator\footnote{Note that the operator $v_1 P_{M_9}$ is well-defined with domain $M_9$ and range equal to the image of $M_9$ under multiplication by $v_1$, which can be understood as a subset of a Hilbert space with Gaussian weights, defined as in \eqref{eq:sigma:2}. We are therefore justified in considering $P_{M_9} v_1 P_{M_9}$ as an operator from $M_9$ to itself.  We also apologize for the abuse of notation; this operator $A$ is not the same as the $A$ defined in subsection~\ref{sec:propertiesofLandGamma}.}
\begin{equation}\notag
A = P_{M_{9}} v_1 P_{M_{9}} : M_9 \rightarrow M_9 \, .
\end{equation}
If $A$ satisfies the Kawashima condition, then we can construct an anti-symmetric operator $K:M_9\rightarrow M_9$ such that
\begin{equation}\notag
\| P_X f \|_{L^2}^2 \lesssim {\rm Re} \langle KA f, f \rangle + \| P_Y f \|_{L^2}^2 \, .
\end{equation}
Translated into the notation of sections~\ref{sec:outline}--\ref{sekk:linear}, this implies that
\begin{equation}\label{kawa:useful}
\| P g \|_{L^2}^2 \les \langle [K P_{\leq 9}, v_1 ] P_{\leq 9} g , P_{\leq 9} g \rangle + \| \langle v \rangle^{-100} P_{\leq 9} P^\perp g \|_{L^2}^2 \, ,
\end{equation}
where the $\langle v \rangle^{-100}$ is allowed due to the fact that all elements of $M_9$ have Gaussian decay and so can afford any polynomial weight. The computation (actually a more detailed version than is required in our setting) is already contained in the book of Glassey \cite[pg. 75]{Glassey}.  We note that by combining this estimate with the standard $L^2$ energy estimate, one can obtain time-asymptotic linear stability. 
\begin{lemma}
$A$ satisfies the Kawashima condition.
\end{lemma}
\begin{proof}
Since every element of $M_9$ contains a factor of $\mu^{\sfrac 12}$, we can simply focus on the prefactors, which are polynomials in $\{v_1, v_2, v_3 \}$ of degree at most $2$.  We claim that given any non-zero polynomial $p(v) = a + b_i v_i + c|v|^2 \in {\rm span}\{1, v, |v|^2\}$, either $v_1 p(v) \notin {\rm span} \{1, v, |v|^2\}$, or $v_1^2 p(v) \notin {\rm span} \{1,v, |v|^2\}$.  If $c\neq 0$, then $v_1 p(v)$ is degree 3 and we are done.  If $c = 0$ and there exists $b_{i_0}\neq 0$, then $v_1 b_i v_i$ is a homogeneous quadratic polynomial which is however not in the span of $|v|^2$, and we are done.  Finally, if $c=b_i=0$, then $a\neq 0$, and $a v_1^2$ is a homogeneous quadratic polynomial which is again not in the span of $|v|^2$.
\end{proof}

\begin{remark}
The micro-macro energy method introduced by Guo \cite{guo2006boltzmann} (see also \cite{duan2009stability,duan2011hypocoercivity,strain2012optimal}) is an explicit example of a Kawashima compensator for the linearized Landau equations. Indeed, if we define higher order moments
$\Theta[g]=(\Theta_{ij}[g])_{d\times d}$ and $\Lambda[g]=(\Lambda_1[g],\dots,\Lambda_d[g] )$ by
\begin{align*}
	&\Theta_{ij}[g]=\brak{(v_iv_j-1)\sqrt{\mu},g}_{L^2_v},\\
    &\Lambda_i[g]=\brak{(|v|^2-(d+2))v_i\sqrt{\mu},g}_{L^2_v} \, , 
\end{align*}
the micro-macro energy used in \cite{guo2006boltzmann} is written as the following, for each frequency in $x$: 
	\begin{align*}
\cM:=\,&\Re(\Lambda[(I-P_X)\widehat{f}_k]\cdot  (ik\widehat{E}_k) +b_1\Re\big(\big(\Theta[(I-P_X)\widehat{f}_k]+(2\widehat{E}_k I)\big):(ik\widehat{m}_k+(ik\widehat{m}_k)^T)\big) \notag \\
	&+b_2\Re \big(\widehat{m}_k\cdot (ik\widehat{\varrho}_k)\big),
\end{align*}
where $0<b_2\ll b_1\ll 1$ are small universal constants. 
However, note that this energy can be written in terms of an anti-symmetric operator $K:M_{13} \to M_{13}$: 
\begin{align*}
\cM = \brak{P_{\leq 13} f, K P_{\leq 13} f}. 
\end{align*}
\end{remark}

\begin{remark}
Heuristically, the continuity equation for the density $\varrho$ is $\p_t \varrho + \Div_x (\varrho u) = 0$, which is undamped. However, the density exchanges energy with the momentum, which has its own damping.  We verify the Kawashima condition for the linearized compressible Navier-Stokes equations in its own section, since it requires the notion of symmetrizability.
\end{remark}

\begin{remark}
The Kawashima compensator strategy also gives upper bounds on Taylor dispersion of a passive scalar in a shear flow. Upon applying the Fourier transform in the longitudinal variable, $X$ is the subspace corresponding to the mean in the cross-sectional variable.
\end{remark}

\section{Proof of Lemmas~\ref{lem:propertiesofL} and~\ref{lem:lemma10:new}}
\label{sec:proofofnonlinearestimate}

\begin{proof}[Proof of Lemma~\ref{lem:propertiesofL}, estimate~\eqref{new:co}]
We first prove~\eqref{new:co}.  We begin by recalling~\eqref{some:IDs} and writing that 
\begin{align*}
&-\left \langle \langle v \rangle^{2\theta} L^{\rm R} g, g \right \rangle \\
&\qquad= \left \langle \langle v \rangle^{2\theta} A_{\rm R} g, g \right \rangle + \left \langle \langle v \rangle^{2\theta} K_{\rm R} g, g \right \rangle \\
&\qquad= \left \langle \langle v \rangle^{2\theta} \partial_i [\sigma_{\rm R}^{ij} \partial_j  g], g \right \rangle - \left \langle \langle v \rangle^{2\theta} \sigma_{\rm R}^{ij} v_i v_j g , g \right \rangle + \left \langle \langle v \rangle^{2\theta} \partial_i \sigma_{\rm R}^i g , g \right \rangle + \left \langle \langle v \rangle^{2\theta} K_{\rm R} g, g \right \rangle  \\
&\qquad=  - | g |_{\sigma_{\rm R}, \theta}^2 - \left \langle \partial_i \langle v \rangle^{2\theta} \sigma^{ij}_{\rm R} \partial_j g , g \right\rangle + \left \langle \langle v \rangle^{2\theta} \partial_i \sigma_{\rm R}^i g , g \right \rangle + \left \langle \langle v \rangle^{2\theta} K_{\rm R} g, g \right \rangle \, .
\end{align*}
From~\cite[Lemma~5]{G02}, we have that for any $m>1$, there exists $0<C(m)<\infty$ such that
\begin{align*}
\left| \left \langle \langle v \rangle^{2\theta} K_{\rm R} g, g \right \rangle \right| \leq \frac{C}{m} | g |_{\sigma_{\rm R}, \theta}^2 + C(m) \int_{|v| \leq C(m)} \left| \langle v \rangle^{\theta} g \right|^2 \, \dee v \leq \frac{C}{m} + C(m,\theta) | g |_{\sigma_{\rm R}, \theta-\sfrac 12}
\end{align*}
Choosing $m$ large enough achieves a bound for $K_{\rm R}$ commensurate with that of~\eqref{new:co}, and so we focus on the remaining error terms from $A_{\rm R}$.    From~\eqref{waits} and~\cite[Lemma~3]{G02}, we have that
\begin{equation}
\left| \partial_i {\rm w}_{\rm R}^{2\theta} \sigma_{\rm R}^{ij} \right| + \left| {\rm w}_{\rm R}^{2\theta} \partial_i \sigma_i \right| \les {\rm w}_{\rm R}^{2\theta} \, ,\notag
\end{equation}
and so
\begin{align*}
\left|  - \left \langle \partial_i \langle v \rangle^{2\theta} \sigma^{ij}_{\rm R} \partial_j g , g \right\rangle + \left \langle \langle v \rangle^{2\theta} \partial_i \sigma_{\rm R}^i g , g \right \rangle \right| &= \left|  - \left \langle \partial_i {\rm w}_{\rm R}^{2\theta} \sigma^{ij}_{\rm R} \partial_j g , g \right\rangle + \left \langle {\rm w}_{\rm R}^{2\theta} \partial_i \sigma_{\rm R}^i g , g \right \rangle \right| \\
&\les \int_{\R^3} {\rm w}_{\rm R}^{2\theta} |g | \left( |g| + |\partial_ j g | \right) \, \dee v \\
&\les \int_{\R^3} {\rm w}_{\rm R}^{2(\theta-\sfrac 12)} \langle v \rangle |g | \left( |g| + |\partial_ j g | \right) \, \dee v \\
&\lesssim | g |_{\sigma_{\rm R}, \theta-\sfrac 12}^2 \, .
\end{align*}
Therefore the bound for these terms matches the required bound for~\eqref{new:co}, concluding the proof.  

In order to prove~\eqref{new:co:kappa}, we recall the definition of the $|\cdot|_{\sigma_\kappa, \ell}$ norm from~\eqref{H1kappa} and apply~\eqref{GS:Lemma:9b} and~\eqref{new:co}.  In order to handle the second term from the right-hand side of~\eqref{GS:Lemma:9b}, we bound the integrand, which is confined to a ball around the origin of finite radius, by a multiple of the $|\cdot|_{\sigma, \vartheta+\sfrac 12}$ norm.  
\end{proof}

\begin{proof}[Proof of Lemma~\ref{lem:lemma10:new}]
These nonlinear estimates are similar to \cite[Lemma~10, eqn.~(71)]{GS08}.  The first difference is that in \cite{GS08}, certain terms which contain derivatives of $g_1$ and $g_2$ are bounded using norms which do not include the weights $\sigma$ and $w$.  The second difference is that our estimate must treat the harder potential $\phi_{\rm R}$, which is not considered in~\cite{GS08}.  However, the proof from \cite{GS08} can be adjusted easily to prove the more general estimates we require, and thus the proof below will follow closely the structure of that from \cite{GS08}.  The idea is to use the appearance of $\mu_0^{\sfrac 12}$ and properly balanced applications of H\"{o}lder's inequality to deal with any decay issues.  We will emphasize these slight adjustments when they arise.  We shall also freely use notations and labels from the proof of \cite[Lemma~10]{GS08} so that the reader may easily compare the two.  Finally, since we must treat both $\Gamma$ and $\Gamma_{\rm R}$, we use will use the parameter $\gamma$ to distinguish between the two cases; specifically, $\gamma=-3$ refers to $Q$ (following the convention of~\cite{GS08}, for which the potential in the Landau operator is written as $|v|^{\gamma+2}$, and thus $\gamma=-3$ corresponds to the operator in~\eqref{symmed}), and $\gamma=-1$ refers to $Q_{\rm R}$.   We shall append $\gamma$ to $\sigma$ (in the ``sigma'' norms), $\Gamma$ (the nonlinear operator), $\phi^{ij}$ (the kernel in the collision operator), and $\ell$ (the index of the polynomial moment); more precisely, we set $\ell_\gamma$ to be $\ell$ when $\gamma=-3$, and $\ell_\gamma=-\ell$ when $\gamma=-1$, which matches the difference between ${\rm w}(\ell,0,0)={\rm w}_{\rm R}(-\ell)$.

Using the product rule and \cite[eqn.~(49)]{GS08}, which does not depend on the value of $\gamma$, we expand the inner product as
\begin{align}
\langle {\rm w}^2 \partial_\beta^\alpha \Gamma_\gamma[g_1,g_2] , \partial_\beta^\alpha g_3 \rangle &=  \sum_{\alpha_1, \beta_1} C(\alpha_1,\beta_1,\alpha,\beta)  \bigg{[}
-\left\langle {\rm w}^2 \{ \phi_\gamma^{ij} \ast \partial_{\beta_1} [\mu^{\sfrac 12} \partial^{\alpha_1} g_1] \} \partial_j \partial_{\beta-\beta_1}^{\alpha-\alpha_1} g_2 , \partial_i \partial_\beta^\alpha g_3 \right\rangle \notag\\
&\qquad\qquad\qquad\qquad - \frac 12 \left\langle {\rm w}^2 \{ \phi_\gamma^{ij} \ast \partial_{\beta_1} [ v_i \mu^{\sfrac 12} \partial^{\alpha_1} g_1] \} \partial_j \partial_{\beta-\beta_1}^{\alpha-\alpha_1} g_2 , \partial_\beta^\alpha g_3 \right\rangle \notag\\
&\qquad\qquad\qquad\qquad + \left\langle {\rm w}^2 \{ \phi_\gamma^{ij} \ast \partial_{\beta_1} [\mu^{\sfrac 12} \partial_j\partial^{\alpha_1} g_1] \}\partial_{\beta-\beta_1}^{\alpha-\alpha_1} g_2 , \partial_i \partial_\beta^\alpha g_3 \right\rangle \notag\\
&\qquad\qquad\qquad\qquad + \frac 12 \left\langle {\rm w}^2 \{ \phi_\gamma^{ij} \ast \partial_{\beta_1} [v_i \mu^{\sfrac 12} \partial_j \partial^{\alpha_1} g_1] \}  \partial_{\beta-\beta_1}^{\alpha-\alpha_1} g_2 , \partial_\beta^\alpha g_3 \right\rangle \notag\\
&\qquad\qquad\qquad\qquad - \left\langle \partial_i [{\rm w}^2] \{ \phi_\gamma^{ij} \ast \partial_{\beta_1} [ \mu^{\sfrac 12} \partial^{\alpha_1} g_1] \} \partial_j \partial_{\beta-\beta_1}^{\alpha-\alpha_1} g_2 , \partial_\beta^\alpha g_3 \right\rangle \notag\\
&\qquad\qquad\qquad\qquad + \left\langle \partial_i [{\rm w}^2] \{ \phi_\gamma^{ij} \ast \partial_{\beta_1} [ \mu^{\sfrac 12} \partial_j \partial^{\alpha_1} g_1] \} \partial_{\beta-\beta_1}^{\alpha-\alpha_1} g_2 , \partial_\beta^\alpha g_3 \right\rangle \notag
\bigg{]} \\
&=: (73) + \cdots + (78) \, .  \notag
\end{align}
These six terms are labeled (73)--(78) in \cite[Lemma~10]{GS08}, and in keeping with the order of proof, we estimate (78) first.  After defining ${\rm w}_1\in L^\infty$ by $\left(-\ell \brak{v}^{-2} - \sfrac{q\theta}{4} \brak{v}^{\theta-2} \right)$, we notice that $\partial_i [{\rm w}^2]={\rm w}^2(v){\rm w}_1(v) v_i$.  Now using that $\phi_\gamma^{ij}v_i\equiv\phi_\gamma^{ij}v_j \equiv 0$, we rewrite (78) as
\begin{equation}\label{e:78:rewritten}
\left\langle [{\rm w}^2 {\rm w}_1] \{ \phi_\gamma^{ij} \ast v_i \partial_{\beta_1} [ \mu_0^{\sfrac 12} \partial_j \partial^{\alpha_1} g_1] \} \partial_{\beta-\beta_1}^{\alpha-\alpha_1} g_2 , \partial_\beta^\alpha g_3 \right\rangle \, .
\end{equation}
We then bound
\begin{align}
\{ \phi_\gamma^{ij} \ast v_i \partial_{\beta_1} [\mu_0^{\sfrac 12} \partial_j \partial^{\alpha_1} g_1 ] \}(v) &\leq [|\phi_\gamma^{ij}|^2 \ast \mu_0^{\sfrac 18}]^{\sfrac 12}(v) \sum_{\bar\beta \leq \beta_1} |\mu_0^{\sfrac{1}{32}} \partial_j \partial_{\bar\beta}^{\alpha_1} g_1|_{\ell_\gamma} \notag\\
&\lesssim [1+|v|]^{\gamma+2} \sum_{\bar\beta\leq \beta_1} \left| \partial_{\bar\beta}^{\alpha_1} g_1 \right|_{\sigma_\gamma,{\ell_\gamma}} \, . \label{e:lem10:1}
\end{align}
These inequalities match \cite[eqn.~(79)]{GS08}; note that we have used the appearance of $\mu_0^{\sfrac{1}{32}}$ inside the ${\ell_\gamma}$ norm to control $ |\mu_0^{\sfrac{1}{32}} \partial_j \partial_{\bar\beta}^{\alpha_1} g_1|_{{\ell_\gamma}}$ with $\left| \partial_{\bar\beta}^{\alpha_1} g_1 \right|_{\sigma_\gamma,{\ell_\gamma}}$.  We then integrate in $v$ to obtain that \eqref{e:78:rewritten} is bounded by a constant multiplied by
\begin{align}
&\sum_{\bar\beta \leq \beta_1} \left| \partial_{\bar \beta}^{\alpha_1} g_1 \right|_{\sigma_\gamma,{\ell_\gamma}} \int_{\R^3} {\rm w}^2(\ell,q,\theta) [1+|v|]^{\gamma+2} \left| \partial_{\beta-\beta_1}^{\alpha-\alpha_1} g_2 \partial_\beta^\alpha g_3 \right| \, \dee v \notag\\
&\quad\leq \sum_{\bar\beta \leq \beta_1} \left| \partial_{\bar \beta}^{\alpha_1} g_1 \right|_{\sigma_\gamma,{\ell_\gamma}} \left( \int_{\R^3} {\rm w}^2(\ell,q,\theta) [1+|v|]^{\gamma+2} \left| \partial_{\beta-\beta_1}^{\alpha-\alpha_1} g_2 \right|^2 \, \dee v \right)^{\sfrac 12} \notag\\
&\qquad \qquad \times \left( \int_{\R^3} {\rm w}^2(\ell,q,\theta) [1+|v|]^{\gamma+2} \left| \partial_{\beta}^{\alpha} g_3 \right|^2 \, \dee v \right)^{\sfrac 12}\notag \\
&\quad\leq \sum_{\bar\beta \leq \beta_1} \left| \partial_{\bar \beta}^{\alpha_1} g_1 \right|_{\sigma_\gamma,{\ell_\gamma}} \left| \partial_{\beta-\beta_1}^{\alpha-\alpha_1} g_2 \right|_{\sigma_\gamma,{\ell_\gamma},q,\theta} \left| \partial_{\beta}^{\alpha} g_3 \right|_{\sigma_\gamma,{\ell_\gamma},q,\theta} \, ,  \label{eq:lem:10:holder}
\end{align}
where we have used Lemma~\ref{l:norm:control} to achieve the last inequality. The above string of inequalities differs from the estimates in \cite{GS08} only in how we split ${\rm w}^2 [1+|v|]^{\gamma+2}$ in the application of H\"{o}lder, allowing us to use $|\cdot|_{\sigma_\gamma,{\ell_\gamma},q,\theta}$ norms on $g_2$ and $g_3$. This concludes the proof for (78).

We now split the estimates for (73)-(77) into several cases, depending on the variables of integration $v',v$ in the convolution and inner product:
\begin{equation}\notag
\{ |v| \leq 1 \} \, , \qquad \{2|v'|\geq |v|, \, |v|\geq 1\}\, , \qquad \{ 2|v'| \leq |v|, \, |v|\geq 1  \} \, .
\end{equation}
In fact the last case will be further split into cases depending on the terms in a Taylor expansion of $\phi^{ij}_\gamma$.\\

\noindent\textbf{Case 1: } Terms (73)--(77) over $\{|v|\leq 1\}$.  In this region, we may ignore issues due to $v$-weights on $g_2$ and $g_3$.  Focusing on the convolutions with $g_1$, for (73) we have that for $|v|\leq 1$,
\begin{align}\label{eq:73:1}
\int_{\R^3} \phi_\gamma^{ij}(v-v') \partial_{\beta_1} [\mu_0^{\sfrac 12} \partial^{\alpha_1} g_1](v') \, dv' &= \int_{\R^3} \phi_\gamma^{ij}(v-v') \sum_{\bar\beta\leq \beta_1} C(\bar\beta,\beta_1) \partial_{\bar\beta} \mu_0^{\sfrac 12}(v') \partial_{\beta_1-\bar\beta}^{\alpha_1} g_1(v') \, .
\end{align}
Each term in the sum contains polynomial terms in $v'$ and a Gaussian weight in $v'$, which we may split between $\partial_{\beta_1-\bar\beta}^{\alpha_1}g_1$ and $\phi_\gamma^{ij}(v-v')$ to make the latter integrable in $L^2$ and bound the former using the $|\partial_{\beta_1-\bar\beta}^{\alpha_1}g_1|_{\sigma_\gamma,{\ell_\gamma}}$.  Applying Young's inequality for convolutions, we may then bound \eqref{eq:73:1} by a constant multiplied by
\begin{equation}\label{eq:case:1:1}
\sum_{\bar\beta\leq\beta_1}  |\partial_{\beta_1-\bar\beta}^{\alpha_1}g_1|_{\sigma_\gamma,{\ell_\gamma}} \, .
\end{equation}
Applying the same reasoning for (74)-(77), we may bound the $g_1$ contribution of all five terms in (73)-(77) by the quantity in \eqref{eq:case:1:1}.  Then using that the inner product in $v$ is restricted to $|v|\leq 1$, we bound the $g_2$ and $g_3$ contributions using the $|\cdot|_{\sigma_\gamma,{\ell_\gamma},q,\theta}$ norm.  This concludes the proof of the first case. \\

\noindent\textbf{Case 2: } Terms (73)--(77) over $\{2|v'| \geq |v|, \, |v|\geq 1 \}$.  Using that $|v|\leq 2 |v'|$, we have that for any $\delta>0$, there exist implicit constants such that
\begin{equation}\notag
|\partial_{\beta_1} \mu_0^{\sfrac 12}(v')| + |\partial_{\beta_1}\{ v'_j \mu_0^{\sfrac 12}(v')\} | \lesssim \mu_0^{\sfrac{\delta}{100}}(v') \mu_0^{\sfrac{1}{4}-\delta}(v) \, .
\end{equation}
Using the same type of arguments as in \eqref{e:lem10:1} and \eqref{eq:case:1:1}, we may bound the contribution to the integrand in $v$ in this region from the convolution with $g_1$ in (73) by
\begin{align*}
&\left| \int_{\R^3} \phi^{ij}(v-v') \sum_{\bar\beta\leq\beta_1} \partial_{\bar\beta} \mu_0^{\sfrac 12}(v') \partial_{\beta_1-\bar\beta}^{\alpha_1} g_1(v') \, dv' \right| \\
&\qquad \leq \int_{\R^3} |\phi^{ij}(v-v')| \mu_0^{\sfrac{\delta}{100}}(v') \mu_0^{\sfrac 14-\delta}(v) \sum_{\bar\beta\leq\beta_1} \left|\partial_{\beta_1-\bar\beta}^{\alpha_1} g_1(v') \right| \, dv' \\
&\qquad \leq \mu_0^{\sfrac{1}{4}-\delta}(v) (1+|v|)^{\gamma+2} \sum_{\bar\beta\leq\beta_1} \left| \partial_{\beta_1-\bar\beta}^{\alpha_1} g_1 \right|_{\sigma_\gamma,{\ell_\gamma}} \, .
\end{align*}
Then applying H\"{o}lder as in \eqref{eq:lem:10:holder} and using the leftover power of a Maxwellian in $v$ and Lemma~\ref{l:norm:control} to use the weighted norms on $g_2$ and $g_3$, we bound (73) in this region by
\begin{equation}\notag
\sum_{\bar\beta\leq\beta_1} \left| \partial_{\beta_1-\bar\beta}^{\alpha_1} g_1 \right|_{\sigma_\gamma,{\ell_\gamma}}  \left| \partial_{\beta_1-\bar\beta}^{\alpha-\alpha_1} g_2 \right|_{\sigma_\gamma,{\ell_\gamma},q,\theta}  \left| \partial_{\beta}^{\alpha} g_3 \right|_{\sigma_\gamma,{\ell_\gamma},q,\theta} \, .
\end{equation}
Note here that we are using that $q\in(0,1)$, and that the choice of $q$ depends on $\delta$.  Note also that if $\theta <2$, then $q$ does not have to be taken to be small. Similar arguments for (74)-(77) conclude the proof of the second case.  As in the previous estimates, the only differences with \cite{GS08} are the usage of the weighted norms and the extension to the $Q_{\rm R}$ operator.\\

\noindent\textbf{Case 3: } Terms (73)--(77) with $\{ 2|v'|\leq |v|, \, |v| \geq 1 \}$ and the first term in the expansion for $\phi_\gamma^{ij}$.  Before beginning the analysis of the remaining region, we Taylor expand $\phi_\gamma^{ij}$ as
\begin{align}\notag
\phi^{ij}_\gamma(v-v') = \phi^{ij}_\gamma(v) - \partial_k \phi_\gamma^{ij}(v) v_k' + \frac 12 \partial_{km} \phi_\gamma^{ij}(\tilde v) v_k' v_m' \, ,
\end{align}
where $\tilde v$ lies on the line segment connecting $v$ and $v-v'$.  We then have three more cases, where each corresponds to the region $\{ 2|v'|\leq |v|, \, |v| \geq 1 \}$ but paired with the first, second, and third terms from the above expansion, respectively.

We now first consider (73) with $\phi^{ij}(v-v')$ replaced with $\phi^{ij}(v)$.  In this case, there is no convolution, and so we use H\"{o}lder to bound
\begin{align*}
\left| \int_{\R^3} \partial_{\beta_1} \left(  \mu_0^{\sfrac 12} \partial^{\alpha_1} g_1 \right)(v') \, dv' \right| \lesssim \sum_{\bar\beta\leq \beta_1} \left| \partial_{\bar\beta}^{\alpha_1} g_1 \right|_{\sigma_\gamma,{\ell_\gamma}} \, .
\end{align*}
The above bound holds due to the presence of the Gaussian $\mu_0^{\sfrac 12}$, which allows us to ignore decay issues and use the $|\cdot|_{\sigma_\gamma,\ell}$ norm.  Then using again that $\phi_\gamma^{ij}v_i \equiv \phi_\gamma^{ij} v_j \equiv 0$, we may bound the entire inner product by 
\begin{align*}
& \sum_{\bar\beta\leq \beta_1} \left| \partial_{\bar\beta}^{\alpha_1} g_1 \right|_{\sigma_\gamma,{\ell_\gamma}} \left|\left\langle {\rm w}^2 \phi_\gamma^{ij} (I-P_v) \partial_j \partial_{\beta-\beta_1}^{\alpha-\alpha_1} g_2 , (I-P_v) \partial_i \partial_\beta^\alpha g_3 \right\rangle \right| \\
&\qquad \leq \sum_{\bar\beta\leq \beta_1} \left| \partial_{\bar\beta}^{\alpha_1} g_1 \right|_{\sigma_\gamma,{\ell_\gamma}} \left| \partial_{\beta-\beta_1}^{\alpha-\alpha_1} g_2 \right|_{\sigma_\gamma,{\ell_\gamma},q,\theta} \left| \partial_{\beta}^{\alpha} g_3 \right|_{\sigma_\gamma,{\ell_\gamma},q,\theta} \, ,
\end{align*}
where we have used the fact that $\phi_\gamma^{ij}$ is $O(|v|^{\gamma+2})$ and Lemma~\ref{l:norm:control}, which uses precisely this decay rate for the projection off of $v$.  This finishes the analysis of (73) with the first term from the Taylor expansion of $\phi_\gamma^{ij}$.  In order to bound (74)--(76), we use similar arguments; indeed the Gaussian $\mu_0^{\sfrac 12}(v')$ allows us to bound all the contributions from $g_1$ using $\sum_{\bar\beta\leq \beta_1} \left| \partial_{\beta_1}^{\alpha_1} g_1 \right|_{\sigma_\gamma,{\ell_\gamma}}$, and any $\partial_i$'s or $\partial_j$'s landing on $g_2$ and $g_3$ only see the projection off of $v$, for which the $|\cdot|_{\sigma_\gamma,{\ell_\gamma},q,\theta}$ norm provides a good control.  For (77) we have that $\partial_i [{\rm w}^2] \phi_\gamma^{ij} \equiv 0$, concluding the analysis of this case.\\

\noindent\textbf{Case 4: } Terms (73)--(77) with $\{ 2|v'|\leq |v|, \, |v| \geq 1 \}$ and the second term in the expansion for $\phi_\gamma^{ij}$.  Recalling that $\phi_\gamma^{ij}v_iv_j \equiv \phi_\gamma^{kj} v_j \equiv 0$, we have that
\begin{align}\notag
\partial_k (\phi_\gamma^{ij}(v) v_i v_j) \equiv 0 \implies \partial_k \phi_\gamma^{ij} v_i v_j \equiv 0 \, .
\end{align}
We also note that $\partial_k \phi_\gamma^{ij}$ is $O([1+|v|]^{\gamma+1})$ as $v\rightarrow \infty$.  Applying the first of these observations to (73), we must then bound
\begin{align*}
& \left\langle {\rm w}^2 \partial_k\phi_\gamma^{ij} \left\{ \int v_k' \partial_{\beta_1} [\mu_0^{\sfrac 12} \partial^{\alpha_1} g_1](v') \,dv' \right\} \partial_j \partial_{\beta-\beta_1}^{\alpha-\alpha_1} g_2 , \partial_i \partial_\beta^\alpha g_3 \right\rangle \\
&\qquad = \left\langle {\rm w}^2 \partial_k\phi_\gamma^{ij} \left\{ \int v_k' \partial_{\beta_1} [\mu_0^{\sfrac 12} \partial^{\alpha_1} g_1](v') \,dv' \right\} (I-P_v) [\partial_j \partial_{\beta-\beta_1}^{\alpha-\alpha_1} g_2] , P_v [\partial_i \partial_\beta^\alpha g_3] \right\rangle \\
&\qquad\qquad + \left\langle {\rm w}^2 \partial_k\phi_\gamma^{ij} \left\{ \int v_k' \partial_{\beta_1} [\mu_0^{\sfrac 12} \partial^{\alpha_1} g_1](v') \,dv' \right\} (I-P_v) [\partial_j \partial_{\beta-\beta_1}^{\alpha-\alpha_1} g_2] , (I-P_v) [\partial_i \partial_\beta^\alpha g_3] \right\rangle \\
&\qquad\qquad + \left\langle {\rm w}^2 \partial_k\phi_\gamma^{ij} \left\{ \int v_k' \partial_{\beta_1} [\mu_0^{\sfrac 12} \partial^{\alpha_1} g_1](v') \,dv' \right\} P_v [\partial_j \partial_{\beta-\beta_1}^{\alpha-\alpha_1} g_2] , (I-P_v) [\partial_i \partial_\beta^\alpha g_3] \right\rangle \, ,
\end{align*}
where the worst term with $P_v$ on both $g_2$ and $g_3$ has vanished.  We note that all of the above identities are identical to those from \cite{GS08}.  Plugging in the decay of $\partial_k \phi_\gamma^{ij}$ to the first of these terms, we may bound
\begin{align*}
&\left| \left\langle {\rm w}^2 \partial_k\phi_\gamma^{ij} \left\{ \int v_k' \partial_{\beta_1} [\mu_0^{\sfrac 12} \partial^{\alpha_1} g_1](v') \,dv' \right\} (I-P_v) [\partial_j \partial_{\beta-\beta_1}^{\alpha-\alpha_1} g_2] , P_v [\partial_i \partial_\beta^\alpha g_3] \right\rangle \right| \\
&\qquad \leq \sum_{\bar\beta\leq \beta_1} \left| \partial_{\bar\beta}^{\alpha_1} g_1 \right|_{\sigma_\gamma,{\ell_\gamma}} \left( \int_{\R^3} {\rm w}^2 [1+|v|]^{\gamma+2} \left|(I-P_v) \partial_j \partial_{\beta-\beta_1}^{\alpha-\alpha_1} g_2 \right|^2 \, dv \right)^{\sfrac 12} \notag\\
&\qquad \qquad \qquad \times \left( \int_{\R^3} {\rm  w}^2 [1+|v|]^{\gamma} \left| P_v \partial_i \partial_{\beta}^{\alpha} g_3 \right|^2 \, dv \right)^{\sfrac 12} \\
&\qquad \leq \sum_{\bar\beta\leq \beta_1} \left| \partial_{\bar\beta}^{\alpha_1} g_1 \right|_{\sigma_\gamma,{\ell_\gamma}}  \left| \partial_{\beta-\beta_1}^{\alpha-\alpha_1} g_2 \right|_{\sigma_\gamma,{\ell_\gamma},q,\theta} \left| \partial_{\beta}^{\alpha} g_3 \right|_{\sigma_\gamma,{\ell_\gamma},q,\theta} \, ,
\end{align*}
where we have crucially used the different decay rates for $P_v$ and $I-P_v$ in the application of Lemma~\ref{l:norm:control} to obtain a bound slightly better than that of \cite{GS08}.  It is clear that similar arguments apply to the second and third terms from the decomposition of (73).  For (74)--(77), we note that either $\partial_i$ or $\partial_j$ has always landed off of $g_3$ or $g_2$, respectively.  Therefore we can always split $(1+|v|)^{\gamma+1}$ into $(1+|v|)^{\sfrac\gamma2}$ to pair with the extra derivative, and $(1+|v|)^{\sfrac{\gamma+2}{2}}$ for the remainder and obtain a bound identical to the one above, concluding the analysis of this case.\\

\noindent\textbf{Case 5: } Terms (73)--(77) with $\{ 2|v'|\leq |v|, \, |v| \geq 1 \}$ and the third term in the expansion for $\phi_\gamma^{ij}$.  To analyze the final term, we first note that due to the region of integration, we have that as in \cite[eqn.~(83)]{GS08},
\begin{align}\notag
\frac 12 |v| \leq |v| - |v'| \leq |\tilde v| \leq |v'| + |v| \leq \frac 32 |v| \quad \implies \quad |\partial_{km}\phi_\gamma^{ij}(\tilde v) | \lesssim (1+|\tilde v|)^{\gamma} \lesssim (1+|v|)^{\gamma} \, .
\end{align}
Plugging this into (73), we have 
\begin{align*}
&\left| \left\langle {\rm w}^2 \left\{ \int \partial_{km}\phi_\gamma^{ij}(\tilde v) v_m' v_k' \partial_{\beta_1} [\mu_0^{\sfrac 12} \partial^{\alpha_1} g_1](v') \,dv' \right\} \partial_j \partial_{\beta-\beta_1}^{\alpha-\alpha_1} g_2 , \partial_i \partial_\beta^\alpha g_3 \right\rangle \right| \notag\\
&\qquad \lesssim \left\langle {\rm w}^2 [1+|v|]^{\gamma} \left| \left\{ \int v_m' v_k' \partial_{\beta_1} [\mu_0^{\sfrac 12} \partial^{\alpha_1} g_1](v') \,dv' \right\} \right| \left|\partial_j \partial_{\beta-\beta_1}^{\alpha-\alpha_1} g_2 \right| , \left| \partial_i \partial_\beta^\alpha g_3 \right| \right\rangle \\
&\qquad \lesssim \sum_{\bar\beta\leq \beta_1} \left| \partial_{\bar\beta}^{\alpha_1} g_1 \right|_{\sigma_\gamma,{\ell_\gamma}} \left( \int_{\R^3} {\rm w}^2 [1+|v|]^{\gamma} \left| \partial_j \partial_{\beta-\beta_1}^{\alpha-\alpha_1} g_2 \right|^2 \, \dee v \right)^{\sfrac 12} \notag\\
&\qquad \qquad \qquad \times \left( \int_{\R^3} {\rm w}^2 [1+|v|]^{\gamma} \left| \partial_i \partial_{\beta}^{\alpha} g_3 \right|^2 \, \dee v \right)^{\sfrac 12} \\
&\qquad \leq \sum_{\bar\beta\leq \beta_1} \left| \partial_{\bar\beta}^{\alpha_1} g_1 \right|_{\sigma_\gamma,{\ell_\gamma}} \left| \partial_{\beta-\beta_1}^{\alpha-\alpha_1} g_2 \right|_{\sigma_\gamma,{\ell_\gamma},q,\theta} \left| \partial_{\beta}^{\alpha} g_3 \right|_{\sigma_\gamma,{\ell_\gamma},q,\theta} \, .
\end{align*}
Arguing similarly for the terms (74)--(77) concludes estimates for the final case.
\end{proof}

\printindex
\addcontentsline{toc}{section}{Index of notations and conventions}


\bibliographystyle{abbrv}
\bibliography{biblio}

\end{document}